\documentclass{amsart}
\usepackage{setspace}
\usepackage[fontsize=11pt]{scrextend}
\usepackage{amsmath,graphicx,amsfonts,amssymb,amsthm,hyperref,amscd,esint,bbm,verbatim,mathabx,abstract}
\usepackage[margin=1in]{geometry}
\usepackage{mathtools}
\mathtoolsset{showonlyrefs}
\newtheorem{remark}{Remark}[section]
\newtheorem{prelimdefinition}{Preliminary Definition}[section]
\newtheorem{definition}{Definition}[section]
\newtheorem{theorem}{Theorem}[section]
\newtheorem{proposition}{Proposition}[section]
\newtheorem{corollary}{Corollary}[theorem]
\newtheorem{lemma}[theorem]{Lemma}
\title{Global, Non-Scattering Solutions to the Energy Critical Yang-Mills Problem}
\date{}
\author{Mohandas Pillai}
\numberwithin{equation}{section}
\begin{document}
\maketitle
\begin{abstract} We consider the Yang-Mills problem on $\mathbb{R}^{1+4}$ with gauge group $SO(4)$. In an appropriate equivariant reduction, this Yang-Mills problem reduces to a single scalar semilinear wave equation. This semilinear equation admits a one-parameter family of solitons, each of which is a re-scaling of a fixed solution. In this work, we construct a class of solutions, each of which consists of a soliton whose length scale is asymptotically constant, coupled to large radiation, plus corrections which slowly decay to zero in the energy norm. Our class of solutions includes ones for which the radiation component is only ``logarithmically'' better than energy class. As such, the solutions are not constructed by apriori assuming the length scale to be constant. Instead, we use an approach similar to a previous work of the author regarding wave maps. In the setup of this work, the soliton length scale asymptoting to a constant is a necessary condition for the radiation profile to have finite energy. An interesting point of our construction is that, for each radiation profile, there exist one-parameter families of solutions consisting of the radiation profile coupled to a soliton, which has any asymptotic value of the length scale.
\end{abstract}
\tableofcontents
\section{Introduction}
We consider the Yang-Mills equation in $4+1$ dimensions, with gauge group $SO(4)$. This equation can be described by a gauge field, $A$, which is a $\text{Lie}(SO(4))$-valued one-form on $\mathbb{R}^{4+1}$. We write $A=A_{\mu} dx^{\mu}$, where, for each $\mu$, $A_{\mu}$ is a $\text{Lie}(SO(4))$-valued function, defined on $\mathbb{R}^{4+1}$. Defining $F$, a $\text{Lie}(SO(4))$-valued two-form on $\mathbb{R}^{4+1}$ by
\begin{equation}\label{Fintroeqn}F=\frac{1}{2}F_{\mu\nu} dx^{\mu} \wedge dx^{\nu}, \quad F_{\mu\nu} = \partial_{\mu}A_{\nu}-\partial_{\nu}A_{\mu}+[A_{\mu},A_{\nu}]\end{equation}
\noeqref{Fintroeqn}
the Yang-Mills equation can be written as 
\begin{equation}\label{fulleqns}-\partial_{t}F_{0\nu}-[A_{0},F_{0\nu}] + \sum_{\mu=1}^{4}\left(\partial_{\mu} F_{\mu \nu} + [A_{\mu},F_{\mu\nu}]\right)=0, \quad \text{ for } \nu=0,1,2,3,4\end{equation}
where $0$ on the right-hand is the zero in $\text{Lie}(SO(4))$. The Yang-Mills equation has the conserved energy 
\begin{equation}E_{\text{Yang-Mills}} = -\frac{1}{48 \pi^{2}} \int_{\mathbb{R}^{4}} \text{Tr}\left(F_{\mu\nu}(t,x) F_{\mu\nu}(t,x)\right) dx\end{equation}
(where repeated indices are summed over). The equation is invariant under the scaling symmetry $A_{\mu}(t,x) \rightarrow \lambda A_{\mu}(\lambda t, \lambda x)$. The components of $F$ transform under this symmetry as $F_{\mu\nu}(t,x) \rightarrow \lambda^{2} F_{\mu\nu}(\lambda t,\lambda x)$, which means that the energy $E_{\text{Yang-Mills}}$ is invariant under the scaling symmetry, because the equation is considered in 4 spatial dimensions. The Yang-Mills equation is also invariant under gauge transformations, which are transformations of $A$ of the form $A_{\mu} \rightarrow g A_{\mu} g^{-1} - \partial_{\mu}g g^{-1}$ where $g: \mathbb{R}^{1+4} \rightarrow SO(4)$. \\
\\
Small energy global well posedness for the $(4+1)$ dimensional Yang-Mills problem was established by Krieger and Tataru, \cite{kt}. In addition, the works of Tataru and Oh, \cite{tatoh0}, \cite{tatoh1}, \cite{tatoh2}, \cite{tatoh3}, \cite{tatoh4}, established a threshold theorem and dichotomy theorem for this problem, with any compact, non-abelian gauge group.

We make the equivariant ansatz (see also \cite{rr}, \cite{kstym})
\begin{equation}\label{equivgaugefield}A_{\mu}^{i,j}(t,x) = \left(\delta^{i}_{\mu}x^{j}-\delta^{j}_{\mu}x^{i}\right)\left(\frac{u(t,|x|)-1}{|x|^{2}}\right), \quad 0 \leq \mu \leq 4, \quad 1 \leq i,j \leq 4.\end{equation}
Note that $A_{0}(t,x)=0$ and $\sum_{k=1}^{4} x_{k} A_{k}(t,x)=0$. In particular, a general gauge field $A$ can not be brought into the form in \eqref{equivgaugefield} by performing a gauge transformation. 
With the equivariant ansatz, the Yang-Mills equation, \eqref{fulleqns}, reduces to
\begin{equation}\label{ym}-\partial_{tt}u(t,r)+\partial_{rr}u(t,r)+\frac{1}{r}\partial_{r}u(t,r)+\frac{2u(t,r)(1-u(t,r)^{2})}{r^{2}}=0, \quad (t,r) \in \mathbb{R} \times (0,\infty).\end{equation}
The energy $E_{\text{Yang-Mills}}$ reduces to the following energy, which is conserved by \eqref{ym}
\begin{equation}\label{eym} E_{YM}(u,\partial_{t}u) = \frac{1}{2} \int_{0}^{\infty} \left(\left(\partial_{t}u\right)^{2}+\left(\partial_{r}u\right)^{2}+\frac{(1-u^{2})^{2}}{r^{2}}\right) r dr.\end{equation}
\noeqref{eym}
The equation \eqref{ym} admits a soliton solution, namely $u(t,r) = Q_{1}(r) = \frac{1-r^{2}}{1+r^{2}}$. In addition, for any $\lambda>0$, $Q_{\lambda}(r) = Q_{1}(r\lambda)$ is a solution, and also a minimizer of $E_{YM}(u,\partial_{t}u)$ within a class of functions with appropriate boundary conditions. We will study perturbations of $Q_{\frac{1}{\lambda(t)}}$, and it will turn out that the ``main" component of such perturbations will involve solutions to the following linear wave equation
\begin{equation}\label{4poteqn}-\partial_{tt}u+\partial_{rr}u+\frac{1}{r}\partial_{r}u-\frac{4}{r^{2}} u =0.\end{equation}
The formally conserved energy for this equation is
\begin{equation}\label{linwaveeqnenergyintro}E(u,\partial_{t}u) = \frac{1}{2} \int_{0}^{\infty} \left((\partial_{t}u)^{2}+(\partial_{r}u)^{2}+\frac{4 u^{2}}{r^{2}}\right) r dr\end{equation}
\noeqref{linwaveeqnenergyintro}
Our goal in this work is to construct global, non-scattering solutions to \eqref{ym}. More precisely, our goal is to construct solutions to \eqref{ym} which can be decomposed as follows.
\begin{equation}\label{goalsoln} u(t,r) = Q_{\frac{1}{\lambda(t)}}(r) + v_{1}(t,r) + v_{e}(t,r)\end{equation} 
where $v_{1}$ represents radiation of the soliton, and solves \eqref{4poteqn}. On the other hand, the function $v_{e}$ is a correction such that $E(v_{e},\partial_{t}v_{e}) \rightarrow 0, \quad t \rightarrow \infty$.

For the 1-equivariant, critical wave maps equation with $\mathbb{S}^{2}$ target, global non-scattering solutions with topological degree 0 or 1, and energy in an appropriate range were classified in \cite{ckls1}, \cite{ckls}. As remarked in Appendix A of \cite{ckls1}, and remark 4 of \cite{ckls}, the methods used in this classification result also apply to \eqref{ym}. Our procedure to construct solutions of the form \eqref{goalsoln} to \eqref{ym} is outlined in a self-contained manner in section 3, but we remark that it is overall similar to that used by the author in \cite{wm}. In particular, our solutions are described as in \eqref{goalsoln}, and our procedure to construct these solutions will be to find a precise relation between the radiation and the dynamics of $\lambda(t)$. The main difference in the technical steps of this work arise from the fact that the initial data for the radiation considered in this work belong to a class of functions which can have much worse singularities at low frequencies, compared to the data for the radiation considered in \cite{wm}. This gives rise to new technical problems not encountered in \cite{wm}. To describe our main result, we define the following set of functions.
\begin{definition}\label{fdef}
For $b>\frac{2}{3}$, let $F_{b}$ denote the set of functions $f$ such that there exists $M>50$, and $C_{f,k}>0$, such that
\begin{equation}\label{fineqindef}f \in C^{\infty}([M,\infty)), \quad |f^{(k)}(t)| \leq \frac{C_{f,k}}{t^{k} \log^{b}(t)}, \text{ for }t \geq M \text{ and }k \geq 0.\end{equation}
\end{definition}
\begin{remark} We define the set $F_{b}$ only for $b>\frac{2}{3}$ because this is the range of the parameter $b$ for which our methods work. Restricting to $b>\frac{1}{2}$ is natural from the point of view of the energy of the radiation component of our solutions, see the discussion surrounding \eqref{v11withoutstr}. The further restriction $b>\frac{2}{3}$ is required for sufficient accuracy of our ansatz.\end{remark}
The class of radiation profiles of our solutions can be labeled by $F_{b}$ in the following way.  For $f \in F_{b}$, we have
\begin{equation}\label{v11eqnintro1}\widehat{v_{1,1}}(\xi) = \frac{8}{3 \pi \xi^{2}} \int_{0}^{\infty} \frac{(\psi \cdot f)'(t)}{t} \sin(t\xi) dt\end{equation}
\noeqref{v11eqnintro1}
where $\psi$ is an unimportant cutoff function defined in \eqref{psidefinition}, $\widehat{v_{1,1}}$ denotes the Hankel transform of order 2 of $v_{1,1}$, and the radiation profile $v_{1}$ is given by
\begin{equation}\label{v1eqnintro1}\begin{cases} -\partial_{tt}v_{1}+\partial_{rr}v_{1}+\frac{1}{r}\partial_{r}v_{1}-\frac{4 v_{1}}{r^{2}} =0\\
v_{1}(0)=0\\
\partial_{t}v_{1}(0) = v_{1,1}\end{cases}.\end{equation}
\noeqref{v1eqnintro1}
In order to describe the leading order behavior of $\lambda(t)$, we introduce the following family of functions. 
\begin{definition}For $b>\frac{2}{3}$, let $\Lambda_{b}$ denote the set of functions $\lambda_{0}$ for which there exists $T_{\lambda_{0}}>50$ such that $\lambda_{0} \in C^{\infty}([T_{\lambda_{0}},\infty))$, and the following two conditions hold: Firstly, there exists $f \in F_{b}$ such that 
\begin{equation}\label{lambda0str}\frac{\lambda_{0}''(t)}{\lambda_{0}(t)} = \frac{f'(t)}{t}, \quad t \geq T_{\lambda_{0}}.\end{equation} 
Secondly, 
\begin{equation}\label{lambda0req}\lambda_{0}(t)>0, \quad \frac{|\lambda_{0}'(t)|}{\lambda_{0}(t)} \leq \frac{C}{t \log^{b}(t)}, \quad t \geq T_{\lambda_{0}}.\end{equation}
\end{definition}
We remark that, given any $f \in F_{b}$, there exists $T_{\lambda_{0}}>50$, and a one-parameter family of $\lambda_{0}\in \Lambda_{b}$ satisfying \eqref{lambda0str} and \eqref{lambda0req}, see Remark \ref{lambda0fromfremark}.

The condition  \eqref{lambda0str}, rather than simply the symbol type estimates, is imposed so as to guarantee that the radiation profile of our solution has finite energy. Once we pick $f \in F_{b}$ and $\lambda_{0} \in \Lambda_{b}$ satisfying \eqref{lambda0str},  we define $\widehat{v_{1,1}}$ by \eqref{v11eqnintro1}. (Without the structural condition, \eqref{lambda0str}, we would have had
\begin{equation}\label{v11withoutstr} \widehat{v_{1,1}}(\xi) = \frac{8}{3\pi \xi^{2}} \int_{0}^{\infty} h(t) \sin(t\xi) dt\end{equation}
where $h \in C^{\infty}([0,\infty))$ is some extension of the function $\frac{\lambda_{0}''(t)}{\lambda_{0}(t)}$ ). The details of how \eqref{v11eqnintro1} leads to the radiation profile of our solution having finite energy are given just below \eqref{v11int}.

Note that the definition of $\Lambda_{b}$ implies that any $\lambda_{0} \in \Lambda_{b}$ satisfies $\lambda_{0}(t) \rightarrow \lambda_{1} >0$ as $t \rightarrow \infty$, despite the fact that some $\lambda_{0}\in \Lambda_{b}$ (for $b \leq 1$) satisfy 
\begin{equation}\label{absintegralintro1}\int_{t}^{\infty} \int_{x}^{\infty} \frac{|\lambda_{0}''(s)|}{\lambda_{0}(s)} ds dx = \infty.\end{equation}
\noeqref{absintegralintro1}
To see this, we write
\begin{equation}\label{feqnintro1}f(t) = - \lim_{M \rightarrow \infty} \int_{t}^{M} \frac{s \lambda_{0}''(s)}{\lambda_{0}(s)} ds = -\lim_{M \rightarrow \infty} \int_{t}^{M} \frac{\frac{d}{ds}\left(\lambda_{0}'(s) s -\lambda_{0}(s)\right)}{\lambda_{0}(s)} ds, \quad t > T_{\lambda_{0}}.\end{equation}
\noeqref{feqnintro1}
Integrating by parts and using \eqref{lambda0req}, and the fact that $b>\frac{2}{3}$, we see that
\begin{equation}\label{lambda0constintro1}\lim_{M \rightarrow \infty}\log(\lambda_{0}(M)) < \infty.\end{equation}
\noeqref{lambda0constintro1} 
Despite the fact that any $\lambda_{0} \in \Lambda_{b}$ is asymptotically constant, we do not directly use this fact in any quantitative estimates of the terms in our ansatz, and their associated error terms. For estimating the radiation profile, we use $\frac{\lambda_{0}''(t)}{\lambda_{0}(t)} = \frac{f'(t)}{t}$, but for the entirety of the rest of the argument, we only use the symbol-type estimates on $\frac{\lambda_{0}'(t)}{\lambda_{0}(t)}$ (which can be satisfied by non-asymptotically constant $\lambda_{0}$).\\
\\
Our main result is the following theorem.
\begin{theorem}\label{mainthm} For all $b>\frac{2}{3}$ and $f \in F_{b}$, let $\lambda_{0}$ be any element of $\Lambda_{b}$ satisfying \eqref{lambda0str} (and \eqref{lambda0req}). Then, there exists $T_{0}=T_{0}(\lambda_{0})$ and a finite energy solution, $u$, to \eqref{ym}, with the following properties.
$$u(t,r) = Q_{\frac{1}{\lambda(t)}}(r) + v_{1}(t,r) + v_{e}(t,r)$$
where
$\lambda(t) \in C^{4}([T_{0},\infty))$
$$-\partial_{tt}v_{1}+\partial_{rr}v_{1}+\frac{1}{r}\partial_{r}v_{1}-\frac{4}{r^{2}} v_{1}=0, \quad E(v_{1},\partial_{t}v_{1})<\infty$$
$$E(v_{e},\partial_{t}v_{e}) < \frac{C}{\log^{4b-2}(t)}, \quad t \geq T_{0}$$
and $\lambda(t) = \lambda_{0}(t)\left(1+e(t)\right)$ where, for some $\epsilon_{0}>0$, we have
$$|e^{(k)}(t)| \leq \frac{C}{t^{k} \log^{\epsilon_{0}}(t)}, \quad 4 \geq k \geq 0.$$
\end{theorem}
\begin{remark}\label{lambda0fromfremark}Given any $f \in F_{b}$, there exists $T_{\lambda_{0}}>50$, and a one-parameter family of $\lambda_{0}\in \Lambda_{b}$ satisfying \eqref{lambda0str} and \eqref{lambda0req}. This can be seen as follows. Given $f \in F_{b}$, we can first find $\omega$ satisfying 
\begin{equation}\label{omegaeqnintro1}\omega'(t) + \omega(t)^{2} = \frac{f'(t)}{t}, \quad |\omega(t)| \leq \frac{C}{t \log^{b}(t)}, \quad t \geq N\end{equation}
\noeqref{omegaeqnintro1}
(where $N > 50$ is sufficiently large) with a fixed point argument. By inspection of this equation, $\omega \in C^{\infty}([N,\infty))$. Then, we can define $T_{\lambda_{0}} = N+1$, and let $\lambda_{0}$ be given by
\begin{equation}\label{lambda0intermsoffeqnintro}\lambda_{0}(t) = c \exp\left(\int_{N+1}^{t} \omega(s) ds\right), \quad t \geq N+1, \quad \text{ any }c>0.\end{equation}
\noeqref{lambda0intermsoffeqnintro}
Then, we have \eqref{lambda0req} and \eqref{lambda0str}.
\end{remark}
\begin{remark}
An interesting feature of our solutions is that the radiation profile depends only on $f$ (as per \eqref{v11eqnintro1}). As we just showed, there is a one-parameter family of $\lambda_{0} \in \Lambda_{b}$, corresponding to a given $f \in F_{b}$. In particular, our family of solutions includes functions of the form $Q_{\frac{1}{\lambda(t)}}(r) + v_{1}(t,r)+o(1)$, for a one-parameter family of possible asymptotic values of $\lambda(t)$, and the \textbf{same} $v_{1}$. 
\end{remark}
\begin{remark} For $\frac{2}{3} < \beta < \alpha < 1$, we can let  
\begin{equation}\label{examplefintro}f(t) = \frac{\sin(\log^{\alpha}(t))}{\log^{\beta}(t)}, \quad t \geq 50.\end{equation}
\noeqref{examplefintro}
Then, $f \in F_{b}$ for any $\frac{2}{3} < b < \beta$. We then carry out the procedure discussed in Remark \ref{lambda0fromfremark}, to recover a $\lambda_{0} \in \Lambda_{b}$ satisfying \eqref{lambda0str} and \eqref{lambda0req}. In this case, we have
\begin{equation}\label{examplelambda0eqnintro}\frac{\lambda_{0}'(t)}{\lambda_{0}(t)} \sim \frac{-\alpha \log^{\alpha-1}(t) \cos(\log^{\alpha}(t))}{t \log^{\beta}(t)}.\end{equation}
\noeqref{examplelambda0eqnintro}
Since $1+\beta-\alpha < 1$, this gives rise to $\lambda_{0} \in \Lambda_{b}$ with $$\int_{t}^{\infty} \frac{|\lambda_{0}'(s)|}{\lambda_{0}(s)}ds = \infty.$$
Nevertheless, as pointed out in the computation \eqref{feqnintro1}, $\lambda_{0}$ is asymptotically constant.
\end{remark}
\begin{remark} By choosing 
$$f(t) = \frac{1}{\log^{b}(t)}, \quad t \geq 50, \quad b>\frac{2}{3}$$
we can show (see \eqref{v11def}) that 
$$\widehat{v_{1,1}}(\xi) = \frac{8}{3 \pi \xi \log^{b}(\frac{1}{\xi})} + O\left(\frac{1}{\xi \log^{b+1}(\frac{1}{\xi})}\right), \quad \xi \rightarrow 0$$
which shows that we can have radiation whose initial velocity has quite a large singularity at low frequencies. In fact, the condition for the radiation to have finite energy in our setting is $\widehat{v_{1,1}}(\xi) \in L^{2}((0,\infty), \xi d\xi)$. The initial velocity therefore satisfies this condition only ``logarithmically''. 
\end{remark}
\begin{remark} A more precise set of estimates on $e$ is given in Proposition \ref{choosinglambdaprop}, where $\delta$ is defined in \eqref{deltadef} and $\delta_{j}$ is defined in \eqref{delta2def}, \eqref{delta4def}, and \eqref{delta5def}, for $j=2,4$, and $5$, respectively.
\end{remark}

Now, we review previous related works. As mentioned before, the work \cite{kt} established small energy global well-posedness for the $(4+1)$ dimensional Yang-Mills problem. Regarding the large energy global well posedness of the Yang-Mills equation in $4+1$ dimensions, the works of Tataru and Oh, \cite{tatoh0}, \cite{tatoh1}, \cite{tatoh2}, \cite{tatoh3}, \cite{tatoh4}, established a threshold theorem and dichotomy theorem for the $(4+1)$-dimensional Yang-Mills equation, associated to any compact, non-abelian gauge group. \\
\\
As previously mentioned, our procedure in this paper is similar to that used in the previous work of the author, \cite{wm}. That work constructed infinite time blow-up solutions to the energy critical, 1-equivariant wave maps problem with $\mathbb{S}^{2}$ target, with a symbol class of possible asymptotic behaviors of the soliton length scale, $\lambda(t)$. The main difference in this work is that the class of initial data of the radiation considered here includes functions that are much more singular at low frequencies than that considered in \cite{wm}. This leads to extra technical difficulties related to the slow decay of the radiation, $v_{1}$. In addition, the constraint that the radiation has finite energy implies, in our setting, that $\lambda_{0}(t)$ must be asymptotically constant for large $t$, in contrast with \cite{wm}.\\
\\
The work \cite{jlr} constructs finite time blow-up solutions to the same wave maps problem just mentioned, by also understanding the relation between a prescribed radiation field and the dynamics of the soliton length scale, in the context of finite time blow-up. (The problem of finite time blow-up for this wave maps equation has also been studied in the preceding works \cite{rs}, \cite{rr}, \cite{kst}, \cite{gk}, \cite{km}). Another key reference for our work is the paper of Krieger, Schlag and Tataru \cite{kstym}, which constructs finite time blow-up solutions to the same equation considered here, \eqref{ym}. In our argument, we use the ``distorted Fourier transform'' of \cite{kstym}, as well as related technical information, most importantly, the transference identity of that paper. For completeness we also mention that there is an analog of \cite{kstym} for the energy critical, focusing semilinear wave equation in $\mathbb{R}^{1+3}$, namely \cite{kstslw}.\\
\\
Regarding other constructions of non-scattering solutions to \eqref{ym}, the work \cite{j} (which also applies to other energy critical wave equations) constructed two bubble solutions to \eqref{ym}. The work \cite{dk} constructed infinite time blow-up and infinite time relaxation solutions to the focusing, energy critical semilinear wave equation on $\mathbb{R}^{1+3}$. Finally, the work \cite{bkt} constructed global solutions to the energy critical wave maps problem with $\mathbb{S}^{2}$ target associated to a codimension two manifold of data. Also, given that our result can be interpreted as a preliminary step towards some form of stability of the soliton under perturbations, we mention the work \cite{ksmanifold}, which constructed a stable manifold for the quintic, focusing semilinear wave equation in $\mathbb{R}^{1+3}$, centered around the Aubin-Talentini soliton solution.
\subsection{Acknowledgments}
The author thanks his adviser, Daniel Tataru, for suggesting the problem, and many useful discussions. This material is based upon work partially supported by the National Science Foundation under Grant No.  DMS-1800294.
\section{Notation}
We will make use of the following notation. It will be slightly more convenient for our purposes to modify the usual definition of $\langle x\rangle$ as follows $\langle x \rangle = \sqrt{50^{2}+x^{2}}.$ If $f:D \subset \mathbb{R}^{k} \rightarrow \mathbb{R}$ is a differentiable function of $k$ arguments, we will use the notation $\partial_{i}f(x)$, for $1 \leq i \leq k$, to denote the $i$th partial derivative of $f$, evaluated at the point $x \in D$.
Occasionally, we will use $dA$ to denote two-dimensional Lebesgue measure.\\
\\
For $f:(0,\infty)\rightarrow \mathbb{R}$, we define the following norm (see also \cite{bkt})
$$||f||^{2}_{\dot{H}^{1}_{e}} = ||\partial_{r}f||^{2}_{L^{2}((0,\infty), rdr)} + ||\frac{f}{r}||^{2}_{L^{2}((0,\infty),r dr)}.$$
The elliptic part of the linear wave equation obtained by linearizing \eqref{ym} around $Q_{1}$ is
$$-\partial_{rr}u-\frac{1}{r}\partial_{r}u-\frac{2}{r^{2}}\left(1-3Q_{1}(r)^{2}\right) u.$$
As noted in \cite{rr}, this operator can be expressed as $L^{*}L$, for
\begin{equation}\label{Ldef}L(f) = -f'(r) +2 \left(\frac{1-r^{2}}{1+r^{2}}\right) \frac{f(r)}{r}\end{equation}
which has the formal adjoint on $L^{2}(r dr)$ given by
\begin{equation}\label{Lstardef}L^{*}(f) = f'(r) +2\left(\frac{1-r^{2}}{1+r^{2}}\right) \frac{f(r)}{r} + \frac{f(r)}{r}.\end{equation} 
We denote by $\phi_{0}$, the following eigenfunction of $L^{*}L$, with eigenvalue $0$.
\begin{equation}\label{phi0notation}\phi_{0}(r) = \frac{r^{2}}{(1+r^{2})^{2}}\end{equation}
Note that this definition of $\phi_{0}$ is different from that of \cite{kstym}, see \eqref{phi0tilde}.
\begin{definition}
Following Theorem 4.3 of \cite{kstym}, we define $\phi(r,\xi)$ to be the solution to
\begin{equation}\label{phiefuncdef}L^{*}L\left(\frac{\phi(r,\xi)}{\sqrt{r}}\right) = \xi \frac{\phi(r,\xi)}{\sqrt{r}}, \quad r,\xi >0\end{equation}
which satisfies $\phi(r,\xi) \sim r^{5/2} \text{ as } r \rightarrow 0.$ 
\end{definition}
We will also use the following representation formula for $\phi(r,\xi)$ from proposition 4.5 of \cite{kstym}.
\begin{equation}\label{phiseries}\phi(r,\xi) = \widetilde{\phi_{0}}(r) + \frac{1}{r^{3/2}}\sum_{j=1}^{\infty}(r^{2}\xi)^{j}\widetilde{\phi_{j}}(r^{2}), \quad |\widetilde{\phi_{j}}(r^{2})| \leq \frac{C^{j}}{j!} \frac{r^{4}}{\langle r^{2} \rangle}, \quad j \geq 1,\end{equation}
where we denote by $\widetilde{\phi_{0}}$ what is denoted by $\phi_{0}$ in that paper. So, we have
\begin{equation}\label{phi0tilde}\widetilde{\phi_{0}}(r) = \frac{r^{5/2}}{(1+r^{2})^{2}} = \sqrt{r} \phi_{0}(r)\end{equation}
Similarly, following Theorem 4.3 of \cite{kstym}, we will make use of the functions $\psi^{+}(r,\xi)$, which satisfy
$$L^{*}L(\frac{\psi^{+}(\cdot,\xi)}{\sqrt{\cdot}})(r) = \xi \frac{\psi^{+}(r,\xi)}{\sqrt{r} }, \quad r,\xi >0, \quad \psi^{+}(r,\xi) \sim \xi^{-1/4} e^{i \sqrt{\xi} r}, \quad r \rightarrow \infty.$$
From Lemma 4.7 of \cite{kstym}, we write
\begin{equation}\label{philgrasymp}\phi(r,\xi) = 2 \text{Re}\left(a(\xi) \psi^{+}(r,\xi)\right)\end{equation}
where\footnote{Following \cite{kstym}, when $a,b>0$, we write $a \approx b$ if there exists $C>0$ such that $C^{-1} a \leq b \leq Ca$, and $a \ll b$ if $a \leq \epsilon b$ for some small $\epsilon>0$.}
\begin{equation}\label{asymbol}|a(\xi)| \approx \begin{cases} 1, \quad \xi \ll 1\\
\frac{1}{\xi}, \quad \xi \gtrapprox 1\end{cases}, \quad |\left(\xi\partial_{\xi}\right)^{k}a(\xi)| \leq C_{k} |a(\xi)|, \quad\xi >0.\end{equation}
We will also make use of Proposition 4.6 of \cite{kstym}, which provides the formulae
\begin{equation}\label{psiplusdef}\psi^{+}(r,\xi) = \xi^{-1/4}e^{i r \sqrt{\xi}} \sigma(r \sqrt{\xi},r), \quad r^{2} \xi \gtrapprox 1\end{equation}
where, for all $a, b \geq 0$, $q>1$, and all sufficiently large integers $j_{0}$,
$$\sup_{r >0}\Bigl|\left(r \partial_{r}\right)^{a} \left(q \partial_{q}\right)^{b}\left(\sigma(q,r)-\sum_{j=0}^{j_{0}}q^{-j}\psi_{j}^{+}(r)\right)\Bigr| \leq C_{a,b,j_{0}} q^{-j_{0}-1}, \quad \sup_{r >0}\Bigl|\left(r\partial_{r}\right)^{k}\psi_{j}^{+}(r)\Bigr| <\infty.$$
Another important notion from \cite{kstym} which we use is the distorted Fourier transform. We use $\mathcal{F}$ to denote the distorted Fourier transform, rather than $\widehat{\cdot}$ used in \cite{kstym}. 
\begin{definition} From Theorem 4.3 of \cite{kstym}, the distorted Fourier transform is given by
\begin{equation}\label{distortedfouriernotation}\mathcal{F}(f)(\xi) = \lim_{M \rightarrow \infty} \int_{0}^{M} \phi(r,\xi) f(r) dr, \quad \xi \geq 0.\end{equation}\end{definition}
We follow the convention of \cite{kstym}, which regards the distorted Fourier transform of $f \in L^{2}((0,\infty))$ as a two-component vector $\begin{bmatrix} a\\
g(\cdot)\end{bmatrix}$, where $a=\langle f, \widetilde{\phi_{0}}\rangle_{L^{2}(dr)}$. The inversion formula is then
$$f(r) = a \frac{\widetilde{\phi_{0}}(r)}{||\widetilde{\phi_{0}}(r)||^{2}_{L^{2}((0,\infty),dr)}}+\lim_{M \rightarrow \infty} \int_{0}^{M} g(\xi) \phi(r,\xi) \rho(\xi) d\xi$$
where we use the function $\rho$, from Lemma 4.7 of \cite{kstym}, which satisfies
\begin{equation}\label{rhoest}\rho(\xi) = \frac{1}{\pi |a(\xi)|^{2}}, \quad \rho(\xi) \approx \begin{cases} 1, \quad \xi \ll 1\\
 \xi^{2}, \quad \xi \gtrapprox 1\end{cases}.\end{equation}
The distorted Fourier transform, $\mathcal{F}$, is an isometry from $L^{2}((0,\infty),dr)$ to $\mathbb{R}\oplus L^{2}((0,\infty),\rho(\xi)d\xi)$.\\
Following \cite{kstym}, we also let $\mathcal{K}$ denote the transference operator, defined by
\begin{equation}\label{transference}\mathcal{F}(R \partial_{R}u)(\xi) = \begin{bmatrix} 0\\
-2 \xi \partial_{\xi} \mathcal{F}(u)_{1}\end{bmatrix} + \mathcal{K}(\mathcal{F}(u))(\xi), \quad \text{where } \mathcal{F}(u)=\begin{bmatrix} \mathcal{F}(u)_{0}\\
\mathcal{F}(u)_{1}\end{bmatrix}.\end{equation}
We also recall the definition of the norm $L_{\rho}^{2,\alpha}$ from \cite{kstym}, namely
\begin{equation}\label{l2rhoalphanorm}||f||_{L_{\rho}^{2,\alpha}}^{2} = |f(0)|^{2} + \int_{0}^{\infty} |f(\xi)|^{2}\langle \xi \rangle^{2\alpha}\rho(\xi) d\xi.\end{equation}
We will make use of the Hankel transform of order 2, $\widehat{\cdot}$, and its inverse, $\widecheck{\cdot}$, which are defined by
$$\widehat{f}(\xi) = \int_{0}^{\infty} f(r) J_{2}(r \xi) r dr, \quad \widecheck{f}(r) = \int_{0}^{\infty} f(\xi) J_{2}(r \xi) \xi d\xi. $$
We let $K_{1}$ denote the modified Bessel function of the second kind and order 1. $K_{1}$ satisfies the asymptotics (see \cite{gr})
\begin{equation}\label{k1notation} K_{1}(x) \sim \begin{cases}\frac{1}{x} + O(x \left(1+|\log(x)|\right)),\quad \text{ as } x \rightarrow 0\\
\sqrt{\frac{\pi}{2 x}} e^{-x}, \quad \text{ as } x \rightarrow \infty\end{cases}.\end{equation}

\section{Summary of the proof}
The argument can be split broadly into two steps: constructing an ansatz, and then completing this ansatz to an exact solution. These two steps are explained in more detail below. \\
\\
\textbf{1. Strategy for constructing the ansatz} For $b > \frac{2}{3}$, we start by taking some $f \in F_{b}$, and $\lambda_{0} \in \Lambda_{b}$ satisfying \eqref{lambda0str}. Then, we  let $\lambda(t)$ (which will be chosen later) be any function of the form 
$$\lambda(t) = \lambda_{0}(t) \left(1+e(t)\right)$$
where $e$ is small in a $C^{2}$ sense, and consider first, $u_{1}(t,r) = Q_{\frac{1}{\lambda(t)}}(r) + v_{1}(t,r)$, where $v_{1}$ solves
$$\begin{cases} -\partial_{tt}v_{1}+\partial_{rr}v_{1}+\frac{1}{r}\partial_{r}v_{1}-\frac{4 v_{1}}{r^{2}} =0\\
v_{1}(0)=0\\
\partial_{t}v_{1}(0) = v_{1,1}\end{cases}$$ 
and $v_{1,1}$ is yet to be chosen. We note that the equation \eqref{ym}, linearized around $Q_{1}(r)$ has a zero eigenfunction, which we denote $\phi_{0}(r)$ (recall \eqref{phi0notation}). Therefore, we will choose $v_{1,1}$ (in Section \ref{cauchydatav1}), depending on $\lambda_{0}$, so that the principal part of the error term of our final ansatz (which is more complicated than $u_{1}$) is orthogonal to $\phi_{0}\left(\frac{r}{\lambda(t)}\right)$ in $L^{2}((0,\infty),r dr)$, for a choice of $\lambda(t)$ which is equal to $\lambda_{0}$ to leading order. The usefulness of such an orthogonality condition can be seen by noting that the inner product between the final correction to our ansatz and $\phi_{0}(\frac{\cdot}{\lambda(t)})$ has roughly two powers of $t$ less decay as $t \rightarrow \infty$ than the inner product between the error term of our ansatz and $\phi_{0}(\frac{\cdot}{\lambda(t)})$, see \eqref{Tmapintro}.\\
\\
In order to further describe how we choose the initial data of $v_{1}$, let us note that substituting $u=u_{1}+u_{2}$ into \eqref{ym} gives the equation 
\begin{equation}\begin{split}&-\partial_{tt}u_{2}+\partial_{rr}u_{2}+\frac{1}{r}\partial_{r}u_{2}+\frac{2}{r^{2}} u_{2}(t,r)\left(1-3Q_{\frac{1}{\lambda(t)}}^{2}(r)\right) \\
&= \text{err}_{1}(t,r):= \partial_{tt}Q_{\frac{1}{\lambda(t)}} - \frac{6}{r^{2}}\left(1-Q_{\frac{1}{\lambda(t)}}^{2}\right) v_{1} + \frac{6}{r^{2}}Q_{\frac{1}{\lambda(t)}}\left(v_{1}+u_{2}\right)^{2}  + \frac{2}{r^{2}}\left(v_{1}+u_{2}\right)^{3}.\end{split}\end{equation}
The function $u_{1}$ is not our final ansatz. However, computing the inner product of its error term with $\phi_{0}\left(\frac{\cdot}{\lambda(t)}\right)$ still allows us to see how to choose $v_{1,1}$. The $u_{2}$-independent terms on the right-hand side of the above equation which contribute to leading order to $\langle \text{err}_{1}(t,R\lambda(t)),\phi_{0}(R)\rangle_{L^{2}(R dR)}$  are the soliton error term $\partial_{tt}Q_{\frac{1}{\lambda(t)}}$ and the linear error term associated to $v_{1}$, which is $- \frac{6}{r^{2}}\left(1-Q_{\frac{1}{\lambda(t)}}^{2}\right) v_{1}$. We compute the relevant inner products of these terms in the following way. (In the main body of the paper, these computations are done just after Proposition \ref{choosinglambdaprop}). We have 
$$v_{1}(t,r) = \int_{0}^{\infty} d\xi J_{2}(r\xi) \sin(t\xi) \widehat{v_{1,1}}(\xi)$$
and
$$\int_{0}^{\infty} \frac{24 R^{3}}{(1+R^{2})^{4} \lambda(t)^{2}} J_{2}(R\lambda(t) \xi) dR = \frac{\xi^{3} \lambda(t)}{2} K_{1}(\xi \lambda(t))$$
(which follows from combining identities from \cite{gr}). Here, $K_{1}$ denotes the modified Bessel function of the second kind and order 1 (recall \eqref{k1notation}). We therefore get
\begin{equation}\label{v11ip}\langle - \frac{6}{r^{2}}\left(1-Q_{\frac{1}{\lambda(t)}}^{2}\right) v_{1}\Bigr|_{r = R\lambda(t)},\phi_{0}(R)\rangle_{L^{2}(R dR)} = -\int_{0}^{\infty}  \sin(t\xi) \widehat{v_{1,1}}(\xi) \frac{\xi^{3} \lambda(t)}{2} K_{1}(\xi \lambda(t))d\xi.\end{equation}
We also have 
$$\langle \partial_{tt}Q_{\frac{1}{\lambda(t)}}\Bigr|_{r=R\lambda(t)},\phi_{0}(R)\rangle_{L^{2}(R dR)} = \frac{2 \lambda''(t)}{3\lambda(t)}.$$
The modulation equation that we use to choose $\lambda(t)$ (which is done in Section \ref{choosinglambdasection}) is not simply to set the sum of these two inner products equal to zero, since we will need more to add more corrections to our ansatz. (The modulation equation is given in \eqref{lambdaeqn}). However, the leading order contribution to the modulation equation is indeed given by the sum of these two terms. Therefore, we choose the initial data of $v_{1}$ so as to make the sum of these two inner products vanish to leading order when $\lambda(t) = \lambda_{0}(t)$. We recall that $\frac{\lambda_{0}''(t)}{\lambda_{0}(t)} = \frac{f'(t)}{t}, \quad t \geq T_{\lambda_{0}}$, and extend this to a function $\frac{(\psi f)'(t)}{t}$ defined on $[0,\infty)$ with a cutoff $\psi$ (whose properties are not so important, as long as $\psi(x) =1$ for $x$ large enough, and which is defined in \eqref{psidefinition}). Then, we let
$$ \widehat{v_{1,1}}(\xi) = \frac{8}{3 \pi \xi^{2}} \int_{0}^{\infty} \frac{(\psi \cdot f)'(t)}{t} \sin(t\xi) dt.$$
This leads to $\widehat{v_{1,1}}(\xi)$ potentially having a singularity of size $\frac{1}{\xi \log^{b}(\frac{1}{\xi})}$ \footnote{Strictly speaking, the small $\xi$ singularity of $\widehat{v_{1,1}}(\xi)$ depends on $f$. What is meant here is that $f$ ranges through a class of functions which includes ones which would produce the aforementioned singularities of $\widehat{v_{1,1}}(\xi)$. This comment also applies to any discussion in this section, of the large $r$ behavior of $v_{1}(t,r)$.}. In addition to causing technical difficulties associated to very slow decay of the radiation $v_{1}$, this also constrains $\lambda_{0}(t)$ to asymptote to a constant for large $t$, in order that the radiation has finite energy, see Remark \ref{necessarycondremark}.

We have $\psi(x) =1$ for $x \geq 2T_{\lambda_{0}}$, which implies, by the inversion of the sine transform, that
$$\int_{0}^{\infty}  \sin(t\xi) \widehat{v_{1,1}}(\xi) \frac{\xi^{2}}{2} d\xi = \frac{2}{3} \frac{\lambda_{0}''(t)}{\lambda_{0}(t)}, \quad t \geq 2 T_{\lambda_{0}}.$$
This will be sufficient to allow $\lambda_{0}(t)$ to be a leading order solution to the eventual modulation equation for $\lambda$. In particular, in our setting, we can replace $K_{1}(\xi \lambda(t))$ appearing in \eqref{v11ip} by $\frac{1}{\xi \lambda(t)}$ (recall the asymptotics of $K_{1}$ given in \eqref{k1notation})  to get the leading order behavior of the integral as a function of $t$.\\
\\
As described earlier, the singularity of $\widehat{v_{1,1}}(\xi)$ for small $\xi$ causes technical difficulties, in part due to the fact that $v_{1}$ has a very slow ($\frac{1}{\log^{b}(r)}$) decay for large $r$. Recall also that $Q_{1}(r) = \frac{1-r^{2}}{1+r^{2}}$. In particular, $Q_{1}$ does not decay at infinity. So, the quadratic and cubic nonlinear error terms involving $v_{1}$ are estimated by $\frac{C}{r^{2} \log^{2b}(r)}$ for $r \geq \frac{t}{2}$, and are thus very far from having sufficient decay for large $r$, in order that the rest of our argument can be carried out. In contrast, the principal error, $F_{4}$, of our final ansatz satisfies
\begin{equation}\label{f4estforcomp}F_{4}(t,r) =0, \quad r \geq \frac{t}{2}, \quad |F_{4}(t,r)| \leq \frac{C \lambda(t)^{4} \log^{2}(t)}{t^{6} \log^{b}(t)}, \quad \frac{t}{4} \leq r \leq \frac{t}{2}.\end{equation}
This allows certain integrals in the $r$ variable involving $F_{4}(t,r)$ (as in \eqref{fourierest}) to have sufficient decay in $t$.\\
\\
Our first correction, introduced in Section \ref{v2section}, to improve the quadratic and cubic nonlinear error terms involving $v_{1}$ is denoted by $v_{2}$, which solves the following equation with zero Cauchy data at infinity.
$$-\partial_{tt}v_{2}+\partial_{rr}v_{2}+\frac{1}{r}\partial_{r}v_{2}-\frac{4}{r^{2}}v_{2}=\frac{6 Q_{\frac{1}{\lambda(t)}}(r)}{r^{2}}v_{1}^{2}+\frac{2}{r^{2}} v_{1}^{3}$$
On the other hand, we only estimate $v_{2}(t,r)$ by $\frac{C}{\log^{2b}(t)}$ for $r \geq \frac{t}{2}.$ Hence, there is only a logarithmic improvement in the $v_{1}$ and $v_{2}$ interactions, relative to the $v_{1}$ self interactions, and this is still much worse than the decay of $F_{4}$ in \eqref{f4estforcomp}. Therefore, we really need to add to $Q_{\frac{1}{\lambda(t)}}(r)+v_{1}(t,r)$, a correction, say $q(t,r)$, which satisfies
$$-\partial_{t}^{2} q(t,r) + \partial_{r}^{2}q(t,r) + \frac{1}{r}\partial_{r}q(t,r)-\frac{4}{r^{2}}q(t,r) = \frac{6 Q_{\frac{1}{\lambda(t)}}(r)}{r^{2}} \left(v_{1}(t,r)+q(t,r)\right)^{2}+\frac{2}{r^{2}} (v_{1}(t,r)+q(t,r))^{3}.$$
We choose to construct $q$ by summing a series of corrections, $v_{k}$, starting with the function $v_{2}$ we just discussed. In particular, in Section \ref{vjsection}, we successively add corrections, $v_{j}$, which solve
$$ -\partial_{tt}v_{j}+\partial_{rr}v_{j}+\frac{1}{r}\partial_{r}v_{j} - \frac{4}{r^{2}} v_{j} = RHS_{j}(t,r)$$
where
\begin{equation}\begin{split} RHS_{j}(t,r) &=\frac{6 Q_{\frac{1}{\lambda(t)}}(r)}{r^{2}} \left(\left(\sum_{k=1}^{j-1}v_{k}\right)^{2} - \left(\sum_{k=1}^{j-2}v_{k}\right)^{2}\right) + \frac{2}{r^{2}}\left(\left(\sum_{k=1}^{j-1}v_{k}\right)^{3} - \left(\sum_{k=1}^{j-2}v_{k}\right)^{3}\right)\end{split}\end{equation}
and prove that the series 
$$v_{s}:=\sum_{j=3}^{\infty} v_{j}$$
(as well as the series resulting from applying any first or second order derivative termwise) converges absolutely and uniformly on the set $\{(t,r)|t \geq T_{1}, r \geq 0\}$, where $T_{1}$ is some sufficiently large number. Moreover, we get that
\begin{equation}\nonumber \begin{split} &-\partial_{tt}v_{s}+\partial_{rr}v_{s}+\frac{1}{r}\partial_{r}v_{s}-\frac{4}{r^{2}}v_{s} \\
&= \frac{6 Q_{\frac{1}{\lambda(t)}}}{r^{2}} \left(2 v_{1} \left(v_{2}+v_{s}\right) + \left(v_{2}+v_{s}\right)^{2}\right) + \frac{2}{r^{2}} \left(3 v_{1}\left(v_{2}+v_{s}\right)^{2} + 3 v_{1}^{2} \left(v_{2}+v_{s}\right) + \left(v_{2}+v_{s}\right)^{3}\right).\end{split}\end{equation}
Let $v_{c}=v_{1}+v_{2}+v_{s}$. Then, the error term of the refined ansatz $u_{3}(t,r) = Q_{\frac{1}{\lambda(t)}}(r) + v_{c}(t,r)$ is estimated by \begin{equation}\label{prewcerrorterm}|\partial_{t}^{2} Q_{\frac{1}{\lambda(t)}}(r)| + |\frac{6v_{c}}{r^{2}}\left(1-Q_{\frac{1}{\lambda(t)}}^{2}(r)\right)| \leq \frac{C \lambda(t)^{2}}{r^{2} t^{2} \log^{b}(t)}, \quad r \geq \frac{t}{2}.\end{equation}
which has roughly two powers of $t$ improved decay in the region $r \geq \frac{t}{2}$ compared with the nonlinear error terms involving $v_{1}$. It turns out that even this major improvement over the ansatz $u_{1}$ still does not have an error term with sufficient decay for large $r$ and $t$. (Compare the inequality above with \eqref{f4estforcomp}). In order to rectify this, we introduce a length scale $g(t)$, and eliminate the portion of the $u_{3}$ error localized to the region $r \geq \frac{g(t)}{2}$. On one hand, we can not have $g(t)$ too small, since doing so would change the leading order behavior of the inner product between the error term of the final ansatz and $\phi_{0}(\frac{\cdot}{\lambda(t)})$, which is not desired. On the other hand, we can not have $g(t)$ too large, since the whole purpose of the next set of corrections is to improve the large $r$ behavior of the error term of $u_{3}$. We therefore find an intermediate scale $g(t)$ which suffices for our purposes. In particular, we choose
\begin{equation}\label{introgdef}g(t) = \lambda(t) \log^{b-2\epsilon}(t), \text{ where } 0 < \epsilon < \min\{\frac{3b-2}{1600},\frac{2b-1}{200},\frac{1}{900000},\frac{b}{900}\}.\end{equation} 
In section \ref{w2sec}, we then add a first correction, $w_{2}$ which improves the error term of $u_{3}$. On the other hand, there are now nonlinear interactions between $w_{2}$ and the previous corrections, which, due to the slow decay of $v_{1}$, are not perturbative. Similarly to the case with $v_{k}$, in section \ref{wjsection}, we add another series of corrections 
$$w_{s} = \sum_{k=3}^{\infty} w_{k}$$ 
to eventually eliminate all the nonlinear error terms involving $w_{j}$ and $v_{k}$. If we let $w_{c} = w_{2}+w_{s}$, then, we end up with our final ansatz 
\begin{equation}\label{ufinalansatzintro}u_{5}(t,r) = Q_{\frac{1}{\lambda(t)}}(r)+v_{c}(t,r)+w_{c}(t,r).\end{equation}
The error term of $u_{5}$ is then decomposed as $F_{4}+F_{5}$, where $F_{5}$ is sufficiently small in sufficiently many norms so as to allow it to be eventually treated perturbatively, even though it will end up not necessarily being orthogonal to $\phi_{0}(\frac{\cdot}{\lambda(t)})$. 
More precisely, if we substitute $u(t,r) = u_{5}(t,r) + v(t,r)$ into \eqref{ym}, we get the following equation (see also \eqref{veqnfinal})
$$-\partial_{tt}v+\partial_{rr}v+\frac{1}{r}\partial_{r}v+\frac{2}{r^{2}}\left(1-3Q_{\frac{1}{\lambda(t)}}(r)^{2}\right) v = F_{4}(t,r)+F_{5}(t,r)+F_{3}(t,r)$$
(where the term $F_{3}$ depends on $v$). We have
$$F_{5}(t,r) = \left(1-\chi_{\leq 1}(\frac{2 r}{t})\right)\left(\frac{-6}{r^{2}}\left(1-Q^{2}_{\frac{1}{\lambda(t)}}(r)\right) w_{c}(t,r)\right)$$
where $\chi_{\leq 1}$ is a cutoff whose properties are unimportant for the purposes of this discussion. The smallness of $F_{5}$ is more precisely
$$||F_{5}(t,R \lambda(t))||_{L^{2}(R dR)} \leq \frac{C \lambda(t)^{3}}{t^{5} \log^{b-2}(t)}, \quad ||L^{*}L\left(F_{5}(t,R\lambda(t))\right)||_{L^{2}(R dR)} \leq \frac{C \lambda(t)^{5} \log^{2}(t)}{g(t)^{2} \log^{b}(t) t^{5}}$$
where we recall that $L$ has been defined in \eqref{Ldef}. We then choose $\lambda(t)$ so that the term $F_{4}(t,r)$ is orthogonal to $\phi_{0}(\frac{\cdot}{\lambda(t)})$. Once we solve this equation for $\lambda$, we then prove that $\lambda$ has symbol-type estimates up to the fourth derivative. This then implies that $F_{4}(t,r)$ itself has symbol-type estimates, which will be important later on.\\
\\
\textbf{2. Completion of the ansatz to an exact solution of \eqref{ym}} To complete the ansatz $u_{5}$ to an exact solution of \eqref{ym}, we use the following approach. We substitute $u=u_{5}+v$ into \eqref{ym}, and use the ``distorted Fourier transform'' of \cite{kstym}, which we denote as $\mathcal{F}$ (recall \eqref{distortedfouriernotation}) to recast the resulting equation for $v$ into one for $y$, given by
\begin{equation}\label{yscalingintro}\mathcal{F}(\sqrt{\cdot} v(t,\cdot \lambda(t)))(\xi) = \begin{bmatrix} y_{0}(t)\\
y_{1}(t,\frac{\xi}{\lambda(t)^{2}})\end{bmatrix}.\end{equation}
The choice of re-scaling in $y_{1}$ is explained by noting that the resulting system of equations (which we will solve for all $t$ sufficiently large) for $y_{0}$ and $y_{1}$ takes the form
\begin{equation}\label{yeqnintro}\begin{bmatrix} -\partial_{tt}y_{0}\\
-\partial_{tt}y_{1}-\omega y_{1}\end{bmatrix} = F_{2} + \mathcal{F}\left(\sqrt{\cdot}\left(F_{3}+F_{4}+F_{5}\right)\left(t,\cdot \lambda(t)\right)\right)\left(\omega \lambda(t)^{2}\right), \quad \omega >0\end{equation}
where $F_{2}$ contains perturbative terms depending on $y$ and $\partial_{t}y$, some of which are estimated using the transference identity of \cite{kstym}, and $F_{3}$ contains other linear and nonlinear error terms depending on $v(y)$. To give the reader some idea of how the transference identity of \cite{kstym} is used, we recall that the soliton term in \eqref{ufinalansatzintro} is re-scaled by $\lambda(t)$, so it is natural to consider re-scaling the spatial argument of $v$ by $\lambda(t)$. When $t$ derivatives act on $v$ re-scaled in this way, one obtains terms involving the operator $r\partial_{r}$. Therefore, one must understand the composition of the distorted Fourier transform with $r \partial_{r}$. The transference identity does precisely this, see \eqref{transference}. The reference \cite{kst} has some more intuition regarding the transference identity.
\\
\\
We solve the equation \eqref{yeqnintro} for $t \geq T_{0}$ by finding a fixed-point of the map $T$ given by
\begin{equation}\label{Tmapintro}T(\begin{bmatrix} y_{0}\\
y_{1}\end{bmatrix})(t,\omega) =\begin{bmatrix} -\int_{t}^{\infty} ds \int_{s}^{\infty} ds_{1} \left(F_{2,0} + \mathcal{F}\left(\sqrt{\cdot}\left(F_{3}+F_{4}+F_{5}\right)\left(s_{1},\cdot \lambda(s_{1})\right)\right)_{0}\right)\\
\int_{t}^{\infty} dx \frac{\sin((t-x)\sqrt{\omega})}{\sqrt{\omega}} \left(F_{2,1}(x,\omega) + \mathcal{F}\left(\sqrt{\cdot}\left(F_{3}+F_{4}+F_{5}\right)\left(x,\cdot \lambda(x)\right)\right)_{1}\left(\omega \lambda(x)^{2}\right)\right)\end{bmatrix}\end{equation}
where the subscripts $i$ after, for example $F_{2}$, or $\mathcal{F}\left(\sqrt{\cdot}\left(F_{3}+F_{4}+F_{5}\right)\left(x,\cdot \lambda(x)\right)\right)$ mean the $i+1$st component of the vector, for $i=0,1$. $T$ is defined on a space $Z$, whose norm is precisely given in \eqref{znormdef}, but is roughly a weighted $L^{\infty}_{t} L^{2}_{\omega}$ norm of $y$ and $\partial_{t}y$. The most delicate terms on the right-hand side of \eqref{Tmapintro} are those involving $F_{4}$. Because of the orthogonality condition on $F_{4}$, we have 
$$\mathcal{F}\left(\sqrt{\cdot}\left(F_{4}\right)\left(s_{1},\cdot \lambda(s_{1})\right)\right)_{0}=0.$$
On the other hand, for the second component of \eqref{Tmapintro}, we treat the $F_{4}$ term by integrating by parts in the $x$ variable, using both the fact that  $$\mathcal{F}\left(\sqrt{\cdot}\left(F_{4}\right)\left(x,\cdot \lambda(x)\right)\right)_{1}(\xi) \rightarrow 0, \quad \xi \rightarrow 0$$
(which follows from the orthogonality condition on $F_{4}$) and the fact that $F_{4}$ has symbol-type estimates. The other terms in \eqref{Tmapintro} can be estimated without such a delicate argument. 

\section{Construction of the ansatz}
Let $b > \frac{2}{3}$, $f \in F_{b}$, and $\lambda_{0} \in \Lambda_{b}$ satisfying \eqref{lambda0str}. We let $M_{f}$ denote the constant $M$ for which \eqref{fineqindef} holds. By replacing $T_{\lambda_{0}}$ with $T_{\lambda_{0}}+M_{f}$ if needed, we may assume that $T_{\lambda_{0}}>M_{f}$. We will have to introduce some constants and parameters to describe our setup. Let $T_{0} > \exp{\left(900! + 2^{-\frac{3}{2(2b-1)}}\right)}\left(1+T_{\lambda_{0}}\right)$. 
\begin{prelimdefinition}
 Let $0 < \epsilon < \min\{\frac{3b-2}{1600},\frac{2b-1}{200},\frac{1}{900000},\frac{b}{900}\}$, and define \begin{equation} \label{deltadef} \delta = \min\{2b-1,3b-4\epsilon-2,5b-8\epsilon-3\}.\end{equation}
 \end{prelimdefinition}
Note that $\delta >0$, because $b > \frac{2}{3}$, and because of the constraints on $\epsilon$. Also, $1+\delta > b$. 
\begin{prelimdefinition}Let 
\begin{equation}\label{delta2def}\delta_{2}=\min\{\frac{1}{2}\left(\delta+1-b\right),\frac{\delta}{2}\}.\end{equation} 
Define a Banach space $X$ to be the set of functions $e \in C^{2}([T_{0},\infty))$ satisfying $||e||_{X} < \infty$, where
\begin{equation} \label{xnorm} ||e||_{X} = \sup_{t \geq T_{0}}\left(|e(t)| \log^{\delta-\delta_{2}}(t) + |e'(t)| t \log^{1+\delta-\delta_{2}}(t) + |e''(t)| t^{2} \log^{1+\delta-\delta_{2}}(t)\right).\end{equation}
\end{prelimdefinition}
Until more precisely chosen, $\lambda$ denotes any function of the form 
\begin{equation}\label{lambdarestr}\lambda(t) = \lambda_{0}(t)\cdot \left(1+e(t)\right), \quad e \in \overline{B}_{1}(0) \subset X.\end{equation} 
In particular, since $1+\delta-\delta_{2} > b$, we have
$$\frac{|\lambda'(t)|}{\lambda(t)} \leq \frac{C}{t \log^{b}(t)}, \quad \frac{|\lambda''(t)|}{\lambda(t)} \leq \frac{C}{t^{2} \log^{b}(t)}.$$ 
For later use, we make the following definition.
\begin{prelimdefinition}\label{gprelimdef} $g(t) = \lambda(t) \log^{b-2\epsilon}(t)$\end{prelimdefinition} 
Some motivation of this definition is given just before \eqref{introgdef} and in the beginning of Section \ref{w2sec}, which is the first place in the proof of Theorem \ref{mainthm} where we use $g(t)$.\\
\\
By the definition of $g$, constraints on $\lambda_{0}$, and \eqref{lambdarestr}, there exists $M_{1}$ sufficiently large so that
\begin{equation}\label{T0const}\log(t) \geq 2 |\log(g(t))|,\quad \frac{t |g'(t)|}{g(t)} \leq \frac{1}{900!}, \quad \frac{t}{g(t)} \geq 1600, \quad \text{ for } t \geq M_{1}.\end{equation}
Then, we further constrain $T_{0}$ to satisfy 
\begin{equation}\label{T0introconst}T_{0} > \exp{\left(900! + 2^{-\frac{3}{2(2b-1)}}\right)}\left(1+T_{\lambda_{0}}\right)+ M_{1}.\end{equation}
The main result of this section is the following theorem concerning the existence of an approximate solution to \eqref{ym}.
\begin{theorem}[Approximate solution to \eqref{ym}]\label{approxsolnthm}
For all $b>\frac{2}{3}$, $f \in F_{b}$, and all $\lambda_{0} \in \Lambda_{b}$ satisfying \eqref{lambda0str}, there exists $T_{3}>0$ such that for all $T_{0}>T_{3}$, there exists $v_{corr} \in C^{2}([T_{0},\infty),C^{2}((0,\infty)))$ and $\lambda \in C^{4}([T_{0},\infty))$, such that, if 
$$u(t,r)=Q_{\frac{1}{\lambda(t)}}(r)+v_{corr}(t,r)$$ 
then
$$ E_{YM}(u,\partial_{t}u) <\infty, \quad -\partial_{tt}u+\partial_{rr}u+\frac{1}{r}\partial_{r}u+\frac{2 u (1-u^{2})}{r^{2}} = -F_{4}(t,r)-F_{5}(t,r)$$
where
$$||F_{5}(t,R \lambda(t))||_{L^{2}(R dR)} \leq \frac{C \lambda(t)^{3}}{t^{5} \log^{b-2}(t)}, \quad ||L^{*}L\left(F_{5}(t,R\lambda(t))\right)||_{L^{2}(R dR)} \leq \frac{C \lambda(t)^{5} \log^{2}(t)}{g(t)^{2} \log^{b}(t) t^{5}}.$$
$$\langle F_{4}(t,R \lambda(t)),\phi_{0}(R)\rangle_{L^{2}(R dR)} =0$$
For $0 \leq j, k \leq 2$, and $j+k \leq 2$,
\begin{equation}\nonumber\begin{split}|r^{k}\partial_{r}^{k}t^{j}\partial_{t}^{j}F_{4}(t,r)| &\leq C \mathbbm{1}_{\{r \leq g(t)\}} \frac{r^{2} \lambda(t)^{2}}{t^{2} \log^{b}(t) (\lambda(t)^{2}+r^{2})^{2}} \\
&+ C \frac{\mathbbm{1}_{\{r \leq \frac{t}{2}\}} \lambda(t)^{2}}{(\lambda(t)^{2}+r^{2})^{2}} \begin{cases} \frac{r^{2} \lambda(t)^{2} \log(t)}{t^{2} g(t)^{2} \log^{b}(t)}, \quad r \leq g(t)\\
\frac{\lambda(t)^{2} \log(t)}{t^{2} \log^{b}(t)} \left(\log(2+\frac{r}{g(t)}) + \frac{\log(t)}{\log^{b}(t)}\right), \quad g(t) < r < \frac{t}{2}\end{cases}.\end{split}\end{equation}
$\lambda$ is given by $\lambda(t) = \lambda_{0}(t)\left(1+e(t)\right).$
\begin{equation}\nonumber\begin{split} |e(t)| \leq \frac{C}{\log^{\delta-\delta_{2}}(t)}, \quad |e^{k}(t)| \leq \begin{cases} \frac{C}{t^{k} \log^{1+\delta-\delta_{2}}(t)}, \quad k=1,2\\
\frac{C}{t^{3} \log^{b+\delta_{4}}(t)}, \quad k=3\\
\frac{C}{t^{4} \log^{b+\delta_{5}}(t)}, \quad k=4\end{cases}\end{split}\end{equation}
where $\delta, \delta_{2}$ are defined in \eqref{deltadef} and \eqref{delta2def}, respectively and $\delta_{4},\delta_{5}>0$.
Finally,
\begin{equation}\label{vcorrinftyest}\begin{split}&\left(||\frac{2 v_{corr}(t,R\lambda(t)) Q_{1}(R) + v_{corr}^{2}(t,R\lambda(t))}{R^{2}\lambda(t)^{2}}||_{L^{\infty}_{R}} + \sup_{R \geq 1}\left(\frac{|\partial_{R}\left(2 v_{corr}(t,R\lambda(t))Q_{1}(R) + v_{corr}^{2}(t,R\lambda(t))\right)|}{\lambda(t)^{2}R^{2}}\right)\right.\\
&+\left.\sup_{R \geq 1}\left(\frac{|\partial_{R}^{2}\left(2 v_{corr}(t,R\lambda(t))Q_{1}(R) + v_{corr}^{2}(t,R\lambda(t))\right)|}{R^{2}\lambda(t)^{2}}\right) \right.\\
&+\left. \sup_{R \leq 1}\left(\frac{|\partial_{R}\left(2 v_{corr}(t,R\lambda(t))Q_{1}(R) + v_{corr}^{2}(t,R\lambda(t))\right)|}{\lambda(t)^{2}R}\right)\right. \\
&+\left. \sup_{R \leq 1}\left(\frac{|\partial_{R}^{2}\left(2 v_{corr}(t,R\lambda(t))Q_{1}(R) + v_{corr}^{2}(t,R\lambda(t))\right)|}{\lambda(t)^{2}}\right)\right)\\
&\leq \frac{C}{t^{2} \log^{b}(t)}.\end{split}\end{equation}
\end{theorem}
\subsection{The Cauchy data for the radiation $v_{1}$}\label{cauchydatav1}
In this section, we will introduce the initial velocity for the first addition to the soliton in our ansatz, which we will denote by $v_{1}$.\\
\\
Let $\psi \in C^{\infty}([0,\infty))$ satisfy  
\begin{equation}\label{psidefinition}\psi(x) = \begin{cases} 0, \quad x \leq T_{\lambda_{0}}\\
1, \quad x \geq 2 T_{\lambda_{0}}\end{cases}, \quad \text{ and } 0 \leq \psi(x) \leq 1, \quad x \geq 0.\end{equation} 
Then, $t \mapsto \psi(t) f(t)$, apriori only defined on $[T_{\lambda_{0}},\infty)$, extends to a smooth function on $[0,\infty)$. Note that 
$$\frac{\lambda_{0}''(t)}{\lambda_{0}(t)} = \frac{(\psi \cdot f)'(t)}{t}, \quad t \geq 2 T_{\lambda_{0}}.$$
Finally, we define $v_{1,1}$ by specifying its Hankel transform of order 2:
\begin{equation}\label{v11def}\widehat{v_{1,1}}(\xi) = \frac{8}{3 \pi \xi^{2}} \int_{0}^{\infty} \frac{(\psi \cdot f)'(t)}{t} \sin(t\xi) dt.\end{equation}
As in \cite{wm}, this definition is made so as to allow $\lambda(t) = \lambda_{0}(t)$ to be a leading order solution to the eventual modulation equation for $\lambda$. Now, we record some pointwise estimates on $\widehat{v_{1,1}}$ and its derivatives. 
\begin{lemma}\label{v11lemma} For $k \geq 0$, there exist  constants $C_{k}, C(k,N)$, such that
\begin{equation}\label{v11hatptwse} |\partial_{\xi}^{k}\left(\xi^{2} \widehat{v_{1,1}}(\xi)\right)| \leq \begin{cases} \frac{C_{k}}{\xi^{k-1}\log^{b}(\frac{1}{\xi})}, \quad \xi \leq \frac{1}{4}\\
\frac{C(k,N)}{\xi^{k+4N}}, \quad \xi > \frac{1}{4}, N \geq 1\end{cases}. \end{equation}\end{lemma}
\begin{proof} We have
\begin{equation}\label{v11int} \begin{split} \widehat{v_{1,1}}(\xi) &=\frac{8}{3 \pi \xi^{2}} \int_{0}^{\frac{1}{\xi}} \frac{\left(\psi f\right)'(t)}{t} t \xi dt + \frac{8}{3 \pi \xi^{2}} \int_{0}^{\frac{1}{\xi}} \frac{\left(\psi f\right)'(t)}{t} \left(\sin(t\xi)-t\xi\right) dt + \frac{8}{3 \pi \xi^{2}} \int_{\frac{1}{\xi}}^{\infty} \frac{\left(\psi f\right)'(t)}{t} \sin(t\xi) dt.\end{split}\end{equation}
We start with the region $\xi \leq \frac{1}{4}$. For the first term on the right-hand side of \eqref{v11int}, we have
$$|\frac{8}{3 \pi \xi^{2}} \int_{0}^{\frac{1}{\xi}} \frac{\left(\psi f\right)'(t)}{t} t \xi dt| = |\frac{8}{3 \pi \xi} \left(\psi \cdot f\right)(\frac{1}{\xi})| \leq \frac{C}{\xi \log^{b}(\frac{1}{\xi})}.$$
For the other two terms on the right-hand side of \eqref{v11int}, we have
\begin{equation}\begin{split}&|\frac{8}{3 \pi \xi^{2}} \int_{0}^{\frac{1}{\xi}} \frac{\left(\psi f\right)'(t)}{t} \left(\sin(t\xi)-t\xi\right) dt|+|\frac{8}{3 \pi \xi^{2}} \int_{\frac{1}{\xi}}^{\infty} \frac{\left(\psi f\right)'(t)}{t} \sin(t\xi) dt|\\
&\leq \frac{C}{\xi^{2}} \begin{cases} 0, \quad \frac{1}{T_{\lambda_{0}}}< \xi\\
\int_{T_{\lambda_{0}}}^{\frac{1}{\xi}} \frac{t \xi^{3} dt}{\log^{b}(t)}, \quad \xi \leq \frac{1}{T_{\lambda_{0}}}\end{cases} + \frac{C}{\xi^{2}} \int_{\frac{1}{\xi}}^{\infty} \frac{dt}{t^{2}\log^{b}(t)}\leq \frac{C}{\xi \log^{b}(\frac{1}{\xi})}.\end{split}\end{equation}
In total, we thus have
\begin{equation}\label{v20nearorigin}|\widehat{v_{1,1}}(\xi)| \leq \frac{C}{\xi \log^{b}(\frac{1}{\xi})}, \quad \xi \leq \frac{1}{4}.\end{equation}
Note that the first term on the right-hand side of \eqref{v11int} is where we use the condition that 
\begin{equation}\label{lambda0str2}\frac{\lambda_{0}''(t)}{\lambda_{0}(t)} = \frac{f'(t)}{t}, \quad |f(t)| \leq \frac{C}{\log^{b}(t)}, \quad t \geq T_{\lambda_{0}}.\end{equation}
This condition, along with $b>\frac{2}{3}$ guarantees that $\widehat{v_{1,1}} \in L^{2}((0,\infty), \xi d\xi)$. (We will see shortly that $\widehat{v_{1,1}}(\xi)$ is rapidly decreasing for large $\xi$). Continuing our estimates, for each $k \geq 1$, there exist constants $C_{j,k}$ such that
$$\partial_{\xi}^{k}\left(\left(\psi f\right)'(\frac{\sigma}{\xi})\right) = \sum_{j=2}^{k+1} \frac{\left(\psi f\right)^{(j)}(\frac{\sigma}{\xi}) \sigma^{j-1} C_{j,k}}{\xi^{k+j-1}}.$$
To estimate $|\partial_{\xi}^{k}\left(\xi^{2} \widehat{v_{1,1}}(\xi)\right)|$ for $k \geq 1$, in the region $\xi \leq \frac{1}{4}$, it suffices to consider the case $\xi \leq \frac{4}{T_{\lambda_{0}}}$, since $$|\partial_{\xi}^{k}\left(\xi^{2} \widehat{v_{1,1}}(\xi)\right)|\leq C_{k}, \quad \frac{4}{T_{\lambda_{0}}} \leq \xi \leq \frac{1}{4}.$$ 
In the case $T_{\lambda_{0}}\xi \leq 4$, we let $\sigma =t \xi$ in the integral defining $\widehat{v_{1,1}}$, and differentiate under the integral sign. Then, we use the support properties of $\psi$, and treat the integral over $\sigma \in [T_{\lambda_{0}}\xi,4]$ and $(4,\infty)$ separately.  Hence, for some constant $C_{k}$, whose value may change from line to line:
\begin{equation}\label{dxikv20int} \partial_{\xi}^{k}\left(\xi^{2} \widehat{v_{1,1}}(\xi)\right) = \frac{8}{3\pi} \sum_{j=2}^{k+1} \frac{C_{j,k}}{\xi^{k+j-1}} \int_{T_{\lambda_{0}}\xi}^{4} \left(\psi f\right)^{(j)}(\frac{\sigma}{\xi}) \sigma^{j-1} d\sigma + \text{Err}(\xi)\end{equation}
where
\begin{equation}\nonumber\begin{split}|Err(\xi)| &\leq C \sum_{j=2}^{k+1} \int_{T_{\lambda_{0}}\xi}^{4} \frac{C_{j,k} C_{j} \sigma d\sigma}{\log^{b}(\frac{\sigma}{\xi}) \xi^{k-1}} + \sum_{j=2}^{k+1}\frac{C_{j,k}}{\xi^{k-1} \log^{b}(\frac{1}{\xi})}\leq \frac{C_{k}}{\xi^{k-1} \log^{b}(\frac{1}{\xi})}, \quad \xi \leq \frac{4}{T_{\lambda_{0}}}.\end{split}\end{equation}
By induction, for $j \geq 1$,
$$\int_{a}^{b} \left(\psi f\right)^{(j)}(x) x^{j-1} dx = (-1)^{j-1} (j-1)! \sum_{q=0}^{j-1}\frac{(-1)^{q}}{q!} b^{q}\left(\psi f\right)^{(q)}(b), \quad \text{ if } \left(\psi f\right)^{(n)}(a) =0, \text{ for }n\geq 0.$$
Using this fact, and the support properties of $\psi$, we return to \eqref{dxikv20int} and \eqref{v20nearorigin} and get, for $k \geq 0$,
\begin{equation}\label{dxikv20nearorigin} |\partial_{\xi}^{k}\left(\xi^{2} \widehat{v_{1,1}}(\xi)\right)| \leq \frac{C_{k}}{\xi^{k-1}\log^{b}(\frac{1}{\xi})}, \quad \xi \leq \frac{1}{4}.\end{equation}

Finally, for $k \geq 1$,
$$\partial_{\xi}^{k}\left(\xi^{2} \widehat{v_{1,1}}(\xi)\right) = \frac{8}{3 \pi} \sum_{j=2}^{k+1} \int_{0}^{\infty} \frac{\left(\psi f\right)^{(j)}(\frac{\sigma}{\xi})}{\xi^{k+j-1}} C_{j,k} \sigma^{j-2} \sin(\sigma) d\sigma$$
implies 
\begin{equation} |\partial_{\xi}^{k}\left(\xi^{2} \widehat{v_{1,1}}(\xi)\right)| \leq \frac{C(k,N)}{\xi^{k+4N}}, \quad \xi > \frac{1}{4}, N \geq 1, k \geq 0.\end{equation}
This completes the proof of the lemma.\end{proof}
\begin{remark}\label{necessarycondremark} \textup{Since $v_{1,1}$ will end up being the initial velocity of our radiation component of the solution, $v_{1}$, Lemma \ref{v11lemma} implies that $v_{1}$ has finite energy}. \textup{The condition \eqref{lambda0str2}, which was important in the proof that $\widehat{v_{1,1}}\in L^{2}((0,\infty),\xi d\xi)$, also implies that $\lambda(t)$ must asymptote to a constant as $t$ approaches infinity, as shown in \eqref{feqnintro1}.}\\
\\
\textup{If we did not assume the condition \eqref{lambda0str2}, we would have defined $\widehat{v_{1,1}}(\xi)$ by
$$\widehat{v_{1,1}}(\xi) = \frac{8}{3 \pi \xi^{2}} \int_{0}^{\infty} h(t) \sin(t\xi) dt$$
where $h$ is some extension of $\frac{\lambda_{0}''(t)}{\lambda_{0}(t)}$, which satisfies
$$h(t) = \frac{\lambda_{0}''(t) }{\lambda_{0}(t)}, \quad t \geq 2T_{\lambda_{0}}, \quad h\in C^{\infty}([0,\infty)).$$
Now we show that if we didn't assume the structural condition \eqref{lambda0str2}, but rather, only the symbol-type estimates
\begin{equation}\label{lambda0symb2}\frac{|\lambda_{0}^{(k)}(t)|}{\lambda_{0}(t)} \leq \frac{C_{k}}{t^{k}\log^{b}(t)}, \quad k \geq 1, \quad t \geq T_{\lambda_{0}}\end{equation} 
and regardless of how one chooses to extend $\frac{\lambda_{0}''(t)}{\lambda_{0}(t)}$ from a function defined on, say $[2T_{\lambda_{0}},\infty)$ to one defined on $[0,\infty)$ we would still need $\lambda_{0}(t)\rightarrow c$ to have $\widehat{v_{1,1}}\in L^{2}((0,\infty),\xi d\xi)$. We show this as follows. Let $h(t)$ be any extension of $\frac{\lambda_{0}''(t)}{\lambda_{0}(t)}$, which satisfies
$$h(t) = \frac{\lambda_{0}''(t)}{\lambda_{0}(t)}, \quad t \geq 2T_{\lambda_{0}}, \quad h\in C^{\infty}([0,\infty)).$$ 
If $\widehat{v_{1,1}} \in L^{2}((0,\infty),\xi d\xi)$, then, $\widehat{v_{1,1}} \in L^{2}((0,\frac{1}{2T_{\lambda_{0}}}),\xi d\xi)$. The analog of \eqref{v11int} is 
\begin{equation}\label{v11int2}\widehat{v_{1,1}}(\xi) =\frac{8}{3 \pi \xi^{2}} \int_{0}^{\frac{1}{\xi}} h(t) t \xi dt + \frac{8}{3 \pi \xi^{2}} \int_{0}^{\frac{1}{\xi}} h(t) \left(\sin(t\xi)-t\xi\right) dt + \frac{8}{3 \pi \xi^{2}} \int_{\frac{1}{\xi}}^{\infty}h(t) \sin(t\xi) dt.\end{equation}
Even without the structural condition \eqref{lambda0str2}, the second and third terms of \eqref{v11int2} are bounded above in absolute value by
$$\frac{C}{\xi \log^{b}(\frac{1}{\xi})}, \quad \xi \leq \frac{1}{2 T_{\lambda_{0}}}.$$
(Recall that $\frac{1}{2T_{\lambda_{0}}} \leq \frac{1}{4})$. Therefore, the condition $\widehat{v_{1,1}} \in L^{2}((0,\frac{1}{2T_{\lambda_{0}}}),\xi d\xi)$ implies that the first term on the right-hand side of \eqref{v11int2} is in $L^{2}((0,\frac{1}{2T_{\lambda_{0}}}),\xi d\xi)$. 
If we let
$$G(x) = \frac{8}{3\pi}\left(\int_{0}^{2T_{\lambda_{0}}} h(t) t dt+ \int_{2T_{\lambda_{0}}}^{x} \frac{t \lambda_{0}''(t)}{\lambda_{0}(t)} dt\right), \quad x > 2 T_{\lambda_{0}}$$
then, the first term of \eqref{v11int2} is
$$\frac{G(\frac{1}{\xi})}{\xi} \in L^{2}((0,\frac{1}{2T_{\lambda_{0}}}),\xi d\xi).$$
Therefore,
$$\int_{0}^{\frac{1}{2T_{\lambda_{0}}}} \frac{|G(\frac{1}{\xi})|^{2}}{\xi} d\xi = \int_{\log(2T_{\lambda_{0}})}^{\infty} (G(e^{u}))^{2} du < \infty, \quad u = \log(\frac{1}{\xi}).$$
But,
$$\frac{d}{du} \left(G(e^{u})\right)^{2} = 2 G(e^{u}) G'(e^{u}) e^{u} = 2 G(e^{u}) \frac{8}{3\pi}\left(  \frac{e^{u} \lambda_{0}''(e^{u})}{\lambda_{0}(e^{u})}\right) e^{u}.$$
Therefore,
$$|\frac{d}{du}\left(G(e^{u})\right)^{2}| \leq C |G(e^{u})| \frac{1}{\log^{b}(e^{u})} \leq C \left(1+\frac{1}{\log^{b-1}(e^{u})}\right) \frac{1}{\log^{b}(e^{u})} \leq C, \quad u \geq \log(2T_{\lambda_{0}})$$
where we used $b>\frac{2}{3} > \frac{1}{2}$, and
$$|G(x)| \leq C +  C \int_{2T_{\lambda_{0}}}^{x} \frac{dt}{t \log^{b}(t)} \leq C \left(1+\frac{1}{\log^{b-1}(x)}\right), \quad x \geq 2T_{\lambda_{0}}$$
and we stress again that we only use the symbol-type estimates \eqref{lambda0symb2}, and \emph{not} the structural condition \eqref{lambda0str2} for this discussion.\\
\\
But, now we can conclude that $u \mapsto (G(e^{u}))^{2}$ is Lipschitz, whence, the condition
$$\int_{2T_{\lambda_{0}}}^{\infty} (G(e^{u}))^{2} du < \infty$$
implies that $\lim_{u \rightarrow \infty} (G(e^{u}))^{2} =0$, which is to say that 
$$\lim_{x \rightarrow \infty} \int_{2T_{\lambda_{0}}}^{x} \frac{t \lambda_{0}''(t)}{\lambda_{0}(t)} dt < \infty, \quad \text{ or, equivalently, that } \lim _{\xi \rightarrow 0^{+}}\int_{2T_{\lambda_{0}}}^{\frac{1}{\xi}} \frac{t \lambda_{0}''(t)}{\lambda_{0}(t)} dt < \infty.$$
Therefore, a necessary (but in general insufficient) condition for $\widehat{v_{1,1}} \in L^{2}((0,\infty), \xi d\xi)$ is that
$$\lim_{\xi \rightarrow 0} \int_{2T_{\lambda_{0}}}^{\frac{1}{\xi}} \frac{\lambda_{0}''(t)}{\lambda_{0}(t)} t dt <\infty.$$
(In particular, the limit has to exist). Repeating the same computation done in \eqref{feqnintro1}, we again get
$$\lim_{\xi \rightarrow 0} \log(\lambda_{0}(\frac{1}{\xi})) <\infty$$
which implies that, in our setting, $\lambda_{0}(t)$ asymptoting to a non-zero constant for large $t$ is necessary for the radiation $v_{1}$ to have finite energy.}
\end{remark}
\subsection{Estimates on $v_{1}$}
We define $v_{1}$ to be the solution to the following Cauchy problem
\begin{equation}\begin{cases} -\partial_{tt}v_{1}+\partial_{rr}v_{1}+\frac{1}{r}\partial_{r}v_{1}-\frac{4 v_{1}}{r^{2}} =0\\
v_{1}(0)=0\\
\partial_{t}v_{1}(0) = v_{1,1}\end{cases}.\end{equation}
From now onwards, we always restrict the $t$ coordinate to satisfy $t \geq T_{0}$. Recall that $T_{0}>0$ has so far been constrained to satisfy \eqref{T0introconst}, and will be further constrained as the proof progresses.
\begin{lemma}\label{v1lemma} We have the following estimates 
\begin{equation}\label{v1est} |v_{1}(t,r)| \leq \begin{cases} \frac{C r^{2}}{t^{2} \log^{b}(t)}, \quad r \leq \frac{t}{2}\\
\frac{C}{\log^{b}(r)}, \quad r \geq \frac{t}{2}\end{cases}.\end{equation}
For $1 \leq j+k$, and $0 \leq j, k \leq 2$,
\begin{equation}\label{dv1est} |\partial_{t}^{j}\partial_{r}^{k} v_{1}(t,r)| \leq \begin{cases} C \frac{r^{2-k}}{t^{2+j} \log^{b}(t)}, \quad r \leq \frac{t}{2}\\
\frac{C}{\sqrt{r}\log^{b}(\langle t-r\rangle) \langle t-r\rangle^{\frac{1}{2}+j+k-1} },\quad r > \frac{t}{2} \end{cases}.\end{equation}
\end{lemma}
\begin{proof}
From \eqref{v11def}, we have
\begin{equation} \frac{4 (\psi\cdot f)'(t)}{3 t} = \int_{0}^{\infty} \sin(t\xi) \widehat{v_{1,1}}(\xi) \xi^{2} d\xi, \quad t \geq 0.\end{equation}
We study the region $r \leq \frac{t}{2}$ (we could just as well study any region $r \leq a t$ with $a<1$) and have
\begin{equation}\label{v1rep}\begin{split} v_{1}(t,r) &= \int_{0}^{\infty} \sin(t\xi) J_{2}(r\xi) \widehat{v_{1,1}}(\xi) d\xi = \frac{r^{2}}{6\pi} \int_{0}^{\pi} \sin^{4}(\theta) d\theta \int_{0}^{\infty} \xi^{2}\left(\sin(\xi t_{+})+\sin(\xi t_{-})\right) \widehat{v_{1,1}}(\xi) d\xi\\
&=\frac{2r^{2}}{9\pi} \int_{0}^{\pi} \sin^{4}(\theta) \left(\frac{(\psi\cdot f)'(t_{+})}{t_{+}} + \frac{(\psi\cdot f)'(t_{-})}{t_{-}}\right)d\theta\end{split}\end{equation}
where $t_{\pm} = t\pm r \cos(\theta) \geq \frac{t}{2}, \text{ for } r \leq \frac{t}{2}.$ So,
\begin{equation}|v_{1}(t,r)| \leq \frac{C r^{2}}{t^{2} \log^{b}(t)}, \quad r \leq \frac{t}{2}.\end{equation}
With the same procedure, we get, for $0 \leq j,k \leq 2$
\begin{equation}|\partial_{r}^{k}\partial_{t}^{j}v_{1}(t,r)| \leq C \frac{r^{2-k}}{t^{2+j} \log^{b}(t)}, \quad r \leq \frac{t}{2}.\end{equation}
Because of the singularity\footnote{What is meant here is that $f$ ranges through a class of functions including ones for which the associated $\widehat{v_{1,1}}(\xi)$ could have singularities at low frequencies $\xi$.} of $\widehat{v_{1,1}}(\xi)$ for small $\xi$, $v_{1}$ does not decay like $\frac{1}{\sqrt{r}}$ for large $r$. On the other hand, its derivatives do decay like $\frac{1}{\sqrt{r}}$ near the cone, because of their improved low frequency behavior. In particular, we prove \eqref{v1est} for $r \geq \frac{t}{2}$ as follows.
\begin{equation}\label{v1fourierrep}\begin{split}v_{1}(t,r) &= \int_{0}^{\infty} \sin(t\xi) J_{2}(r\xi) \widehat{v_{1,1}}(\xi) d\xi= \int_{0}^{\frac{1}{r}} \sin(t\xi) J_{2}(r\xi) \widehat{v_{1,1}}(\xi) d\xi + \int_{\frac{1}{r}}^{\infty} \sin(t\xi) J_{2}(r\xi) \widehat{v_{1,1}}(\xi) d\xi.\end{split}\end{equation}
So,
\begin{equation}\begin{split}|v_{1}(t,r)|&\leq C \int_{0}^{\frac{1}{r}} \frac{r^{2} \xi^{2} d\xi}{\xi \log^{b}(\frac{1}{\xi})} + C \int_{\frac{1}{r}}^{\frac{1}{4}} \frac{d\xi}{\sqrt{r} \xi^{3/2} \log^{b}(\frac{1}{\xi})} + C \int_{\frac{1}{4}}^{\infty} \frac{d\xi}{\sqrt{r \xi} \xi^{7}}\leq \frac{C}{\log^{b}(r)}, \quad r \geq \frac{t}{2}.\end{split}\end{equation}
We remark that, in the above estimate, we used
$$|J_{2}(x)| \leq \begin{cases} C x^{2}, \quad x \leq 1\\
 \frac{C}{\sqrt{x}}, \quad x \geq 1\end{cases}.$$
To establish \eqref{dv1est} in the region $r \geq \frac{t}{2}$, we start by using the following simple argument, considering first $\partial_{t}v_{1}$.
\begin{equation}\begin{split} |\partial_{t}v_{1}(t,r)| &= |\int_{0}^{\infty} \cos(t\xi) \xi J_{2}(r\xi) \widehat{v_{1,1}}(\xi) d\xi| \leq C \int_{0}^{\frac{1}{4}} \frac{\sqrt{\xi} d\xi}{\xi \log^{b}(\frac{1}{\xi}) \sqrt{r}} + C \int_{\frac{1}{4}}^{\infty} \frac{\sqrt{\xi}}{\sqrt{r}} \frac{1}{\xi^{900}} d\xi\leq \frac{C}{\sqrt{r}}\end{split}\end{equation}
In other words, we simply estimate $J_{2}(x)$ by $\frac{C}{\sqrt{x}}$ globally, even though $J_{2}(x)$ is significantly smaller for small $x$. The identical procedure is used for all other derivatives of $v_{1}$, resulting in the following.
\begin{equation}\label{v1sqrt} |\partial_{t}^{j}\partial_{r}^{k} v_{1}(t,r)| \leq \frac{C}{\sqrt{r}}, \quad r \geq \frac{t}{2}, \quad 1 \leq j+k \leq 2\end{equation} Next, we start with
\begin{equation}\begin{split} \partial_{t}v_{1}(t,r) &= \int_{0}^{\infty} J_{2}(r\xi) \xi \cos(t\xi) \widehat{v}_{1,1}(\xi) d\xi\end{split}\end{equation}
Using $|J_{2}(x)| \leq C x^{2}, \quad x \leq 1$, and the large $x$ asymptotics of $J_{2}(x)$, namely
\begin{equation} J_{2}(x) = -\sqrt{\frac{2}{\pi x}} \cos(\frac{\pi}{4}-x) + O\left(\frac{1}{x^{3/2}}\right), \quad x \rightarrow\infty,\end{equation}
we get
\begin{equation}\label{dtv1nearconeexp} \partial_{t}v_{1}(t,r) = \text{Err}(t,r) +\frac{F(t-r)}{\sqrt{r}}\end{equation}
with
$$|\text{Err}(t,r)| \leq \frac{C}{r \log^{b}(r)}, \quad r \geq \frac{t}{2}$$
and
$$F(x) = \frac{-1}{2 \sqrt{\pi}} \int_{0}^{\infty} \sqrt{\xi} \widehat{v_{1,1}}(\xi)\left(\cos(x \xi) - \sin(x\xi)\right) d\xi.$$
Then, we make the change of variable $\xi =\omega^{2}$ to get
$$ \int_{0}^{\infty} \cos(|t-r|\xi) \sqrt{\xi} \widehat{v_{1,1}}(\xi) d\xi = \int_{-\infty}^{\infty} \cos(|t-r|\omega^{2}) \widehat{v_{1,1}}(\omega^{2}) \omega^{2} d\omega.$$
We directly estimate the integral in the region $\omega^{2}|t-r|\lesssim 1$, and integrate by parts in $\omega$ (integrating $\omega \cos(|t-r| \omega^{2})$) otherwise. This leads to
\begin{equation}\label{Fcosterm}|\int_{0}^{\infty} \cos(|t-r|\xi) \sqrt{\xi} \widehat{v_{1,1}}(\xi) d\xi| \leq \frac{C}{\sqrt{|t-r|} \log^{b}(|t-r|)}, |t-r| \geq 50.\end{equation}
We then use \eqref{dtv1nearconeexp}, \eqref{Fcosterm}, along with \eqref{v1sqrt} in the region $|r-t| < 50$. The $\sin$ term in the expression for $F$, and the other derivatives of $v_{1}$ are treated similarly.\end{proof}

\subsection{Estimates on $v_{2}$, the first iterate}\label{v2section}
$v_{2}$ is defined as the solution to
\begin{equation}-\partial_{tt}v_{2}+\partial_{rr}v_{2}+\frac{1}{r}\partial_{r}v_{2}-\frac{4}{r^{2}}v_{2}=\frac{6 Q_{\frac{1}{\lambda(t)}}(r)}{r^{2}}v_{1}^{2}+\frac{2}{r^{2}} v_{1}^{3}:=RHS_{2}(t,r)\end{equation}
with zero Cauchy data at infinity.
Now, we record estimates on $RHS_{2}$, and its various derivatives. 
\begin{lemma}\label{rhs2lemma}For $0 \leq j, k \leq 2$,
\begin{equation} |\partial_{t}^{j}\partial_{r}^{k} RHS_{2}(t,r)| \leq \frac{C r^{2-k}}{t^{4+j} \log^{2b}(t)}, \quad r \leq \frac{t}{2}.\end{equation}
\begin{equation} |RHS_{2}(t,r)| \leq  \frac{C}{r^{2} \log^{2b}(r)}, \quad r \geq \frac{t}{2}\end{equation}
For $1 \leq j+k \leq 2$, and $j,k \geq 0$,
\begin{equation}\label{firstderivsrhs2larger} |\partial_{t}^{j}\partial_{r}^{k}RHS_{2}(t,r)| \leq  \frac{C}{r^{5/2}  \langle t-r \rangle^{-\frac{1}{2}+j+k} \log^{b}(r)  \log^{b}(\langle t-r\rangle)}, \quad r \geq \frac{t}{2} \end{equation}
If $3 \leq j+k$ and $0 \leq j, k \leq 2$, then,
\begin{equation} |\partial_{t}^{j} \partial_{r}^{k} RHS_{2}(t,r)| \leq \frac{C}{r^{5/2} \langle t-r \rangle^{\frac{1}{2}+j+k-1} \log^{2b}(\langle t-r \rangle)}, \quad r \geq \frac{t}{2}.\end{equation}
Let $s\geq s_{0} \geq T_{0}$, and $\frac{s_{0}}{2} \leq r_{0} < s_{0}$. Then,
\begin{equation}\label{d2rhs2}\begin{split}&||\left(\partial_{r}+\frac{2}{r}\right)\partial_{s}RHS_{2}(s,r)\mathbbm{1}_{\leq 0}(r-(s-s_{0}+r_{0}))||_{L^{2}(r dr)} + ||\partial_{s}^{2} RHS_{2}(s,r) \mathbbm{1}_{\leq 0}(r-(s-s_{0}+r_{0}))||_{L^{2}(r dr)} \\
&\leq \frac{C}{s^{2} \langle s_{0}-r_{0}\rangle\log^{b}(s)\log^{b}(\langle s_{0}-r_{0}\rangle) }. \end{split}\end{equation}
\end{lemma}
\begin{proof} The estimates in the lemma follow from elementary manipulations using Lemma \ref{v1lemma}. The only important features to note are the following. Note that, although the expression for (for instance) $\partial_{tr}RHS_{2}$ includes a term involving $\partial_{t}v_{1} \partial_{r}v_{1}$, and estimates for both $\partial_{t}v_{1}$ and $\partial_{r}v_{1}$ only have a factor of $\frac{1}{\log^{b}(\langle t-r \rangle)}$, as opposed to a factor of $\frac{1}{\log^{b}(r)}$, we still obtain the stated estimates above. This is because
$$\frac{1}{\sqrt{r} \log^{b}(\langle t-r \rangle)} \leq \frac{C}{\log^{b}(r) \sqrt{\langle t-r \rangle}}, \quad r > \frac{t}{2}$$
which can be proven by noting that
$$x \mapsto \frac{\sqrt{x}}{\log^{b}(x)} \text{ is increasing for } x > e^{2b}, \text{ and } \frac{1}{r} \leq \frac{1}{|t-r|}, \quad r > \frac{t}{2}.$$
In addition, we remark that, for any $a>0$, there exists $C_{a}>0$ such that 
\begin{equation}\label{lambdagrowth}\lambda(t) \leq C_{a}t^{a}\end{equation} 
(which follows from $\frac{|\lambda'(t)|}{\lambda(t)} \leq \frac{C}{t \log^{b}(t)}$ and $b > \frac{2}{3}$). This is used (for some fixed, sufficiently small $a>0$) to estimate some terms involving $t$ derivatives of $Q_{\frac{1}{\lambda(t)}}(r)$.
\end{proof}
We note one more useful estimate. By the definition of $v_{2}$, and $L^{2}$ isometry property of the Hankel transform of order 2, we have
\begin{equation}\label{enest}\begin{split} |v_{2}(t,r)| &= |\int_{t}^{\infty}  \int_{0}^{\infty}  \sin((t-s)\xi) \widehat{RHS_{2}}(s,\xi) J_{2}(r\xi)d\xi ds| \\
&\leq C \int_{t}^{\infty}  \left(\int_{0}^{\frac{1}{r}}  r^{2} \xi^{2} |\widehat{RHS_{2}}(s,\xi)|d\xi + \int_{\frac{1}{r}}^{\infty}  \frac{|\widehat{RHS_{2}}(s,\xi)|}{\sqrt{r\xi}}d\xi\right)ds\\
&\leq C \int_{t}^{\infty}  \left(||RHS_{2}(s)||_{L^{2}(r dr)}  r^{2} \left(\int_{0}^{\frac{1}{r}}  \xi^{3}d\xi\right)^{1/2} + \frac{C}{\sqrt{r}} ||RHS_{2}(s)||_{L^{2}(r dr)} \left(\int_{\frac{1}{r}}^{\infty} \frac{d\xi}{\xi^{2}}\right)^{1/2}\right)ds\\
&\leq C \int_{t}^{\infty}  ||RHS_{2}(s)||_{L^{2}(r dr)}ds. \end{split}\end{equation}
Then, we use an $8$ step procedure to estimate all quantities related to $v_{2}$:
\begin{lemma} \label{v2lemma} We have the following estimates on $v_{2}$. For $0 \leq j, k \leq 2$,
\begin{equation}\label{v2nearorigin} |\partial_{t}^{j}\partial_{r}^{k}v_{2}(t,r)| \leq C \frac{r^{2-k}}{t^{2+j} \log^{2b}(t)}, \quad r \leq \frac{t}{2}.\end{equation}
\begin{equation}\label{v2est} |v_{2}(t,r)|+r |\partial_{t}v_{2}(t,r)| + r |\partial_{r}v_{2}(t,r)| \leq \frac{C}{\log^{2b}(t)},\quad r > \frac{t}{2}\end{equation}
\begin{equation}\label{d2v2est} |\partial_{tr}v_{2}(t,r)| + |\partial_{r}^{2}v_{2}(t,r)|+|\partial_{t}^{2}v_{2}(t,r)| \leq \frac{C}{t\langle t-r\rangle \log^{b}(t) \log^{b}(\langle t-r\rangle )}, \quad t > r >\frac{t}{2}\end{equation}
\begin{equation}\label{d2v2linfty} ||\partial_{tr}v_{2}(t,r)||_{L^{\infty}_{r}(\{r \geq \frac{t}{2}\})} + ||\partial_{t}^{2}v_{2}(t,r)||_{L^{\infty}_{r}(\{r \geq \frac{t}{2}\})} + ||\partial_{r}^{2} v_{2}(t,r)||_{L^{\infty}_{r}(\{r \geq \frac{t}{2}\})} \leq \frac{C}{t^{3/2} \log^{b}(t)}, \end{equation}
For $j+k \geq 3$ and $0 \leq j, k \leq 2$, 
\begin{equation} \label{d3v2est} |\partial_{t}^{j}\partial_{r}^{k}v_{2}(t,r)| \leq \frac{C}{\sqrt{t} \langle t-r \rangle ^{\frac{1}{2}+j+k-1} \log^{2b}(\langle t-r \rangle)}, \quad t > r > \frac{t}{2}.\end{equation}
\end{lemma}
\begin{proof}
\textbf{Step 1}: We use the fact that, if $v_{2}=r^{2}\widetilde{w_{2}}$, then, $\widetilde{w_{2}}$ solves the following equation, with zero Cauchy data at infinity.
$$-\partial_{t}^{2}\widetilde{w_{2}}+\partial_{r}^{2}\widetilde{w_{2}}+\frac{5}{r}\partial_{r}\widetilde{w_{2}}=\frac{RHS_{2}(t,r)}{r^{2}}$$
We then estimate $v_{2}$ in the region $r \leq \frac{t}{2}$ by using Duhamel's principle, and the $6+1$ dimensional spherical means formula, as follows.
\begin{equation}\label{wforv2form} \widetilde{w_{2}}(t,r) = -\int_{t}^{\infty} f_{2}(s-t,r \textbf{e}_{1}) ds\end{equation}
and we use the spherical means formula for $f_{2}$, namely (see, for instance, \cite{evans})
$$f_{2}(t,x) = \frac{1}{8 \pi^{3}} \left(\frac{1}{t}\partial_{t}\right)^{2} \int_{B_{t}(0)}\frac{RHS_{2}(s,|x+y|)}{|x+y|^{2} \sqrt{t^{2}-|y|^{2}}} dy = \frac{1}{8 \pi^{3}} \left(\frac{1}{t}\partial_{t}\right)^{2}\left(t^{5} \int_{B_{1}(0)} \frac{RHS_{2}(s,|x+t z|)}{|x+tz|^{2} \sqrt{1-|z|^{2}}} dz\right)$$
where, to ease notation, we write $x=r \textbf{e}_{1} \in \mathbb{R}^{6}$. We then differentiate under the integral sign in the equation above, and let $z=\frac{y}{t}$. Using spherical  coordinates,
\begin{equation}\label{ydefstep1}y=\rho(\cos(\phi),\sin(\phi)\cos(\phi_{2}),\dots,\sin(\phi)\sin(\phi_{2})\cdot \dots \cdot \sin(\phi_{5})) \in \mathbb{R}^{6}\end{equation}
and recalling \eqref{wforv2form}, we get
\begin{equation}\begin{split} |v_{2}(t,r)| \leq C r^{2} &\int_{t}^{\infty}\frac{ds}{(s-t)^{4}}\int_{0}^{s-t} \frac{\rho^{5}d\rho}{\sqrt{(s-t)^{2}-\rho^{2}}} \int_{0}^{\pi} I\text{ } d\phi\end{split}\end{equation}
where
\begin{equation}\begin{split}I=  \frac{\sin^{4}(\phi)}{|x+y|^{2}}&\left(|RHS_{2}(s,|x+y|)|\left(1+\frac{\rho^{2}}{|x+y|^{2}}\right)+|\partial_{2}RHS_{2}(s,|x+y|)|\rho\left(1+\frac{\rho}{|x+y|}\right)\right.\\
&\left.+|\partial_{2}^{2} RHS_{2}(s,|x+y|)| \rho^{2}\right)\end{split}\end{equation} 
and we note that $|x+y| = \sqrt{\rho^{2}+r^{2}+2 r \rho\cos(\phi)}.$ By using Cauchy's residue theorem appropriately, we get, for all $\rho \neq r$,
\begin{equation}\label{phiint} \int_{0}^{\pi} \frac{\sin^{4}(\phi) d\phi}{\rho^{2}+r^{2}+2 r \rho\cos(\phi)} = \frac{1}{2} \int_{0}^{2\pi} \frac{\sin^{4}(\phi) d\phi}{\rho^{2}+r^{2}+2 r \rho\cos(\phi)} = \frac{-\pi \left(\text{min}\{r,\rho\}^{2}-3\text{max}\{r,\rho\}^{2}\right)}{8(\text{max}\{\rho,r\})^{4}}.\end{equation}
Then, we use the estimates for $RHS_{2}$ from Lemma \ref{rhs2lemma} to get \eqref{v2nearorigin} for $j=k=0$.\\
\\
\textbf{Step 2}: To estimate $\partial_{r}v_{2}$ in the region $r \leq \frac{t}{2}$, we first use the fact that, if $r u_{2}:=\left(\partial_{r}+\frac{2}{r}\right) v_{2} $, then, $u_{2}$ solves
$$-\partial_{tt}u_{2}+\partial_{rr}u_{2}+\frac{3}{r}\partial_{r}u_{2} = \frac{1}{r}\left(\partial_{r}+\frac{2}{r}\right)RHS_{2}(t,r)$$ 
with zero Cauchy data at infinity. Then, we use the spherical means formula for $u_{2}$. We first get
\begin{equation}|\left(\partial_{r}+\frac{2}{r}\right)v_{2}| \leq C r \int_{t}^{\infty} ds \frac{1}{(s-t)^{2}} \int_{0}^{s-t} \frac{\rho^{3} d\rho}{\sqrt{(s-t)^{2}-\rho^{2}}} \int_{0}^{\pi} \sin^{2}(\phi)  I_{2} d\phi\end{equation}
where
\begin{equation}\begin{split}I_{2}= &\left(\frac{|\partial_{2}^{2}RHS_{2}(s,|x+y|)|\rho}{|x+y|} + \frac{|\partial_{2}RHS_{2}(s,|x+y|)|}{|x+y|}\left(1+\frac{\rho}{|x+y|}\right)\right.\\
&\left.+\frac{|RHS_{2}(s,|x+y|)|}{|x+y|^{2}}\left(1+\frac{\rho}{|x+y|}\right)\right).\end{split}\end{equation}
We then use \eqref{phiint} to get
\begin{equation} \int_{0}^{\pi} \frac{\rho \sin^{2}(\phi) d\phi}{\sqrt{\rho^{2}+r^{2}+2 r \rho\cos(\phi)}} \leq C \int_{0}^{\pi}\left(1+\frac{\sin^{4}(\phi) \rho^{2}}{\rho^{2}+r^{2}+2 r \rho\cos(\phi)}\right)d\phi \leq C\end{equation}
which gives
\begin{equation} |\left(\partial_{r}+\frac{2}{r}\right)v_{2}(t,r)| \leq \frac{C r}{t^{2} \log^{2b}(t)}, \quad r \leq \frac{t}{2}.\end{equation}
\textbf{Step 3}: We estimate $\partial_{r}^{2} v_{2}$ in the region $r \leq \frac{t}{2}$ using the fact that, if $z_{2}:=\left(\partial_{r}+\frac{1}{r}\right)\left(\partial_{r}+\frac{2}{r}\right) v_{2}$, then $z_{2}$ solves
$$-\partial_{tt}z_{2}+\partial_{rr}z_{2}+\frac{1}{r}\partial_{r}z_{2}=\left(\partial_{r}+\frac{1}{r}\right)\left(\partial_{r}+\frac{2}{r}\right) RHS_{2}(t,r)$$
with zero Cauchy data at infinity, and using the spherical means formula for $z_{2}$. The details for this step are very similar to those of Steps 1 and 2.\\
\\
\textbf{Step 4}: Differentiating the formulae for $v_{2}, r u_{2}$, and $z_{2}$ with respect to $t$, we show that, for $j=1,2$, $\partial_{t}^{j}v_{2}$ solves the same equation as $v_{2}$, except with $\partial_{t}^{j}RHS_{2}$ on the right-hand side, and zero Cauchy data at infinity. Then, we use the same procedure as in steps 1-3 to  obtain estimates on all remaining derivatives of $v_{2}$ of the form $\partial_{t}^{j}\partial_{r}^{k} v_{2}$ in the region $r \leq \frac{t}{2}$, for $0 \leq j, k \leq 2$.\\
\\
\textbf{Step 5}: Next, we estimate $\left(\partial_{r}+\frac{2}{r}\right)v_{2}$ in the region $r \geq \frac{t}{2}$, using a slightly different representation formula than what was used in Step 2. Using the Fundamental Theorem of Calculus, we then estimate $v_{2}$ in the region $r \geq \frac{t}{2}$.\\
\\
We first consider the case $r > 2t$, and note that $p_{2}(t,r):=\left(\partial_{r}+\frac{2}{r}\right) v_{2}(t,r)$ solves
$$-\partial_{t}^{2}p_{2}+\partial_{r}^{2}p_{2}+\frac{1}{r}\partial_{r}p_{2}-\frac{p_{2}}{r^{2}}=\left(\partial_{r}+\frac{2}{r}\right)RHS_{2}(t,r)$$
with zero Cauchy data at infinity. Then, we use the following procedure to get a representation formula for $p_{2}$  which does not involve any derivatives of $\left(\partial_{r}+\frac{2}{r}\right) RHS_{2}$:
Let $g_{2}:[T_{0},\infty) \times \mathbb{R}^{2} \rightarrow \mathbb{R}$ be defined by
$$g_{2}(t,r\cos(\theta),r\sin(\theta)) = p_{2}(t,r) \cos(\theta).$$
Then,
\begin{equation}\begin{split}\left(-\partial_{t}^{2} +\Delta_{\mathbb{R}^{2}}\right)g_{2}(t,r\cos(\theta),r\sin(\theta)) &= \cos(\theta)\left(-\partial_{t}^{2}p_{2}+\partial_{r}^{2}p_{2}+\frac{1}{r}\partial_{r}p_{2}-\frac{1}{r^{2}}p_{2}\right)=\cos(\theta) \left(\partial_{r}+\frac{2}{r}\right)RHS_{2}(t,r).\end{split}\end{equation}
We then use Duhamel's principle and the 2+1 dimensional spherical means formula to get
\begin{equation} \begin{split} g_{2}(t,x) &= -\frac{1}{2\pi} \int_{t}^{\infty} ds \int_{B_{s-t}(0)} \left(\frac{x_{1}+y_{1}}{|x+y|}\right) \frac{\left(\partial_{2}+\frac{2}{|x+y|}\right)RHS_{2}(s,|x+y|)}{\sqrt{(s-t)^{2}-|y|^{2}}} dy.\end{split}\end{equation}
where we write $x=(x_{1},x_{2}) \in \mathbb{R}^{2}$. Finally, we get
\begin{equation}\label{p2repform} p_{2}(t,r)=g_{2}(t,r,0)= \frac{-1}{2\pi} \int_{t}^{\infty} ds \int_{0}^{s-t}\frac{\rho d\rho}{\sqrt{(s-t)^{2}-\rho^{2}}} \int_{0}^{2\pi} d\theta \frac{\left(\partial_{2}+\frac{2}{|x+y|}\right)RHS_{2}(s,|x+y|)}{|x+y|} \left(\widehat{x}\cdot(x+y)\right)\end{equation}
where we now regard $x=r \textbf{e}_{1} \in \mathbb{R}^{2}$, $y=(\rho \cos(\theta),\rho\sin(\theta))$, and we have $|x+y| = \sqrt{r^{2}+\rho^{2}+2 r \rho \cos(\theta)}.$ Then, we treat several pieces of $p_{2}$ separately. Precisely, we make the following definitions. 
\begin{equation}\label{p2Idef}p_{2,I}(t,r):=\frac{-1}{2\pi} \int_{t}^{\infty} ds \int_{0}^{s-t}\frac{\rho d\rho}{\sqrt{(s-t)^{2}-\rho^{2}}} \int_{0}^{2\pi} d\theta \frac{\left(\partial_{2}+\frac{2}{|x+y|}\right)RHS_{2}(s,|x+y|)}{|x+y|} \left(\widehat{x}\cdot(x+y)\right)\mathbbm{1}_{\{|x+y| \leq \frac{s}{2}\}}\end{equation}
The term $p_{2,I}$ is defined this way simply because our estimates for $RHS_{2}(t,r)$ and its derivatives (from Lemma \ref{rhs2lemma}) are recorded in the regions $r \leq \frac{t}{2}$ and $r \geq \frac{t}{2}$ separately.\\
\\
We then define $p_{2,II}$ by $p_{2,II} := p_{2}-p_{2,I}$. We have $p_{2,II}=p_{2,II,a}+p_{2,II,b}+p_{2,II,c}$, for the following definitions:
\begin{equation} \label{p2IIadef} \begin{split}&p_{2,II,a}(t,r) \\
&= \frac{-1}{2\pi} \int_{t}^{\infty} ds \int_{0}^{s-t}\frac{\rho d\rho}{\sqrt{(s-t)^{2}-\rho^{2}}} \int_{0}^{2\pi} d\theta \frac{\left(\partial_{2}+\frac{2}{|x+y|}\right)RHS_{2}(s,|x+y|)}{|x+y|} \left(\widehat{x}\cdot(x+y)\right)\mathbbm{1}_{\{s-t-r>|x+y| > \frac{s}{2}\}}\end{split}\end{equation}
The point of the definition of $p_{2,II,a}$ is the following. We insert the estimate \eqref{firstderivsrhs2larger} into the integrand of $p_{2,II,a}$, and note that, by the support properties of the integrand, $\frac{1}{\sqrt{\langle s-|x+y|\rangle}} \leq \frac{C}{\sqrt{r+t}}$; hence, we can use a relatively simple argument to study $p_{2,II,a}$.
\begin{equation}\label{p2IIbdef}\begin{split}&p_{2,II,b}(t,r) \\
&= \frac{-1}{2\pi} \int_{t}^{t+\frac{r}{8}} ds \int_{0}^{s-t}\frac{\rho d\rho}{\sqrt{(s-t)^{2}-\rho^{2}}} \int_{0}^{2\pi} d\theta \frac{\left(\widehat{x}\cdot(x+y)\right)\left(\partial_{2}+\frac{2}{|x+y|}\right)RHS_{2}(s,|x+y|)}{|x+y|} \mathbbm{1}_{\{|x+y| > \frac{s}{2}\}} \mathbbm{1}_{\{|x+y| \geq s-t-r\}}\end{split}\end{equation}
The reason for truncating the $s$ integral in the definition of $p_{2,II,b}$ is also because of additional smallness of the factor $\frac{1}{\sqrt{\langle s-|x+y|\rangle}}$ in this region.\\
\\
Finally, we are left with the last piece of $p_{2}$, which we denote $p_{2,II,c}$.
$$p_{2,II,c}=p_{2}-p_{2,I}-p_{2,II,a}-p_{2,II,b}.$$
\begin{equation}\label{p2IIcdef}\begin{split}&p_{2,II,c}(t,r)\\
&=\frac{-1}{2\pi} \int_{t+\frac{r}{8}}^{\infty} ds \int_{0}^{s-t}\frac{\rho d\rho}{\sqrt{(s-t)^{2}-\rho^{2}}} \int_{0}^{2\pi} d\theta \frac{\left(\partial_{2}+\frac{2}{|x+y|}\right)RHS_{2}(s,|x+y|)}{|x+y|} \left(\widehat{x}\cdot(x+y)\right)\mathbbm{1}_{\{|x+y| > \frac{s}{2}\}} \mathbbm{1}_{\{|x+y| \geq s-t-r\}}\\
& = \frac{-1}{2\pi} \int_{t+\frac{r}{8}}^{\infty} ds \int_{B_{s-t}(0) \cap (B_{\frac{s}{2}} (-x))^{c} \cap (B_{s-t-r}(-x))^{c}}\frac{dA(y)}{\sqrt{(s-t)^{2}-|y|^{2}}} \frac{\left(\partial_{2}+\frac{2}{|x+y|}\right)RHS_{2}(s,|x+y|)}{|x+y|} \left(\widehat{x}\cdot(x+y)\right)\end{split}\end{equation}
We start by estimating $p_{2,I}$, which we recall is given by \eqref{p2Idef}. We first note that the integrand in the definition of $p_{2,I}$ vanishes unless $s \geq \frac{2}{3}(t+r)$: If $s < \frac{2}{3}(t+r)$ and $|x+y| \leq \frac{s}{2}$, then, $\rho = |y| \geq |x|-|x+y| \geq r-\frac{s}{2} \geq \frac{-t+2r}{3} > s-t$. On the other hand, the $\rho$ integration is constrained to the region $\rho \leq s-t$. 
So,
$$|p_{2,I}(t,r)| \leq C \int_{\frac{2}{3}(t+r)}^{\infty} ds \int_{0}^{s-t} \frac{\rho d\rho}{\sqrt{(s-t)^{2}-\rho^{2}}} \int_{0}^{2\pi} \frac{d\theta}{s^{3} \log^{2b}(s)} \leq \frac{C}{r \log^{2b}(r)}.$$
We recall the definition of $p_{2,II,a}$ in \eqref{p2IIadef}. For $|x+y|$ in the support of the characteristic functions appearing in \eqref{p2IIadef}, $|s-|x+y|| = s-|x+y| \geq s-(s-t-r) = t+r.$ Finally, the integrand vanishes unless $s > 2(t+r)$. Then, we use the estimates on $RHS_{2}$ from Lemma \ref{rhs2lemma} to get
$$|p_{2,II,a}(t,r)| \leq C \int_{2(t+r)}^{\infty} ds \int_{0}^{s-t} \frac{\rho d\rho}{\sqrt{(s-t)^{2}-\rho^{2}}} \int_{0}^{2\pi} d\theta \frac{1}{s^{5/2} \sqrt{r} \log^{2b}(r)} \leq \frac{C}{r \log^{2b}(r)}.$$
Next, we treat $p_{2,II,b}$ (defined in \eqref{p2IIbdef}). This time, for $|x+y|$ in the support of the characteristic functions appearing in \eqref{p2IIbdef}, we have
$$|s-|x+y|| = |x+y|-s \geq |x|-|y|-s \geq r-(s-t)-s \geq r-2t-\frac{r}{4} + t \geq \frac{r}{4} + \frac{r}{2} -t \geq \frac{r}{4}.$$
Also, $|x+y| \geq |x|-|y| \geq r-(s-t) \geq \frac{7r}{8}$. So, we get
$$|p_{2,II,b}(t,r)| \leq C \int_{t}^{t+\frac{r}{8}} ds \int_{0}^{s-t} \frac{\rho d\rho}{\sqrt{(s-t)^{2}-\rho^{2}}} \int_{0}^{2\pi} d\theta \frac{1}{r^{5/2} \sqrt{r} \log^{2b}(r)} \leq \frac{C}{r \log^{2b}(r)}$$
Finally, we estimate $p_{2,II,c}$, whose definition is reproduced here, for the reader's convenience.
\begin{equation}\begin{split}&p_{2,II,c}(t,r)\\
& = \frac{-1}{2\pi} \int_{t+\frac{r}{8}}^{\infty} ds \int_{B_{s-t}(0) \cap (B_{\frac{s}{2}} (-x))^{c} \cap (B_{s-t-r}(-x))^{c}}\frac{dA(y)}{\sqrt{(s-t)^{2}-|y|^{2}}} \frac{\left(\partial_{2}+\frac{2}{|x+y|}\right)RHS_{2}(s,|x+y|)}{|x+y|} \left(\widehat{x}\cdot(x+y)\right)\end{split}\end{equation}
We will prove estimates on $p_{2,II,c}$ which are valid for any $r \geq \frac{t}{2}$ for later use, even though we assumed $r > 2t$ in the very beginning of this argument. Note that the intersection of the balls in the integral is empty, unless $s \geq 2(t-r)$, since $\frac{s}{2} \leq |x+y| \leq s-t+r$. Then, we use polar coordinates centered at $x$. More precisely, we write $z=y+x = (\rho\cos(\theta),\rho\sin(\theta))$.  The integrand of the $s$ integral above is then bounded above in absolute value by
$$ C \int_{|r-(s-t)|}^{s-t+r} \rho d\rho \int_{0}^{\theta^{*}} \frac{|\left(\partial_{2}+\frac{2}{\rho}\right) RHS_{2}(s,\rho)| \mathbbm{1}_{\{\rho > \max{\left(\frac{s}{2}, s-t-r\right)}\}} d\theta}{\sqrt{(s-t)^{2}-r^{2}-\rho^{2}+2 r \rho \cos(\theta)}}$$
where $$\theta^{*} = \arccos\left(\frac{\rho^{2}+r^{2}-(s-t)^{2}}{2 r \rho}\right).$$
To get this, we first used the inequality $\frac{|\left(\widehat{x}\cdot(x+y)\right)|}{|x+y|} \leq 1$. Then, the only $\theta$-dependent term remaining in the integrand of the $s$ integral in the expression for $p_{2,II,c}$ is 
$$\frac{1}{\sqrt{(s-t)^{2}-|z-x|^{2}}} = \frac{1}{\sqrt{(s-t)^{2}-\rho^{2}-r^{2}+2 r \rho \cos(\theta)}}.$$ 
Next, we used the facts that $\cos(2\pi -\theta) = \cos(\theta)$, and the integrand is supported in the region $\rho > s-t-r$. Note also that $\theta^{*}$ is well-defined, for all $\rho$ in the region of integration in the expression above, and $0 \leq \theta^{*} \leq \pi$. Then, we note
\begin{equation} \int_{0}^{\theta^{*}} \frac{d\theta}{\sqrt{(s-t)^{2}-r^{2}-\rho^{2}+2 r \rho \cos(\theta)}} = \frac{1}{\sqrt{2 r \rho}} \int_{0}^{\theta^{*}} \frac{d\delta}{\sqrt{\cos(\theta^{*})\left(\cos(\delta)-1\right) + \sin(\theta^{*}) \sin(\delta)}}=\frac{1}{\sqrt{2 r \rho}}f(\theta^{*})\end{equation}
where we made the substitution $\delta =  \theta^{*}-\theta$ in the first integral. We then have
$$|f(\theta^{*})| \leq C \langle \log(\pi - \theta^{*})\rangle \leq \frac{C}{\sqrt{\pi-\theta^{*}}}.$$
(Although the singularity of $f$ as $\theta^{*}$ approaches $\pi$ is much better than $\frac{1}{\sqrt{\pi-\theta^{*}}}$, the above inequality suffices for our purposes, and slightly simplifies some of our estimates). Using
$$\pi-\theta^{*} = \arccos\left(1+\frac{(s-t)^{2}-(\rho+r)^{2}}{2 \rho r}\right)$$
we get
\begin{equation} |\int_{0}^{\theta^{*}} \frac{d\theta}{\sqrt{(s-t)^{2}-r^{2}-\rho^{2}+2 r \rho \cos(\theta)}}| \leq \frac{C}{(r\rho)^{1/4} (\rho+r-(s-t))^{1/4} (\rho+r+s-t)^{1/4}}.\end{equation} So, we have
\begin{equation}\begin{split}&|p_{2,II,c}(t,r)| \\
&\leq C \int_{\max\{t+\frac{r}{8},2(t-r)\}}^{\infty} ds \int_{\max\{s-t-r,\frac{s}{2}\}}^{s-t+r} \rho d\rho |\left(\partial_{2}+\frac{2}{\rho}\right) RHS_{2}(s,\rho)|  \int_{0}^{\theta^{*}} \frac{d\theta}{\sqrt{(s-t)^{2}-\rho^{2}-r^{2}+2 r \rho \cos(\theta)}}\end{split}\end{equation}
and our above estimates give
\begin{equation}\label{p2iicintest} |p_{2,II,c}(t,r)| \leq \frac{C}{r^{1/4}} \int_{t+\frac{r}{8}}^{\infty} ds \int_{s-t-r}^{s-t+r} \frac{\sqrt{\rho} d\rho}{(\rho+r-(s-t))^{1/4}} \frac{1}{s^{5/2} \sqrt{\langle s-\rho \rangle} \log^{b}(\rho) \log^{b}(\langle s-\rho \rangle)} \mathbbm{1}_{\{\rho \geq \frac{s}{2}\}}.\end{equation}
Let 
$$p_{2,II,c,i}(t,r):=\frac{1}{r^{1/4}} \int_{t+r}^{\infty} ds \int_{s-t-r}^{s-t+r} \frac{\sqrt{\rho} d\rho}{(\rho+r-(s-t))^{1/4}} \frac{1}{s^{5/2} \sqrt{\langle s-\rho \rangle} \log^{b}(\rho) \log^{b}(\langle s-\rho \rangle)} \mathbbm{1}_{\{\rho \geq \frac{s}{2}\}}$$
and 
$$p_{2,II,c,ii}(t,r):=\frac{1}{r^{1/4}} \int_{t+\frac{r}{8}}^{t+r} ds \int_{s-t-r}^{s-t+r} \frac{\sqrt{\rho} d\rho}{(\rho+r-(s-t))^{1/4}} \frac{1}{s^{5/2} \sqrt{\langle s-\rho \rangle} \log^{b}(\rho) \log^{b}(\langle s-\rho \rangle)} \mathbbm{1}_{\{\rho \geq \frac{s}{2}\}}.$$
Note that, in the expression for $p_{2,II,c,i}$, $s \geq t+r$, so that $s-t-r >0$. Then, we  consider separately two regions of the $\rho$ integration. In the region $s-t-r \leq \rho \leq s-t-\frac{r}{2}$, we have
$$|s-\rho| = s-\rho \geq s-(s-t-\frac{r}{2}) = t+\frac{r}{2}.$$
In the region $s-t-\frac{r}{2} \leq \rho \leq s-t+r$, we have $\frac{1}{(p-(s-t-r))^{1/4}} \leq \frac{C}{r^{1/4}}.$ So,
\begin{equation}\begin{split}|p_{2,II,c,i}(t,r)|&\leq \frac{C}{r^{1/4}} \int_{t+r}^{\infty} ds \int_{s-t-r}^{s-t-\frac{r}{2}} \frac{\sqrt{\rho} d\rho}{(\rho+r-(s-t))^{1/4}} \frac{1}{s^{5/2} \sqrt{r} \log^{2b}(r)}\\
&+\frac{C}{r^{1/4}} \int_{t+r}^{\infty} ds \int_{s-t-\frac{r}{2}}^{s-t+r} \frac{\sqrt{\rho}d\rho}{r^{1/4} s^{5/2} \sqrt{\langle s-\rho \rangle} \log^{b}(s)\log^{b}(\langle s-\rho \rangle)}.\end{split}\end{equation}
Then, we use
$$\int_{t-r}^{t+\frac{r}{2}} \frac{dx}{\sqrt{\langle x\rangle} \log^{b}(\langle x\rangle)} \leq C \frac{\sqrt{\langle t-r \rangle}}{\log^{b}(\langle t-r\rangle)}+C \frac{\sqrt{t+r}}{\log^{b}(t+r)}$$
along with $|t-r| \leq r, \quad r \geq \frac{t}{2}$ to get 
$$|p_{2,II,c,i}(t,r)| \leq \frac{C}{r \log^{2b}(r)}.$$
On the other hand, in the expression for $p_{2,II,c,ii}$, $s < t+r$. So, we make use of the $\mathbbm{1}_{\{\rho \geq \frac{s}{2}\}}$ in the integrand of \eqref{p2iicintest}. Also, in this case
$$(\rho+r-(s-t))^{1/4} \geq \rho^{1/4}.$$
This gives
$$|p_{2,II,c,ii}(t,r)| \leq \frac{C}{r^{1/4}} \int_{t+\frac{r}{8}}^{t+r} ds \int_{\frac{s}{2}}^{s-t+r} \frac{\rho^{1/4} d\rho}{s^{5/2} \sqrt{\langle s-\rho \rangle }\log^{b}(\rho) \log^{b}(\langle s-\rho \rangle)}$$
which can be treated in the same way as we treated $p_{2,II,c,i}$. In the very beginning of this argument, we considered the region $r > 2t$. This was so that we could estimate $p_{2,II,b}$. If $\frac{t}{2} \leq r \leq 2t$, then, we instead decompose $p_{2,II}$ as
$$p_{2,II}(t,r) = p_{2,II,a}(t,r) + p_{2,II,d}(t,r).$$
We then estimate $p_{2,II,d}$ with the identical procedure used to estimate $p_{2,II,c}$. We obtain the same final estimate for $p_{2,II,d}$ as we did for $p_{2,II,c}$. Even though we only have $s \geq t$ in the integral defining $p_{2,II,d}$, as opposed to $s \geq t+\frac{r}{8}$ for $p_{2,II,c}$, the fact that $\frac{t}{2} \leq r \leq 2t$ ensures that we do indeed get the same final estimate for $p_{2,II,d}$. In total, we finally get
$$|p_{2}(t,r)| \leq \frac{C}{r \log^{2b}(r)}, \quad r \geq \frac{t}{2}.$$
Then, we recover $v_{2}$ from $p_{2}$:
$$v_{2}(t,r) = \frac{1}{r^{2}} \int_{0}^{r} p_{2}(t,x) x^{2} dx$$
If $r \geq \frac{t}{2}$, then,
$$|v_{2}(t,r)| \leq \frac{C}{r^{2}} \int_{0}^{\frac{t}{2}} \frac{x^{3} dx}{t^{2} \log^{2b}(t)} + \frac{C}{r^{2}} \int_{\frac{t}{2}}^{r} \frac{x dx}{\log^{2b}(x)}$$
and we finally get
$$|v_{2}(t,r)| \leq \frac{C}{\log^{2b}(t) }, \quad r \geq \frac{t}{2}.$$
\textbf{Step 6}: Similarly, we estimate $\partial_{t}v_{2}$ in the region $r \geq \frac{t}{2}$, using the fact that it solves the same equation as $v_{2}$, except with $\partial_{t}RHS_{2}$ on the right-hand side. In particular, we first note that, if $u$ solves
$$-\partial_{tt}u(t,r)+\partial_{rr}u(t,r)+\frac{1}{r}\partial_{r}u(t,r)-\frac{4}{r^{2}} u(t,r)=F(t,r), \quad t \geq T_{0},\quad r>0$$
and $w:[T_{0},\infty) \times \mathbb{R}^{2}$ is defined by
$$w(t,r \cos(\theta),r \sin(\theta))=u(t,r) \cos(2\theta), \quad r>0$$
then, $w$ solves
$$-\partial_{tt}w+\Delta_{\mathbb{R}^{2}} w =\left(2\frac{x_{1}^{2}}{|x|^{2}}-1\right) F(t,|x|), \quad x \neq 0.$$
Now, we apply this procedure to the case $u=\partial_{t}v_{2}$, and $F=\partial_{t}RHS_{2}$. Using $u(t,r) = w(t,r,0)$, we get
\begin{equation}\label{dtv2noderiv}\begin{split}&\partial_{t}v_{2}(t,r) = \frac{-1}{2\pi} \int_{t}^{\infty} ds \int_{0}^{s-t} \frac{\rho d\rho}{\sqrt{(s-t)^{2}-\rho^{2}}} \int_{0}^{2\pi} d\theta  I_{dtv2}\end{split}\end{equation}
where
$$I_{dtv2}= \partial_{1}RHS_{2}(s,\sqrt{r^{2}+\rho^{2}+2 r \rho \cos(\theta)}) \left(\frac{r^{2}+2 r \rho \cos(\theta) + \rho^{2}(1-2\sin^{2}(\theta))}{r^{2}+2 r \rho \cos(\theta)+\rho^{2}}\right).$$
Because our estimates on $\partial_{t}RHS_{2}$ from Lemma \ref{rhs2lemma} are just as good (in fact, slightly better in the region $r \leq \frac{t}{2}$) as those for $\partial_{r} RHS_{2}$, we can repeat the same procedure used to estimate $p_{2}$, to get
$$|\partial_{t}v_{2}(t,r)| \leq \frac{C}{r \log^{2b}(r)}, \quad r \geq \frac{t}{2}.$$
\textbf{Step 7}: We estimate $\partial_{t}^{2} v_{2}$ and $\partial_{tr}v_{2}$ in the region $t> r \geq \frac{t}{2}$ by using a procedure based on \eqref{enest}, which also takes advantage of the finite speed of propagation. Then, we estimate $||\partial_{tt}v_{2}(t,\cdot)||_{L^{\infty}_{r}(\{r \geq \frac{t}{2}\})}$ by using the fact that $\partial_{tt}v_{2}$ solves the same equation as $v_{2}$, with zero Cauchy data at infinity, except with $\partial_{t}^{2} RHS_{2}$ on the right-hand side. We estimate $||\partial_{tr}v_{2}(t,\cdot)||_{L^{\infty}_{r}(\{r \geq \frac{t}{2}\})}$ similarly. Finally, we use the equation solved by $v_{2}$ to read off estimates on $\partial_{r}^{2} v_{2}$ in the region $t> r \geq \frac{t}{2}$, and to estimate $||\partial_{r}^{2}v_{2}(t,\cdot)||_{L^{\infty}_{r}(\{r \geq \frac{t}{2}\})}$\\
\\
Using \eqref{d2rhs2}, the finite speed of propagation, as well as an appropriate analog of \eqref{enest}, we get 
$$|\partial_{t}^{2}v_{2}(t,r)|+|\partial_{tr}v_{2}(t,r)| \leq \frac{C}{t \langle t-r \rangle \log^{b}(t) \log^{b}(\langle t-r \rangle)}, \quad t>r > \frac{t}{2}.$$
We then argue as we did for $\partial_{t}v_{2}$, to get
$$\partial_{t}^{2}v_{2}(t,r) = \frac{-1}{2\pi} \int_{t}^{\infty} ds \int_{0}^{s-t} \frac{\rho d\rho}{\sqrt{(s-t)^{2}-\rho^{2}}} \int_{0}^{2\pi} d\theta I_{dttv2}$$
where
$$I_{dttv2} = \partial_{1}^{2}RHS_{2}(s,\sqrt{r^{2}+\rho^{2}+2 r \rho \cos(\theta)}) \left(\frac{r^{2}+2 r \rho \cos(\theta) + \rho^{2}(1-2\sin^{2}(\theta))}{r^{2}+2 r \rho \cos(\theta)+\rho^{2}}\right).$$
Then, we carry out the same procedure used to estimate $p_{2}$. The difference here is that we have an extra factor of $\frac{1}{\langle t-r \rangle}$ in the pointwise estimates for $\partial_{t}^{2}RHS_{2}(t,r)$, relative to those for $\partial_{r}RHS_{2}(t,r)$ (recall Lemma \ref{rhs2lemma}). This leads to 
$$|\partial_{t}^{2} v_{2}(t,r)| \leq \frac{C}{r^{3/2} \log^{b}(t)}, \quad r > \frac{t}{2}.$$
Also, if $m_{2}=\left(\partial_{r}+\frac{2}{r}\right)\partial_{t}v_{2}$, then, $m_{2}$ solves
$$-\partial_{tt}m_{2}+\partial_{rr}m_{2}+\frac{1}{r}\partial_{r}m_{2}-\frac{1}{r^{2}} m_{2} = \left(\partial_{r}+\frac{2}{r}\right)\partial_{t} RHS_{2}(t,r)$$
with zero Cauchy data at infinity. Using the analog of the $p_{2}$ representation formula, \eqref{p2repform}, we can repeat the same argument used for $\partial_{t}^{2} v_{2}$, and use the previous estimates on $\partial_{t}v_{2}$, to get
$$|\partial_{tr}v_{2}(t,r)| \leq \frac{C}{r^{3/2} \log^{b}(t)}, \quad r > \frac{t}{2}.$$
\textbf{Step 8}: We estimate $\partial_{ttr}v_{2}$ in the region $t > r \geq \frac{t}{2}$ by using the same procedure as for $\partial_{tr}v_{2}$. Then, we estimate $\partial_{trr} v_{2}$ and $\partial_{ttrr}v_{2}$ in the region $t > r \geq \frac{t}{2}$, by using the same representation formulae for $\left(\partial_{r}+\frac{1}{r}\right)\left(\partial_{r}+\frac{2}{r}\right) \partial_{t}^{j}v_{2}$ (for $j=1,2$) as was used in step 4.
\end{proof}
\subsection{Summation of the higher iterates, $v_{j}$}\label{vjsection}
We now proceed to recursively define subsequent corrections, $v_{j}$. For $j \geq 3$, define $RHS_{j}$ by
\begin{equation}\label{rhsj}\begin{split} RHS_{j}(t,r) &=\frac{6 Q_{\frac{1}{\lambda(t)}}(r)}{r^{2}} \left(\left(\sum_{k=1}^{j-1}v_{k}\right)^{2} - \left(\sum_{k=1}^{j-2}v_{k}\right)^{2}\right) + \frac{2}{r^{2}}\left(\left(\sum_{k=1}^{j-1}v_{k}\right)^{3} - \left(\sum_{k=1}^{j-2}v_{k}\right)^{3}\right)\\
&=\frac{6 Q_{\frac{1}{\lambda(t)}}}{r^{2}} \left(2\sum_{k=1}^{j-2}v_{k}v_{j-1}+v_{j-1}^{2}\right) + \frac{2}{r^{2}}\left(3\left(\sum_{k=1}^{j-2}v_{k}\right)^{2} v_{j-1} + 3\sum_{k=1}^{j-2}v_{k} v_{j-1}^{2} + v_{j-1}^{3}\right).\end{split}\end{equation}
Then, we let $v_{j}$ be the solution to the following equation, with zero Cauchy data at infinity
\begin{equation} -\partial_{tt}v_{j}+\partial_{rr}v_{j}+\frac{1}{r}\partial_{r}v_{j} - \frac{4}{r^{2}} v_{j} = RHS_{j}(t,r).\end{equation}
We proceed to prove estimates on $v_{k}$ by induction.
\begin{lemma} \label{vkinductionlemma}Let $C_{1}>9$ be such that \eqref{v1est}, \eqref{dv1est}, and \eqref{v2nearorigin} through \eqref{d3v2est} hold, with the constant $C = C_{1}$ on the right-hand side. Then, there exists $n>900$ and $T_{1}>0$ such that, if
$$D_{n,k} = \begin{cases} C_{1}, \quad k=2\\
C_{1}^{nk}, \quad k \geq 3\end{cases}$$
then, we have the following estimates for $k \geq 2$, and all $t \geq T_{1}$.
\begin{equation} \label{vknearorigin} |\partial_{t}^{p} \partial_{r}^{m} v_{k}(t,r)| \leq \frac{D_{n,k} r^{2-m}}{t^{2+p} \log^{bk}(t)}, \quad r \leq \frac{t}{2}, \quad 0 \leq p,m \leq 2\end{equation}
\begin{equation}\label{vkassump} |v_{k}(t,r)|+r |\partial_{t}v_{k}(t,r)| +r |\partial_{r}v_{k}(t,r)| \leq \frac{D_{n,k}}{\log^{bk}(t)}, \quad r > \frac{t}{2}\end{equation}
\begin{equation}\label{d2vkassump} |\partial_{rr}v_{k}(t,r)|+ |\partial_{tt}v_{k}(t,r)| +|\partial_{tr}v_{k}(t,r)| \leq \frac{D_{n,k}}{t \langle t -r\rangle \log^{b}(\langle t-r \rangle) \log^{b(k-1)}(t)}, \quad t> r > \frac{t}{2}\end{equation}
\begin{equation}\label{d2vkinftyassump} ||\partial_{tr}v_{k}(t,r)||_{L^{\infty}_{r}\{r \geq \frac{t}{2}\}}+||\partial_{rr}v_{k}(t,r)||_{L^{\infty}_{r}\{r \geq \frac{t}{2}\}}+||\partial_{tt}v_{k}(t,r)||_{L^{\infty}_{r}\{r \geq \frac{t}{2}\}} \leq \frac{D_{n,k}}{t^{3/2} \log^{b(k-1)}(t)}\end{equation}
\begin{equation} \label{d3vkassump} |\partial_{trr}v_{k}(t,r)| + |\partial_{ttr}v_{k}(t,r)| \leq \frac{D_{n,k}}{\sqrt{t} \langle t-r \rangle^{5/2} \log^{2b}(\langle t-r \rangle) \log^{b(k-2)}(t)}, \quad t > r > \frac{t}{2}\end{equation}
and, for $t > r > \frac{t}{2}$,
\begin{equation}\label{dttrrvkassump2} |\partial_{ttrr}v_{k}(t,r)| \leq \begin{cases} \frac{D_{n,k}}{\sqrt{t} \langle t-r \rangle^{7/2} \log^{2b}(\langle t-r \rangle)}, \quad k=2\\
\frac{D_{n,k}}{\sqrt{t} \langle t-r \rangle^{7/2} \log^{3b}(\langle t-r \rangle) \log^{b(k-3)}(t)}, \quad 3 \leq k \end{cases}\end{equation}
\end{lemma}
After proving the lemma, we then prove the following corollary.
\begin{corollary}\label{vkinductioncorollary} The series 
$$v_{s}:=\sum_{j=3}^{\infty} v_{j}$$
as well as the series resulting from applying any first or second order derivative termwise, converges absolutely and uniformly on the set $\{(t,r)|t \geq T_{1}, r \geq 0\}$. Moreover, 
\begin{equation}\nonumber \begin{split} &-\partial_{tt}v_{s}+\partial_{rr}v_{s}+\frac{1}{r}\partial_{r}v_{s}-\frac{4}{r^{2}}v_{s} \\
&= \frac{6 Q_{\frac{1}{\lambda(t)}}}{r^{2}} \left(2 v_{1} \left(v_{2}+v_{s}\right) + \left(v_{2}+v_{s}\right)^{2}\right) + \frac{2}{r^{2}} \left(3 v_{1}\left(v_{2}+v_{s}\right)^{2} + 3 v_{1}^{2} \left(v_{2}+v_{s}\right) + \left(v_{2}+v_{s}\right)^{3}\right).\end{split}\end{equation}
\end{corollary}
For completeness, we also include the following corollary, which is the result of combining Lemma \ref{v1lemma}, Lemma \ref{v2lemma}, and Lemma \ref{vkinductionlemma}.
\begin{corollary}\label{vkinductioncorollary2} Let $v_{c}$ be given by $v_{c}(t,r) := v_{1}(t,r)+v_{2}(t,r)+v_{s}(t,r).$ Then, for $0 \leq p,m \leq 2$, and $t \geq T_{1}$, we have the following estimates.
$$|\partial_{t}^{p}\partial_{r}^{m} v_{c}(t,r)| \leq \frac{C r^{2-m}}{t^{2+p}\log^{b}(t)}, \quad r \leq \frac{t}{2}$$
$$|v_{c}(t,r)| \leq \frac{C}{\log^{b}(r)} + \frac{C}{\log^{2b}(t)}, \quad r \geq \frac{t}{2}$$
$$|\partial_{t}v_{c}(t,r)| + |\partial_{r}v_{c}(t,r)| \leq \frac{C}{\sqrt{r} \log^{b}(\langle t-r \rangle) \sqrt{\langle t-r \rangle}}+\frac{C}{r \log^{2b}(t)}, \quad r \geq \frac{t}{2}$$
For $r \geq \frac{t}{2}$,
\begin{equation}\label{d2vcest}\begin{split}&|\partial_{r}^{2}v_{c}(t,r)| + |\partial_{t}^{2}v_{c}(t,r)| + |\partial_{tr}v_{c}(t,r)| \\
&\leq \frac{C}{\sqrt{r} \log^{b}(\langle t-r \rangle) \langle t-r \rangle ^{3/2}}+\begin{cases} \frac{C}{t \langle t-r \rangle \log^{b}(t) \log^{b}(\langle t-r \rangle)}, \quad t>r>\frac{t}{2}\\
\frac{1}{t^{3/2} \log^{b}(t)}, \quad r > \frac{t}{2}\end{cases}\end{split}\end{equation}
$$|\partial_{trr}v_{c}(t,r)| + |\partial_{ttr}v_{c}(t,r)| \leq \frac{C}{\sqrt{t} \log^{b}(\langle t-r\rangle) \langle t-r \rangle^{5/2}}, \quad t > r > \frac{t}{2}$$
$$|\partial_{ttrr}v_{c}(t,r) \leq \frac{C}{\sqrt{t} \log^{b}(\langle t-r \rangle) \langle t-r \rangle^{7/2}}, \quad t > r > \frac{t}{2}$$
\end{corollary}
\emph{Proof} (of Lemma \ref{vkinductionlemma})
Let $n > 900$ be otherwise arbitrary, and let $T_{0,n}>e^{(900!)^{1+\frac{1}{b}}}$ satisfy $\frac{C_{1}^{90 n}}{\log^{b}(T_{0,n})} < e^{-(900!)}$ and be otherwise arbitrary. Our goal is to show that, for a sufficiently large $T_{0,n}$, we can prove estimates on $v_{j}$ (and its derivatives), valid for all $t \geq T_{0,n}+\exp{\left(900! + 2^{-\frac{3}{2(2b-1)}}\right)}\left(1+T_{\lambda_{0}}\right)+ M_{1}$ and all $r \geq 0$, by induction. In the following estimates, we assume $t \geq \exp{\left(900! + 2^{-\frac{3}{2(2b-1)}}\right)}\left(1+T_{\lambda_{0}}\right)+ M_{1}+T_{0,n}.$ Let 
$$D_{n,k} = \begin{cases} C_{1}, \quad k=2\\
C_{1}^{nk}, \quad k \geq 3\end{cases}.$$
Suppose, for any $j \geq 3$, and all $k$ with $2 \leq k \leq j-1$, that
\begin{equation}  |\partial_{t}^{p} \partial_{r}^{m} v_{k}(t,r)| \leq \frac{D_{n,k} r^{2-m}}{t^{2+p} \log^{bk}(t)}, \quad r \leq \frac{t}{2}, \quad 0 \leq p,m \leq 2\end{equation}
\begin{equation} |v_{k}(t,r)| +r |\partial_{t}v_{k}(t,r)| + r|\partial_{r}v_{k}(t,r)| \leq \frac{D_{n,k}}{\log^{bk}(t)}, \quad r > \frac{t}{2}\end{equation}
\begin{equation} |\partial_{rr}v_{k}(t,r)|+ |\partial_{tt}v_{k}(t,r)| +|\partial_{tr}v_{k}(t,r)| \leq \frac{D_{n,k}}{t \langle t -r\rangle \log^{b}(\langle t-r \rangle) \log^{b(k-1)}(t)}, \quad t> r > \frac{t}{2}\end{equation}
\begin{equation} ||\partial_{tr}v_{k}(t,r)||_{L^{\infty}_{r}\{r \geq \frac{t}{2}\}}+||\partial_{rr}v_{k}(t,r)||_{L^{\infty}_{r}\{r \geq \frac{t}{2}\}}+||\partial_{tt}v_{k}(t,r)||_{L^{\infty}_{r}\{r \geq \frac{t}{2}\}} \leq \frac{D_{n,k}}{t^{3/2} \log^{b(k-1)}(t)}\end{equation}
\begin{equation}  |\partial_{trr}v_{k}(t,r)| + |\partial_{ttr}v_{k}(t,r)| \leq \frac{D_{n,k}}{\sqrt{t} \langle t-r \rangle^{5/2} \log^{2b}(\langle t-r \rangle) \log^{b(k-2)}(t)}, \quad t > r > \frac{t}{2}\end{equation}
and, for $t > r > \frac{t}{2}$,
\begin{equation} |\partial_{ttrr}v_{k}(t,r)| \leq \begin{cases} \frac{D_{n,k}}{\sqrt{t} \langle t-r \rangle^{7/2} \log^{2b}(\langle t-r \rangle)}, \quad j=3,k=2\\
\begin{cases} \frac{D_{n,k}}{\sqrt{t} \langle t-r \rangle^{7/2} \log^{2b}(\langle t-r \rangle)}, \quad k=2\\
\frac{D_{n,k}}{\sqrt{t} \langle t-r \rangle^{7/2} \log^{3b}(\langle t-r \rangle) \log^{b(k-3)}(t)}, \quad 3 \leq k \leq j-1\end{cases}, \quad j > 3\end{cases}.\end{equation}
Then, for some $C$ \emph{independent} of $n$ (and $t$) we have
$$|\partial_{t}^{p}\partial_{r}^{m} RHS_{j}(t,r)| \leq C\left( C_{1}^{1-n}+\frac{C_{1}^{n}}{\log^{b}(t)}\right) \frac{C_{1}^{nj} r^{2-m}}{t^{4+p} \log^{bj}(t)}, \quad r \leq \frac{t}{2}, \quad 0 \leq p,m \leq 2$$
$$|RHS_{j}(t,r)| \leq \frac{C_{1}^{nj}}{r^{2} \log^{bj}(t)}, \quad r > \frac{t}{2}$$
\begin{equation}\begin{split}|&\partial_{r}RHS_{j}(t,r)|+|\partial_{t}RHS_{j}(t,r)| \\
&\leq C \left(\frac{C_{1}^{nj}}{r^{3} \log^{bj}(t)} \left(C_{1}^{1-n}+\frac{C_{1}^{n}}{\log^{b}(t)}\right) + \frac{C_{1}^{nj} C_{1}^{1-n}}{r^{5/2} \sqrt{\langle t-r \rangle} \log^{b(j-1)}(t) \log^{b}(\langle t-r \rangle)}\right), \quad r > \frac{t}{2}\end{split}\end{equation}
\begin{equation}\nonumber \begin{split}&|\partial_{tr}RHS_{j}(t,r)| + |\partial_{r}^{2}RHS_{j}(t,r)| + |\partial_{t}^{2}RHS_{j}(t,r)| \\
&\leq \frac{C C_{1}^{nj}\left(C_{1}^{1-n}+\frac{C_{1}^{n}}{\log^{b}(t)}\right)}{\log^{b(j-1)}(t) \log^{b}(\langle t-r \rangle) r^{5/2} \langle t-r \rangle^{3/2}} + \frac{C C_{1}^{nj}\left(\frac{C_{1}^{2-n}}{\log^{b}(\langle t-r \rangle)} + \frac{C_{1}^{3n}}{\log^{2b}(t) \log^{b}(\langle t-r \rangle)}\right)}{\log^{b(j-1)}(t) r^{3} \langle t-r \rangle \log^{b}(\langle t-r \rangle)}, \quad t > r > \frac{t}{2}\end{split}.\end{equation}
We will also require another estimate on $\partial_{t}^{2}RHS_{j}$ and $\partial_{tr}RHS_{j}$, which is valid for all $r >\frac{t}{2}$:
\begin{equation}\begin{split}&|\partial_{t}^{2}RHS_{j}(t,r)|+|\partial_{tr}RHS_{j}(t,r)|\\
& \leq \frac{C C_{1}^{nj}C_{1}^{1-n}}{r^{5/2}\langle t-r \rangle^{3/2} \log^{b}(\langle t-r \rangle) \log^{b(j-1)}(t)} + \frac{C C_{1}^{nj}}{r^{2}t^{3/2} \log^{b(j-1)}(t)} \left(\frac{C_{1}^{n}}{\log^{b}(t)} + \frac{C_{1}^{1-n}}{\log^{b}(t)}\right)\\
& + \frac{C C_{1}^{nj}}{r^{3} \langle t-r \rangle} \frac{\left(C_{1}^{-n+2}+\frac{C_{1}^{3n}}{\log^{2b}(t)}\right)}{\log^{b(j-1)}(t) \log^{2b}(\langle t-r \rangle)}, \quad r > \frac{t}{2}\end{split}\end{equation}
$$|\partial_{ttr}RHS_{j}(t,r)| + |\partial_{trr}RHS_{j}(t,r)| \leq \frac{C C_{1}^{nj} \left(C_{1}^{1-n}+\frac{C_{1}^{n}}{\log^{b}(t)}\right)}{r^{5/2} \log^{2b}(\langle t-r \rangle) \log^{b(j-2)}(t) \langle t-r \rangle^{5/2}}, \quad t>r>\frac{t}{2}$$
$$|\partial_{ttrr}RHS_{j}(t,r)| \leq \frac{C C_{1}^{nj}}{r^{5/2} \langle t-r \rangle^{7/2}}\frac{\left(C_{1}^{1-n}+\frac{C_{1}^{n}}{\log^{b}(t)}\right)}{\log^{3b}(\langle t-r \rangle) \log^{b(j-1)}(t)}, \quad t>r>\frac{t}{2}$$
Then, we repeat the analogs of steps 1-8 used to estimate $v_{2}$, and get (for $C$ \emph{independent} of $n$ (and $t$)):
\begin{equation} |\partial_{t}^{p}\partial_{r}^{m}v_{j}(t,r)| \leq C \frac{C_{1}^{nj}r^{2-m}}{t^{2+p} \log^{bj}(t)} \left(C_{1}^{1-n}+\frac{C_{1}^{n}}{\log^{b}(t)}\right), \quad r \leq \frac{t}{2}, \quad 0 \leq p,m \leq 2\end{equation}
\begin{equation} |v_{j}(t,r)| +r |\partial_{r}v_{j}(t,r)| + r|\partial_{t}v_{j}(t,r)| \leq \frac{C C_{1}^{nj}\left(C_{1}^{1-n}+\frac{C_{1}^{n}}{\log^{b}(t)}\right)}{\log^{bj}(t)}, \quad r > \frac{t}{2}\end{equation}
\begin{equation} |\partial_{t}^{2}v_{j}(t,r)| + |\partial_{r}^{2}v_{j}(t,r)| + |\partial_{tr}v_{j}(t,r)| \leq \frac{C C_{1}^{nj}}{t \langle t-r\rangle \log^{b(j-1)}(t) \log^{b}(\langle t-r \rangle)}\left(C_{1}^{2-n} + \frac{C_{1}^{2n}}{\log^{b}(t)}\right), \quad t > r > \frac{t}{2}\end{equation}
\begin{equation} ||\partial_{t}^{2}v_{j}(t,r)||_{L^{\infty}_{r}\{r > \frac{t}{2}\}} +  ||\partial_{tr}v_{j}(t,r)||_{L^{\infty}_{r}\{r > \frac{t}{2}\}}+||\partial_{r}^{2}v_{j}(t,r)||_{L^{\infty}_{r}\{r > \frac{t}{2}\}} \leq \frac{C C_{1}^{nj}}{t^{3/2} \log^{b(j-1)}(t)} \left(\frac{C_{1}^{2n}}{\log^{b}(t)} + C_{1}^{2-n}\right)\end{equation}
\begin{equation} |\partial_{ttr}v_{j}(t,r)| + |\partial_{trr}v_{j}(t,r)| \leq \frac{C C_{1}^{nj}\left(C_{1}^{2-n}+\frac{C_{1}^{3n}}{\log^{b}(t)}\right)}{\sqrt{t} \langle t-r \rangle^{5/2} \log^{b(j-2)}(t) \log^{2b}(\langle t-r \rangle)}, \quad t > r > \frac{t}{2}\end{equation}
Finally,
\begin{equation} |\partial_{ttrr}v_{j}(t,r)| \leq \frac{C C_{1}^{nj} \left(C_{1}^{2-n}+\frac{C_{1}^{3n}}{\log^{b}(t)}\right)}{\sqrt{t} \langle t-r \rangle^{7/2} \log^{b(j-3)}(t) \log^{3b}(\langle t-r \rangle)}, \quad t > r >\frac{t}{2}.\end{equation}
Since $C$ is independent of $n$, there exists $n_{0}$ such that $\max\{C,1\} C_{1}^{90-n_{0}} < e^{-(900!)}.$ Now, since $C$ is also independent of $T_{0,n_{0}}$, we can choose $T_{0,n_{0}}$ to satisfy, in addition to our previous constraints at the beginning of this argument, the following inequality  $$\max\{C,1\} \frac{C_{1}^{90n_{0}}}{\log^{b}(T_{0,n_{0}})} < e^{-(900!)}.$$
By mathematical induction, the above results imply that \eqref{vknearorigin} through \eqref{dttrrvkassump2} are true (with $n=n_{0}$) for all $k \geq 2$, and $t \geq T_{0,n_{0}}+\exp{\left(900! + 2^{-\frac{3}{2(2b-1)}}\right)}\left(1+T_{\lambda_{0}}\right)+ M_{1}:=T_{1}$.\qed\\
\\
Now, we prove Corollary \ref{vkinductioncorollary}. From here on, for the rest of the paper, we further restrict $T_{0}$ to satisfy $T_{0} \geq T_{1}$. Also, for all $j \geq 2$,
$$||v_{j}(t,r)||_{L^{\infty}\{r \geq 0, t \geq T_{1}\}} \leq e^{-j(900!)}$$
$$||\partial_{t}v_{j}(t,r)||_{L^{\infty}\{r \geq 0, t \geq T_{1}\}}+||\partial_{r}v_{j}(t,r)||_{L^{\infty}\{r \geq 0, t \geq T_{1}\}} \leq \frac{2}{T_{1}} e^{-j(900!)}$$
$$||\partial_{tr}v_{j}(t,r)||_{L^{\infty}\{r \geq 0, t \geq T_{1}\}}+||\partial_{t}^{2}v_{j}(t,r)||_{L^{\infty}\{r \geq 0, t \geq T_{1}\}}+||\partial_{r}^{2}v_{j}(t,r)||_{L^{\infty}\{r \geq 0, t \geq T_{1}\}} \leq \frac{e^{-(900!)(j-1)}}{T_{1}^{3/2}}.$$
The series (in Corollary \ref{vkinductioncorollary}) defining $v_{s}$ (as well as the series resulting from applying any first or second order derivative termwise) converges absolutely and uniformly on the set $\{(t,r)|t \geq T_{1}, r \geq 0\}$. Moreover, using the first line of \eqref{rhsj}, we get, for any $N \geq 4$,
$$\sum_{j=3}^{N} RHS_{j}(t,r) = \frac{6 Q_{\frac{1}{\lambda(t)}}(r)}{r^{2}} \left(2 v_{1}\sum_{k=2}^{N-1} v_{k} + \left(\sum_{k=2}^{N-1} v_{k}\right)^{2}\right) + \frac{2}{r^{2}} \left(3 v_{1} \left(\sum_{k=2}^{N-1} v_{k}\right)^{2} + 3v_{1}^{2} \sum_{k=2}^{N-1} v_{k} + \left(\sum_{k=2}^{N-1} v_{k}\right)^{3}\right)$$
(where the argument $(t,r)$ of all instances of $v_{k},v_{1}$ has been omitted, for clarity).
Using the uniformity of the convergence of the series defining $v_{s}$, we get
\begin{equation}\nonumber \begin{split} &-\partial_{tt}v_{s}+\partial_{rr}v_{s}+\frac{1}{r}\partial_{r}v_{s}-\frac{4}{r^{2}}v_{s} = \lim_{N\rightarrow \infty} \sum_{j=3}^{N} \left(-\partial_{tt}v_{j}+\partial_{rr}v_{j}+\frac{1}{r}\partial_{r}v_{j}-\frac{4}{r^{2}}v_{j}\right)=\lim_{N \rightarrow \infty} \sum_{j=3}^{N} RHS_{j}(t,r) \\
&= \frac{6 Q_{\frac{1}{\lambda(t)}}}{r^{2}} \left(2 v_{1} \left(v_{2}+v_{s}\right) + \left(v_{2}+v_{s}\right)^{2}\right) + \frac{2}{r^{2}} \left(3 v_{1}\left(v_{2}+v_{s}\right)^{2} + 3 v_{1}^{2} \left(v_{2}+v_{s}\right) + \left(v_{2}+v_{s}\right)^{3}\right).\end{split}\end{equation}
\subsection{Improvement of the large $r$ behavior of the remaining error terms}\label{w2sec}
We recall the function $v_{c}$ defined in Corollary \ref{vkinductioncorollary2}. Despite the major improvement of the decay of the error terms accomplished via the resummation of the $v_{k}$ above, we will still need to improve the decay of the error terms which result from substituting $Q_{\frac{1}{\lambda(t)}} + v_{c}$ into \eqref{ym}. The soliton error term, and the linear error term resulting from $v_{c}$, namely $\partial_{tt}Q_{\frac{1}{\lambda(t)}} - \frac{6v_{c}}{r^{2}}\left(1-Q_{\frac{1}{\lambda(t)}}^{2}(r)\right),$ both contribute to leading order in the modulation equation. Therefore, any improvement of these error terms should not change their leading order inner product with the appropriately re-scaled zero eigenfunction $\phi_{0}$. Keeping this in mind, we let
\begin{equation}\label{wrhs2firstdef}WRHS_{2}(t,r):= \chi_{\geq 1}(\frac{r}{g(t)}) \left(\partial_{tt}Q_{\frac{1}{\lambda(t)}} - \frac{6v_{c}}{r^{2}}\left(1-Q_{\frac{1}{\lambda(t)}}^{2}(r)\right)\right), \text{ where }g(t) = \lambda(t) \log^{b-2\epsilon}(t)\end{equation}
where $\chi_{\geq 1}(x) \in C^{\infty}(\mathbb{R}), \quad 0 \leq \chi_{\geq 1}(x) \leq 1, \quad \chi_{\geq 1}(x) = \begin{cases} 0, \quad x \leq \frac{1}{2}\\
1, \quad x \geq 1\end{cases}$ and define $w_{2}$ to be the solution to the following equation, with zero Cauchy data at infinity
\begin{equation}\label{w2eqn}-\partial_{tt}w_{2}+\partial_{rr}w_{2}+\frac{1}{r}\partial_{r}w_{2}-\frac{4}{r^{2}}w_{2}=WRHS_{2}(t,r).\end{equation}
Now, we will prove estimates on $w_{2}$. Note that $WRHS_{2}$ depends on $\lambda''$. Because our only assumptions on $\lambda$ do not include any more regularity than $\lambda \in C^{2}([T_{0},\infty))$, we can not consider $\partial_{t}WRHS_{2}(t,r)$ until we first choose a specific $\lambda$ by solving the modulation equation, and then prove that this $\lambda$ has higher than $C^{2}$ regularity. In terms of estimating $w_{2}$, this means that we can not differentiate $WRHS_{2}$ in time, at this point. Therefore, we will first record a set of preliminary estimates on derivatives of $w_{2}$, which only involve $WRHS_{2}$, and its $r$ derivatives. After choosing $\lambda$, we can then obtain more optimal, final estimates on the derivatives of $w_{2}$. This is very similar to the argument used in the wave maps paper of the author, \cite{wm}. 
\begin{lemma}\label{w2prelim}[Preliminary estimates on $w_{2}$]
We have the following \emph{preliminary} estimates on $w_{2}$
\begin{equation} |w_{2}(t,r)| \leq \begin{cases} \frac{C r^{2} \lambda(t)^{2} \log(2+\frac{r}{g(t)}) \log(t)}{(g(t)^{2}+r^{2}) t^{2} \log^{b}(t)}, \quad r \leq \frac{t}{2}\\
\frac{C \lambda(t)^{2} \log(t)}{t^{2}\log^{b}(t)}, \quad r > \frac{t}{2}\end{cases} \quad |\partial_{r}w_{2}(t,r)| \leq \begin{cases} \frac{C r \lambda(t)^{2} \log(t)}{t^{2} \log^{b}(t) g(t)^{2}}, \quad r \leq g(t)\\
\frac{C \lambda(t)^{2} \log(t)}{t^{2} \log^{b}(t) g(t)}, \quad r > g(t)\end{cases}\end{equation}
For $j+k=2$ or $j=1,k=0$,
\begin{equation}|\partial_{t}^{j}\partial_{r}^{k}w_{2}(t,r)| \leq \frac{C \lambda(t)^{2} \log(t)}{t^{2} \log^{b}(t) g(t)^{j+k}}, \quad r >0\end{equation}
\end{lemma}
\begin{proof}
We start proving the estimates of the lemma statement by considering the region $r \leq \frac{t}{2}$.
Using the estimates  for $v_{c}(s,|x+y|)$, from Corollary \ref{vkinductioncorollary2}, we get 
\begin{equation}\begin{split}&\sum_{k=0}^{2}\frac{|\partial_{2}^{k}WRHS_{2}(s,|x+y|)|}{|x+y|^{2-k}} \leq \begin{cases}\frac{C \lambda(s)^{2} \mathbbm{1}_{\{ |x+y| \geq \frac{g(s)}{2}\}}}{s^{2} \log^{b}(s) (g(s)^{2}+|x+y|^{2})^{2}}, \quad |x+y| < \frac{s}{2}\\
\frac{C \lambda(s)^{2}}{|x+y|^{4} \sqrt{s} \langle s-|x+y|\rangle^{3/2} \log^{b}(\langle s-|x+y|\rangle)}, \quad \frac{s}{2} \leq |x+y| < s\end{cases}.\end{split}\end{equation}
Using the same procedure and notation as in step 1 of the proof of Lemma \ref{v2lemma}, we get  
\begin{equation}\label{w2splitting}\begin{split}|w_{2}(t,r)| &\leq C r^{2} \int_{t}^{\infty} ds \int_{0}^{s-t}\frac{\rho d\rho}{\sqrt{(s-t)^{2}-\rho^{2}}}\int_{0}^{\pi} d\phi \sin^{4}(\phi) \left(\frac{\mathbbm{1}_{\{|x+y| > \frac{g(s)}{2}\}}\mathbbm{1}_{\{|x+y| \leq \frac{s}{2}\}} \lambda(s)^{2}}{s^{2} \log^{b}(s) (g(s)^{2}+|x+y|^{2})^{2}}\right)\left(1+\frac{\rho^{2}}{|x+y|^{2}}\right)\\
&+C r^{2} \int_{t}^{\infty} ds \int_{0}^{s-t} \frac{\rho d\rho}{\sqrt{(s-t)^{2}-\rho^{2}}} \int_{0}^{\pi} d\phi \sin^{4}(\phi) \left(\frac{\lambda(s)^{2} \mathbbm{1}_{\{|x+y| > \frac{g(s)}{2}\}}\mathbbm{1}_{\{|x+y|>\frac{s}{2}\}}\left(1+\frac{\rho^{2}}{|x+y|^{2}}\right)}{|x+y|^{4} \sqrt{s} \langle s-|x+y|\rangle^{3/2} \log^{b}(\langle s-|x+y|\rangle)}\right)\end{split}\end{equation}
Denote the first line of \eqref{w2splitting} by $w_{2,I}$. Then, we consider several pieces of $w_{2,I}$ separately. Let
$$w_{2,I,a}(t,r) := r^{2} \int_{t}^{t+\frac{r}{8}} ds \int_{0}^{s-t}\frac{\rho d\rho}{\sqrt{(s-t)^{2}-\rho^{2}}}\int_{0}^{\pi} d\phi \sin^{4}(\phi) \left(\frac{\mathbbm{1}_{\{|x+y| > \frac{g(s)}{2}\}}\mathbbm{1}_{\{|x+y| \leq \frac{s}{2}\}} \lambda(s)^{2}}{s^{2} \log^{b}(s) (g(s)^{2}+|x+y|^{2})^{2}}\right)\left(1+\frac{\rho^{2}}{|x+y|^{2}}\right).$$
Recall that $|x+y| = \sqrt{\rho^{2}+r^{2}+2 r \rho \cos(\phi)}$. For $w_{2,I,a}$, we have $\rho \leq s-t \leq \frac{r}{8} \implies |x+y| \geq C(r+\rho).$ Therefore,
$$|w_{2,I,a}(t,r)| \leq C \int_{t}^{t+\frac{r}{8}}ds \int_{0}^{s-t} \frac{\rho d\rho}{\sqrt{(s-t)^{2}-\rho^{2}}} \frac{r^{2} \lambda(s)^{2}}{s^{2} \log^{b}(s)(g(s)^{2}+r^{2})^{2}} \leq \frac{C r^{2} \lambda(t)^{2}}{(g(t)^{2}+r^{2}) t^{2} \log^{b}(t)}$$
where we used \eqref{T0const} to conclude that
\begin{equation}\label{gdec}s \mapsto \frac{1}{s (g(s)^{2}+r^{2})^{2}} \text{ is decreasing on }[T_{0},\infty).\end{equation}
For the next integrals to consider, we first appropriately use Cauchy's residue theorem to conclude
$$ \int_{0}^{2\pi} \frac{\sin^{4}(\phi) d\phi}{(g(s)^{2}+\rho^{2}+r^{2}+2 r \rho \cos(\phi))^{2}} \leq \frac{C}{(g(s)^{2}+\rho^{2}+r^{2})^{2}}$$
$$\int_{0}^{2\pi} \frac{\sin^{4}(\phi) d\phi}{(g(s)^{2}+r^{2}+\rho^{2}+2 r \rho \cos(\phi))^{3}} \leq \frac{C}{(g(s)^{2}+r^{2}+\rho^{2})^{5/2} \sqrt{g(s)^{2}+(\rho-r)^{2}}}.$$
The important point in the above integrals is that the factor $\sin^{4}(\phi)$ vanishes when $\phi=\pi$, which is precisely when the vectors $x=r\textbf{e}_{1}$ and $y$ (defined in \eqref{ydefstep1}) are antiparallel. Therefore, we get much more decay in $\rho^{2}+r^{2}+g(s)^{2}$ than we would have without the $\sin$ factor. Now, we consider $w_{2,I,b}$ defined by
\begin{equation}\nonumber\begin{split}&w_{2,I,b}(t,r) \\
&= r^{2} \int_{t+\frac{r}{8}}^{t+\frac{r}{8}+g(t)} ds \int_{0}^{s-t}\frac{\rho d\rho}{\sqrt{(s-t)^{2}-\rho^{2}}}\int_{0}^{\pi} d\phi \sin^{4}(\phi) \left(\frac{\mathbbm{1}_{\{|x+y| > \frac{g(s)}{2}\}}\mathbbm{1}_{\{|x+y| \leq \frac{s}{2}\}} \lambda(s)^{2}}{s^{2} \log^{b}(s) (g(s)^{2}+|x+y|^{2})^{2}}\right)\left(1+\frac{\rho^{2}}{|x+y|^{2}}\right)\end{split}\end{equation}
which gives
$$|w_{2,I,b}(t,r)| \leq C r^{2} \int_{t+\frac{r}{8}}^{t+\frac{r}{8}+g(t)} ds \int_{0}^{s-t} \frac{\rho d\rho}{\sqrt{(s-t)^{2}-\rho^{2}}} \frac{\lambda(s)^{2}}{s^{2} \log^{b}(s) g(s)(g(s)^{2}+r^{2})^{3/2}} \leq C \frac{r^{2} \lambda(t)^{2}}{t^{2} \log^{b}(t) (g(t)^{2}+r^{2})}$$
where we again use \eqref{T0const} to justify the analog of \eqref{gdec}. Next, we consider
$$w_{2,I,c}(t,r) := r^{2} \int_{t+\frac{r}{8}+g(t)}^{\infty} \frac{ds}{(s-t)} \int_{0}^{s-t}\rho d\rho\int_{0}^{\pi} d\phi \sin^{4}(\phi) \left(\frac{\mathbbm{1}_{\{|x+y| > \frac{g(s)}{2}\}}\mathbbm{1}_{\{|x+y| \leq \frac{s}{2}\}} \lambda(s)^{2}}{s^{2} \log^{b}(s) (g(s)^{2}+|x+y|^{2})^{2}}\right)\left(1+\frac{\rho^{2}}{|x+y|^{2}}\right).$$
The point of this definition is to utilize the decay of the integrand in $s-t$ in order to do the $s$ integral. As we will show later, the difference $\frac{1}{\sqrt{(s-t)^{2}-\rho^{2}}}-\frac{1}{(s-t)}$ decays sufficiently fast in $s$ so as to allow an argument which does the $s$ integral first, before the $\rho$ integral. A similar procedure was also used in the author's work regarding wave maps\cite{wm}, in the estimation of the correction denoted by $v_{4}$ in that paper. We start by noting
\begin{equation}\label{rhoint1forw2Ic}\int_{0}^{s-t} \rho d\rho \frac{\mathbbm{1}_{\{\rho \leq \frac{r}{2}\}}}{(g(s)^{2}+\rho^{2}+r^{2})^{3/2} \sqrt{g(s)^{2}+(\rho-r)^{2}}} \leq C \int_{0}^{s-t} \frac{\rho d\rho}{(g(s)^{2}+r^{2}+\rho^{2})^{3/2}\sqrt{g(s)^{2}+r^{2}}} \leq \frac{C}{g(s)^{2}+r^{2}}\end{equation}
$$\int_{0}^{s-t} \rho d\rho \frac{\mathbbm{1}_{\{\frac{r}{2} \leq \rho \leq 2 r\}}}{(g(s)^{2}+\rho^{2}+r^{2})^{3/2} \sqrt{g(s)^{2}+(\rho-r)^{2}}} \leq \frac{C r}{(g(s)^{2}+r^{2})^{3/2}} \int_{\frac{r}{2}}^{2r} \frac{d\rho}{\sqrt{g(s)^{2}+(\rho-r)^{2}}} \leq C \frac{\log(1+\frac{r}{g(s)})}{g(s)^{2}+r^{2}}$$
\begin{equation}\label{rhoint3forw2Ic}\int_{0}^{s-t} \rho d\rho \frac{\mathbbm{1}_{\{\rho > 2r\}}}{(g(s)^{2}+\rho^{2}+r^{2})^{3/2} \sqrt{g(s)^{2}+(\rho-r)^{2}}} \leq C \int_{0}^{s-t} \frac{\rho d\rho}{(g(s)^{2}+r^{2}+\rho^{2})^{2}} \leq \frac{C}{g(s)^{2}+r^{2}}\end{equation}
which gives
$$|w_{2,I,c}(t,r)| \leq \frac{C r^{2} \lambda(t)^{2} \log(2+\frac{r}{g(t)}) \log(t)}{t^{2} \log^{b}(t) (g(t)^{2}+r^{2})}.$$
Finally, we consider
$$w_{2,I,d}(t,r) :=r^{2} \int_{t+\frac{r}{8}+g(t)}^{\infty} ds \int_{0}^{s-t}\rho d\rho\left(\frac{1}{\sqrt{(s-t)^{2}-\rho^{2}}}-\frac{1}{(s-t)}\right)\int_{0}^{\pi}  I_{w_{2,I,d} }d\phi$$
where
$$I_{w_{2,I,d}} = \sin^{4}(\phi) \left(\frac{\mathbbm{1}_{\{|x+y| > \frac{g(s)}{2}\}}\mathbbm{1}_{\{|x+y| \leq \frac{s}{2}\}} \lambda(s)^{2}}{s^{2} \log^{b}(s) (g(s)^{2}+|x+y|^{2})^{2}}\right)\left(1+\frac{\rho^{2}}{|x+y|^{2}}\right).$$
Using again an analog of the observation \eqref{gdec}, and the $\phi$ integrals noted above, we first do the $\phi$ integral. Then, we switch the order of the $s$ and $\rho$ and $\phi$ integrals, to do the $s$ integral first:
\begin{equation}\nonumber\begin{split}&|w_{2,I,d}(t,r)| \\
&\leq Cr^{2} \int_{0}^{\infty}  \rho d\rho \int_{\rho+t}^{\infty} ds \left(\frac{1}{\sqrt{(s-t)^{2}-\rho^{2}}}-\frac{1}{(s-t)}\right) \frac{\lambda(t)^{2}}{\log^{b}(t) t^{2} (g(t)^{2}+r^{2}+\rho^{2})^{3/2} \sqrt{g(t)^{2}+(\rho-r)^{2}}}.\end{split}\end{equation}
Then, we use the same method used to study the $\rho$ integrals \eqref{rhoint1forw2Ic} through \eqref{rhoint3forw2Ic}, to get
$$|w_{2,I,d}(t,r)| \leq \frac{ C r^{2} \lambda(t)^{2} \log(2+\frac{r}{g(t)})}{t^{2} \log^{b}(t) (g(t)^{2}+r^{2})}.$$
In total, we get
$$|w_{2,I}(t,r)| \leq \frac{C r^{2} \lambda(t)^{2} \log(2+\frac{r}{g(t)}) \log(t)}{t^{2} \log^{b}(t) (g(t)^{2}+r^{2})}.$$
It remains to treat the second line of \eqref{w2splitting}, which we denote as $w_{2,II}$. If $r \leq \frac{t}{2}$, then, as in the case of estimating $v_{2}$, we use that $s-|x+y| \geq s-(r+(s-t)) \geq t-r \geq \frac{t}{2}$, and we use the estimates on $v_{c}(s,|x+y|)$ in the region $|x+y| \geq \frac{s}{2}$ to get
$$|w_{2,II}(t,r)| \leq C r^{2} \int_{t}^{\infty} ds \frac{(s-t) \lambda(s)^{2}}{s^{9/2} t^{3/2} \log^{b}(t)} \leq \frac{C r^{2} \lambda(t)^{2}}{t^{4} \log^{b}(t)}, \quad r \leq \frac{t}{2}.$$
Finally, we need to estimate $w_{2}(t,r)$ for $r > \frac{t}{2}$. We use the analog  of \eqref{dtv2noderiv} to get
\begin{equation}\label{w2noderiv}w_{2}(t,r) = \frac{-1}{2\pi} \int_{t}^{\infty} ds \int_{0}^{s-t} \frac{\rho d\rho}{\sqrt{(s-t)^{2}-\rho^{2}}} \int_{0}^{2\pi} d\theta WRHS_{2}(s,|x+y|) \left(1-\frac{2 \rho^{2} \sin^{2}(\theta)}{r^{2}+\rho^{2}+2 r \rho \cos(\theta)}\right)\end{equation}
Again, we insert $1=\mathbbm{1}_{\{|x+y| \leq \frac{s}{2}\}} + \mathbbm{1}_{\{|x+y| > \frac{s}{2}\}}$ into the integrand of the above expression, and define $w_{2,IV}$ by
$$w_{2,IV}(t,r) = \frac{-1}{2\pi} \int_{t}^{\infty} ds \int_{0}^{s-t} \frac{\rho d\rho}{\sqrt{(s-t)^{2}-\rho^{2}}} \int_{0}^{2\pi} d\theta \mathbbm{1}_{\{|x+y| > \frac{s}{2}\}} WRHS_{2}(s,|x+y|) \left(1-\frac{2 \rho^{2} \sin^{2}(\theta)}{r^{2}+\rho^{2}+2 r \rho \cos(\theta)}\right)$$
and $w_{2,III}(t,r) = w_{2}(t,r)-w_{2,IV}(t,r).$
$$|w_{2,IV}(t,r)| \leq C \int_{t}^{\infty} ds \int_{0}^{s-t} \frac{\rho d\rho}{\sqrt{(s-t)^{2}-\rho^{2}}} \int_{0}^{2\pi} d\theta \left(\frac{\lambda(s)^{2}}{s^{4} \log^{b}(s)}\right) \leq \frac{C \lambda(t)^{2}}{t^{2} \log^{b}(t)}$$
For $w_{2,III}$ in the region $r > \frac{t}{2}$, we again consider several integrals separately. We have
\begin{equation}\nonumber\begin{split}&w_{2,III,a}(t,r) \\
&= \frac{-1}{2\pi} \int_{t}^{t+\frac{r}{8}} ds \int_{0}^{s-t} \frac{\rho d\rho}{\sqrt{(s-t)^{2}-\rho^{2}}} \int_{0}^{2\pi} d\theta \mathbbm{1}_{\{|x+y| \leq \frac{s}{2}\}} WRHS_{2}(s,|x+y|) \left(1-\frac{2 \rho^{2} \sin^{2}(\theta)}{r^{2}+\rho^{2}+2 r \rho \cos(\theta)}\right)\end{split}\end{equation}
which gives, via the same reasoning as used to estimate $w_{2,I,a}$,
$$|w_{2,III,a}(t,r)| \leq C \int_{t}^{t+\frac{r}{8}} ds \int_{0}^{s-t} \frac{\rho d\rho}{\sqrt{(s-t)^{2}-\rho^{2}}} \int_{0}^{2\pi} d\theta \frac{r^{2} \lambda(s)^{2}}{s^{2}\log^{b}(s)(g(s)^{2}+r^{2})^{2}} \leq \frac{C \lambda(t)^{2}}{t^{2} \log^{b}(t)}, \quad r > \frac{t}{2}.$$
Next, we have
\begin{equation}\nonumber\begin{split}&w_{2,III,b}(t,r)\\
&=\frac{-1}{2\pi} \int_{t+\frac{r}{8}}^{t+\frac{r}{8}+g(t)} ds \int_{0}^{s-t} \frac{\rho d\rho}{\sqrt{(s-t)^{2}-\rho^{2}}} \int_{0}^{2\pi} d\theta \mathbbm{1}_{\{|x+y| \leq \frac{s}{2}\}} WRHS_{2}(s,|x+y|) \left(1-\frac{2 \rho^{2} \sin^{2}(\theta)}{r^{2}+\rho^{2}+2 r \rho \cos(\theta)}\right).\end{split}\end{equation}
We use
\begin{equation}\label{thetaint1}\int_{0}^{2\pi} \frac{r^{2}+\rho^{2} + 2 r \rho \cos(\theta)}{(g(s)^{2}+r^{2}+\rho^{2}+2 r \rho \cos(\theta))^{2}} d\theta \leq \frac{C}{\sqrt{g(s)^{2}+\rho^{2}+r^{2}}\sqrt{g(s)^{2}+(\rho-r)^{2}}}\end{equation}
to get
\begin{equation}\begin{split}|w_{2,III,b}(t,r)| &\leq C \int_{t+\frac{r}{8}}^{t+\frac{r}{8}+g(t)}ds \int_{0}^{s-t} \frac{\rho d\rho}{\sqrt{(s-t)^{2}-\rho^{2}}} \frac{\lambda(s)^{2}}{s^{2}\log^{b}(s)} \frac{1}{g(s) \sqrt{g(s)^{2}+r^{2}}} \\
&\leq C \int_{t+\frac{r}{8}}^{t+\frac{r}{8}+g(t)} ds (s-t) \frac{\lambda(s)^{2}}{s^{2}\log^{b}(s) g(s) \sqrt{g(s)^{2}+r^{2}}} \leq \frac{C \lambda(t)^{2}}{t^{2} \log^{b}(t)}.\end{split}\end{equation}
The next integral to treat is
\begin{equation}\nonumber\begin{split}&w_{2,III,c}(t,r)\\
&=\frac{-1}{2\pi} \int_{t+\frac{r}{8}+g(t)}^{\infty} \frac{ds}{(s-t)} \int_{0}^{s-t} \rho \mathbbm{1}_{\{\frac{r}{2} \leq \rho \leq 2r\}} d\rho \int_{0}^{2\pi} d\theta \mathbbm{1}_{\{|x+y| \leq \frac{s}{2}\}} WRHS_{2}(s,|x+y|) \left(1-\frac{2 \rho^{2} \sin^{2}(\theta)}{r^{2}+\rho^{2}+2 r \rho \cos(\theta)}\right).\end{split}\end{equation}
So,
\begin{equation}\begin{split}|w_{2,III,c}(t,r)| &\leq C \int_{t+\frac{r}{8}+g(t)}^{\infty} \frac{ds}{(s-t)} \frac{\lambda(s)^{2}}{s^{2} \log^{b}(s)} \int_{\frac{r}{2}}^{2r} \frac{\rho d\rho}{\sqrt{g(s)^{2}+(\rho-r)^{2}}\sqrt{g(s)^{2}+r^{2}}} \\
&\leq C r \int_{t+\frac{r}{8}+g(t)}^{\infty} \frac{ds}{s(s-t)} \frac{\log(1+\frac{r}{g(s)}) \lambda(s)^{2}}{s \log^{b}(s) \sqrt{g(s)^{2}+r^{2}}} \leq \frac{C \lambda(t)^{2} \log(1+\frac{r}{g(t)})}{r t \log^{b}(t)}\end{split}\end{equation}
where we used the fact that $\frac{1}{(s-t)} \leq \frac{C}{r}$. Next, we have
\begin{equation}\begin{split}w_{2,III,d}(t,r):=\frac{-1}{2\pi} \int_{t+\frac{r}{8}+g(t)}^{\infty} \frac{ds}{(s-t)} &\int_{0}^{s-t}  \rho \left(\mathbbm{1}_{\{\rho < \frac{r}{2}\}}+\mathbbm{1}_{\{\rho > 2r\}}\right) d\rho \\
&\int_{0}^{2\pi} d\theta \mathbbm{1}_{\{|x+y| \leq \frac{s}{2}\}} WRHS_{2}(s,|x+y|) \left(1-\frac{2 \rho^{2} \sin^{2}(\theta)}{r^{2}+\rho^{2}+2 r \rho \cos(\theta)}\right)\end{split}\end{equation}
which gives
\begin{equation}\begin{split}|w_{2,III,d}(t,r)| &\leq C \int_{t+\frac{r}{8}+g(t)}^{\infty} \frac{ds}{(s-t)} \int_{0}^{s-t} \frac{\rho d\rho \lambda(s)^{2}}{s^{2} \log^{b}(s)} \left(\frac{\mathbbm{1}_{\{\rho < \frac{r}{2}\}}+\mathbbm{1}_{\{2r < \rho\}}}{(r^{2}+\rho^{2}+g(s)^{2})}\right) \\
&\leq C \int_{t+\frac{r}{8} +g(t)}^{\infty} \frac{ds}{(s-t)} \frac{\lambda(s)^{2}}{s^{2} \log^{b}(s)} \left(1+\log(s-t)+\log(r)\right)\leq \frac{C \lambda(t)^{2}}{t \log^{b}(t)}\frac{\log(r)}{r}, \quad r \geq \frac{t}{2}.\end{split}\end{equation}
The final integral to estimate is
\begin{equation}\begin{split}w_{2,III,e}(t,r):=\frac{-1}{2\pi} \int_{t+\frac{r}{8}+g(t)}^{\infty} ds &\int_{0}^{s-t}  \rho \left(\frac{1}{\sqrt{(s-t)^{2}-\rho^{2}}}-\frac{1}{(s-t)}\right) d\rho \\
&\int_{0}^{2\pi} d\theta \mathbbm{1}_{\{|x+y| \leq \frac{s}{2}\}} WRHS_{2}(s,|x+y|) \left(1-\frac{2 \rho^{2} \sin^{2}(\theta)}{r^{2}+\rho^{2}+2 r \rho \cos(\theta)}\right).\end{split}\end{equation}
We again switch the order of integration, as previously done to estimate $w_{2,I,d}$. The only difference here is that we use $\frac{1}{s} \leq \frac{1}{\rho+t}$ for one of the factors of $\frac{1}{s}$ which appear in the integrand of $w_{2,III,e}$ once the estimates for $WRHS_{2}$ are substituted. This gives
$$|w_{2,III,e}(t,r)| \leq C \int_{0}^{\infty} \frac{\rho d\rho \lambda(t)^{2}}{(\rho+t) \log^{b}(t) t \sqrt{g(t)^{2}+\rho^{2}+r^{2}}\sqrt{g(t)^{2}+(\rho-r)^{2}}}$$
and
$$|w_{2,III,e}(t,r)| \leq \frac{C \lambda(t)^{2} \log(r) }{r t\log^{b}(t)}, \quad r \geq \frac{t}{2}.$$
Combining the above gives the final pointwise estimate on $w_{2}$:
\begin{equation} |w_{2}(t,r)| \leq \begin{cases} \frac{C r^{2}\lambda(t)^{2} \log(2+\frac{r}{g(t)})\log(t)}{(g(t)^{2}+r^{2}) t^{2} \log^{b}(t)}, \quad r \leq \frac{t}{2}\\
\frac{C \lambda(t)^{2} \log(t)}{t^{2}\log^{b}(t)}, \quad r > \frac{t}{2}\end{cases}\end{equation}
where we used the fact that $\frac{\log(1+\frac{r}{g(t)})}{r} \leq \frac{C \log(t)}{t}, \quad r \geq \frac{t}{2}.$ Next, we consider the derivatives of $w_{2}$. We recall the remarks prior to the estimation of $w_{2}$ regarding preliminary estimates, and prove a preliminary estimate on $\partial_{r}w_{2}$. Here, we will start with the case $r \leq g(t)$. We first note that
$$\sin^{2}(\phi)\frac{\rho}{\sqrt{r^{2}+\rho^{2}+2 r \rho \cos(\phi)}} \leq \begin{cases} C , \quad \rho < \frac{r}{2}, \text{ or } \rho > 2r\\
\frac{C \sin^{2}(\frac{\phi}{2}) \cos^{2}(\frac{\phi}{2})r}{\sqrt{(r-\rho)^{2}+4 r \rho \cos^{2}(\frac{\theta}{2})}} \leq C, \quad \frac{r}{2} \leq \rho \leq 2r\end{cases}.$$
We also note that 
$$\int_{0}^{2\pi} \frac{d\phi}{(g(t)^{2}+\rho^{2}+r^{2}+2 r \rho \cos(\phi))^{2}} \leq \frac{C}{\sqrt{g(t)^{2}+r^{2}+\rho^{2}}(g(t)^{2}+(r-\rho)^{2})^{3/2}}.$$
Then, we proceed as in step 2 of the estimation of $\partial_{r}v_{2}+\frac{2}{r}v_{2}$, to get
\begin{equation}\begin{split}|\left(\partial_{r}+\frac{2}{r}\right)w_{2}(t,r)|&\leq C r \int_{t}^{\infty} ds \int_{0}^{s-t} \frac{\rho d\rho}{\sqrt{(s-t)^{2}-\rho^{2}}} \int_{0}^{\pi} d\phi \left(\frac{\mathbbm{1}_{\{|x+y| < \frac{s}{2}\}}\lambda(s)^{2}}{s^{2}\log^{b}(s) (g(s)^{2}+r^{2}+\rho^{2}+2 r \rho \cos(\phi))^{2}}\right)\\
&+ C r \int_{t}^{\infty} ds \int_{0}^{s-t} \frac{\rho d\rho}{\sqrt{(s-t)^{2}-\rho^{2}}} \int_{0}^{\pi} d\phi \frac{\mathbbm{1}_{\{|x+y| > \frac{s}{2}\}}\lambda(s)^{2}}{|x+y|^{4} \sqrt{s} \langle s-|x+y| \rangle^{3/2} \log^{b}(\langle s-|x+y|\rangle)}.\end{split}\end{equation}
Denote the first line of the above expression by $q_{I}(t,r)$, and the second by $q_{II}(t,r)$. Then, we estimate $q_{I}(t,r)$ using the same procedure used for $w_{2}$. The main difference here is that we have the factor 
$$\frac{1}{\sqrt{g(t)^{2}+r^{2}+\rho^{2}}(g(t)^{2}+(r-\rho)^{2})^{3/2}} \text{ instead of } \frac{1}{\sqrt{g(t)^{2}+(r-\rho)^{2}}(g(t)^{2}+r^{2}+\rho^{2})^{3/2}}.$$ 
This leads to an extra factor of $\frac{r}{g(t)}$ when we estimate certain $\rho$ integrals in the region $\frac{r}{2} \leq \rho \leq 2r$. Since we are considering the region $r \leq g(t)$, we end up with
$$|q_{I}(t,r)| \leq \frac{C r \lambda(t)^{2}}{t^{2} \log^{b}(t) g(t)^{2}} \log(t), \quad r \leq g(t).$$
The same procedure used for $w_{2,II}$ gives
$$|q_{II}(t,r)| \leq \frac{C r \lambda(t)^{2}}{t^{4} \log^{b}(t)}, \quad r \leq g(t).$$
For the region $r \geq g(t)$, we will differentiate our formula \eqref{w2noderiv} directly. We emphasize again that the estimate on $\partial_{r}w_{2}$ which we will obtain now, in the region $r \geq g(t)$ is a preliminary estimate, and it will be improved later, once we choose $\lambda(t)$, and show that $\lambda \in C^{3}([T_{0},\infty))$, thereby allowing us to estimate $\partial_{t}w_{2}$ and $\left(-\partial_{t}+\partial_{r}\right)w_{2}$. We have
\begin{equation}\begin{split} &\partial_{r}w_{2}(t,r) \\
&= \frac{-1}{2\pi} \int_{t}^{\infty} ds \int_{0}^{s-t} \frac{\rho d\rho}{\sqrt{(s-t)^{2}-\rho^{2}}} \int_{0}^{2\pi} d\theta \partial_{r}\left(WRHS_{2}(s,|x+y|) \left(1-\frac{2 \rho^{2}\sin^{2}(\theta)}{r^{2}+\rho^{2}+2 r \rho \cos(\theta)}\right)\right) \mathbbm{1}_{\{|x+y| \leq \frac{s}{2}\}}\\
&- \frac{1}{2\pi} \int_{t}^{\infty} ds \int_{0}^{s-t} \frac{\rho d\rho}{\sqrt{(s-t)^{2}-\rho^{2}}} \int_{0}^{2\pi} d\theta \partial_{r}\left(WRHS_{2}(s,|x+y|) \left(1-\frac{2 \rho^{2}\sin^{2}(\theta)}{r^{2}+\rho^{2}+2 r \rho \cos(\theta)}\right)\right) \mathbbm{1}_{\{|x+y| > \frac{s}{2}\}}\end{split}\end{equation}
where we again denote the first line of the right-hand side of the above expression by $q_{III}$, and the second by $q_{IV}$.
\begin{equation}\begin{split}|q_{III}(t,r)| \leq C \int_{t}^{\infty}ds \int_{0}^{s-t} \frac{\rho d\rho}{\sqrt{(s-t)^{2}-\rho^{2}}} \int_{0}^{2\pi} d\theta &\left(\frac{|x+y| \lambda(s)^{2}}{s^{2} \log^{b}(s) (g(s)^{2}+|x+y|^{2})^{2}} \frac{|r+\rho \cos(\theta)|}{|x+y|} \right.\\
&+\left. \frac{|x+y|^{2} \lambda(s)^{2}}{s^{2} \log^{b}(s) (g(s)^{2}+|x+y|^{2})^{2}} \frac{|r+\rho \cos(\theta)| \rho^{2}\sin^{2}(\theta)}{|x+y|^{4}}\right)\end{split}\end{equation}
We then use $|r+\rho \cos(\theta)| = |(r-\rho)+\rho(\cos(\theta)+1)| \leq |r-\rho| + \rho (1+\cos(\theta))$
$$\int_{0}^{2\pi} \frac{d\theta}{(g(t)^{2}+\rho^{2}+r^{2}+2 r \rho \cos(\theta))^{2}} \leq \frac{C}{\sqrt{g(t)^{2}+\rho^{2}+r^{2}} (g(t)^{2}+(\rho-r)^{2})^{3/2}}$$
$$\frac{\rho^{2}\sin^{2}(\theta)}{\rho^{2}+r^{2}+2 r \rho \cos(\theta)} \leq C \frac{\rho^{2} \cos^{2}(\frac{\theta}{2})}{(r-\rho)^{2}+4 r \rho \cos^{2}(\frac{\theta}{2})} \leq C$$
$$\int_{0}^{2\pi} \frac{(1+\cos(\theta))}{(g(s)^{2}+r^{2}+\rho^{2}+2 r \rho \cos(\theta))^{2}} d\theta \leq \frac{C}{\sqrt{(\rho-r)^{2}+g(s)^{2}}((\rho+r)^{2}+g(s)^{2})^{3/2}}$$
and the identical procedure used to estimate $w_{2,III}$, to get
$$|q_{III}(t,r)| \leq \frac{C \lambda(t)^{2}}{t^{2} \log^{b}(t) g(t)} \log(t), \quad r \geq g(t).$$
Finally, 
$$|q_{IV}(t,r)| \leq C \int_{t}^{\infty} ds \int_{0}^{s-t} \frac{\rho d\rho}{\sqrt{(s-t)^{2}-\rho^{2}}} \left(\frac{\lambda(s)^{2}}{s^{9/2} \langle s-|x+y|\rangle^{1/2} \log^{b}(\langle s-|x+y|\rangle)}\right) \leq \frac{C \lambda(t)^{2}}{t^{5/2}}.$$
Then, we use the same observation as in step 3 of estimating $v_{2}$, along with the identical argument used to estimate $\left(\partial_{r}+\frac{2}{r}\right) w_{2}$, to get
$$|\partial_{r}^{2}w_{2}(t,r) + \frac{3}{r}\partial_{r}w_{2}(t,r)| \leq \frac{C \lambda(t)^{2} \log(t)}{t^{2} \log^{b}(t) g(t)^{2}}, \quad r > 0.$$
We can estimate $\partial_{t}^{2}w_{2}$ by using the equation solved by $w_{2}$. It then remains to estimate $\partial_{t}w_{2}$. For this, we return to \eqref{w2noderiv}, and make the substitution $\rho = q(s-t)$.
\begin{equation}\label{w2fordt}w_{2}(t,r)= \frac{-1}{2\pi} \int_{t}^{\infty} ds \int_{0}^{1} \frac{(s-t)q dq}{\sqrt{1-q^{2}}}\int_{0}^{2\pi} d\theta I_{w_{2,1}}\end{equation}
where
$$I_{w_{2,1}} = WRHS_{2}(s,\sqrt{q^{2}(s-t)^{2}+r^{2}+2 r q(s-t) \cos(\theta)})\left(1-\frac{2q^{2}(s-t)^{2} \sin^{2}(\theta)}{r^{2}+q^{2}(s-t)^{2} + 2 r q (s-t) \cos(\theta)}\right).$$
Then, we can differentiate under the integral sign. The resulting integrals can be estimated with the same procedure used to estimate analogous integrals arising in the expressions for $\partial_{r}w_{2}$ and $w_{2}$. We get
$$|\partial_{t}w_{2}(t,r)| \leq \frac{C \lambda(t)^{2} \log(t)}{t^{2} \log^{b}(t) g(t)}, \quad r >0.$$
The same procedure is used to estimate $\partial_{tr}w_{2}$, and this concludes the proof of the lemma. \end{proof}
\subsection{Summation of the higher corrections, $w_{k}$}\label{wjsection}
The nonlinear interactions between $w_{2}$ and $v_{c}$ can not be treated perturbatively in our final argument. Therefore, we will need to define corrections $w_{k}$, in a similar manner as the corrections $v_{k}$ were defined, and sum a series of the form $\sum_{k=3}^{\infty} w_{k}$. Because the estimates for $w_{j}$ and $w_{2}$ will be of a slightly different form, we will first (define and) estimate $w_{3}$, then prove estimates on $w_{j}$ for $j \geq 4$ (these $w_{j}$ are defined in \eqref{wjeqn})  by induction. We let 
$$WRHS_{3}(t,r) = \frac{6}{r^{2}}\left(Q_{\frac{1}{\lambda(t)}}+v_{c}\right) w_{2}^{2}+\frac{2 w_{2}^{3}}{r^{2}} + \frac{6 w_{2}}{r^{2}}\left(v_{c}^{2}+2v_{c}Q_{\frac{1}{\lambda(t)}}\right)$$
and define $w_{3}$ to be the solution to the following equation with 0 Cauchy data at infinity.
$$-\partial_{t}^{2}w_{3}+\partial_{r}^{2}w_{3}+\frac{1}{r}\partial_{r}w_{3}-\frac{4}{r^{2}}w_{3}=WRHS_{3}(t,r)$$

Then,
$$|WRHS_{3}(t,r)| \leq C \begin{cases} \frac{r^{2} \lambda(t)^{2} \log(2+\frac{r}{g(t)}) \log(t)}{(g(t)^{2}+r^{2})t^{4}\log^{2b}(t)}, \quad r \leq \frac{t}{2}\\
\frac{\lambda(t)^{2} \log^{2}(t)}{t^{5/2}r^{3/2} \log^{2b}(t)}, \quad r > \frac{t}{2}\end{cases}\quad |\partial_{r}WRHS_{3}(t,r)| \leq \begin{cases} \frac{C r \lambda(t)^{2} \log(t)}{g(t)^{2} t^{4} \log^{2b}(t)}, \quad r \leq g(t)\\
\frac{C \lambda(t)^{2} \log(t)}{t^{4} \log^{2b}(t) g(t)}, \quad g(t) < r \leq \frac{t}{2}\\
\frac{C \lambda(t)^{2} \log(t)}{t^{5/2} \log^{2b}(t) r^{3/2} g(t)}, \quad r > \frac{t}{2}\end{cases}$$
$$|\partial_{r}^{2}WRHS_{3}(t,r)| \leq \begin{cases} \frac{C \lambda(t)^{2} \log(t)}{t^{4} \log^{2b}(t) g(t)^{2}}, \quad r \leq \frac{t}{2}\\
\frac{C \lambda(t)^{2} \log(t)}{t^{5/2}r^{3/2} \log^{2b}(t) g(t)^{2}}, \quad r > \frac{t}{2}\end{cases}.$$
Now, we estimate $w_{3}$, starting with the region $r \leq g(t)$. The same remarks concerning the nature of the preliminary estimates on the derivatives of $w_{2}$ apply here for $w_{3}$, and eventually for $w_{j}$.
\begin{lemma}\label{w3prelim}[Preliminary estimates on $w_{3}$]
We have the following \emph{preliminary} estimates on $w_{3}$
\begin{equation}\label{w3est}|w_{3}(t,r)| \leq \begin{cases} \frac{C r^{2} \lambda(t)^{2} \log(t)}{t^{2}g(t)^{2} \log^{2b}(t)}, \quad r \leq g(t)\\
\frac{C \lambda(t)^{2} \log^{2}(t)}{t^{2} \log^{2b}(t)}, \quad r > g(t)\end{cases} \quad |\partial_{r}w_{3}(t,r)| \leq \begin{cases} \frac{C r \lambda(t)^{2} \log(t)}{t^{2} g(t)^{2} \log^{2b}(t)}, \quad r \leq g(t)\\
\frac{C \lambda(t)^{2} \log(t)}{t^{2} \log^{2b}(t) g(t)}, \quad r > g(t)\end{cases}\end{equation}
$$|\partial_{r}^{2}w_{3}(t,r)| \leq  \frac{C \lambda(t)^{2} \log(t)}{t^{2} g(t)^{2} \log^{2b}(t)} \quad |\partial_{t}^{2}w_{3}(t,r)| \leq \begin{cases} \frac{C \lambda(t)^{2} \log(t)}{t^{2} g(t)^{2} \log^{2b}(t)}, \quad r \leq g(t)\\
\frac{C \lambda(t)^{2} \log^{2}(t)}{t^{2} g(t)^{2} \log^{2b}(t)}, \quad r >g(t)\end{cases}$$
For $j=0,1$,
$$|\partial_{t}\partial_{r}^{j}w_{3}(t,r)| \leq \frac{C \lambda(t)^{2} \log(t)}{t^{2} g(t)^{1+j} \log^{2b}(t)}$$
\end{lemma}
\begin{proof}
 Using the analog of step 1 of the proof of Lemma \ref{v2lemma}, we get
$$|w_{3}(t,r)| \leq C r^{2} \int_{t}^{\infty} ds \int_{0}^{s-t} \frac{\rho d\rho}{\sqrt{(s-t)^{2}-\rho^{2}}} \int_{0}^{\pi} d\phi I_{w_{3}}$$
where
$$I_{w_{3}}= \sin^{4}(\phi) \left(1+\frac{\rho^{2}}{|x+y|^{2}}\right) \left(\frac{\lambda(s)^{2} \log(2+\frac{|x+y|}{g(s)}) \log(s)}{(g(s)^{2}+|x+y|^{2}) s^{4} \log^{2b}(s)} + \frac{\lambda(s)^{2} \log(s)}{g(s)^{2}s^{4} \log^{2b}(s)}\right).$$
We use $\frac{\log(2+\frac{|x+y|}{g(s)})}{g(s)^{2}+|x+y|^{2}} \leq \frac{C}{g(s)^{2}}$ to get
$$|w_{3}(t,r)| \leq \frac{C r^{2} \lambda(t)^{2} \log(t)}{t^{2} g(t)^{2} \log^{2b}(t)}, \quad r \leq g(t).$$
Using the same procedure as in steps 2 and 3 of the proof of Lemma \ref{v2lemma}, we get
$$|\partial_{r}w_{3}(t,r)| \leq \frac{C r \lambda(t)^{2} \log(t)}{t^{2} g(t)^{2} \log^{2b}(t)}, \quad r \leq g(t)$$
and
$$|\partial_{r}^{2}w_{3}(t,r)| \leq \frac{C \lambda(t)^{2} \log(t)}{t^{2} g(t)^{2} \log^{2b}(t)}, \quad r \leq g(t).$$
To estimate $w_{3}$ in the region $r > g(t)$, we use the analog of \eqref{w2noderiv}, and we get
$$|w_{3}(t,r)| \leq C \int_{t}^{\infty} ds \int_{0}^{s-t} \frac{\rho d\rho}{\sqrt{(s-t)^{2}-\rho^{2}}} \int_{0}^{2\pi} d\theta \frac{\lambda(s)^{2} \log^{2}(s)}{s^{4} \log^{2b}(s)} \leq \frac{C \lambda(t)^{2} \log^{2}(t)}{t^{2} \log^{2b}(t)}, \quad r > g(t).$$
Differentiating the analog of \eqref{w2noderiv}, we get
$$|\partial_{r}w_{3}(t,r)| \leq \frac{C \lambda(t)^{2} \log(t)}{t^{2} \log^{2b}(t) g(t)}, \quad r > g(t).$$
We then use the observation of step 3 of the proof of Lemma \ref{v2lemma} to get
$$|\partial_{r}^{2} w_{3}(t,r)| \leq \frac{C \lambda(t)^{2} \log(t)}{t^{2} \log^{2b}(t) g(t)^{2}}, \quad r \geq g(t).$$
Using the same procedure that was used to estimate $\partial_{t}w_{2}$, we get
$$|\partial_{t}w_{3}(t,r)| \leq \frac{C \lambda(t)^{2} \log(t)}{t^{2} g(t) \log^{2b}(t)}, \quad r >0.$$
We then read off estimates on $\partial_{t}^{2}w_{3}$, based on the equation solved by $w_{3}$, and the previous estimates. This completes the proof of the lemma. \end{proof}
Now, for $j\geq 4$, we define $w_{j}$ to be the solution to the following equation, with zero Cauchy data at infinity.
\begin{equation}\label{wjeqn}-\partial_{tt}w_{j}+\partial_{rr}w_{j}+\frac{1}{r}\partial_{r}w_{j}-\frac{4}{r^{2}}w_{j}=WRHS_{j}(t,r)\end{equation}
where
\begin{equation}\begin{split}WRHS_{j}(t,r)&:= \frac{6}{r^{2}}\left(Q_{\frac{1}{\lambda(t)}}+v_{c}\right)\left(w_{j-1}^{2}+2 \sum_{k=2}^{j-2} w_{k}w_{j-1}\right)+\frac{2}{r^{2}}\left(w_{j-1}^{3}+3 w_{j-1}^{2} \sum_{k=2}^{j-2} w_{k} + 3 w_{j-1} \left(\sum_{k=2}^{j-2} w_{k}\right)^{2}\right) \\
&+ \frac{6 w_{j-1}}{r^{2}}\left(v_{c}^{2}+2v_{c}Q_{\frac{1}{\lambda(t)}}\right).\end{split}\end{equation}
As with $v_{j}$, we will now prove estimates on $w_{j}$ by induction.
\begin{lemma}\label{wkinductionlemma}  Let $C_{2}>9$ be such that the estimates of lemmas \ref{w2prelim} and \ref{w3prelim} hold, with the constant $C = C_{2}$ on the right-hand side. Then, there exists $p>900$ and $T_{2}>0$ such that, if
$$D_{p,k} = \begin{cases} C_{2}, \quad k=3\\
C_{2}^{pk}, \quad k \geq 4\end{cases}, \quad \text{ and } q_{k} = \begin{cases} 1, \quad k =3\\
2, \quad k \geq 4\end{cases}$$
then, the following estimates are true, for all $k \geq 3$ and $t \geq T_{2}$.
\begin{equation} \label{wkassump} |w_{k}(t,r)| \leq \begin{cases} \frac{D_{p,k} r^{2} \lambda(t)^{2} \log^{q_{k}}(t)}{t^{2} g(t)^{2}\log^{b(k-1)}(t)}, \quad r \leq g(t)\\
\frac{D_{p,k} \lambda(t)^{2} \log^{2}(t)}{t^{2} \log^{b(k-1)}(t)}, \quad r > g(t)\end{cases} \quad |\partial_{r}w_{k}(t,r)| \leq \begin{cases} \frac{D_{p,k} r \lambda(t)^{2} \log^{2}(t)}{t^{2} g(t)^{2} \log^{b(k-1)}(t)}, \quad r \leq g(t)\\
\frac{D_{p,k} \lambda(t)^{2} \log^{2}(t)}{t^{2} \log^{b(k-1)}(t) g(t)}, \quad r > g(t)\end{cases}\end{equation}
\begin{equation}\label{dtwkassump} |\partial_{t}w_{k}(t,r)| \leq \frac{D_{p,k} \lambda(t)^{2} \log^{2}(t)}{t^{2} g(t) \log^{b(k-1)}(t)}\end{equation}
For $j+k=2$,
\begin{equation}\label{dtrwkassump} |\partial_{t}^{j}\partial_{r}^{k}w_{k}(t,r)| \leq \frac{D_{p,k} \lambda(t)^{2} \log^{2}(t)}{t^{2} g(t)^{2} \log^{b(k-1)}(t)}\end{equation}
\end{lemma}
Just like for the $v_{k}$ corrections, we then obtain the following corollary.
\begin{corollary}\label{wkinductioncorollary} The series
$$w_{s}:=\sum_{j=3}^{\infty} w_{j}$$ 
and the series resulting from applying any first or second order derivative termwise, converges absolutely and uniformly on the set $\{(t,r)| t \geq T_{2}, \quad r >0\}$. Moreover, if 
$$w_{c}(t,r):= w_{2}(t,r)+w_{s}(t,r)$$ 
then
\begin{equation}\label{wceqn}\begin{split}-\partial_{tt}w_{c}+\partial_{rr}w_{c}+\frac{1}{r}\partial_{r}w_{c}-\frac{4 w_{c}}{r^{2}} &= \chi_{\geq 1}(\frac{r}{g(t)}) \left(\partial_{t}^{2}Q_{\frac{1}{\lambda(t)}} - \frac{6 v_{c}}{r^{2}}\left(1-Q_{\frac{1}{\lambda(t)}}^{2}(r)\right)\right) + \frac{6}{r^{2}}\left(Q_{\frac{1}{\lambda(t)}}+v_{c}\right) w_{c}^{2}\\
&+\frac{2}{r^{2}}w_{c}^{3} + \frac{6}{r^{2}} w_{c} \left(v_{c}^{2}+2 v_{c}Q_{\frac{1}{\lambda(t)}}\right).\end{split}\end{equation}
\end{corollary}
\noindent\emph{Proof} (of Lemma \ref{wkinductionlemma}. Let $C_{2}$ be as in the lemma statement, and let $p > 900$ be otherwise arbitrary, and let $T_{0,p}>e^{(900!)^{1+\frac{1}{b}}}$ satisfy $\frac{C_{2}^{90 p}}{\log^{b}(t)} + \frac{C_{2}^{90p} \lambda(t)^{2} \log^{2}(t)}{g(t)^{2}\log^{b}(t)} < e^{-(900!)}, \quad t \geq T_{0,p}$ and be otherwise arbitrary. (We recall that $\frac{\lambda(t)}{g(t)} = \frac{1}{\log^{b-2\epsilon}(t)}$, and $b > \frac{2}{3}$, so, such a $T_{0,p}$ exists). Our goal is to show that, for a sufficiently large $T_{0,p}$, we can prove estimates on $w_{j}$ (and its derivatives), valid for all $t \geq T_{1}+T_{0,p}$ and all $r \geq 0$, by induction. In the following estimates, we assume $t \geq T_{1}+T_{0,p}$. Let 
$$D_{p,k} = \begin{cases} C_{2}, \quad k=3\\
C_{2}^{pk}, \quad k \geq 4\end{cases}, \quad q_{k} = \begin{cases} 1, \quad k =3\\
2, \quad k \geq 4\end{cases}.$$
Suppose, for any $j \geq 4$, and all $k$ with $3 \leq k \leq j-1$, that
\begin{equation}  |w_{k}(t,r)| \leq \begin{cases} \frac{D_{p,k} r^{2} \lambda(t)^{2} \log^{q_{k}}(t)}{t^{2} g(t)^{2}\log^{b(k-1)}(t)}, \quad r \leq g(t)\\
\frac{D_{p,k} \lambda(t)^{2} \log^{2}(t)}{t^{2} \log^{b(k-1)}(t)}, \quad r > g(t)\end{cases} \quad |\partial_{r}w_{k}(t,r)| \leq \begin{cases} \frac{D_{p,k} r \lambda(t)^{2} \log^{2}(t)}{t^{2} g(t)^{2} \log^{b(k-1)}(t)}, \quad r \leq g(t)\\
\frac{D_{p,k} \lambda(t)^{2} \log^{2}(t)}{t^{2} \log^{b(k-1)}(t) g(t)}, \quad r > g(t)\end{cases}\end{equation}
For $n+m=2$ or $n=1,m=0$,
\begin{equation} |\partial_{r}^{m}\partial_{t}^{n}w_{k}(t,r)|\leq  \frac{D_{p,k} \lambda(t)^{2} \log^{2}(t)}{t^{2} g(t)^{n+m} \log^{b(k-1)}(t)}, \quad r >0.\end{equation}
Then, for some constant $C$ \emph{independent} of $t,p,j$, we have the following estimates, for $t \geq T_{1}+T_{0,p}$.
$$|WRHS_{j}(t,r)| \leq C \begin{cases} \frac{C_{2}^{p(j-1)} r^{2} \lambda(t)^{2} \log^{2}(t)}{t^{4} g(t)^{2} \log^{b(j-1)}(t)}, \quad r \leq g(t)\\
\frac{C_{2}^{p(j-1)} \lambda(t)^{2} \log^{2}(t)}{t^{4} \log^{b(j-1)}(t)}, \quad g(t) < r \leq \frac{t}{2}\\
\frac{C_{2}^{p(j-1)} \lambda(t)^{2} \log^{2}(t)}{t^{5/2} r^{3/2} \log^{b(j-1)}(t)} \left(1+\frac{\lambda(t)^{2} \log^{2}(r)}{t^{2} \log^{b}(t)}\right), \quad r > \frac{t}{2}\end{cases}$$
$$|\partial_{r}WRHS_{j}(t,r)| \leq C \begin{cases} \frac{C_{2}^{p(j-1)} r \lambda(t)^{2} \log^{2}(t)}{t^{4} g(t)^{2} \log^{b(j-1)}(t)}, \quad r \leq g(t)\\
\frac{C_{2}^{p(j-1)} \lambda(t)^{2} \log^{2}(t)}{t^{4} g(t) \log^{b(j-1)}(t)}, \quad g(t) < r \leq \frac{t}{2}\\
\frac{C_{2}^{p(j-1)} \lambda(t)^{2} \log^{2}(t)}{t^{3} r \log^{b(j-1)}(t) g(t)} + \frac{C_{2}^{p(j-1)} \lambda(t)^{2} \log^{2}(t)}{r^{2}t^{5/2} \log^{b(j-2)}(t) \sqrt{\langle t-r \rangle} \log^{b}(\langle t-r \rangle)}, \quad r > \frac{t}{2}\end{cases}$$
$$|\partial_{r}^{2}WRHS_{j}(t,r)| \leq C \begin{cases} \frac{C_{2}^{p(j-1)} \lambda(t)^{2} \log^{2}(t)}{t^{4} g(t)^{2} \log^{b(j-1)}(t)}, \quad r \leq \frac{t}{2}\\
\frac{C_{2}^{p(j-1)} \lambda(t)^{2} \log^{2}(t)}{r^{3/2} t^{5/2} g(t)^{2} \log^{b(j-1)}(t)} \left(1+\frac{\lambda(t)^{2} \log^{2}(r)}{t^{2} \log^{b}(t)}\right), \quad r > \frac{t}{2}\end{cases}$$
Using the same procedure used to estimate $w_{3}$, we get the following estimates, where the constant $C$ is \emph{independent} of $t,j,p$, and $t \geq T_{1}+T_{0,p}$.
\begin{equation}\label{result1}|w_{j}(t,r)| \leq \begin{cases} \frac{C r^{2} C_{2}^{p(j-1)} \lambda(t)^{2} \log^{2}(t)}{t^{2} g(t)^{2} \log^{b(j-1)}(t)}, \quad r \leq g(t)\\
\frac{C C_{2}^{p(j-1)} \lambda(t)^{2} \log^{2}(t)}{t^{2} \log^{b(j-1)}(t)}, \quad r>g(t)\end{cases}\quad |\partial_{r}w_{j}(t,r)| \leq \begin{cases} \frac{C r C_{2}^{p(j-1)} \lambda(t)^{2} \log^{2}(t)}{t^{2} g(t)^{2} \log^{b(j-1)}(t)}, \quad r \leq g(t)\\
\frac{C C_{2}^{p(j-1)} \lambda(t)^{2} \log^{2}(t)}{t^{2} g(t) \log^{b(j-1)}(t)}, \quad r > g(t)\end{cases}\end{equation}
For $n+m=2$ or $n=1,m=0$,
$$|\partial_{r}^{m}\partial_{t}^{n}w_{j}(t,r)| \leq  \frac{C C_{2}^{p(j-1)} \lambda(t)^{2} \log^{2}(t)}{t^{2} g(t)^{m+n} \log^{b(j-1)}(t)}.$$
Therefore, there exists $p_{0}>900$ such that $C C_{2}^{p_{0}(j-1)} \leq C_{2}^{p_{0}j}$. Then, by mathematical induction, \eqref{wkassump} through \eqref{dtrwkassump} are true for all $j \geq 3$, provided that $T_{0,p_{0}}$ is chosen sufficiently large (though we have slightly better estimates on $w_{3}$ than what we supposed for the purposes of the induction argument).\qed\\
\\ 
Now, we can prove Corollary \ref{wkinductioncorollary}. By Lemma \ref{wkinductionlemma} (with $T_{2}=T_{1}+T_{0,p_{0}}$) the series (in Corollary \ref{wkinductioncorollary}) defining $w_{s}$, and the series resulting from applying any first or second order derivative termwise converges absolutely and uniformly on the set $\{(t,r)| t \geq T_{1}+T_{0,p_{0}}, \quad r >0\}$. From here on, we will further restrict $T_{0}$ to satisfy $T_{0} > T_{1}+T_{0,p_{0}}$. Then, we define $w_{c}$ as in Corollary \ref{wkinductioncorollary}. Using the fact that
\begin{equation}\begin{split} WRHS_{j}(t,r)&= \frac{6}{r^{2}}\left(Q_{\frac{1}{\lambda(t)}}+v_{c}\right) \left(\left(\sum_{k=2}^{j-1} w_{k}\right)^{2}-\left(\sum_{k=2}^{j-2} w_{k}\right)^{2}\right) + \frac{2}{r^{2}}\left(\left(\sum_{k=2}^{j-1} w_{k}\right)^{3}-\left(\sum_{k=2}^{j-2} w_{k}\right)^{3}\right)\\
&+\frac{6}{r^{2}} w_{j-1} \left(v_{c}^{2}+2v_{c}Q_{\frac{1}{\lambda(t)}}\right)\end{split}\end{equation}
we proceed as in the case of $v_{c}$, to get \eqref{wceqn}, which will be useful for us in the next section.
\subsection{Choosing $\lambda(t)$}\label{choosinglambdasection}
Let \begin{equation}\label{f4def}\begin{split}F_{4}(t,r)& = \left(1-\chi_{\geq 1}(\frac{r}{g(t)})\right)\left(\partial_{t}^{2}Q_{\frac{1}{\lambda(t)}}(r) - \frac{6 v_{c}(t,r)}{r^{2}} \left(1-Q_{\frac{1}{\lambda(t)}}^{2}(r)\right)\right) -\frac{6}{r^{2}} \chi_{\leq 1}(\frac{2 r}{t})\left(1-Q^{2}_{\frac{1}{\lambda(t)}}(r)\right) w_{c}(t,r)\end{split}\end{equation}
\begin{equation}\label{f5def}F_{5}(t,r) =-\frac{6}{r^{2}} \left(1-\chi_{\leq 1}(\frac{2 r}{t})\right)\left(1-Q^{2}_{\frac{1}{\lambda(t)}}(r)\right) w_{c}(t,r)\end{equation}
where
\begin{equation}\label{chidef}\chi_{\leq 1}(x) \in C^{\infty}(\mathbb{R}), \quad 0 \leq \chi_{\leq 1}(x) \leq 1, \quad \chi_{\leq 1}(x) = \begin{cases} 1, \quad x \leq \frac{1}{2}\\
0, \quad x \geq 1\end{cases}.\end{equation}
If we substitute $u(t,r) = Q_{\frac{1}{\lambda(t)}}(r) + v_{c}(t,r) + w_{c}(t,r)+ v(t,r)$ into \eqref{ym}, we get
\begin{equation}\begin{split}-\partial_{tt}v+\partial_{rr}v+\frac{1}{r}\partial_{r}v+\frac{2}{r^{2}}\left(1-3Q_{\frac{1}{\lambda(t)}}(r)^{2}\right) v &=F_{4}(t,r)+F_{5}(t,r) + \frac{2 v^{3}}{r^{2}} + \frac{6}{r^{2}}\left(Q_{\frac{1}{\lambda(t)}}(r) + v_{c}+w_{c}\right) v^{2} \\
&+ \frac{6 v}{r^{2}}\left(\left(v_{c}+Q_{\frac{1}{\lambda(t)}}+w_{c}\right)^{2}-Q_{\frac{1}{\lambda(t)}}^{2}\right).\end{split}\end{equation}
\subsubsection{Estimates on $F_{5}$}
We will now show that $F_{5}$ decays sufficiently quickly in sufficiently many norms, so that we do not need to include it in the modulation equation for $\lambda$. By directly substituting the estimates of the previous sections into the definition of $F_{5}$, we get: there exists $C>0$ such that for all $\lambda$ satisfying \eqref{lambdarestr}, we have
\begin{equation}\label{f5l2}||F_{5}(t,R \lambda(t))||_{L^{2}(R dR)} \leq \frac{C \lambda(t)^{3}}{t^{5} \log^{b-2}(t)}\end{equation}
\begin{equation}\label{lstarlf5l2}||L^{*}L\left(F_{5}(t,R\lambda(t))\right)||_{L^{2}(R dR)} \leq \frac{C \lambda(t)^{5} \log^{2}(t)}{g(t)^{2} \log^{b}(t) t^{5}}.\end{equation}
\subsubsection{Solving the modulation equation}
Now, we will choose $\lambda(t)$ so that
$$\langle F_{4}(t,R\lambda(t)),\phi_{0}(R)\rangle_{L^{2}(R dR)} =0.$$
This equation can be re-written in the form
\begin{equation}\label{lambdaeqn}\begin{split}&\langle \partial_{tt}Q_{\frac{1}{\lambda(t)}} - \frac{6}{r^{2}} v_{1} \left(1-Q_{\frac{1}{\lambda(t)}}^{2}(r)\right)\Bigr|_{r=R\lambda(t)},\phi_{0}(R)\rangle_{L^{2}(R dR)} \\
&= \langle \frac{6}{r^{2}} \sum_{k=2}^{\infty} v_{k} \left(1-Q_{\frac{1}{\lambda(t)}}^{2}(r)\right)\Bigr|_{r=R\lambda(t)}, \phi_{0}(R)\rangle_{L^{2}(R dR)}\\
&+\langle \chi_{\geq 1}(\frac{R \lambda(t)}{g(t)}) \left(\partial_{tt}Q_{\frac{1}{\lambda(t)}} - \frac{6}{r^{2}} v_{c} \left(1-Q_{\frac{1}{\lambda(t)}}^{2}(r)\right)\right)\Bigr|_{r=R\lambda(t)}, \phi_{0}(R)\rangle_{L^{2}(R dR)}\\
&+\langle \frac{\chi_{\leq 1}(\frac{2R \lambda(t)}{t})}{r^{2}} 6 \left(1-Q_{\frac{1}{\lambda(t)}}^{2}(r)\right)w_{c}(t,r)\Bigr|_{r=R\lambda(t)}, \phi_{0}(R)\rangle_{L^{2}(R dR)}\end{split} \end{equation}
The main result of this section is
\begin{proposition}\label{choosinglambdaprop} There exists $T_{3}>0$ such that for all $T_{0} \geq T_{3}$, there exists a solution, $\lambda$ (which is of the form \eqref{lambdarestr}) to \eqref{lambdaeqn}, for $t \geq T_{0}$. In addition, $\lambda(t) \in C^{4}([T_{0},\infty))$, and satisfies
$$\lambda(t) = \lambda_{0}(t)\left(1+e(t)\right)$$
where
\begin{equation}\nonumber\begin{split} |e(t)| \leq \frac{C}{\log^{\delta-\delta_{2}}(t)}, \quad |e^{k}(t)| \leq \begin{cases} \frac{C}{t^{k} \log^{1+\delta-\delta_{2}}(t)}, \quad k=1,2\\
\frac{C}{t^{3} \log^{b+\delta_{4}}(t)}, \quad k=3\\
\frac{C}{t^{4} \log^{b+\delta_{5}}(t)}, \quad k=4\end{cases}\end{split}\end{equation}
where $\delta$, $\delta_{2}$ are defined in \eqref{deltadef}, \eqref{delta2def}, respectively, and $\delta_{4}, \delta_{5}>0$.
\end{proposition}
We start by computing the left-hand side of \eqref{lambdaeqn}. Firstly, we have
$$\langle \partial_{tt}Q_{\frac{1}{\lambda(t)}}\Bigr|_{r=R\lambda(t)},\phi_{0}(R)\rangle_{L^{2}(R dR)} = \frac{2 \lambda''(t)}{3\lambda(t)}.$$
Next, we start by noting that
$$\frac{6\left(1-Q_{1}^{2}(R)\right)}{R^{2} \lambda(t)^{2}} \phi_{0}(R) = \frac{24 R^{2}}{\lambda(t)^{2}(1+R^{2})^{4}}.$$
Then, we note that
$$\frac{24}{\lambda(t)^{2}} \int_{0}^{\infty} \frac{R^{3} J_{2}(\xi R \lambda(t)) dR}{(1+R^{2})^{4}} = \frac{\xi^{3}\lambda(t)}{2} K_{1}(\xi\lambda(t))$$
(which follows from combining integral identities of \cite{gr}). Finally, we recall
$$v_{1}(t,R \lambda(t)) = \int_{0}^{\infty} \widehat{v_{1,1}}(\xi) \sin(t\xi) J_{2}(R \lambda(t)\xi) d\xi.$$
 Therefore,
 \begin{equation} \begin{split} &\langle- \frac{6}{r^{2}} v_{1} \left(1-Q_{\frac{1}{\lambda(t)}}^{2}(r)\right)\Bigr|_{r=R\lambda(t)},\phi_{0}(R)\rangle_{L^{2}(R dR)} \\
 &= \frac{-24}{\lambda(t)^{2}} \int_{0}^{\infty} v_{1}(t,R\lambda(t)) \frac{R^{3}}{(1+R^{2})^{4}} dR= \frac{-24}{\lambda(t)^{2}} \int_{0}^{\infty} d\xi \sin(t\xi) \widehat{v_{1,1}}(\xi) \int_{0}^{\infty} J_{2}(R \lambda(t) \xi) \frac{R^{3} dR}{(1+R^{2})^{4}}\\
 &=-\int_{0}^{\infty} d\xi \sin(t\xi) \widehat{v_{1,1}}(\xi) \frac{\xi^{3} \lambda(t)}{2} K_{1}(\xi\lambda(t))\\
 &=-\frac{\lambda(t)}{2}\int_{0}^{\infty} d\xi \sin(t\xi) \widehat{v_{1,1}}(\xi) \xi^{3} \cdot \frac{1}{\xi \lambda(t)}-\frac{\lambda(t)}{2} \int_{0}^{\infty} d\xi \sin(t\xi) \widehat{v_{1,1}}(\xi) \xi^{3} \left(K_{1}(\xi \lambda(t))-\frac{1}{\xi \lambda(t)}\right).\end{split}\end{equation}
Next, recalling the formula for $\widehat{v_{1,1}}$, namely \eqref{v11def}, and using the sin transform inversion formula, we get
\begin{equation}\label{v1firstip}\langle- \frac{6}{r^{2}} v_{1} \left(1-Q_{\frac{1}{\lambda(t)}}^{2}(r)\right)\Bigr|_{r=R\lambda(t)},\phi_{0}(R)\rangle_{L^{2}(R dR)} = -\frac{2 \lambda_{0}''(t)}{3\lambda_{0}(t)} + E_{v_{1},ip}(t,\lambda(t))\end{equation}
where
\begin{equation}\nonumber \begin{split}E_{v_{1},ip}(t,\lambda(t)) &= \frac{-\lambda(t)}{2} \int_{0}^{\infty} d\xi \sin(t\xi) \widehat{v_{1,1}}(\xi) \xi^{3}\left(K_{1}(\xi\lambda(t))-\frac{1}{\xi \lambda(t)}\right)\\
& = \frac{\lambda(t)}{2} \int_{0}^{\infty} d\xi \frac{\cos(t\xi)}{t^{3}} \partial_{\xi}^{3}\left(\widehat{v_{1,1}}(\xi) \xi^{3}\left(K_{1}(\xi \lambda(t))-\frac{1}{\xi\lambda(t)}\right)\right).\end{split}\end{equation}
Using the symbol-type estimates on $\widehat{v_{1,1}}$ (from Lemma \ref{v11lemma}), asymptotics of the modified Bessel function of the second kind (from \eqref{k1notation}, and its analogs for derivatives of $K_{1}$), \eqref{lambdarestr}, and the observation \eqref{lambdagrowth}  we get
$$|E_{v_{1},ip}(t,\lambda(t))| \leq \frac{C \log(t)}{t^{5/2}}$$
(where the power of $t$ in the denominator could be improved, but is sufficient for our purposes). On the other hand, we have
$$\sum_{k=2}^{\infty} |v_{k}(t,r)| \leq \begin{cases} \frac{C_{1}^{2n_{0}} r^{2}}{t^{2} \log^{2b}(t)}, \quad r \leq \frac{t}{2}\\
\frac{C_{1}^{2n_{0}}}{\log^{2b}(t)}, \quad r > \frac{t}{2}\end{cases}$$
and this gives
$$|v_{sip}(t,\lambda(t))| \leq \frac{C}{t^{2} \log^{2b}(t)} $$
where
$$v_{sip}(t,\lambda(t)) = \langle \frac{6}{r^{2}} \sum_{k=2}^{\infty} v_{k} \left(1-Q_{\frac{1}{\lambda(t)}}^{2}(r)\right)\Bigr|_{r=R\lambda(t)}, \phi_{0}(R)\rangle_{L^{2}(R dR)}.$$
Using our estimates from previous sections, we get
$$|lin_{ip}(t,\lambda(t))| \leq \frac{C}{t^{2} \log^{b}(t) \log^{2(b-2\epsilon)}(t)}$$
where
$$lin_{ip}(t,\lambda(t)) := \langle \chi_{\geq 1}(\frac{R \lambda(t)}{g(t)}) \left(\partial_{tt}Q_{\frac{1}{\lambda(t)}} - \frac{6}{r^{2}} v_{c} \left(1-Q_{\frac{1}{\lambda(t)}}^{2}(r)\right)\right)\Bigr|_{r=R\lambda(t)}, \phi_{0}(R)\rangle_{L^{2}(R dR)}$$
and we recall the definition of $g$: $g(t) = \lambda(t) \log^{b-2\epsilon}(t)$. Next, for $j \geq 3$, we use the estimates on $w_{j}$ given in \eqref{wkassump}, to get
$$|w_{c,ip}(t,\lambda(t))| \leq \frac{C}{t^{2}}\left(\frac{1}{\log^{3b-4\epsilon-1}(t)}+\frac{1}{\log^{5b-8\epsilon-2}(t)}\right)$$
where
$$w_{c,ip}(t,\lambda(t)) = \langle \frac{\chi_{\leq 1}(\frac{2R \lambda(t)}{t})}{r^{2}} 6 \left(1-Q_{\frac{1}{\lambda(t)}}^{2}(r)\right)w_{c}(t,r)\Bigr|_{r=R\lambda(t)}, \phi_{0}(R)\rangle_{L^{2}(R dR)}.$$
Substituting $\lambda(t) = \lambda_{0}(t) \left(1+e(t)\right), \quad e \in \overline{B}_{1}(0) \subset X$ (where we recall \eqref{lambdarestr} and \eqref{xnorm}) into \eqref{lambdaeqn}, we get 
\begin{equation} \label{eeqn} e''(t) + \frac{2 \lambda_{0}'(t)}{\lambda_{0}(t)} e'(t) = \frac{3 G(t,\lambda_{0}(t)\left(1+e(t)\right))}{2}\left(1+e(t)\right)\end{equation}
where
\begin{equation} G(t,\lambda(t)) = v_{sip}(t,\lambda(t))+lin_{ip}(t,\lambda(t))+w_{c,ip}(t,\lambda(t))-E_{v_{1},ip}(t,\lambda(t)).\end{equation}
Let $\mathcal{B} :=\overline{B}_{1}(0) \subset X$. Our goal is to solve \eqref{eeqn} for $e \in \mathcal{B}$ using a fixed point argument. So, we define $T$ on  $\mathcal{B}$ by
$$T(e)(t) = \int_{t}^{\infty} \frac{dx}{\lambda_{0}(x)^{2}} \int_{x}^{\infty} \frac{3}{2} \lambda_{0}(s)^{2} G(s,\lambda_{0}(s)\left(1+e(s)\right))\left(1+e(s)\right) ds.$$
Combining our estimates above, we get
$$|G(t,\lambda_{0}(t)\left(1+e(t)\right))| \leq \frac{C}{t^{2} \log^{1+\delta}(t)}$$
where we recall the definition of $\delta$ in \eqref{deltadef}.
 This gives
$$|\int_{x}^{\infty} \frac{3}{2}\lambda_{0}(s)^{2} G(s,\lambda_{0}(s)\left(1+e(s)\right))\left(1+e(s)\right) ds| \leq C \int_{x}^{\infty} \frac{\lambda_{0}(s)^{2} ds}{s^{2} \log^{1+\delta}(s)}.$$
Then, we integrate by parts to get
$$\int_{x}^{\infty} \frac{\lambda_{0}(s)^{2} ds}{s^{2} \log^{1+\delta}(s)} = \frac{\lambda_{0}(x)^{2}}{x \log^{1+\delta}(x)} + \int_{x}^{\infty}\frac{ds}{s}\left(\frac{2 \lambda_{0}(s)\lambda_{0}'(s)}{\log^{1+\delta}(s)}-\frac{(\delta+1) \lambda_{0}(s)^{2}}{s \log^{\delta+2}(s)}\right). $$
Therefore, 
$$|\int_{x}^{\infty} \frac{\lambda_{0}(s)^{2} ds}{s^{2} \log^{1+\delta}(s)}| \leq C \frac{\lambda_{0}(x)^{2}}{x \log^{1+\delta}(x)} + \left(\frac{C}{\log^{b}(x)}+\frac{C}{\log(x)}\right) \int_{x}^{\infty} \frac{\lambda_{0}(s)^{2} ds}{s^{2} \log^{1+\delta}(s)}$$
So, there exists $T_{2}>T_{1}+T_{0,p}$, and $C>0$ such that, for all $x \geq T_{2}$,
\begin{equation}\label{i2b}|\int_{x}^{\infty} \frac{\lambda_{0}(s)^{2} ds}{s^{2} \log^{1+\delta}(s)}| \leq C \frac{\lambda_{0}(x)^{2}}{x \log^{1+\delta}(x)}.\end{equation}
So, for all $T_{0} \geq T_{2}$, we have
$$|T(e)'(t)| \leq \frac{C}{t \log^{\delta+1}(t)}, \quad t \geq T_{0}$$
$$|T(e)(t)| \leq C \int_{t}^{\infty} \frac{dx}{x \log^{\delta+1}(x)} \leq \frac{C}{\log^{\delta}(t)}, \quad t \geq T_{0}$$
and
$$|T(e)''(t)| \leq \frac{C}{t^{2} \log^{1+\delta}(t)}, \quad t \geq T_{0}.$$
In particular, $T: \mathcal{B} \rightarrow \mathcal{B}$. Now, we will study the Lipschitz properties of $T$. We recall that $v_{sip}$ depends on $\lambda$. To emphasize the dependence of $v_{k}$ on $\lambda$, we will write $v_{k}=v_{k}^{\lambda}$. Similarly, we denote the previously defined functions $RHS_{k}$ by $RHS_{k}^{\lambda}$. Our goal is to understand the Lipschitz (in $e$) dependence of $v_{sip}(t,\lambda_{0}(t)\left(1+e(t)\right))$ and $E_{v_{1},ip}(t,\lambda_{0}(t)\left(1+e(t)\right))$, for $e \in \mathcal{B}$. For $i=1,2$, let $e_{i} \in \mathcal{B}$, and let $\lambda_{i}(t) = \lambda_{0}(t)\left(1+e_{i}(t)\right)$. Let 
$$F(r,\lambda(t)) = \frac{\left(1-Q_{1}^{2}(\frac{r}{\lambda(t)})\right) \phi_{0}(\frac{r}{\lambda(t)})}{r^{2}\lambda(t)^{2}} r.$$
Then, 
$$v_{sip}(t,\lambda(t)) = 6 \int_{0}^{\infty} \sum_{k=2}^{\infty} v_{k}^{\lambda}(t,r) F(r,\lambda(t)) dr.$$
We start with 
$$|\partial_{2} F(r,\lambda(t))| \leq \frac{C r^{3}\lambda(t)}{(r^{2}+\lambda(t)^{2})^{4}}, \quad |F(r,\lambda(t))| \leq \frac{C r^{3} \lambda(t)^{2}}{(r^{2}+\lambda(t)^{2})^{4}}.$$
To understand the Lipschitz (in $e$) dependence of $v_{k}^{\lambda_{0}\left(1+e\right)}$, we start by noting that $v_{2}^{\lambda_{1}}-v_{2}^{\lambda_{2}}$ solves the following equation with zero Cauchy data at infinity.
\begin{equation} \left(-\partial_{tt}+\partial_{rr}+\frac{1}{r}\partial_{r}-\frac{4}{r^{2}}\right)\left(v_{2}^{\lambda_{1}}-v_{2}^{\lambda_{2}}\right) = \frac{6 v_{1}(t,r)^{2}}{r^{2}} \left(Q_{\frac{1}{\lambda_{1}}(t)}(r)-Q_{\frac{1}{\lambda_{2}(t)}}(r)\right)\end{equation}
There exists an absolute constant $C$ such that, for all $e \in \mathcal{B}$ we have
$$C^{-1} \lambda_{0}(t) \leq \lambda_{0}(t)|\left(1+e(t)\right)| \leq C \lambda_{0}(t), \quad t \geq T_{0}.$$
Using this, we get $|Q_{\frac{1}{\lambda_{1}(t)}}(r)-Q_{\frac{1}{\lambda_{2}(t)}}(r)| \leq \frac{C|\lambda_{2}(t)-\lambda_{1}(t)| \lambda_{0}(t)}{r^{2}}$ and this gives
$$||RHS_{2}^{\lambda_{1}}(s,r)-RHS_{2}^{\lambda_{2}}(s,r)||_{L^{2}(r dr)} \leq \frac{C |\lambda_{2}(s)-\lambda_{1}(s)| \lambda_{0}(s)}{s^{3} \log^{2b}(s)}.$$
Using the procedure of \eqref{enest}, and the estimate \eqref{i2b}, we get
$$|v_{2}^{\lambda_{1}}-v_{2}^{\lambda_{2}}|(t,r) \leq \frac{C ||e_{1}-e_{2}||_{X}}{t \log^{\delta-\delta_{2}}(t)}\frac{\lambda_{0}(t)^{2}}{t \log^{2b}(t)}, \quad r \geq 0, \quad t \geq T_{0}.$$
Then, a similar induction procedure used to construct $v_{s}$ shows that there exists $C_{2},m_{0},T_{3}>T_{2}$ such that, for $t \geq T_{3}$,
$$|v_{j}^{\lambda_{1}}-v_{j}^{\lambda_{2}}|(t,r) \leq \frac{C_{2}^{m_{0} j} ||e_{1}-e_{2}||_{X} \lambda_{0}(t)^{2}}{t^{2} \log^{bj+\delta-\delta_{2}}(t)}, \quad j \geq 2.$$
(The main difference between the procedure used to establish the above estimates, and that used to construct $v_{s}$ is that here, we need only inductively prove estimates on $RHS_{j}^{\lambda_{1}}-RHS_{j}^{\lambda_{2}}$, and use the procedure of \eqref{enest} to estimate $v_{j}^{\lambda_{1}}-v_{j}^{\lambda_{2}}$). The above estimates imply that, if $T_{0} \geq T_{3}$ (which we will assume from now on) then, we have
$$\sum_{k=2}^{\infty} |v_{k}^{\lambda_{1}}-v_{k}^{\lambda_{2}}|(t,r) \leq \frac{C ||e_{1}-e_{2}||_{X} \lambda_{0}(t)^{2}}{t^{2} \log^{2b}(t) \log^{\delta-\delta_{2}}(t)}, \quad t \geq T_{0}, \quad r \geq 0.$$
This gives
$$|v_{sip}(t,\lambda_{1}(t))-v_{sip}(t,\lambda_{2}(t))| \leq \frac{C ||e_{1}-e_{2}||_{X}}{t^{2} \log^{2b}(t) \log^{\delta-\delta_{2}}(t)}, \quad t \geq T_{0}.$$
Again using properties of the modified Bessel function of the second kind, we get
$$|E_{v_{1},ip}(t,\lambda_{1}(t))-E_{v_{1},ip}(t,\lambda_{2}(t))| \leq C \frac{||e_{1}-e_{2}||_{X} \log(t)}{t^{5/2} \log^{\delta-\delta_{2}}(t)}.$$
Next, we estimate $lin_{ip}(t,\lambda_{1}(t)) - lin_{ip}(t,\lambda_{2}(t))$. We start by noting that
$$|\chi_{\geq 1}(\frac{r}{g_{1}(t)}) - \chi_{\geq 1}(\frac{r}{g_{2}(t)})| \leq \frac{C ||e_{1}-e_{2}||_{X} \mathbbm{1}_{\{r \geq \frac{g_{0}(t)}{4}\}}}{\log^{\delta-\delta_{2}}(t)}.$$ 
Next, we let $F_{s}(r,\lambda(t),\lambda'(t),\lambda''(t)) = \frac{\partial_{tt}Q_{1}(\frac{r}{\lambda(t)}) \phi_{0}(\frac{r}{\lambda(t)})}{\lambda(t)^{2}}.$ Then, we use
\begin{equation}\nonumber\begin{split}&F_{s}(r,\lambda_{1}(t),\lambda_{1}'(t),\lambda_{1}''(t)) - F_{s}(r,\lambda_{2}(t),\lambda_{2}'(t),\lambda_{2}''(t)) \\
&= \int_{0}^{1} DF_{s}(r,\lambda_{\sigma}(t)) \cdot \left(\lambda_{1}(t)-\lambda_{2}(t),\lambda_{1}'(t)-\lambda_{2}'(t),\lambda_{1}''(t)-\lambda_{2}''(t)\right) d\sigma\end{split}\end{equation}
where $DF_{s}$ denotes the gradient in the last three arguments of $F_{s}$, and
$$\lambda_{\sigma}(t) = \left(\sigma \lambda_{1}(t)+(1-\sigma)\lambda_{2}(t),\sigma \lambda_{1}'(t)+(1-\sigma)\lambda_{2}'(t),\sigma \lambda_{1}''(t)+(1-\sigma)\lambda_{2}''(t)\right).$$
This gives
$$|F_{s}(r,\lambda_{1}(t),\lambda_{1}'(t),\lambda_{1}''(t)) - F_{s}(r,\lambda_{2}(t),\lambda_{2}'(t),\lambda_{2}''(t))| \leq \frac{C ||e_{1}-e_{2}||_{X} \lambda_{0}(t)^{2} r^{4}}{t^{2} \log^{\delta-\delta_{2}}(t) (r^{2}+\lambda_{0}(t)^{2})^{4}}\left(\frac{1}{\log^{b}(t)}+\frac{1}{\log(t)}\right).$$
Then, we use our estimates above, and recall that $b > \frac{2}{3}$ to conclude
$$|lin_{ip}(t,\lambda_{1}(t)) - lin_{ip}(t,\lambda_{2}(t))| \leq \frac{C ||e_{1}-e_{2}||_{X}}{t^{2} \log^{\delta-\delta_{2}}(t)} \left(\frac{1}{\log^{b}(t)}+\frac{1}{\log(t)}\right)\frac{1}{\log^{2b-4\epsilon}(t)}.$$
To study $w_{c,ip}(t,\lambda_{1}(t))-w_{c,ip}(t,\lambda_{2}(t))$, we will need to estimate $w_{k}^{\lambda_{1}}-w_{k}^{\lambda_{2}}$, where we use the same notational convention as we used for $v_{k}$. For later use, let $g_{i}(t) = \log^{b-2\epsilon}(t) \lambda_{i}(t), \quad i =0,1,2$. We start with $k=2$. We split $WRHS_{2}^{\lambda_{1}}-WRHS_{2}^{\lambda_{2}}$ as follows. We define
\begin{equation}\begin{split}WRHS_{2,lip,0}(t,r) &:= \chi_{\geq 1}(\frac{r}{g_{2}(t)}) \left(\partial_{t}^{2}Q_{1}(\frac{r}{\lambda_{1}(t)})-\partial_{t}^{2}Q_{1}(\frac{r}{\lambda_{2}(t)})-\frac{6 v_{c}^{\lambda_{2}}(t,r)}{r^{2}}\left(Q_{1}^{2}(\frac{r}{\lambda_{2}(t)})-Q_{1}^{2}(\frac{r}{\lambda_{1}(t)})\right)\right)\\
&+\left(\chi_{\geq 1}(\frac{r}{g_{1}(t)})-\chi_{\geq 1}(\frac{r}{g_{2}(t)})\right)\left(\partial_{t}^{2}Q_{\frac{1}{\lambda_{1}(t)}}(r) - \frac{6 v_{c}^{\lambda_{1}}(t,r)}{r^{2}}\left(1-Q_{1}^{2}(\frac{r}{\lambda_{1}(t)})\right)\right)\end{split}\end{equation}
and $$WRHS_{2,lip,1}(t,r) = WRHS_{2}^{\lambda_{1}}(t,r)-WRHS_{2}^{\lambda_{2}}(t,r)-WRHS_{2,lip,0}(t,r)$$ 
and write $w_{2}^{\lambda_{1}}(t,r)-w_{2}^{\lambda_{2}}(t,r) = w_{2,lip,0}(t,r)+w_{2,lip,1}(t,r)$, where $w_{2,lip,j}$ solves the following equation with zero Cauchy data at infinity.
$$-\partial_{tt}w_{2,lip,j}+\partial_{rr}w_{2,lip,j}+\frac{1}{r}\partial_{r}w_{2,lip,j}-\frac{4}{r^{2}}w_{2,lip,j} = WRHS_{2,lip,j}$$
The point of this splitting is that we will need to use a more complicated procedure to estimate $w_{2,lip,0}$, since too many logarithmic losses in estimates are insufficient for our purposes. We have
$$|WRHS_{2,lip,1}(t,r)| \leq \frac{C ||e_{1}-e_{2}||_{X} \mathbbm{1}_{\{r \geq \frac{g_{0}(t)}{4}\}} \lambda_{0}(t)^{4}}{t^{2} \log^{\delta-\delta_{2}}(t) \log^{2b}(t) (g_{0}(t)^{2}+r^{2})^{2}}.$$
Using the analog of \eqref{w2noderiv}, and a similar procedure used to estimate various integrals arising in the $w_{2}$ estimates above, we get
$$|w_{2,lip,1}(t,r)| \leq \frac{C ||e_{1}-e_{2}||_{X} \lambda_{0}(t)^{4} \log(t)}{t^{2} \log^{\delta-\delta_{2}}(t) \log^{2b}(t) g_{0}(t)^{2}}, \quad r >0.$$
In particular, the procedure used to estimate $w_{2,lip,1}$ does not involve any derivatives of $WRHS_{2,lip,1}$, which is why we did not need to prove any estimates on derivatives of $v_{c}^{\lambda_{1}}-v_{c}^{\lambda_{2}}$. (Note that $v_{c}^{\lambda_{1}}-v_{c}^{\lambda_{2}}$ arises in some terms of $WRHS_{2,lip,1}$). Next, we note that
\begin{equation}\begin{split}&|\partial_{r}WRHS_{2,lip,0}(t,r)| + \frac{|WRHS_{2,lip,0}(t,r)|}{r} \\
&\leq C \mathbbm{1}_{\{r \geq \frac{g_{0}(t)}{4}\}} \left(\frac{r \lambda_{0}(t)^{2} ||e_{1}-e_{2}||_{X}}{t^{2} \log^{\delta-\delta_{2}}(t) (r^{2}+\lambda_{0}(t)^{2})^{2}}\right)\left(\frac{1}{\log^{b}(t)}+\frac{1}{\log(t)}\right) \\
&+ C \mathbbm{1}_{\{r \geq \frac{g_{0}(t)}{4}\}} \frac{\lambda_{0}(t)^{2} ||e_{1}-e_{2}||_{X}}{\log^{\delta-\delta_{2}}(t) (\lambda_{0}(t)^{2}+r^{2})^{2}} \begin{cases} \frac{r}{t^{2} \log^{b}(t)}, \quad r \leq \frac{t}{2}\\
\frac{1}{\sqrt{r} \sqrt{\langle t-r \rangle} \log^{b}(\langle t-r \rangle)}, \quad r > \frac{t}{2}\end{cases}\end{split}\end{equation}
and
\begin{equation}\begin{split} &|\partial_{r}^{2}WRHS_{2,lip,0}(t,r)| \\
&\leq C \mathbbm{1}_{\{r \geq \frac{g_{0}(t)}{4}\}} \frac{\lambda_{0}(t)^{2} ||e_{1}-e_{2}||_{X}}{t^{2} \log^{\delta-\delta_{2}}(t) (r^{2} + \lambda_{0}(t)^{2})^{2}}\left(\frac{1}{\log^{b}(t)}+\frac{1}{\log(t)}\right) \\
&+ \frac{C \mathbbm{1}_{\{r \geq \frac{g_{0}(t)}{4}\}} \lambda_{0}(t)^{2} ||e_{1}-e_{2}||_{X}}{\log^{\delta-\delta_{2}}(t) (\lambda_{0}(t)^{2}+r^{2})^{2}} \begin{cases} \frac{1}{t^{2} \log^{b}(t)}, \quad r \leq \frac{t}{2}\\
\frac{1}{r^{3/2} \sqrt{\langle t-r \rangle} \log^{b}(\langle t-r \rangle)}, \quad r > \frac{t}{2}\end{cases}\\
&+C \frac{\mathbbm{1}_{\{r \geq \frac{g_{0}(t)}{4}\}}\lambda_{0}(t)^{2} ||e_{1}-e_{2}||_{X}}{\log^{\delta-\delta_{2}}(t) (r^{2}+\lambda_{0}(t)^{2})^{2}} \begin{cases} \frac{1}{t^{2} \log^{b}(t)}, \quad r \leq \frac{t}{2}\\
 \begin{cases}\frac{1}{\sqrt{r} \log^{b}(\langle t-r \rangle) \langle t-r \rangle^{3/2}}+\frac{1}{t \langle t-r \rangle \log^{b}(t) \log^{b}(\langle t-r \rangle)}, \quad t > r > \frac{t}{2}\\
\frac{1}{\sqrt{r} \log^{b}(\langle t-r \rangle) \langle t-r \rangle^{3/2}}+\frac{1}{t^{3/2} \log^{b}(t)}, \quad r > \frac{t}{2}\end{cases}\end{cases}.\end{split}\end{equation}
(Note that the estimate \eqref{d2vcest}, on $\partial_{r}^{2}v_{c}$, is what gives rise to the form of the estimates recorded above). Using the same procedure that we used to estimate $\partial_{r}^{k}w_{2}$ for $k=0,1,2$, we get
$$|w_{2,lip,0}(t,r)| \leq\begin{cases} \frac{C r^{2} \lambda_{0}(t)^{2} \log(2+\frac{r}{g_{0}(t)}) \log(t) ||e_{1}-e_{2}||_{X}}{t^{2} (g_{0}(t)^{2}+r^{2}) \log^{\delta-\delta_{2}}(t)}\left(\frac{1}{\log^{b}(t)}+\frac{1}{\log(t)}\right), \quad r \leq \frac{t}{2}\\
\frac{C \lambda_{0}(t)^{2} ||e_{1}-e_{2}||_{X}}{t^{2} \log^{\delta-\delta_{2}}(t)}\left(\frac{1}{\log^{b}(t)}+\frac{1}{\log(t)}\right), \quad r > \frac{t}{2}\end{cases}$$
$$|\partial_{r}w_{2,lip,0}(t,r)| \leq \begin{cases} \frac{C r \lambda_{0}(t)^{2} \log(t) ||e_{1}-e_{2}||_{X}}{t^{2} g_{0}(t)^{2} \log^{\delta-\delta_{2}}(t)}\left(\frac{1}{\log^{b}(t)}+\frac{1}{\log(t)}\right), \quad r \leq g_{0}(t)\\
\frac{C \lambda_{0}(t)^{2} \log(t) ||e_{1}-e_{2}||_{X}}{t^{2} g_{0}(t) \log^{\delta-\delta_{2}}(t)}\left(\frac{1}{\log^{b}(t)}+\frac{1}{\log(t)}\right), \quad r > g_{0}(t)\end{cases}$$
$$|\partial_{r}^{2} w_{2,lip,0}(t,r)| \leq \frac{C \lambda_{0}(t)^{2} \log(t) ||e_{1}-e_{2}||_{X}}{t^{2} g_{0}(t) \log^{\delta-\delta_{2}}(t)}\left(\frac{1}{\log^{b}(t)}+\frac{1}{\log(t)}\right), \quad r > 0$$
Next, we consider $w_{3}^{\lambda_{1}}-w_{3}^{\lambda_{2}}$, and define $WRHS_{3,lip,0}$ by
\begin{equation}\begin{split}WRHS_{3,lip,0}(t,r) &:= \frac{6}{r^{2}}\left(Q_{\frac{1}{\lambda_{1}}}+v_{c}^{\lambda_{1}}\right) w_{2,lip,0}\left(w_{2}^{\lambda_{1}}+w_{2}^{\lambda_{2}}\right) + \frac{2}{r^{2}} w_{2,lip,0}\left((w_{2}^{\lambda_{1}})^{2}+w_{2}^{\lambda_{1}}w_{2}^{\lambda_{2}} + (w_{2}^{\lambda_{2}})^{2}\right) \\
&+ \frac{6}{r^{2}} w_{2,lip,0} \left((v_{c}^{\lambda_{1}})^{2}+2 v_{c}^{\lambda_{1}}Q_{\frac{1}{\lambda_{1}}}\right).\end{split}\end{equation}
As before, we also define
$$WRHS_{3,lip,1}(t,r):= WRHS_{3}^{\lambda_{1}}(t,r)-WRHS_{3}^{\lambda_{2}}(t,r) - WRHS_{3,lip,0}(t,r)$$
and, for $j=0,1$, we let $w_{3,lip,j}$ solve the following equation with zero Cauchy data at infinity.
$$-\partial_{tt}w_{3,lip,j}+\partial_{rr}w_{3,lip,j}+\frac{1}{r}\partial_{r}w_{3,lip,j}-\frac{4}{r^{2}}w_{3,lip,j} = WRHS_{3,lip,j}$$
Noting the similarities between $WRHS_{3,lip,0}$ and $WRHS_{3}$, we get
$$|WRHS_{3,lip,0}(t,r)| \leq  \begin{cases} \frac{C r^{2} \lambda_{0}(t)^{2} \log(2+\frac{r}{g_{0}(t)}) \log(t) ||e_{1}-e_{2}||_{X}}{(g_{0}(t)^{2}+r^{2}) t^{4} \log^{b}(t) \log^{\delta-\delta_{2}}(t)}\left(\frac{1}{\log^{b}(t)}+\frac{1}{\log(t)}\right), \quad r \leq \frac{t}{2}\\
\frac{C \lambda_{0}(t)^{2} ||e_{1}-e_{2}||_{X}}{t^{5/2} r^{3/2} \log^{b}(t) \log^{\delta-\delta_{2}}(t)}\left(\frac{1}{\log^{b}(t)}+\frac{1}{\log(t)}\right), \quad r > \frac{t}{2}\end{cases}$$

$$|\partial_{r}WHRS_{3,lip,0}(t,r)| \leq \frac{||e_{1}-e_{2}||_{X}}{\log^{\delta-\delta_{2}}(t)}\left(\frac{1}{\log^{b}(t)}+\frac{1}{\log(t)}\right) \cdot \begin{cases} \frac{C r \lambda_{0}(t)^{2} \log(t)}{g_{0}(t)^{2} t^{4} \log^{b}(t)}, \quad r \leq g_{0}(t)\\
\frac{C \lambda_{0}(t)^{2} \log(t)}{t^{4} \log^{b}(t) g_{0}(t)}, \quad g_{0}(t) \leq r \leq \frac{t}{2}\\
\frac{C \lambda_{0}(t)^{2} \log(t)}{t^{5/2} \log^{b}(t) r^{3/2} g_{0}(t)}, \quad r > \frac{t}{2}\end{cases}$$

$$|\partial_{r}^{2}WRHS_{3,lip,0}(t,r)| \leq \frac{||e_{1}-e_{2}||_{X}}{\log^{\delta-\delta_{2}}(t)}\left(\frac{1}{\log^{b}(t)}+\frac{1}{\log(t)}\right)\cdot \begin{cases} \frac{C \lambda_{0}(t)^{2} \log(t)}{t^{4} \log^{b}(t) g_{0}(t)^{2}}, \quad r \leq \frac{t}{2}\\
\frac{C \lambda_{0}(t)^{2} \log(t)}{t^{5/2}r^{3/2} \log^{b}(t) g_{0}(t)^{2}}, \quad r > \frac{t}{2}\end{cases}$$
The same procedure used to estimate $w_{3}$ then gives
$$|w_{3,lip,0}(t,r)| \leq \begin{cases} \frac{C r^{2} \lambda_{0}(t)^{2} \log(t) ||e_{1}-e_{2}||_{X}}{t^{2} g_{0}(t)^{2} \log^{b}(t) \log^{\delta-\delta_{2}}(t)}\left(\frac{1}{\log^{b}(t)}+\frac{1}{\log(t)}\right), \quad r \leq g_{0}(t)\\
\frac{C \lambda_{0}(t)^{2} \log^{2}(t) ||e_{1}-e_{2}||_{X}}{t^{2} \log^{b}(t) \log^{\delta-\delta_{2}}(t)}\left(\frac{1}{\log^{b}(t)}+\frac{1}{\log(t)}\right), \quad r > g_{0}(t)\end{cases}$$

$$|\partial_{r}w_{3,lip,0}(t,r)| \leq \begin{cases} \frac{C r \lambda_{0}(t)^{2} \log(t) ||e_{1}-e_{2}||_{X}}{t^{2} g_{0}(t)^{2} \log^{b}(t) \log^{\delta-\delta_{2}}(t)}\left(\frac{1}{\log^{b}(t)}+\frac{1}{\log(t)}\right), \quad r \leq g_{0}(t)\\
\frac{C \lambda_{0}(t)^{2} \log(t) ||e_{1}-e_{2}||_{X}}{t^{2} \log^{b}(t) g_{0}(t) \log^{\delta-\delta_{2}}(t)}\left(\frac{1}{\log^{b}(t)}+\frac{1}{\log(t)}\right), \quad r > g_{0}(t)\end{cases}$$

$$|\partial_{rr}w_{3,lip,0}(t,r)| \leq \frac{C \lambda_{0}(t)^{2} \log(t) ||e_{1}-e_{2}||_{X}}{t^{2} g_{0}(t)^{2} \log^{b}(t) \log^{\delta-\delta_{2}}(t)}\left(\frac{1}{\log^{b}(t)}+\frac{1}{\log(t)}\right), \quad r >0$$

Next, we note that
$$|WRHS_{3,lip,1}(t,r)| \leq \frac{C ||e_{1}-e_{2}||_{X} \lambda_{0}(t)^{4}}{\log^{\delta-\delta_{2}}(t)} \begin{cases} \frac{\log(t)}{t^{4} \log^{3b}(t) g_{0}(t)^{2}} + \frac{\log(2+\frac{r}{g_{0}(t)}) \log(t)}{t^{4} \log^{2b}(t)(g_{0}(t)^{2}+r^{2})}, \quad r \leq \frac{t}{2}\\
\frac{\log(t)}{t^{5/2} \log^{3b}(t) g_{0}(t)^{2} r^{3/2}} + \frac{\log^{2}(t)}{r^{3/2} t^{5/2} \log^{2b}(t) t^{2}}, \quad r > \frac{t}{2}\end{cases}.$$
After applying the procedure of \eqref{enest}, we get
$$|w_{3,lip,1}(t,r)| \leq \frac{C ||e_{1}-e_{2}||_{X} \lambda_{0}(t)^{4} \log(t)}{t^{2} \log^{3b}(t) \log^{\delta-\delta_{2}}(t) g_{0}(t)^{2}}.$$
Then, we define, for $j \geq 4$,
\begin{equation}\begin{split}WRHS_{j,lip,0}(t,r):&=\frac{6}{r^{2}}\left(Q_{\frac{1}{\lambda_{2}(t)}}+v_{c}^{\lambda_{2}}\right)\left(w_{j-1,lip,0}(w_{j-1}^{\lambda_{1}}+w_{j-2}^{\lambda_{2}})+2 \sum_{k=2}^{j-2} w_{k,lip,0}w_{j-1}^{\lambda_{1}} + 2 \sum_{k=2}^{j-2} w_{k}^{\lambda_{2}} w_{j-1,lip,0}\right)\\
&+\frac{2}{r^{2}}\left(w_{j-1,lip,0}\left((w_{j-1}^{\lambda_{1}})^{2}+w_{j-1}^{\lambda_{1}}w_{j-1}^{\lambda_{2}}+(w_{j-1}^{\lambda_{2}})^{2}\right) + 3 w_{j-1,lip,0}(w_{j-1}^{\lambda_{1}}+w_{j-1}^{\lambda_{2}}) \sum_{k=2}^{j-2} w_{k}^{\lambda_{1}}\right. \\
&+ \left.3 (w_{j-1}^{\lambda_{2}})^{2} \sum_{k=2}^{j-2} w_{k,lip,0} + 3 w_{j-1,lip,0}\left(\sum_{k=2}^{j-2} w_{k}^{\lambda_{1}}\right)^{2}  \right.\\
&\left.+3 w_{j-1}^{\lambda_{2}} \left(\sum_{k=2}^{j-2} w_{k,lip,0}\right)\left(\sum_{q=2}^{j-2} \left(w_{q}^{\lambda_{1}}+w_{q}^{\lambda_{2}}\right)\right)\right)+\frac{6 w_{j-1,lip,0}}{r^{2}} \left((v_{c}^{\lambda_{1}})^{2}+2v_{c}^{\lambda_{1}}Q_{\frac{1}{\lambda_{1}}}\right)\end{split}\end{equation}
and
$$WRHS_{j,lip,1}(t,r):= WRHS_{j}^{\lambda_{1}}(t,r)-WRHS_{j}^{\lambda_{2}}(t,r) - WRHS_{j,lip,0}(t,r).$$
As with previous estimates, all estimates which we will prove by induction are valid for all $t \geq T_{0}$, provided that $T_{0}$ is sufficiently large. We will no longer explicitly write this after each such estimate. By using a similar procedure used to estimate $w_{j}$ by induction, we get, for some constant $C_{3} \geq C_{2}$, and all $j \geq 4$,
$$|w_{j,lip,0}(t,r)| \leq \frac{||e_{1}-e_{2}||_{X}}{\log^{\delta-\delta_{2}}(t)}\left(\frac{1}{\log^{b}(t)}+\frac{1}{\log(t)}\right) \begin{cases} \frac{C_{3}^{j} r^{2} \lambda_{0}(t)^{2} \log^{2}(t)}{t^{2} g_{0}(t)^{2} \log^{b(j-2)}(t)}, \quad r \leq g_{0}(t)\\
\frac{C_{3}^{j} \lambda_{0}(t)^{2} \log^{2}(t)}{t^{2} \log^{b(j-2)}(t)}, \quad r > g_{0}(t)\end{cases}$$

$$|\partial_{r}w_{j,lip,0}(t,r)| \leq \begin{cases} \frac{C_{3}^{j} r \lambda_{0}(t)^{2} \log^{2}(t) ||e_{1}-e_{2}||_{X}}{t^{2} g_{0}(t)^{2} \log^{b(j-2)}(t) \log^{\delta-\delta_{2}}(t)}\left(\frac{1}{\log^{b}(t)}+\frac{1}{\log(t)}\right), \quad r \leq g_{0}(t)\\
\frac{C_{3}^{j} \lambda_{0}(t)^{2} \log^{2}(t) ||e_{1}-e_{2}||_{X}}{t^{2} \log^{b(j-2)}(t) g_{0}(t) \log^{\delta-\delta_{2}}(t)}\left(\frac{1}{\log^{b}(t)}+\frac{1}{\log(t)}\right), \quad r > g_{0}(t)\end{cases}$$

$$|\partial_{r}^{2} w_{j,lip,0}(t,r)| \leq \frac{C_{3}^{j} \lambda_{0}(t)^{2} \log^{2}(t) ||e_{1}-e_{2}||_{X}}{t^{2} g_{0}(t)^{2} \log^{b(j-2)}(t) \log^{\delta-\delta_{2}}(t)}\left(\frac{1}{\log^{b}(t)}+\frac{1}{\log(t)}\right), \quad r>0.$$
Using a procedure similar to that used to estimate $v_{j}^{\lambda_{1}}-v_{j}^{\lambda_{2}}$, we get
$$|w_{j,lip,1}(t,r)| \leq \frac{C_{3}^{j} ||e_{1}-e_{2}||_{X} \lambda_{0}(t)^{4} \log(t)}{t^{2} \log^{bj}(t) \log^{\delta-\delta_{2}}(t) g_{0}(t)^{2}}.$$
Finally, this gives
$$|w_{c,ip}(t,\lambda_{1}(t))-w_{c,ip}(t,\lambda_{2}(t))| \leq \frac{C ||e_{1}-e_{2}||_{X} \lambda_{0}(t)^{2} \log(t)}{t^{2} g_{0}(t)^{2} \log^{\delta-\delta_{2}}(t)}\left(\frac{1}{\log^{b}(t)}+\frac{1}{\log(t)}\right).$$
When combined with our previous estimates, we get
$$|G(t,\lambda_{1}(t))-G(t,\lambda_{2}(t))| \leq \frac{C ||e_{1}-e_{2}||_{X}}{\log^{\delta-\delta_{2}}(t) t^{2} \log^{1+\delta_{0}}(t)}$$
where $\delta_{0}=\min\{2b-4\epsilon-1,3b-4\epsilon-2\}.$ Note that $\delta_{0}>0$, by the constraints on $\epsilon$. This implies that there exists a constant $C$ \emph{independent} of $T_{0}$ such that, for all $e_{1},e_{2} \in \mathcal{B}$, and all $t \geq T_{0}$,
$$||T(e_{1})-T(e_{2})||_{X} \leq \frac{C ||e_{1}-e_{2}||_{X}}{\log^{\delta_{0}}(T_{0})}.$$
Combined with our previous estimates of $T(e)$ (and its derivatives) for $e \in \mathcal{B}$, we see that there exists $T_{4}>T_{3}$ such that, for all $T_{0} \geq T_{4}$, $T$ has a fixed point, say $e_{0} \in \mathcal{B}$. By inspection of the definition of $T$, this means that $\lambda(t) = \lambda_{0}(t)\left(1+e_{0}(t)\right)$ solves \eqref{lambdaeqn}. From now on, we fix $\lambda(t) = \lambda_{0}(t)\left(1+e_{0}(t)\right)$.
\subsubsection{Estimating $\lambda'''$}
In this section, we will show that $e_{0} \in C^{3}([T_{0},\infty))$, and estimate $e_{0}'''(t)$.  Estimating $e_{0}'''(t)$ will be done in two steps, exactly as in \cite{wm}. First, we obtain a preliminary estimate on $e_{0}'''(t)$ by differentiating an appropriate expression for $e_{0}''(t)$(see \eqref{e0eqndiff} below). Once we establish this preliminary estimate, we can differentiate $WRHS_{j}(t,r)$ in the $t$ variable, and this allows us to justify a different representation formula for $\partial_{t}w_{j}$ than what was used to establish \eqref{dtwkassump}. With this different representation formula for $\partial_{t}w_{j}$, we then proceed to prove an estimate on $e_{0}'''(t)$ which is stronger than our preliminary estimate. As a by-product of this procedure, we obtain an estimate on $\partial_{t}w_{j}$ which is much better than \eqref{dtwkassump}, in the region $r \leq t$. 
\begin{lemma}\label{e0ppplemma} We have the following regularity of $e_{0}$, and estimate. $e_{0} \in C^{3}([T_{0},\infty))$, with 
$$|e_{0}'''(t)| \leq \frac{C}{t^{3} \log^{b+\delta_{4}}(t)}, \quad t \geq T_{0}$$
\end{lemma}
\begin{proof}
From \eqref{lambdaeqn}, $e_{0}$ solves
\begin{equation}\label{e0eqn}e_{0}''(t) + \frac{2 \lambda_{0}'(t)}{\lambda_{0}(t)} e_{0}'(t) = \frac{3}{2} G(t,\lambda(t)) \left(1+e_{0}(t)\right).\end{equation}
We start with \eqref{e0eqn}, written in the following form.
$$e_{0}''(t) + \frac{2 \lambda_{0}'(t)}{\lambda_{0}(t)} e_{0}'(t) = i_{0}(t) e_{0}''(t) + \frac{3}{2} G_{rest}(t,\lambda(t)) (1+e_{0}(t))$$
where
$$i_{0}(t) = \frac{3}{2} \lambda_{0}(t) \int_{0}^{\infty} \chi_{\geq 1}(\frac{r}{g(t)}) \frac{4 r^{2} \lambda(t)}{(\lambda(t)^{2}+r^{2})^{2}} \frac{\phi_{0}(\frac{r}{\lambda(t)}) r dr}{\lambda(t)^{2}} \left(1+e_{0}(t)\right)$$
$$G_{rest}(t,\lambda(t)) = G(t,\lambda(t)) - \frac{2 i_{0}(t) e_{0}''(t)}{3(1+e_{0}(t))}.$$
There exists a constant $C$, independent of $t, T_{0}$, such that $|i_{0}(t)| \leq \frac{C}{\log^{2(b-2\epsilon)}(t)}.$ So, there exists an absolute constant $T_{5}>T_{4}$ such that, for all $T_{0}> T_{5}$, we have (for instance) $|i_{0}(t)| \leq \frac{1}{900000}$. Then,
\begin{equation}\label{e0eqndiff} e_{0}''(t) = \frac{\frac{3}{2} G_{rest}(t,\lambda(t))\left(1+e_{0}(t)\right) - \frac{2 \lambda_{0}'(t)}{\lambda_{0}(t)} e_{0}'(t)}{1-i_{0}(t)}.\end{equation}
Because $v_{s},w_{s} \in C^{2}([T_{0},\infty) \times [0,\infty))$, the right-hand side of \eqref{e0eqndiff} is in $C^{1}([T_{0},\infty))$. In particular, $e_{0}\in C^{3}([T_{0},\infty))$. We will now estimate the $t$-derivative of the right-hand side of \eqref{e0eqndiff}. We have
$$v_{s,ip}(t,\lambda(t)) = \int_{0}^{\infty} \frac{6}{r^{2}} \sum_{k=2}^{\infty} v_{k}(t,r) \left(1-Q_{1}^{2}(\frac{r}{\lambda(t)})\right) \phi_{0}(\frac{r}{\lambda(t)}) \frac{r dr}{\lambda(t)^{2}}$$
which gives
$$|\partial_{t}v_{s,ip}(t,\lambda(t))| \leq \frac{C}{t^{3} \log^{2b}(t)}.$$
Next, we have
$$|\partial_{t}\left(\int_{0}^{\infty} \chi_{\geq 1}(\frac{r}{g(t)}) \left(\frac{4r^{2} (\lambda'(t))^{2}(r^{2}-3 \lambda(t)^{2})}{(r^{2}+\lambda(t)^{2})^{3}} - \frac{6 v_{c}(t,r)}{r^{2}} \left(1-Q_{\frac{1}{\lambda(t)}}^{2}(r)\right)\right)\phi_{0}(\frac{r}{\lambda(t)}) \frac{r dr}{\lambda(t)^{2}}\right)| \leq \frac{C \lambda(t)^{2}}{t^{3} \log^{b}(t) g(t)^{2}}$$
$$|i_{0}'(t)| \leq \frac{C \lambda_{0}(t)^{2}}{t g(t)^{2}}\left(\frac{1}{\log^{b}(t)} + \frac{1}{\log(t)}\right)$$
$$|\partial_{t}\left(\frac{i_{0}(t)\left(\lambda_{0}''(t)\left(1+e_{0}(t)\right)+2 \lambda_{0}'(t) e_{0}'(t)\right)}{(1+e_{0}(t))\lambda_{0}(t)}\right)| \leq \frac{C \lambda_{0}(t)^{2}}{t^{3}\log^{b}(t) g(t)^{2}}.$$
Using the same procedure used to estimate $E_{v_{1},ip}$, we get
$$|\partial_{t}E_{v_{1},ip}(t,\lambda(t))| \leq \frac{C \log(t)}{t^{7/2}}.$$
Using \eqref{dtwkassump}, we get
$$|\partial_{t}w_{c,ip}(t,\lambda(t))| \leq \frac{C \log(t)}{t^{2} \log^{b}(t) g(t)}.$$
As mentioned before, once we obtain a preliminary estimate on $e_{0}'''$, we will be able to prove a much stronger estimate on $e_{0}'''$, via improving \eqref{dtwkassump}.
This gives
$$|\partial_{t}G_{rest}(t,\lambda(t))| \leq \frac{C \log(t)}{t^{2} \log^{b}(t) g(t)}$$
which leads to the following \emph{preliminary} estimate on $e_{0}''''$.
\begin{equation}\label{e0pppprelim}|e_{0}'''(t)| \leq \frac{C \log(t)}{t^{2} \log^{b}(t) g(t)}\end{equation}
Now that we have this preliminary estimate on $e_{0}'''$, we can prove a better estimate on $\partial_{t}w_{j}$. The details for this argument involve some slightly long estimates, and are therefore presented in Appendix \ref{dtwjfinalestappendix}. The improved estimates on $\partial_{t}w_{j}$, proven in Appendix \ref{dtwjfinalestappendix} then give the following.
\begin{equation}\nonumber\begin{split}|\partial_{t}w_{c}(t,r)| &\leq \sum_{k=2}^{\infty} |\partial_{t}w_{k}(t,r)| \\
&\leq \begin{cases} \frac{C r^{2} \lambda(t)^{2} \log(t)}{g(t)^{2}} \left(\frac{1}{t^{3} \log^{b}(t)} + \frac{\sup_{x \geq t}(x^{3/2}|e_{0}'''(x)|)}{t^{3/2}}\right), \quad r \leq g(t)\\
C \lambda(t)^{2} \left(\log(2+\frac{r}{g(t)})+\frac{\log(t)}{\log^{b}(t)}\right) \log(t) \left(\frac{1}{t^{3} \log^{b}(t)} + \frac{\sup_{x \geq t}(x^{3/2} |e_{0}'''(x)|)}{t^{3/2}}\right), \quad g(t) < r < \frac{t}{2}\\
\frac{C \lambda(t)^{2} \log^{2}(t)}{\log^{b}(t) t^{5/2} \sqrt{\langle t-r \rangle} \log^{b}(\langle t-r \rangle)} + C \lambda(t)^{2} \log^{2}(t) \left(\frac{1}{t^{3} \log^{b}(t)} + \frac{\sup_{x \geq t}(x^{3/2}|e_{0}'''(x)|)}{t^{3/2}}\right), \quad \frac{t}{2} < r < t\end{cases}\end{split}\end{equation}
Using this estimate, we then get
$$|\partial_{t}w_{c,ip}(t,\lambda(t))| \leq C \left(\frac{1}{t^{3} \log^{b}(t)} + \frac{\sup_{x \geq t}(x^{3/2}|e_{0}'''(x)|)}{t^{3/2}}\right)\cdot\left(\frac{1}{\log^{2b-4\epsilon-1}(t)}+\frac{1}{\log^{4b-8\epsilon-1}(t)}+\frac{1}{\log^{5b-8\epsilon-2}(t)}\right).$$
Let $\delta_{3} = \min\{2b-4\epsilon-1,4b-8\epsilon-1,5b-8\epsilon-2,b\}.$ Combining our previous estimates on the other terms of $\partial_{t}G_{rest}$ gives
$$|\partial_{t}G_{rest}(t,\lambda(t))| \leq \frac{C}{t^{3} \log^{b+\delta_{3}}(t)} + \frac{C \sup_{x \geq t}(x^{3/2}|e_{0}'''(x)|)}{t^{3/2} \log^{\delta_{3}}(t)}.$$
If 
\begin{equation} \label{delta4def}\delta_{4} = \min\{\delta_{3},1+\delta-\delta_{2}\}\end{equation} then (for $C$ \emph{independent} of $t, T_{0}$),
\begin{equation}\label{e0pppbootstrap}|e_{0}'''(t)| \leq \frac{C}{t^{3} \log^{b+\delta_{4}}(t)} + \frac{C \sup_{x \geq t}(x^{3/2} |e_{0}'''(x)|)}{t^{3/2} \log^{\delta_{3}}(t)}, \quad t \geq T_{0}.\end{equation}
Recalling the preliminary estimate, \eqref{e0pppprelim}, we see that $x \mapsto x^{3/2} |e_{0}'''(x)|$ is a continuous function on $[T_{0},\infty)$, which decays at infinity. Therefore, for each $t \geq T_{0}$, there exists $y(t) \geq t$ such that
$$\sup_{x \geq t}(x^{3/2}|e_{0}'''(x)|) = y(t)^{3/2} |e_{0}'''(y(t))|.$$
But then, for any $t_{1}>T_{0}$, we evaluate \eqref{e0pppbootstrap} at the point $t = y(t_{1})$, and get
\begin{equation}\begin{split}y(t_{1})^{3/2} |e_{0}'''(y(t_{1}))| = \sup_{x \geq t_{1}}(x^{3/2}|e_{0}'''(x)|) &\leq \frac{C}{y(t_{1})^{3/2} \log^{b+\delta_{4}}(y(t_{1}))} + \frac{C \sup_{x \geq y(t_{1})}(x^{3/2} |e_{0}'''(x)|)}{\log^{\delta_{3}}(y(t_{1}))}\\
&\leq \frac{C}{t_{1}^{3/2} \log^{b+\delta_{4}}(t_{1})} + \frac{C \sup_{x \geq t_{1}}(x^{3/2} |e_{0}'''(x)|)}{\log^{\delta_{3}}(t_{1})}\end{split}\end{equation}
where we used $t \mapsto \sup_{x \geq t}(x^{3/2} |e_{0}'''(x)|) \text{ is decreasing, and } y(t_{1}) \geq t_{1}.$ Therefore, there exists $M_{2}>0$, $C>0$, \emph{independent} of $T_{0}$ such that for all $t_{1} > M_{2}$, we have
$$|e_{0}'''(t_{1})| \leq \frac{C}{t_{1}^{3} \log^{b+\delta_{4}}(t_{1})}.$$
Since $e_{0} \in C^{3}([T_{0},\infty))$, we conclude that there exists $C>0$ \emph{independent} of $T_{0}$ such that
$$|e_{0}'''(t)| \leq \frac{C}{t^{3} \log^{b+\delta_{4}}(t)}, \quad t \geq T_{0}$$
which completes the proof of the final estimate on $e_{0}'''$.\end{proof}
\subsubsection{Estimating $\lambda''''$}
\begin{lemma}\label{e0pppplemma}We have the following regularity of $e_{0}$, and estimate. $e_{0} \in C^{4}([T_{0},\infty))$
$$|e_{0}''''(t)| \leq \frac{C}{t^{4} \log^{b+\delta_{5}}(t)}, \quad t \geq T_{0}.$$\end{lemma}
\begin{proof}
By inspection of \eqref{e0eqndiff}, $e_{0} \in C^{4}([T_{0},\infty))$. Our goal in this section will be to estimate $e_{0}''''$. From \eqref{e0eqndiff}, we get
\begin{equation}\begin{split}|e_{0}''''(t)| &\leq C \left(|\partial_{t}^{2}\left(G_{rest}(t,\lambda(t))(1+e_{0}(t))\right)|+|\partial_{t}^{2}\left(\frac{\lambda_{0}'(t) e_{0}'(t)}{\lambda_{0}(t)}\right)| + |i_{0}''(t) e_{0}''(t)| + |i_{0}'(t) e_{0}'''(t)|\right) \\
&+ C |i_{0}'(t)| \left(|\partial_{t}(G_{rest}(t,\lambda(t))\left(1+e_{0}(t)\right))|+|\partial_{t}\left(\frac{\lambda_{0}'(t) e_{0}'(t)}{\lambda_{0}(t)}\right)| + |i_{0}'(t) e_{0}''(t)|\right).\end{split}\end{equation}
Then, we note
$$|\partial_{t}^{2} v_{sip}(t,\lambda(t))| \leq \frac{C}{t^{4} \log^{2b}(t)}$$
$$|\partial_{t}^{2}\left(\frac{-2}{3} \frac{i_{0}(t) \left(\lambda_{0}''(t) \left(1+e_{0}(t)\right)+2\lambda_{0}'(t)e_{0}'(t)\right)}{\left(1+e_{0}(t)\right)\lambda_{0}(t)}\right)| \leq \frac{C \lambda(t)^{2}}{g(t)^{2} t^{4} \log^{b}(t)}$$
$$|i_{0}''(t)| \leq \frac{C \lambda(t)^{2}}{g(t)^{2} t^{2}}\left(\frac{1}{\log^{b}(t)} + \frac{1}{\log(t)}\right)$$
$$|\partial_{t}^{2}\left(\int_{0}^{\infty} \chi_{\geq 1}(\frac{r}{g(t)}) \left(\frac{4 r^{2} \lambda'(t)^{2} (r^{2}-3\lambda(t)^{2})}{(r^{2}+\lambda(t)^{2})^{3}} - \frac{6 v_{c}}{r^{2}} \left(1-Q_{1}^{2}(\frac{r}{\lambda(t)})\right)\right) \phi_{0}(\frac{r}{\lambda(t)}) \frac{r dr}{\lambda(t)^{2}}\right)| \leq \frac{C \lambda(t)^{2}}{g(t)^{2} t^{4} \log^{b}(t)}.$$
Using a similar procedure used to estimate $E_{v_{1},ip}$, we get
$$|\partial_{t}^{2} E_{v_{1},ip}(t,\lambda(t))| \leq \frac{C \log(t)}{t^{9/2}}.$$
Finally, we have the \emph{preliminary} estimate
$$|\partial_{t}^{2} w_{c,ip}(t,\lambda(t))| \leq \frac{C \log(t) \left(1+\frac{\log(t)}{\log^{b}(t)}\right)}{t^{2} g(t)^{2} \log^{b}(t)}.$$
In total, this gives the \emph{preliminary} estimate
\begin{equation}\label{e0ppppprelimest}|e_{0}''''(t)| \leq \frac{C \log(t) \left(1+\frac{\log(t)}{\log^{b}(t)}\right)}{t^{2} g(t)^{2} \log^{b}(t)}.\end{equation}
As before, we now start to record a better estimate on $\partial_{t}^{2} w_{j}$ than what was previously obtained, using a procedure which is justified by the preliminary estimate. We present these estimates and their proof in Appendix \ref{dttwjfinalestimateappendix}. The results of Appendix \ref{dttwjfinalestimateappendix} lead to the following improved estimate
$$|\partial_{t}^{2}w_{c,ip}(t,\lambda(t))| \leq \frac{C \log(t)}{t^{4} \log^{3b-4\epsilon}(t)} + \frac{C \log(t) \sup_{x \geq t}(x^{3/2} |e_{0}''''(x)|)}{\log^{2b-4\epsilon}(t) t^{3/2}}$$
which, when combined with our previous estimates from this section, yields
$$|e_{0}''''(t)| \leq \frac{C}{t^{4} \log^{b+\delta_{5}}(t)} + \frac{C \sup_{x \geq t}(x^{3/2} |e_{0}''''(x)|)}{\log^{2b-4\epsilon-1}(t) t^{3/2}}$$
where
\begin{equation}\label{delta5def}\delta_{5}=\min\{b,2b-4\epsilon-1\}.\end{equation}
Repeating the argument used to estimate $e_{0}'''$, we conclude that there exists $C>0$, \emph{independent} of $T_{0}$ such that
$$|e_{0}''''(t)| \leq \frac{C}{t^{4} \log^{b+\delta_{5}}(t)}, \quad t \geq T_{0}.$$
\end{proof}

\subsubsection{Symbol-Type estimates on $F_{4}$}
Our goal will be to establish the following symbol-type estimates on $F_{4}(t,r)$ in the region $r \leq \frac{t}{2}$. (Recall that $F_{4}(t,r) = 0$ if $r \geq \frac{t}{2}$, as per \eqref{f4def}).
\begin{lemma} For $0 \leq j, k \leq 2$, and $j+k \leq 2$, we have the following estimates.
\begin{equation}\label{f4symb}\begin{split}|r^{k}\partial_{r}^{k}t^{j}\partial_{t}^{j}F_{4}(t,r)| &\leq C \mathbbm{1}_{\{r \leq g(t)\}} \frac{r^{2} \lambda(t)^{2}}{t^{2} \log^{b}(t) (\lambda(t)^{2}+r^{2})^{2}} \\
&+ C \frac{\mathbbm{1}_{\{r \leq \frac{t}{2}\}} \lambda(t)^{2}}{(\lambda(t)^{2}+r^{2})^{2}} \begin{cases} \frac{r^{2} \lambda(t)^{2} \log(t)}{t^{2} g(t)^{2} \log^{b}(t)}, \quad r \leq g(t)\\
\frac{\lambda(t)^{2} \log(t)}{t^{2} \log^{b}(t)} \left(\log(2+\frac{r}{g(t)}) + \frac{\log(t)}{\log^{b}(t)}\right), \quad g(t) < r < \frac{t}{2}\end{cases}\end{split}\end{equation}\end{lemma}
In order to obtain these estimates, we first have to obtain improved estimates on $\partial_{r}w_{c}$, $\partial_{r}^{2} w_{c}$, and $\partial_{tr}w_{c}$ in the region $\frac{t}{2} > r > g(t)$. So, we start with the following lemma.
\begin{lemma} We have the following symbol-type estimates on $w_{c}$. For $0 \leq j,k \leq 2, \quad j+k \leq 2$,
\begin{equation}\label{wcsymb}t^{j} r^{k} |\partial_{r}^{k}\partial_{t}^{j} w_{c}(t,r)| \leq \begin{cases} \frac{C r^{2} \lambda(t)^{2} \log(t)}{t^{2} g(t)^{2} \log^{b}(t)}, \quad r \leq g(t)\\
\frac{C \lambda(t)^{2} \log(t)}{t^{2} \log^{b}(t)}\left(\log(2+\frac{r}{g(t)})+\frac{\log(t)}{\log^{b}(t)}\right), \quad g(t) < r < \frac{t}{2}\end{cases}\end{equation}
\end{lemma}
\begin{proof} First, we recall that $w_{2}(t,r) = \int_{t}^{\infty} ds w_{2,s}(t,r)$, where $w_{2,s}$ solves
$$\begin{cases} -\partial_{tt}w_{2,s}+\partial_{rr}w_{2,s}+\frac{1}{r}\partial_{r}w_{2,s}-\frac{4}{r^{2}}w_{2,s}=0\\
w_{2,s}(s,r)=0\\
\partial_{t}w_{2,s}(s,r) = WRHS_{2}(s,r)\end{cases}.$$
Therefore, since $t<s$, we expect to obtain better decay for $\left(-\partial_{t}+\partial_{r}\right)w_{2,s}$ than what we would have for simply $\partial_{r}w_{2,s}$. We compute $\partial_{t}w_{2}$ by starting with \eqref{w2fordt}, and differentiating under the integral sign. The key point we will use to obtain our estimate is that, for a differentiable function $g$, and $s -t>0, \quad 2\pi > \theta >0$, and $q,r>0$, we have
\begin{equation}\nonumber\begin{split}&-\partial_{t}\left(g(\sqrt{r^{2}+q^{2}(s-t)^{2}+2 r q (s-t) \cos(\theta)})\right) \\
&= g'(\sqrt{r^{2}+q^{2}(s-t)^{2}+2 r q (s-t) \cos(\theta)}) \frac{(s-t)q^{2}+r q \cos(\theta)}{\sqrt{r^{2}+(s-t)^{2}q^{2}+2 r q (s-t) \cos(\theta)}}\end{split}\end{equation}
and
\begin{equation}\nonumber\begin{split}&\partial_{r}\left(g(\sqrt{r^{2}+q^{2}(s-t)^{2}+2 r q (s-t) \cos(\theta)})\right) \\
&= g'(\sqrt{r^{2}+q^{2}(s-t)^{2}+2 r q (s-t) \cos(\theta)}) \frac{r+q(s-t) \cos(\theta)}{\sqrt{r^{2}+(s-t)^{2}q^{2}+2 r q (s-t) \cos(\theta)}}.\end{split}\end{equation}
So,
\begin{equation}\nonumber\begin{split}&\left(-\partial_{t}+\partial_{r}\right)\left(g(\sqrt{r^{2}+q^{2}(s-t)^{2}+2 r q (s-t) \cos(\theta)})\right) \\
&= g'(\sqrt{r^{2}+q^{2}(s-t)^{2}+2 r q (s-t) \cos(\theta)}) \frac{\left((s-t)q+r\right)\left(1+\cos(\theta)\right)}{\sqrt{r^{2}+q^{2}(s-t)^{2}+2 r q (s-t) \cos(\theta)}}\\
&+g'(\sqrt{r^{2}+q^{2}(s-t)^{2}+2 r q (s-t) \cos(\theta)}) \left(q-1\right) \frac{(s-t)q+r  \cos(\theta)}{\sqrt{r^{2}+(s-t)^{2}q^{2}+2 r q (s-t) \cos(\theta)}} .\end{split}\end{equation}
For ease of notation, and later use, we introduce the following vectors in $\mathbb{R}^{2}$: $x = r \textbf{e}_{1}, \text{ } y=q(s-t)(\cos(\theta),\sin(\theta)).$ So, the above can be written as 
\begin{equation}\label{mdtpdrcancel}\begin{split}\left(-\partial_{t}+\partial_{r}\right)\left(g(|x+y|)\right) &= \left(g'(|x+y|) \cdot |x+y|\right) \frac{\left((s-t)q+r\right)\left(1+\cos(\theta)\right)}{r^{2}+q^{2}(s-t)^{2}+2 r q (s-t) \cos(\theta)}\\
&+g'(|x+y|) \left(q-1\right) \frac{ \widehat{y} \cdot (x+y)}{|x+y|}\\
&= \left(g'(|x+y|) \cdot |x+y|\right) \frac{\left((s-t)q+r\right)\left(1+\cos(\theta)\right)}{(r-q(s-t))^{2}+2 r q (s-t)(1+ \cos(\theta))}\\
&+g'(|x+y|) \left(q-1\right) \frac{ \widehat{y} \cdot (x+y)}{|x+y|}. \end{split}\end{equation}
We will apply this observation to the integrand of the representation formula for $w_{2}$. The third line of \eqref{mdtpdrcancel} has a gain of decay, as can be seen below. We note, for $s-t, q,r>0,$ and $0<\theta<2\pi$, that
$$(r-q(s-t))^{2}+2 r q (s-t)(1+ \cos(\theta)) > \begin{cases} C r^{2}, \quad q(s-t) < \frac{r}{2}\\
C r q (s-t)(1+\cos(\theta)), \quad \frac{r}{2} < q(s-t) < 2r\\
C q^{2}(s-t)^{2}, \quad 2r < q(s-t)\end{cases}.$$
Therefore,
\begin{equation}\label{dtplusdrintstep2} \frac{\left((s-t)q+r\right)\left(1+\cos(\theta)\right)}{(r-q(s-t))^{2}+2 r q (s-t)(1+ \cos(\theta))} \leq \frac{C}{\max\{r, q(s-t)\}}.\end{equation}
Note that the factor $1+\cos(\theta)$ in the numerator of the above expression is crucial in obtaining this estimate. On the other hand, the fourth line of \eqref{mdtpdrcancel} has the factor $q-1$. The integrand of the representation formula for $w_{2}$ contains a factor of $\frac{1}{\sqrt{1-q^{2}}}$. Therefore, the factor $q-1$ cancels the singularity of the integrand of the representation formula for $w_{2}$, and we can obtain a gain of decay for this term, by appropriately integrating by parts in $q$, as will be seen below. The decomposition of \eqref{mdtpdrcancel}, into one term which gains decay in $r+q(s-t)$, and another, which vanishes at $q=1$, would not be possible if we only applied, for example, $\partial_{r}$ to $w_{2}$.\\
\\
Letting $\rho =q(s-t)$, we now have $|x+y| = \sqrt{r^{2}+\rho^{2}+2 r \rho \cos(\theta)}$, like before, and we get
\begin{equation}\begin{split}&-\left(-\partial_{t}+\partial_{r}\right)w_{2}(t,r)\\
&= \frac{1}{2\pi} \int_{t}^{\infty} ds \int_{0}^{s-t} \frac{\rho d\rho}{(s-t) \sqrt{(s-t)^{2}-\rho^{2}}} \int_{0}^{2\pi} d\theta WRHS_{2}(s,|x+y|) \left(1-\frac{2\rho^{2}\sin^{2}(\theta)}{r^{2}+\rho^{2}+2 r \rho \cos(\theta)}\right)\\
&+\frac{1}{2\pi} \int_{t}^{\infty} ds \int_{0}^{s-t} \frac{\rho d\rho}{\sqrt{(s-t)^{2}-\rho^{2}}} \int_{0}^{2\pi} d\theta \partial_{2}WRHS_{2}(s,|x+y|)|x+y| \left(1-\frac{2\rho^{2}\sin^{2}(\theta)}{r^{2}+\rho^{2}+2 r \rho \cos(\theta)}\right) \\
&\qquad\qquad\qquad\qquad\qquad\qquad\qquad\qquad\qquad\qquad\qquad\qquad\qquad\qquad\qquad\cdot\left(\frac{(r+\rho)(1+\cos(\theta))}{(r^{2}+\rho^{2}+2 r \rho \cos(\theta))}\right)\\
&-\frac{1}{2\pi} \int_{t}^{\infty} ds \int_{0}^{s-t} \frac{\rho d\rho}{\sqrt{(s-t)^{2}-\rho^{2}}} \int_{0}^{2\pi} d\theta WRHS_{2}(s,|x+y|) \frac{4 \rho^{2} \sin^{2}(\theta)}{(s-t)(r^{2}+\rho^{2} + 2 r \rho \cos(\theta))}\\
&+\frac{1}{2\pi} \int_{t}^{\infty} ds \int_{0}^{s-t} \frac{\rho d\rho}{\sqrt{(s-t)^{2}-\rho^{2}}} \int_{0}^{2\pi} d\theta WRHS_{2}(s,|x+y|) \frac{2\rho^{2}\sin^{2}(\theta)}{(r^{2}+\rho^{2}+2 r \rho\cos(\theta))^{2}} \cdot 2(r+\rho) (1+\cos(\theta))\\
&-\frac{1}{2\pi} \int_{t}^{\infty} ds \int_{0}^{s-t} \rho d\rho \frac{\sqrt{s-t-\rho}}{\sqrt{s-t+\rho}} \int_{0}^{2\pi} \frac{d\theta}{(s-t)} \partial_{\rho}\left(WRHS_{2}(s,\sqrt{r^{2}+\rho^{2}+2 r \rho\cos(\theta)})\left(1-\frac{2\rho^{2}\sin^{2}(\theta)}{r^{2}+\rho^{2}+2 r \rho \cos(\theta)}\right)\right)\\
&-\frac{1}{2\pi} \int_{t}^{\infty} ds \int_{0}^{s-t} \frac{\rho d\rho}{\sqrt{(s-t)^{2}-\rho^{2}}} \int_{0}^{2\pi} d\theta WRHS_{2}(s,|x+y|) \cdot \frac{-4 \rho \sin^{2}(\theta)}{(r^{2}+\rho^{2}+2 r \rho \cos(\theta))} \left(\frac{\rho}{s-t} -1\right).\end{split}\end{equation}
We split the $s$ integration into two regions, one with $s-t \leq \frac{r}{8}$, and another, with $s-t \geq \frac{r}{8}$. In the region with $s-t \leq \frac{r}{8}$, we have $\rho \leq s-t \leq \frac{r}{8}$, and can proceed roughly as we did when previously estimating $\partial_{r}w_{2}$. On the other hand, in the region $s-t \geq \frac{r}{8}$, we use our observation just after \eqref{dtplusdrintstep2}, and integrate by parts in the fifth line of the above expression to get
\begin{equation}\begin{split}&|\left(-\partial_{t}+\partial_{r}\right)w_{2}(t,r)| \leq C \int_{t}^{t+\frac{r}{8}} ds \int_{0}^{s-t} \frac{\rho d\rho}{(s-t)\sqrt{(s-t)^{2}-\rho^{2}}} \int_{0}^{2\pi} d\theta |WRHS_{2}(s,|x+y|)|\\
&+ C \int_{t}^{t+\frac{r}{8}} ds \int_{0}^{s-t} \frac{\rho d\rho}{\sqrt{(s-t)^{2}-\rho^{2}}} \int_{0}^{2\pi} d\theta \left(|\partial_{2}WRHS_{2}(s,|x+y|)|+\frac{|WRHS_{2}(s,|x+y|)|}{|x+y|}\right) \\
&+C \int_{t+\frac{r}{8}}^{\infty} ds \int_{0}^{s-t} \frac{\rho d\rho}{\sqrt{(s-t)^{2}-\rho^{2}}} \int_{0}^{2\pi} d\theta \frac{\left(|\partial_{2}WRHS_{2}(s,|x+y|)|\cdot|x+y|+|WRHS_{2}(s,|x+y|)|\right)}{(r+\rho)}\\
&+C \int_{t+\frac{r}{8}}^{\infty} ds \int_{0}^{s-t} d\rho \int_{0}^{2\pi} \frac{d\theta}{(s-t)} |WRHS_{2}(s,|x+y|)|.\end{split}\end{equation}
Finally, the third and fourth lines of the above expression show the gain we obtain when applying $-\partial_{t}+\partial_{r}$ to $w_{2}$. Using a procedure similar to that used to estimate $w_{2}$, we get
\begin{equation}\label{drw2final}|\partial_{r}w_{2}(t,r)| \leq \begin{cases} \frac{C \lambda(t)^{2} \log(2+\frac{r}{g(t)}) \log(t)}{t^{2} \log^{b}(t) \sqrt{r^{2}+g(t)^{2}}}, \quad \frac{t}{2} > r > g(t)\\
\frac{C \lambda(t)^{2} \log^{2}(t)}{t^{5/2} \sqrt{\langle t-r \rangle} \log^{b}(\langle t-r \rangle)}, \quad t > r > \frac{t}{2}\end{cases}.\end{equation}
Using our estimate on $\partial_{r}w_{2}$ from Lemma \ref{w2prelim} in the region $r \leq g(t)$, as well as our improved estimate, \eqref{drw2final} for $t> r > g(t)$, we get
$$|\partial_{r}WRHS_{3}(t,r)| + \frac{|WRHS_{3}(t,r)|}{r} \leq \begin{cases} \frac{C r \lambda(t)^{2} \log(t) \log(2+\frac{r}{g(t)})}{t^{4} \log^{2b}(t) (r^{2}+g(t)^{2})}, \quad r \leq \frac{t}{2}\\
\frac{C \lambda(t)^{2} \log^{2}(t)}{t^{9/2} \langle t-r \rangle \log^{b}(t) \log^{b}(\langle t-r \rangle)}, \quad t > r > \frac{t}{2}\end{cases}.$$
We then use the same procedure as above to estimate $\left(-\partial_{t}+\partial_{r}\right)w_{3}$. Combining this with the final estimate on $\partial_{t}w_{3}$, namely \eqref{dtw3finalest}, we get
$$|\partial_{r}w_{3}(t,r)| \leq \begin{cases} \frac{C \lambda(t)^{2} \log^{2}(t)}{t^{3} \log^{2b}(t)}, \quad g(t) < r < \frac{t}{2}\\
\frac{C \lambda(t)^{2} \log^{2}(t)}{t^{5/2} \sqrt{\langle t-r \rangle} \log^{b}(t) \log^{b}(\langle t-r \rangle)}, \quad \frac{t}{2} < r < t\end{cases}.$$
With the same procedure used to estimate $\partial_{r}w_{3}$, and an induction argument similar to those previously used, we get: there exists $C_{6}> C_{4}+C_{2}^{p}$ such that, for all $j \geq 4$, we have
$$|\partial_{r}w_{j}(t,r)| \leq \begin{cases} \frac{C_{6}^{j} \lambda(t)^{2} \log^{2}(t)}{t^{3} \log^{b(j-1)}(t)}, \quad g(t) < r < \frac{t}{2}\\
\frac{C_{6}^{j} \lambda(t)^{2} \log^{2}(t)}{t^{5/2} \sqrt{\langle t-r \rangle} \log^{b(j-2)}(t) \log^{b}(\langle t-r \rangle)}, \quad \frac{t}{2} < r < t\end{cases}.$$
Next, we will estimate $\partial_{tr}w_{j}$, for $j \geq 2$. For this, we start with the formula
\begin{equation}\nonumber\begin{split}&\partial_{t}w_{2}(t,r) \\
&= \frac{-1}{2\pi} \int_{t}^{\infty} ds \int_{0}^{s-t} \frac{\rho d\rho}{\sqrt{(s-t)^{2}-\rho^{2}}} \int_{0}^{2\pi} d\theta \partial_{1}WRHS_{2}(s,\sqrt{r^{2}+\rho^{2}+2 r \rho \cos(\theta)}) \left(1-\frac{2\rho^{2}\sin^{2}(\theta)}{r^{2}+\rho^{2}+2 r \rho \cos(\theta)}\right)\end{split}\end{equation}
which was used when obtaining the final estimate on $\lambda'''$, and which follows from the fact that $\partial_{t}w_{2}$ solves the same equation as $w_{2}$ does, also with zero Cauchy data at infinity, except with $\partial_{t}WRHS_{2}$ on the right-hand side. Then, we estimate $\left(-\partial_{t}+\partial_{r}\right)\partial_{t}w_{2}$ with the same argument used to estimate $\left(-\partial_{t}+\partial_{r}\right)w_{2}$. Then, we repeat this argument for $w_{j}$, for $j \geq 3$, and get
$$|\partial_{tr}w_{2}(t,r)| \leq \begin{cases} \frac{C \lambda(t)^{2} \log(t) \log(2+\frac{r}{g(t)})}{\sqrt{r^{2}+g(t)^{2}} t^{3} \log^{b}(t)}, \quad g(t) < r < \frac{t}{2}\\
\frac{C \lambda(t)^{2} \log^{2}(t)}{t^{5/2} \langle t-r \rangle^{3/2} \log^{b}(t)}, \quad \frac{t}{2} < r < t \end{cases}$$
$$|\partial_{tr}w_{3}(t,r)| \leq \begin{cases} \frac{C \lambda(t)^{2} \log^{2}(t)}{t^{4} \log^{2b}(t)}, \quad g(t) < r \leq \frac{t}{2}\\
\frac{C \lambda(t)^{2} \log^{2}(t)}{t^{5/2} \langle t-r \rangle^{3/2} \log^{b}(t) \log^{b}(\langle t- r\rangle)}, \quad \frac{t}{2} \leq r < t\end{cases}.$$
There exists $C_{7} > 3\left(C_{2}^{p}+C_{4}+C_{5}+C_{6}\right)$ such that for $j \geq 4$, we have 
$$|\partial_{tr}w_{j}(t,r)| \leq \begin{cases} \frac{C_{7}^{j} \lambda(t)^{2} \log^{2}(t)}{t^{4} \log^{b(j-1)}(t)}, \quad g(t) < r < \frac{t}{2}\\
\frac{C_{7}^{j} \lambda(t)^{2} \log^{2}(t)}{t^{5/2} \langle t-r \rangle^{3/2} \log^{b(j-2)}(t) \log^{b}(\langle t-r \rangle)}, \quad \frac{t}{2} < r < t\end{cases}$$
Finally, we read off estimates on $r^{2}\partial_{r}^{2}w_{j}$ by inspection of the equation it solves, and our previous estimates. This gives
$$|r^{2}\partial_{r}^{2}w_{2}(t,r)| \leq \frac{C \lambda(t)^{2} \log(2+\frac{r}{g(t)}) \log(t)}{t^{2} \log^{b}(t)}, \quad g(t) \leq r \leq \frac{t}{2}$$
$$|r^{2} \partial_{r}^{2}w_{3}(t,r)| \leq \frac{C \lambda(t)^{2} \log^{2}(t)}{t^{2} \log^{2b}(t)}, \quad g(t) \leq r \leq \frac{t}{2}$$
and, for $j \geq 4$,
$$|r^{2} \partial_{r}^{2} w_{j}(t,r)| \leq \frac{C_{7}^{j} \lambda(t)^{2} \log^{2}(t)}{t^{2} \log^{b(j-1)}(t)}, \quad g(t) \leq r \leq \frac{t}{2}.$$
Combining all of our previous estimates, we get \eqref{wcsymb}\end{proof}
Then, we obtain \eqref{f4symb}. Finally, we make the following definition.
\begin{definition}\label{vcorrdef}$v_{corr}(t,r) = v_{c}(t,r)+w_{c}(t,r).$\end{definition}
\noindent The last estimate of Theorem \ref{approxsolnthm}, namely, \eqref{vcorrinftyest}, follows directly from our estimates on $v_{c}$ and $w_{c}$.
\section{Solving the final equation}
Substituting $u(t,r) = Q_{\frac{1}{\lambda(t)}}(r) + v_{c}(t,r) + w_{c}(t,r) + v(t,r)$ into \eqref{ym}, we get
\begin{equation}\label{veqnfinal}-\partial_{tt}v+\partial_{rr}v+\frac{1}{r}\partial_{r}v+\frac{2}{r^{2}}\left(1-3Q_{\frac{1}{\lambda(t)}}(r)^{2}\right) v = F_{4}(t,r)+F_{5}(t,r)+F_{3}(t,r)\end{equation}
where $F_{4}$ and $F_{5}$ were defined in \eqref{f4def} and \eqref{f5def}, and
\begin{equation}\label{f3def}\begin{split}F_{3}(t,r) &= \frac{2 v(t,r)^{3}}{r^{2}} + \frac{6}{r^{2}}\left(Q_{\frac{1}{\lambda(t)}}(r) + v_{c}(t,r)+w_{c}(t,r)\right)v(t,r)^{2} \\
&+ \frac{6 v(t,r)}{r^{2}}\left(\left(v_{c}(t,r)+Q_{\frac{1}{\lambda(t)}}(r)+w_{c}(t,r)\right)^{2}-Q_{\frac{1}{\lambda(t)}}(r)^{2}\right).\end{split}\end{equation}
We will (formally) derive the equation for a re-scaling of the distorted Fourier transform (as defined in \cite{kstym}, and discussed in \eqref{distortedfouriernotation}) of an appropriate function associated to $v$. We will call this function $y$. Then, we will show that the equation for $y$ has a (weak) solution with enough regularity to rigorously justify the inverse of each step we perform to derive its equation, thereby obtaining a (weak) solution to the original equation, \eqref{veqnfinal}. \\
\\
\textbf{Outline of the argument}
We provide the reader with a more detailed outline of the argument of this section.\\
\textbf{1}. We derive the equation \eqref{yeqn} for $y$ given by
$$\mathcal{F}(\sqrt{\cdot}v(t,\cdot\lambda(t)))(\xi) = \begin{bmatrix} y_{0}(t)& y_{1}(t,\frac{\xi}{\lambda(t)^{2}})\end{bmatrix}^{\text{T}}.$$
We remind the reader that the definition of $\mathcal{F}$, the distorted Fourier transform is discussed in \eqref{distortedfouriernotation}.\\
\\
\textbf{2}. Next, we introduce a space, $Z$, in which we will solve the equation \eqref{yeqn}. This will be done by finding a fixed point of the map $T$ (defined in \eqref{finalTdef}) in $\overline{B_{1}(0)}\subset Z$.  The norm for the space $Z$ is defined in \eqref{znormdef}.\\
\\
\textbf{3}. We then estimate appropriate norms of $F_{2}$ and $F_{3}$ (in section \ref{f2estimatessection} and Proposition \ref{f3estimatesprop}), for all $y \in \overline{B_{1}(0)} \subset Z$. These estimates are sufficient for our purposes because the $F_{2}$ and $F_{3}$ contributions to $T$ (see \eqref{finalTdef}) can be estimated in appropriate norms, by essentially using Minkowski's inequality.
\\
\\
\textbf{4}. In Lemma \ref{f4lemma}, we estimate the $F_{4}$ contributions to $T$ (see again \eqref{finalTdef}). Here, we integrate by parts in the $x$ variable, for example, in the integral
$$\lambda(t)\left(\omega \lambda(t)^{2}\right)^{3/2} \int_{t}^{\infty} dx \frac{\sin((t-x)\sqrt{\omega})}{\sqrt{\omega}} \mathcal{F}(\sqrt{\cdot} F_{4}(x,\cdot \lambda(x)))_{1}(\omega \lambda(x)^{2})$$
which arises from estimating an appropriate weighted norm of \eqref{finalTdef}, taking advantage of both the orthogonality condition, and the symbol-type estimates on $F_{4}$.\\
\\
\textbf{5}. The results from steps \textbf{3} and \textbf{4}, combined with our previous estimates on $F_{5}$, namely \eqref{f5l2} and \eqref{lstarlf5l2}, are sufficient to show that $T$ maps $\overline{B_{1}(0)} \subset Z$ into itself, see \eqref{Test}.\\
\\
\textbf{6}. We then conclude this section by showing that $T$ is a strict contraction on $\overline{B_{1}(0)} \subset Z$, see \eqref{Tlip}.\\
\\
Returning to the main argument of this section, we start by defining $w$ by
$$v(t,r) = w(t,\frac{r}{\lambda(t)})\sqrt{\frac{\lambda(t)}{r}}$$
and evaluate \eqref{veqnfinal} at the point $(t,R\lambda(t))$. Then, we get
\begin{equation}\label{weqn}\begin{split} &-\partial_{tt}w+\frac{2 R \lambda'(t)}{\lambda(t)} \partial_{tR}w - \frac{\lambda'(t)}{\lambda(t)} \partial_{t}w - R^{2} \left(\frac{\lambda'(t)}{\lambda(t)}\right)^{2} \left(\partial_{RR}w-\frac{1}{R} \partial_{R}w+\frac{3}{4R^{2}} w\right)\\
&+ R\left(\partial_{R}w-\frac{w}{2R}\right)\left(\frac{\lambda''(t)}{\lambda(t)} - 2 \left(\frac{\lambda'(t)}{\lambda(t)}\right)^{2}\right)  + \frac{1}{\lambda(t)^{2}}\left(\partial_{RR} w+\frac{24}{(1+R^{2})^{2}} w - \frac{15}{4 R^{2}} w\right) \\
&= \sqrt{R} \left(F_{4}(t,R\lambda(t))+F_{5}(t,R\lambda(t))+F_{3}(t,R\lambda(t))\right).\end{split}\end{equation}
We remind the reader that the definition of the distorted Fourier transform of \cite{kstym} is given, following \cite{kstym}, in \eqref{distortedfouriernotation}, and the transference operator, $\mathcal{K}$, of \cite{kstym}, is recalled in \eqref{transference}. Then, for $y$ defined by
$$\mathcal{F}(w)(t,\xi) = \begin{bmatrix} y_{0}(t)&y_{1}(t,\frac{\xi}{\lambda(t)^{2}})\end{bmatrix}^{\text{T}}$$
we evaluate the distorted Fourier transform of \eqref{weqn} at the point $(t,\omega \lambda(t)^{2})$ to get
\begin{equation}\label{yeqn}\begin{bmatrix} -\partial_{tt}y_{0}&-\partial_{tt}y_{1}-\omega y_{1}\end{bmatrix}^{\text{T}} = F_{2} + \mathcal{F}\left(\sqrt{\cdot}\left(F_{3}+F_{4}+F_{5}\right)\left(t,\cdot \lambda(t)\right)\right)\left(\omega \lambda(t)^{2}\right)\end{equation}
where
\begin{equation}\label{F2def}\begin{split}F_{2} &= -\left(\frac{-\lambda'(t)}{\lambda(t)} \begin{bmatrix}y_{0}'(t)\\
\partial_{1}y_{1}(t,\omega)\end{bmatrix} - \frac{3}{4} \left(\frac{\lambda'(t)}{\lambda(t)}\right)^{2} \begin{bmatrix}y_{0}(t)\\
y_{1}(t,\omega)\end{bmatrix} -\frac{1}{2}\begin{bmatrix} y_{0}(t)\\
y_{1}(t,\omega)\end{bmatrix} \left(\frac{\lambda''(t)}{\lambda(t)} - 2 \left(\frac{\lambda'(t)}{\lambda(t)}\right)^{2}\right)\right. \\
&+\left. 2 \left(\frac{\lambda'(t)}{\lambda(t)}\right) \mathcal{K}\left(\begin{bmatrix} y_{0}'(t)\\
\partial_{1}y_{1}(t,\frac{\cdot}{\lambda(t)^{2}})\end{bmatrix}\right)(\omega \lambda(t)^{2})-2 \left(\frac{\lambda'(t)}{\lambda(t)}\right)^{2} [\mathcal{K},\xi \partial_{\xi}]\left(\begin{bmatrix} y_{0}(t)\\
y_{1}(t,\frac{\cdot}{\lambda(t)^{2}})\end{bmatrix}\right)(\omega\lambda(t)^{2})\right.\\
&\left. - \left(\frac{\lambda'(t)}{\lambda(t)}\right)^{2} \mathcal{K}\left(\mathcal{K}\left(\begin{bmatrix} y_{0}(t)\\
y_{1}(t,\frac{\cdot}{\lambda(t)^{2}})\end{bmatrix}\right)\right)(\omega \lambda(t)^{2})+ \frac{\lambda''(t)}{\lambda(t)} \mathcal{K}\left(\begin{bmatrix} y_{0}(t)\\
y_{1}(t,\frac{\cdot}{\lambda(t)^{2}})\end{bmatrix}\right)( \omega \lambda(t)^{2})\right).\end{split}\end{equation}
As in the definition of the transference operator, in the expression for $F_{2}$, the operator $\xi \partial_{\xi}$ appearing in the term involving $[\mathcal{K},\xi \partial_{\xi}]\left(\begin{bmatrix} y_{0}(t)\\
y_{1}(t,\frac{\cdot}{\lambda(t)^{2}})\end{bmatrix}\right)$ only acts on the second component of $\begin{bmatrix} y_{0}(t)\\
y_{1}(t,\frac{\cdot}{\lambda(t)^{2}})\end{bmatrix}$. Also, the symbols $( \omega \lambda(t)^{2})$ appearing after, for instance $\mathcal{K}\left(\begin{bmatrix} y_{0}(t)\\
y_{1}(t,\frac{\cdot}{\lambda(t)^{2}})\end{bmatrix}\right)$ mean that the second component of $\mathcal{K}\left(\begin{bmatrix} y_{0}(t)\\
y_{1}(t,\frac{\cdot}{\lambda(t)^{2}})\end{bmatrix}\right)$ is evaluated at the point $\omega \lambda(t)^{2}$. (Recall that the first component of $\mathcal{K}\left(\begin{bmatrix} y_{0}(t)\\
y_{1}(t,\frac{\cdot}{\lambda(t)^{2}})\end{bmatrix}\right)$ is a real number, rather than a function of frequency).
\subsection{The Iteration Space}
We will now describe the space in which we plan to solve \eqref{yeqn}. We let $(Z,||\cdot||_{Z})$ denote the normed vector space defined as follows. $Z$ is the set of elements $\begin{bmatrix} y_{0}(t)\\
y_{1}(t,\omega)\end{bmatrix}$ where $y_{0}:[T_{0},\infty) \rightarrow \mathbb{R}$, and $y_{1}$ is an (equivalence class) of measurable functions, $y_{1}:[T_{0},\infty) \times (0,\infty) \rightarrow \mathbb{R}$ such that
$$y_{0}(t) \in C^{1}_{t}([T_{0},\infty))$$
$$y_{1}(t,\omega) \frac{t^{2} \log^{\epsilon}(t)}{\lambda(t)} \langle \omega \lambda(t)^{2}\rangle^{3/2}\sqrt{\rho(\omega \lambda(t)^{2})} \in C^{0}_{t}([T_{0},\infty), L^{2}(d\omega))$$
$$\partial_{t}y_{1}(t,\omega) \frac{t^{3} \log^{\epsilon}(t)}{\lambda(t)} \langle \omega \lambda(t)^{2}\rangle \sqrt{\rho(\omega \lambda(t)^{2})} \in C^{0}_{t}([T_{0},\infty), L^{2}(d\omega))$$
and $||\begin{bmatrix} y_{0}\\
y_{1}\end{bmatrix}||_{Z} < \infty$
where
\begin{equation}\label{znormdef}\begin{split}||\begin{bmatrix} y_{0}\\
y_{1}\end{bmatrix}||_{Z} &= \sup_{t \geq T_{0}}\left(\frac{t^{2} \log^{\epsilon}(t) |y_{0}(t)|}{\lambda(t)^{2}} + \frac{\log^{\epsilon}(t) \lambda(t) t^{2} ||y_{1}(t,\omega)\langle \omega \lambda(t)^{2} \rangle^{3/2}||_{L^{2}(\rho(\omega \lambda(t)^{2}) d\omega)}}{\lambda(t)^{2}} + \frac{t^{3} \log^{\epsilon}(t) |y_{0}'(t)|}{\lambda(t)^{2}}\right.\\
&\left. + \frac{t^{3} \log^{\epsilon}(t) \lambda(t) ||\partial_{t}y_{1}(t,\omega) \langle \omega \lambda(t)^{2} \rangle||_{L^{2}(\rho(\omega \lambda(t)^{2}) d\omega)}}{\lambda(t)^{2}}\right).\end{split}\end{equation} 
\subsection{$F_{2}$ estimates}\label{f2estimatessection}
We will now estimate $F_{2}$, for all $\begin{bmatrix} y_{0}&y_{1}\end{bmatrix}^{\text{T}} \in \overline{B_{1}(0)} \subset Z$. We use Proposition 5.2 of \cite{kstym}, which states that $\mathcal{K}$ and $[\mathcal{K},\xi \partial_{\xi}]$ are bounded on $L^{2,\alpha}_{\rho}$, for all $\alpha \in \mathbb{R}$, which, as defined in \cite{kstym}, has norm given in \eqref{l2rhoalphanorm}. Recalling the definitions of $F_{2}$ (in \eqref{F2def}), and $||\cdot||_{Z}$ (in \eqref{znormdef}), we immediately get
$$|F_{2,0}(x)| \leq \frac{C \lambda(x)^{2}}{x^{4} \log^{b+\epsilon}(x)}$$
$$\lambda(x)||F_{2,1}(x,\omega)||_{L^{2}(\rho(\omega \lambda(x)^{2}) d\omega)} \leq \frac{C \lambda(x)^{2}}{x^{4} \log^{b+\epsilon}(x)}$$
$$\lambda(x)^{4} ||\omega F_{2,1}(x,\omega)||_{L^{2}(\rho(\omega \lambda(x)^{2})d\omega)} \leq C \frac{\lambda(x)^{3}}{x^{4} \log^{b+\epsilon}(x)}$$
where we define $F_{2,i}$ by $F_{2} = \begin{bmatrix}F_{2,0}&F_{2,1}\end{bmatrix}^{\text{T}}.$
\subsection{$F_{3}$ estimates}
The main result of this section is the following. (We remind the reader that $\phi(r,\xi)$ has been defined in \eqref{phiefuncdef}, and that $v_{corr}$ has been defined in Definition \ref{vcorrdef}).
\begin{proposition}\label{f3estimatesprop} Let $\begin{bmatrix} y_{0}\\
y_{1}\end{bmatrix}$ satisfy $y_{1}(t,\omega)\langle \omega \lambda(t)^{2}\rangle^{3/2} \in L^{2}(\rho(\omega \lambda(t)^{2})d\omega)$. Let $F_{3}$ be given by \eqref{f3def}, for
$$v(t,r) = y_{0}(t) \frac{\phi_{0}(\frac{r}{\lambda(t)})}{||\phi_{0}||^{2}_{L^{2}(rdr)}} + \int_{0}^{\infty} y_{1}(t,\frac{\xi}{\lambda(t)^{2}}) \phi(\frac{r}{\lambda(t)},\xi) \sqrt{\frac{\lambda(t)}{r}} \rho(\xi) d\xi.$$ 
Then, we have the following estimates.
\begin{equation}\label{f3totalest}\begin{split} &||F_{3}(t,R\lambda(t))||_{L^{2}(R dR)}+||L^{*}L(F_{3}(t,R\lambda(t)))||_{L^{2}(R dR)} \\
&\leq C \frac{\lambda(t)^{2}}{t^{4} \log^{2\epsilon}(t)} ||\begin{bmatrix} y_{0}\\
y_{1}\end{bmatrix}||_{Z}^{2} +C \frac{\lambda(t)^{4}}{t^{6} \log^{3\epsilon}(t)} ||\begin{bmatrix} y_{0}\\
y_{1}\end{bmatrix}||_{Z}^{3}  + C \frac{\lambda(t)^{2}}{t^{4} \log^{\epsilon+b}(t)} ||\begin{bmatrix} y_{0}\\
y_{1}\end{bmatrix}||_{Z} \end{split}\end{equation}
\end{proposition}
\begin{proof} We recall the operators $L$ and $L^{*}$, which were defined in \eqref{Ldef} and \eqref{Lstardef}, respectively. As noted in \cite{rr}, $L$ satisfies
$$L^{*}L(f)(r) = -\partial_{rr}f-\frac{1}{r}\partial_{r}f-\frac{2}{r^{2}}\left(1-3Q_{1}(r)^{2}\right) f.$$
In order to estimate $F_{3}$, which we recall is given by \eqref{f3def}, we must first understand what kinds of (weighted) estimates we can conclude on $v$, $Lv$, and $L^{*}Lv$, for $v$ given by
\begin{equation}\label{vyrelationbeforelemma51}\sqrt{r} v(t,r \lambda(t)) =\mathcal{F}^{-1}\left( \begin{bmatrix} y_{0}(t)&y_{1}(t,\frac{\cdot}{\lambda(t)^{2}})\end{bmatrix}^{\text{T}}\right)(r)\end{equation}
where
$\begin{bmatrix} y_{0}&y_{1}\end{bmatrix}^{\text{T}} \in \overline{B_{1}(0)} \subset Z$.
This is the purpose of the following lemma, which is similar to Lemma 5.1 of \cite{wm}.
\begin{lemma}\label{ptwselemma} There exists $C>0$ such that, for all $\begin{bmatrix} \overline{y_{0}}&\overline{y_{1}}(\xi)\end{bmatrix}^{\text{T}}$ with $\overline{y_{1}}(\xi) \langle \xi \rangle^{3/2} \in L^{2}((0,\infty),\rho(\xi)d\xi)$, if $\overline{v}$ is given by
$$\overline{v}(r) = \frac{1}{\sqrt{r}}\mathcal{F}^{-1}\left(\begin{bmatrix} \overline{y_{0}}&\overline{y_{1}}\end{bmatrix}^{\text{T}}\right)(r) = \overline{y_{0}} \frac{\phi_{0}(r)}{||\phi_{0}||_{L^{2}(r dr)}^{2}} + \int_{0}^{\infty} \frac{\phi(r,\xi)}{\sqrt{r}} \overline{y_{1}}(\xi) \rho(\xi) d\xi$$
then, $\partial_{r}\overline{v}$ and $\overline{v}$ admit continuous extensions to $[0,\infty)$ (which we also denote by $\partial_{r}\overline{v}$ and $\overline{v}$):
$$\partial_{r}\overline{v}, \overline{v} \in C^{0}([0,\infty))$$ 
$$\overline{v}(0)=\partial_{r}\overline{v}(0)=\lim_{r \rightarrow \infty} \overline{v}(r)=\lim_{r\rightarrow\infty} \partial_{r}\overline{v}(r)=0$$
\begin{equation}\label{vptwse} \frac{|\overline{v}(r)|}{r^{2}\sqrt{\langle \log(r)\rangle}} \leq C \left(|\overline{y_{0}}|+ ||\overline{y_{1}}(\xi) \langle \xi \rangle^{3/2}||_{L^{2}(\rho(\xi) d\xi)}\right), \quad r \leq 1\end{equation}
\begin{equation}\label{lvptwse} |L\overline{v}(r)| \leq C r \sqrt{\langle \log(r)\rangle}  ||\overline{y_{1}}(\xi) \langle \xi \rangle^{3/2}||_{L^{2}(\rho(\xi) d\xi)}, \quad r \leq 1\end{equation}
\begin{equation}\label{lstarlvptwse} |L^{*}L\overline{v}(r)| \leq C \sqrt{\langle \log(r) \rangle} ||\overline{y_{1}}(\xi) \langle \xi \rangle^{3/2}||_{L^{2}(\rho(\xi) d\xi)}, \quad 0< r \leq 1\end{equation}
\begin{equation}\label{l2trans} ||\overline{v}||_{L^{2}(r dr)} \leq C \left(|\overline{y_{0}}| + ||\overline{y_{1}}(\xi)||_{L^{2}(\rho(\xi) d\xi)}\right)\end{equation}
\begin{equation}\label{lstarltrans} ||L^{*}L\overline{v}||_{L^{2}(r dr)} = ||\xi \overline{y_{1}}(\xi)||_{L^{2}(\rho(\xi) d\xi)}\end{equation}
\begin{equation}\label{ltrans} ||L\overline{v}||_{L^{2}(r dr)} = ||\sqrt{\xi} \overline{y_{1}}(\xi)||_{L^{2}(\rho(\xi)d\xi)}\end{equation}
\begin{equation} \label{h1dotest} ||\overline{v}||_{\dot{H}^{1}_{e}} \leq C \left(||L(\overline{v})||_{L^{2}(r dr)} + ||\overline{v}||_{L^{2}(r dr)} \right)\end{equation}
\begin{equation} \label{inftyest} ||\overline{v}||_{L^{\infty}} \leq C ||\overline{v}||_{\dot{H}^{1}_{e}}\end{equation}
\begin{equation}\label{lstarl2} ||L\overline{v}||_{L^{\infty}} + ||L\overline{v}||_{\dot{H}^{1}_{e}} \leq C ||L^{*}L\overline{v}||_{L^{2}(r dr)}\end{equation}
\end{lemma}
\begin{proof}
The proof of this lemma is very similar to that of Lemma 5.1 of \cite{wm}, but, we will include a proof here for completeness. To establish \eqref{vptwse}, it suffices to estimate the following integral, using \eqref{phiseries}, \eqref{philgrasymp}, \eqref{asymbol}, and \eqref{rhoest} (i.e., properties of $\phi(r,\xi)$, $a$, and $\rho$ from section 4 of \cite{kstym}). For $r \leq 1$,
\begin{equation}\begin{split}&|\int_{0}^{\infty} \overline{y_{1}}(\xi) \frac{\phi(r,\xi)}{r^{5/2}} \frac{\rho(\xi) d\xi}{\sqrt{\langle \log(r)\rangle}}| \leq \frac{C}{\sqrt{\langle \log(r) \rangle}}\int_{0}^{\frac{4}{r^{2}}} |\overline{y_{1}}(\xi)| \langle \xi\rangle^{3/2} \sqrt{\rho(\xi)} \frac{\sqrt{\rho(\xi)}d\xi}{\langle \xi\rangle^{3/2}} + C \int_{\frac{4}{r^{2}}}^{\infty} \frac{|\overline{y_{1}}(\xi)| \cdot |a(\xi)| \rho(\xi) d\xi}{\xi^{1/4} r^{5/2} \sqrt{\langle \log(r) \rangle}}\\
&\leq \frac{C}{\sqrt{\langle \log(r) \rangle}} ||\overline{y_{1}}(\xi) \langle \xi \rangle^{3/2}||_{L^{2}(\rho(\xi) d\xi)} \left(\int_{0}^{\frac{4}{r^{2}}} \frac{\rho(\xi) d\xi}{\langle \xi\rangle^{3}}\right)^{1/2} + \frac{C}{\sqrt{\langle \log(r) \rangle}} ||\overline{y_{1}}(\xi) \langle \xi \rangle^{3/2}||_{L^{2}(\rho(\xi) d\xi)}.\end{split}\end{equation}
The proofs of \eqref{lvptwse} and \eqref{lstarlvptwse} are similar. The only new point to note is that $L(\phi_{0})=0$, which explains why there is no $\overline{y_{0}}$ term on the right-hand sides of these inequalities. Also, we use the fact that the functions $\widetilde{\phi_{j}}$ (which appear in the formula \eqref{phiseries} for $\phi(r,\xi)$) in \cite{kstym} satisfy symbol-type estimates. This is not directly stated in Proposition 4.5 of \cite{kstym}. However, this fact can be proven in a straightforward manner by noting that
$$ \widetilde{f_{0}}(u) = \frac{u^{2}}{(1+u)^{2}}, \quad \widetilde{\phi_{j}}(u) = u^{-j} \widetilde{f_{j}}(u), \text{ for } j \geq 1, \quad \widetilde{f_{1}}(u) = \frac{-u^{3}(2+u)}{24(1+u)^{2}}$$
and then using an argument by induction to estimate $\widetilde{f_{j}}'(u)$, given the representation formula of $\widetilde{f_{j}}$ in terms of $\widetilde{f_{j-1}}$, which is
$$\widetilde{f_{j}}(u) = \int_{0}^{u} \left(k(u,v)-k(v,u)\right)\frac{\widetilde{f_{j-1}}(v)}{8v^{2}} dv, \quad \text{where }k(u,v) = \frac{u^{2}\left(-1-8v+12 v^{2} \log(v)+8v^{3}+v^{4}\right)}{(1+u)^{2}(1+v)^{2}}.$$
The lemma statement regarding continuity of $\overline{v}$ and $\partial_{r}\overline{v}$ is proven with the Dominated convergence theorem, again using the properties \eqref{phiseries} and \eqref{philgrasymp} of $\phi(r,\xi)$ (which are from section 4 of \cite{kstym}). The Dominated convergence theorem also shows that $\overline{v}$ and $\partial_{r}\overline{v}$ extend continuously to $[0,\infty)$ with $\overline{v}(0)=\partial_{r}\overline{v}(0)=\lim_{r \rightarrow \infty} \overline{v}(r)=\lim_{r\rightarrow\infty} \partial_{r}\overline{v}(r)=0$. The inequality \eqref{l2trans} follows directly from the $L^{2}$ isometry property of $\mathcal{F}$. To prove \eqref{ltrans}, we first recall the conjugation of $L^{*}L$ used in \cite{kstym}, namely $\mathcal{L}$, which satisfies
$$\mathcal{L}(u)(r) = L^{*}L\left(\frac{u(\cdot)}{\sqrt{\cdot}}\right)(r) \sqrt{r}$$
and for which $\phi(r,\xi)$ are eigenfunctions. Using the Dominated convergence theorem, we have, for $r>0$,
$$\mathcal{L}(\overline{v}(\cdot) \sqrt{\cdot})(r) = \int_{0}^{\infty} \mathcal{L}(\phi(\cdot,\xi))(r) \overline{y_{1}}(\xi) \rho(\xi) d\xi = \int_{0}^{\infty} \xi \phi(r,\xi) \overline{y_{1}}(\xi) \rho(\xi) d\xi = \mathcal{F}^{-1}\left(\begin{bmatrix} 0&\xi \overline{y_{1}}(\xi)\end{bmatrix}^{\text{T}}\right)(r)$$
where we again emphasize that we regard the distorted Fourier transform as a two-component vector, following \cite{kstym}. We can now prove \eqref{lstarltrans} using the $L^{2}$ isometry property of $\mathcal{F}$.
$$||\xi \overline{y_{1}}(\xi)||_{L^{2}(\rho(\xi) d\xi)} = ||\mathcal{L}(\overline{v}(\cdot) \sqrt{\cdot})(r)||_{L^{2}(dr)} = ||L^{*}L(\overline{v})||_{L^{2}(r dr)}$$ 
Continuing the proof of \eqref{ltrans}, we have
\begin{equation}\label{intstep}\begin{split}\langle \mathcal{L}\left(\overline{v}(\cdot)\sqrt{\cdot}\right)(r),\overline{v}(r)\sqrt{r} \rangle_{L^{2}(dr)} &= \langle \mathcal{F}^{-1}\left(\begin{bmatrix} 0&\xi \overline{y_{1}}(\xi)\end{bmatrix}^{\text{T}}\right),\mathcal{F}^{-1}\left(\begin{bmatrix} \overline{y_{0}}&\overline{y_{1}}\end{bmatrix}^{\text{T}}\right)(r)\rangle_{L^{2}(dr)} \\
&= \langle \xi \overline{y_{1}}(\xi),\overline{y_{1}}(\xi)\rangle_{L^{2}(\rho(\xi) d\xi)} = ||\sqrt{\xi} \overline{y_{1}}(\xi)||^{2}_{L^{2}(\rho(\xi) d\xi)}.\end{split}\end{equation}
On the other hand, 
\begin{equation}\begin{split}\langle \mathcal{L}\left(\overline{v}(\cdot)\sqrt{\cdot}\right)(r),\overline{v}(r)\sqrt{r} \rangle_{L^{2}(dr)} &= \int_{0}^{\infty} L^{*}L(\overline{v})(r) \overline{v}(r) r dr.\end{split}\end{equation}
Recalling the assumptions of the lemma, and \eqref{l2trans} and \eqref{lstarltrans}, which we have proven at this point, we see that $L^{*}L(\overline{v})(r) \overline{v}(r) r \in L^{1}((0,\infty),dr).$ Therefore, by the Dominated convergence theorem, 
$$\int_{0}^{\infty} \overline{v}(r) L^{*}L\overline{v}(r) r dr = \lim_{M \rightarrow \infty} \int_{0}^{\infty} \overline{v}(r) L^{*}L\overline{v}(r) \psi_{\leq 1}(\frac{r}{M}) r dr$$
where $0 \leq \psi_{\leq 1}(x) \leq 1, \quad \psi_{\leq 1} \in C^{\infty}(\mathbb{R}), \quad \psi_{\leq 1}'(x) \leq 0, \psi_{\leq 1}(x) = \begin{cases} 1, \quad x \leq \frac{1}{2}\\
0, \quad x \geq 1\end{cases}.$ Then, we inspect integral which arises from the term $\partial_{r}$ in the expression for $L^{*}$ (recall \eqref{Lstardef}), and integrate by parts once to get, for instance, for all $M \geq 1$,
\begin{equation}\begin{split}&\int_{0}^{\infty} \overline{v}(r) \partial_{r}L\overline{v}(r) \psi_{\leq 1}(\frac{r}{M}) r dr = -\int_{0}^{\infty} L\overline{v}(r) \partial_{r}\left(\overline{v}(r) \psi_{\leq 1}(\frac{r}{M}) r\right) dr   \\
&= -\int_{0}^{\infty} L(\overline{v})(r) \left(\partial_{r}\overline{v}(r) + \frac{\overline{v}(r)}{r}\right) \psi_{\leq 1}(\frac{r}{M}) r dr + \int_{0}^{\infty} \frac{\partial_{r}\left((\overline{v}(r))^{2}\right)}{2} \frac{\psi_{\leq 1}'(\frac{r}{M})}{M} r dr \\
&- \int_{0}^{\infty} 2 \left(\frac{1-r^{2}}{1+r^{2}}\right) \frac{\overline{v}(r)^{2}}{r} \frac{\psi_{\leq 1}'(\frac{r}{M})}{M} r dr. \end{split} \end{equation}
For the integral below, we integrate by parts again, to get
$$|\int_{0}^{\infty} \frac{\partial_{r}\left((\overline{v}(r))^{2}\right)}{2} \frac{\psi_{\leq 1}'(\frac{r}{M})}{M} r dr| = |-\int_{0}^{\infty} \frac{\overline{v}(r)^{2}}{2} \left(\frac{\psi_{\leq 1}''(\frac{r}{M})}{M^{2}} + \frac{\psi_{\leq 1}'(\frac{r}{M})}{M r}\right) r dr| \leq \frac{C}{M^{2}} ||\overline{v}||^{2}_{L^{2}(r dr)}.$$
In total, we have (recall \eqref{Ldef})
$$\int_{0}^{\infty} \overline{v}(r) L^{*}L\overline{v}(r) r dr=\lim_{M \rightarrow \infty} \int_{0}^{\infty} \overline{v}(r) L^{*}L\overline{v}(r) \psi_{\leq 1}(\frac{r}{M}) r dr = \lim_{M \rightarrow \infty} \int_{0}^{\infty} (L\overline{v}(r))^{2} \psi_{\leq 1}(\frac{r}{M}) r dr.$$
By the Monotone convergence theorem (recall that $\psi_{\leq 1}'(x) \leq 0$) we get
$$\int_{0}^{\infty} (L\overline{v}(r))^{2} r dr = \int_{0}^{\infty} \overline{v}(r) L^{*}L\overline{v}(r) r dr.$$
Combining this with the observation \eqref{intstep}, we get
$$||\sqrt{\xi} \overline{y_{1}}(\xi)||^{2}_{L^{2}(\rho(\xi) d\xi)} = \int_{0}^{\infty} (L\overline{v}(r))^{2} r dr$$
which is \eqref{ltrans}. The inequality \eqref{h1dotest} is proven similarly to the analogous estimate in \cite{wm}, except that we will not need to use an approximation argument. The details of the proof are as follows. From the lemma hypotheses, and what we have established up to now,  $\overline{v} \in C^{1}([0,\infty)) \cap L^{2}((0,\infty), r dr)$, and $L\overline{v} \in L^{2}((0,\infty), r dr)$ with $\overline{v}(0) = \lim_{r \rightarrow \infty} \overline{v}(r) =0$. So,
$$L\overline{v}(r) = -\overline{v}'(r) + \frac{2}{r} \overline{v}(r) + \overline{v}(r) \cdot \left(\frac{-4 r}{1+r^{2}}\right)$$
which shows that \begin{equation}\label{h1dotint}-\overline{v}'(r) + \frac{2}{r} \overline{v}(r) = L\overline{v}(r)-\overline{v}(r) \cdot \left(\frac{-4 r}{1+r^{2}}\right) \in L^{2}((0,\infty), r dr).\end{equation}
So, for any $M>4$,
\begin{equation}\label{h1dotint2}\int_{\frac{1}{M}}^{M} (-\overline{v}'(r)+\frac{2}{r}\overline{v}(r))^{2} r dr = \int_{\frac{1}{M}}^{M} \left((\overline{v}'(r))^{2} + \frac{4}{r^{2}} \overline{v}(r)^{2}\right) r dr - 2 \int_{\frac{1}{M}}^{M} \frac{d}{dr}((\overline{v}(r)^{2}) dr. \end{equation}
By the Dominated convergence theorem, and the observation \eqref{h1dotint}, we have
$$\lim_{M \rightarrow \infty} \int_{\frac{1}{M}}^{M} (-\overline{v}'(r)+\frac{2}{r}\overline{v}(r))^{2} r dr = \int_{0}^{\infty} (-\overline{v}'(r)+\frac{2}{r}\overline{v}(r))^{2} r dr.$$
We then let $M\rightarrow \infty$ in \eqref{h1dotint2} and use the Monotone convergence theorem, to get
$$\int_{0}^{\infty} \left((\overline{v}'(r))^{2}+\frac{4}{r^{2}}\overline{v}(r)^{2}\right) r dr = \int_{0}^{\infty} (-\overline{v}'(r)+\frac{2}{r}\overline{v}(r))^{2} r dr\leq C \left(||L\overline{v}||^{2}_{L^{2}(r dr)} + ||\overline{v}||^{2}_{L^{2}(r dr)}\right)$$
where the last inequality follows from \eqref{h1dotint}. This in particular, proves that $\overline{v} \in \dot{H}^{1}_{e}$, with the estimate \eqref{h1dotest}. The next inequality to prove, \eqref{inftyest}, is proven in the same way as in \cite{wm}. In particular, since $\overline{v} \in C^{1}([0,\infty))\cap \dot{H}^{1}_{e}$, with $\overline{v}(0)=0$, we use the Fundamental theorem of calculus for $\overline{v}^{2}$, to get
\begin{equation}\label{ptwseh1doteintstep}\overline{v}(r)^{2} \leq C \int_{0}^{r} \frac{|\overline{v}(r)|}{\sqrt{r}} |\partial_{r}\overline{v}| \sqrt{r} dr \leq C ||\frac{\overline{v}}{r}||_{L^{2}(r dr)} ||\partial_{r}\overline{v}||_{L^{2}(r dr)} \leq C ||\overline{v}||^{2}_{\dot{H}^{1}_{e}}.\end{equation} 
The final estimate to prove is \eqref{lstarl2}. If we have $g \in C^{1}([0,\infty))\cap L^{2}((0,\infty), r dr)$, and $L^{*}g \in L^{2}((0,\infty), r dr)$ with $g(0) = \lim_{r \rightarrow \infty} g(r) =0$, then, we recall the definition of $L^{*}$: 
$$L^{*}(f)(r) = f'(r) + V(r) f(r) , \quad \text{ where }V(r) = \frac{2}{r}\left(\frac{1-r^{2}}{1+r^{2}}\right) + \frac{1}{r}.$$
For $M > 4$,
$$\int_{\frac{1}{M}}^{M} \left(L^{*}g\right)^{2} r dr = \int_{\frac{1}{M}}^{M} (g'(r))^{2} r dr + \int_{\frac{1}{M}}^{M} (V(r))^{2} g(r)^{2} r dr + 2 \int_{\frac{1}{M}}^{M} V(r) \frac{d}{dr}(g(r)^{2}) \frac{r}{2} dr.$$
The last term in the expression above is
$$V(r) g(r)^{2} r \Bigr|_{\frac{1}{M}}^{M} - \int_{\frac{1}{M}}^{M} (g(r))^{2} \frac{d}{dr}\left(V(r) r\right) dr$$
So, 
\begin{equation}\begin{split}\int_{\frac{1}{M}}^{M} \left(L^{*}g\right)^{2} r dr &= \int_{\frac{1}{M}}^{M} \left((g'(r))^{2}+(V(r))^{2} g(r)^{2} -\frac{(g(r))^{2}}{r} \frac{d}{dr}\left(V(r)r\right)\right) rdr + V(r) g(r)^{2} r \Bigr|_{\frac{1}{M}}^{M}\\
&= \int_{\frac{1}{M}}^{M} \left((g'(r))^{2}+ g(r)^{2} \left(\frac{9+2r^{2}+r^{4}}{(r+r^{3})^{2}}\right)\right) r dr + V(r) g(r)^{2} r \Bigr|_{\frac{1}{M}}^{M}\end{split}\end{equation}
By our assumptions on $g$,
$$\lim_{M \rightarrow \infty} \left(\int_{\frac{1}{M}}^{M} (L^{*}g)^{2} r dr - \left(V(r) g(r)^{2} r \Bigr|_{\frac{1}{M}}^{M}\right)\right) = \int_{0}^{\infty} (L^{*}g)^{2} r dr$$
(in particular, the limit exists). Therefore,
$$\lim_{M \rightarrow \infty} \left(\int_{\frac{1}{M}}^{M} \left((g'(r))^{2}+ g(r)^{2} \left(\frac{9+2r^{2}+r^{4}}{(r+r^{3})^{2}}\right)\right) r dr\right)=\int_{0}^{\infty} (L^{*}g)^{2} r dr.$$
By the Monotone convergence theorem, we then get that $g \in \dot{H}^{1}_{e}$, and 
$$||L^{*}g||^{2}_{L^{2}(r dr)} = \int_{0}^{\infty} \left((g'(r))^{2}+ g(r)^{2} \left(\frac{9+2r^{2}+r^{4}}{(r+r^{3})^{2}}\right)\right) r dr \geq C \int_{0}^{\infty} \left((g'(r))^{2}+\frac{g(r)^{2}}{r^{2}}\right) r dr$$
In other words, $||g||_{\dot{H}^{1}_{e}} \leq C ||L^{*}g||_{L^{2}(r dr)}.$ Then, we can apply \eqref{ptwseh1doteintstep} (note that $g$ satisfies all of the assumptions we used to establish \eqref{ptwseh1doteintstep}, now that we showed that $g \in \dot{H}^{1}_{e}$)  to get
\begin{equation}\label{gconc}||g||_{L^{\infty}} \leq C ||g||_{\dot{H}^{1}_{e}} \leq C ||L^{*}g||_{L^{2}(r dr)}.\end{equation}
Note that $L\overline{v}$ does not quite satisfy all of the stated assumptions on $g$ used in the preceding argument. Therefore, we define, for $M \geq 4$,
$$\overline{v}_{M}(r) = \overline{y_{0}} \frac{\phi_{0}(r)}{||\phi_{0}||^{2}_{L^{2}(r dr)}} + \int_{0}^{\infty} \frac{\phi(r,\xi)}{\sqrt{r}} \overline{y_{1}}(\xi) \rho(\xi) \chi_{\leq 1}(\frac{\xi}{M}) d\xi.$$
As in the proof of Lemma 5.1 in \cite{wm}, we verify, via the Dominated convergence theorem, that $L\overline{v_{M}} \in C^{1}([0,\infty))$, and that $L\overline{v_{M}} \in  L^{2}((0,\infty), r dr)$, $L^{*}L\overline{v_{M}} \in L^{2}((0,\infty), r dr)$ with $L(\overline{v_{M}})(0) = \lim_{r \rightarrow \infty} L(\overline{v_{M}})(r) =0$. Therefore, we have \eqref{gconc} for $g=L\overline{v_{M}}$. An approximation argument then establishes \eqref{lstarl2}.
\end{proof}
We recall the definition of $F_{3}$. 
\begin{equation}\label{f3def2}\begin{split}F_{3}(t,r) &= \frac{2 v(t,r)^{3}}{r^{2}} + \frac{6}{r^{2}}\left(Q_{\frac{1}{\lambda(t)}}(r) + v_{c}(t,r)+w_{c}(t,r)\right)v(t,r)^{2} \\
&+ \frac{6 v(t,r)}{r^{2}}\left(\left(v_{c}(t,r)+Q_{\frac{1}{\lambda(t)}}(r)+w_{c}(t,r)\right)^{2}-Q_{\frac{1}{\lambda(t)}}(r)^{2}\right)\end{split}\end{equation}
where $v$ is given in terms of our previously defined function $y$, by
$$v(t,r) = y_{0}(t) \frac{\phi_{0}(\frac{r}{\lambda(t)})}{||\phi_{0}||^{2}_{L^{2}(rdr)}} + \int_{0}^{\infty} y_{1}(t,\frac{\xi}{\lambda(t)^{2}}) \phi(\frac{r}{\lambda(t)},\xi) \sqrt{\frac{\lambda(t)}{r}} \rho(\xi) d\xi.$$
Now, we will record estimates on quantities associated to $F_{3}$, for any $\begin{bmatrix} y_{0}&y_{1}\end{bmatrix}^{\text{T}}$ satisfying $y_{1}(t,\omega)\langle \omega \lambda(t)^{2}\rangle^{3/2} \in L^{2}(\rho(\omega \lambda(t)^{2})d\omega)$. We recall that $v_{corr}$ has been defined in Definition \ref{vcorrdef}, and estimated in \eqref{vcorrinftyest}. Then, directly estimating each term of \eqref{f3def2} (and their derivatives), and using the above lemma, we get \eqref{f3totalest}. \end{proof}

\subsection{$F_{4}$ Estimates}
Here, we will estimate certain oscillatory integrals applied to $F_{4}$, for later use. First, we will need an estimate related to the function $\rho$ (which was defined in \eqref{rhoest}, following \cite{kstym}) appearing in the spectral measure associated to the distorted Fourier transform of \cite{kstym}. Using  \eqref{rhoest} (which follows from Lemma 4.7 of \cite{kstym}) we get
\begin{equation}\label{rhoscaling}\frac{\rho(\omega \lambda(t)^{2})}{\rho(\omega \lambda(x)^{2})} \leq C \text{max}\{\frac{\lambda(t)^{4}}{\lambda(x)^{4}} ,1\}.\end{equation}
For ease of notation, we define $\mathcal{F}\left(\sqrt{\cdot}F_{4}\left(x,\cdot \lambda(x)\right)\right)_{i}$ by
$$\mathcal{F}\left(\sqrt{\cdot}F_{4}\left(x,\cdot \lambda(x)\right)\right)\left(\omega \lambda(x)^{2}\right) = \begin{bmatrix} \mathcal{F}\left(\sqrt{\cdot}F_{4}\left(x,\cdot \lambda(x)\right)\right)_{0}&
\mathcal{F}\left(\sqrt{\cdot}F_{4}\left(x,\cdot \lambda(x)\right)\right)_{1}\left(\omega \lambda(x)^{2}\right)\end{bmatrix}^{\text{T}}.$$
Later on, we will use this notation for $F_{3}$ and $F_{5}$ as well. Because $F_{4}(x,R\lambda(x))$ is orthogonal to $\phi_{0}(R)$ in $L^{2}(R dR)$, we have 
$$\mathcal{F}(\sqrt{\cdot} F_{4}(x,\cdot \lambda(x)))(\omega \lambda(x)^{2})=\begin{bmatrix} 0&
\mathcal{F}(\sqrt{\cdot} F_{4}(x,\cdot \lambda(x)))_{1}(\omega \lambda(x)^{2})\end{bmatrix}^{\text{T}}.$$  
Now, we will prove the following lemma
\begin{lemma}\label{f4lemma} We have the following estimates
\begin{equation} ||\lambda(t) \left(\omega \lambda(t)^{2}\right)^{3/2} \int_{t}^{\infty} dx \frac{\sin((t-x)\sqrt{\omega})}{\sqrt{\omega}} \mathcal{F}(\sqrt{\cdot} F_{4}(x,\cdot \lambda(x))_{1}(\omega \lambda(x)^{2})||_{L^{2}(\rho(\omega \lambda(t)^{2})d\omega)} \leq \frac{C \lambda(t)^{2}}{t^{2} \log^{b}(t)}\end{equation}
\begin{equation}\begin{split}&||\lambda(t) \int_{t}^{\infty} \frac{\sin((t-x)\sqrt{\omega})}{\sqrt{\omega}} \mathcal{F}(\sqrt{\cdot} F_{4}(x,\cdot \lambda(x)))_{1}(\omega \lambda(x)^{2}) dx ||_{L^{2}(\rho(\omega \lambda(t)^{2})d\omega)}\\
&\leq \frac{C \lambda(t)^{2}}{t^{2}} \left(\frac{1}{\log^{2b-2\epsilon-1}(t)} + \frac{1}{\log^{3b-2-2\epsilon}(t)} + \frac{1}{\log^{2\epsilon}(t)} + \frac{1}{\log^{b}(t)}\right)\end{split}\end{equation}
\begin{equation}\begin{split} &||\lambda(t) \int_{t}^{\infty} dx \cos((t-x)\sqrt{\omega}) \mathcal{F}(\sqrt{\cdot} F_{4}(x,\cdot \lambda(x)))_{1}(\omega \lambda(x)^{2})||_{L^{2}(\rho(\omega \lambda(t)^{2})d\omega)}\\
&\leq C \frac{\lambda(t)^{2}}{t^{3}}\left(\frac{1}{\log^{b}(t)} + \frac{1}{\log^{2\epsilon}(t)} + \frac{1}{\log^{2b-2\epsilon-1}(t)} + \frac{1}{\log^{3b-2-2\epsilon}(t)}\right)\end{split}\end{equation}
\begin{equation}\begin{split}&||\lambda(t) \left(\omega \lambda(t)^{2}\right) \int_{t}^{\infty} dx \cos((t-x)\sqrt{\omega}) \mathcal{F}(\sqrt{\cdot}F_{4}(x,\cdot\lambda(x)))_{1}(\omega \lambda(x)^{2})||_{L^{2}(\rho(\omega \lambda(t)^{2})d\omega)}\\
&\leq \frac{C \lambda(t)^{2}}{t^{3}} \left(\frac{1}{\log^{2b-2\epsilon-1}(t)} + \frac{1}{\log^{3b-2\epsilon-2}(t)} + \frac{1}{\log^{2\epsilon}(t)} + \frac{1}{\log^{b}(t)}\right)\end{split}\end{equation}\end{lemma}
\begin{proof}
We remind the reader of the symbol-type estimates, \eqref{f4symb}, on $F_{4}(t,r)$. We start by estimating $\mathcal{F}(\sqrt{\cdot} F_{4}(x,\cdot \lambda(x)))_{1}(\omega \lambda(x)^{2})$, and use the same procedure we used in \cite{wm}. For completeness, we write out the steps here. We have
\begin{equation}\begin{split}\mathcal{F}(\sqrt{\cdot} F_{4}(x,\cdot \lambda(x)))_{1}(\omega \lambda(x)^{2}) &= \int_{0}^{\infty} \sqrt{R} F_{4}(x,R\lambda(x)) \phi(R, \omega \lambda(x)^{2}) dR \\
&= \frac{1}{\lambda(x)^{3/2}} \int_{0}^{\infty} \sqrt{u} F_{4}(x,u) \phi(\frac{u}{\lambda(x)}, \omega \lambda(x)^{2}) du.\end{split}\end{equation}
Using \eqref{phiseries} and \eqref{philgrasymp} (which are from \cite{kstym}) we divide into two regions: $r^{2} \xi \leq 4$ and $r^{2} \xi > 4$, and get
\begin{equation}\label{fourierest}\begin{split}&\frac{1}{\lambda(x)^{3/2}} \int_{0}^{\infty} \sqrt{u} F_{4}(x,u) \phi(\frac{u}{\lambda(x)}, \omega \lambda(x)^{2}) du \\
&= \frac{1}{\lambda(x)^{3/2}} \int_{0}^{\frac{2}{\sqrt{\omega}}} \sqrt{u} F_{4}(x,u) \widetilde{\phi_{0}}(\frac{u}{\lambda(x)}) du + \frac{1}{\lambda(x)^{3/2}} \int_{0}^{\frac{2}{\sqrt{\omega}}} \sqrt{u} F_{4}(x,u) \sum_{j=1}^{\infty} \frac{\widetilde{\phi_{j}}(\frac{u^{2}}{\lambda(x)^{2}})}{u^{3/2}} \lambda(x)^{3/2} (u^{2}\omega)^{j} du\\
&+ \frac{2}{\lambda(x)^{3/2}} \text{Re}\left(\int_{\frac{2}{\sqrt{\omega}}}^{\infty} \sqrt{u} F_{4}(x,u) a(\omega \lambda(x)^{2}) \psi^{+}(\frac{u}{\lambda(x)},\lambda(x)^{2} \omega) du\right).\end{split}\end{equation}
We use the orthogonality of $F_{4}(x,R\lambda(x))$ to $\phi_{0}(R)$ in $L^{2}(R dR)$, as part of estimating 
$$I:= \frac{1}{\lambda(x)^{3/2}} \int_{0}^{\frac{2}{\sqrt{\omega}}} \sqrt{u} F_{4}(x,u) \widetilde{\phi_{0}}(\frac{u}{\lambda(x)}) du = \frac{1}{\lambda(x)^{2}} \int_{0}^{\frac{2}{\sqrt{\omega}}} F_{4}(x,u)u \phi_{0}(\frac{u}{\lambda(x)}) du.$$
We split into 4 cases, depending on the range of $\omega$.\\
\\
\textbf{Case 1}: $\omega < \frac{16}{x^{2}}$. In this case, we use the orthogonality to get 
$$I = \frac{1}{\lambda(x)^{2}} \int_{0}^{\frac{2}{\sqrt{\omega}}} F_{4}(x,u)u \phi_{0}(\frac{u}{\lambda(x)}) du = -\frac{1}{\lambda(x)^{2}} \int_{\frac{2}{\sqrt{\omega}}}^{\infty} F_{4}(x,u) u \phi_{0}(\frac{u}{\lambda(x)}) du =0$$
where we used the support conditions on $\chi_{\geq 1}(x)$, and the fact that $\frac{t}{g(t)} > 1600$.\\
\\
\textbf{Case 2}: $\frac{16}{x^{2}} < \omega \leq \frac{4}{g(x)^{2}}$. In this case, we again use the orthogonality to get
\begin{equation}\begin{split}|\frac{1}{\lambda(x)^{2}} \int_{0}^{\frac{2}{\sqrt{\omega}}} F_{4}(x,u)u \phi_{0}(\frac{u}{\lambda(x)}) du| &= |-\frac{1}{\lambda(x)^{2}} \int_{\frac{2}{\sqrt{\omega}}}^{\infty} F_{4}(x,u) \phi_{0}(\frac{u}{\lambda(x)}) u du| \\
&\leq \frac{C \lambda(x)^{4} \log(x)}{x^{2} \log^{b}(x)} \omega^{2} \left(\log(2+\frac{2}{\sqrt{\omega} g(x)}) + \frac{\log(x)}{\log^{b}(x)}\right).\end{split}\end{equation}
\textbf{Case 3}: $\frac{4}{g(x)^{2}} \leq \omega \leq \frac{4}{\lambda(x)^{2}}$. Using the orthogonality, we get
\begin{equation}\nonumber\begin{split}|\frac{1}{\lambda(x)^{2}} \int_{0}^{\frac{2}{\sqrt{\omega}}} F_{4}(x,u)u \phi_{0}(\frac{u}{\lambda(x)}) du| &= |-\frac{1}{\lambda(x)^{2}} \int_{\frac{2}{\sqrt{\omega}}}^{\infty} F_{4}(x,u) \phi_{0}(\frac{u}{\lambda(x)}) u du| \\
&\leq \frac{C \lambda(x)^{4} \log(x)\left(1+\frac{\log(x)}{\log^{b}(x)}\right)}{x^{2} \log^{b}(x) g(x)^{4}} + \frac{C \omega \lambda(x)^{2}}{x^{2}\log^{b}(x)}.\end{split}\end{equation}
\textbf{Case 4}: $\frac{4}{\lambda(x)^{2}} \leq \omega$. Here, we do not need to use the orthogonality, and simply directly estimate
$$|\frac{1}{\lambda(x)^{2}} \int_{0}^{\frac{2}{\sqrt{\omega}}} F_{4}(x,u)u \phi_{0}(\frac{u}{\lambda(x)}) du| \leq \frac{C}{\omega^{3} \lambda(x)^{6} x^{2} \log^{b}(x)}.$$
To estimate the other two integrals of \eqref{fourierest}, we treat the same 4 cases considered while estimating $I$, namely \textbf{1}. $\omega < \frac{16}{x^{2}}$, \textbf{2}. $\frac{16}{x^{2}} < \omega \leq \frac{4}{g(x)^{2}}$, \textbf{3}. $\frac{4}{g(x)^{2}} \leq \omega \leq \frac{4}{\lambda(x)^{2}}$, and \textbf{4}. $\frac{4}{\lambda(x)^{2}} \leq \omega$. For each of the other two integrals of \eqref{fourierest}, and in each of the aforementioned cases, we directly  insert \eqref{phiseries} or \eqref{philgrasymp}, as appropriate, and estimate the integral, using \eqref{f4symb}. 
In total, we get
\begin{equation}\label{f4overomega}\begin{split}&\lambda(x) ||\frac{\mathcal{F}(\sqrt{\cdot} F_{4}(x,\cdot \lambda(x)))_{1}(\omega \lambda(x)^{2})}{\omega}||_{L^{2}(\rho(\omega \lambda(x)^{2}) d\omega)} \\
&\leq \frac{C \lambda(x)^{2}}{x^{2} \log^{b}(x)} + \frac{C \lambda(x)^{2}}{x^{2} \log^{2\epsilon}(x)} + \frac{C \lambda(x)^{2}}{x^{2} \log^{2b-2\epsilon-1}(x)} + \frac{C \lambda(x)^{2}}{x^{2} \log^{3b-2-2\epsilon}(x)}\end{split}\end{equation}
which will be used later on. The first integral involving $\mathcal{F}(\sqrt{\cdot} F_{4}(x,\cdot \lambda(x)))_{1}(\omega \lambda(x)^{2})$ that we need to estimate is
\begin{equation}\begin{split}&\lambda(t)\left(\omega \lambda(t)^{2}\right)^{3/2} \int_{t}^{\infty} dx \frac{\sin((t-x)\sqrt{\omega})}{\sqrt{\omega}} \mathcal{F}(\sqrt{\cdot} F_{4}(x,\cdot \lambda(x)))_{1}(\omega \lambda(x)^{2})\\
&=-\lambda(t)^{4} \mathcal{F}(\sqrt{\cdot} F_{4}(t,\cdot \lambda(t)))_{1}(\omega \lambda(t)^{2}) \sqrt{\omega} - \lambda(t)^{4} \int_{t}^{\infty} \frac{\cos((t-x)\sqrt{\omega})}{\sqrt{\omega}} \omega \partial_{x}\left(\mathcal{F}(\sqrt{\cdot} F_{4}(x,\cdot \lambda(x)))_{1}(\omega \lambda(x)^{2})\right) dx.\end{split}\end{equation}
Using \eqref{ltrans}, we have $||\lambda(t)^{4} \mathcal{F}(\sqrt{\cdot} F_{4}(t,\cdot \lambda(t)))_{1}(\omega \lambda(t)^{2}) \sqrt{\omega}||_{L^{2}(\rho(\omega \lambda(t)^{2})d\omega)} = \lambda(t)^{2} ||L(F_{4}(t,\cdot \lambda(t)))||_{L^{2}(R dR)}.$ Then, we use the symbol-type estimates on $F_{4}$, namely \eqref{f4symb}, to get
$$||L(F_{4}(t,\cdot\lambda(t)))||_{L^{2}(R dR)} \leq \frac{C}{t^{2} \log^{b}(t)}$$
which gives
$$||\lambda(t)^{4} \mathcal{F}(\sqrt{\cdot} F_{4}(t,\cdot \lambda(t)))_{1}(\omega \lambda(t)^{2}) \sqrt{\omega}||_{L^{2}(\rho(\omega \lambda(t)^{2})d\omega)} \leq \frac{C \lambda(t)^{2}}{t^{2} \log^{b}(t)}.$$
Next, we note that 
$$\partial_{x}\left(\mathcal{F}(\sqrt{\cdot} F_{4}(x,\cdot \lambda(x)))_{1}(\omega \lambda(x)^{2})\right) = \mathcal{F}(\sqrt{\cdot} \partial_{x}\left(F_{4}(x,\cdot\lambda(x))\right))_{1}(\omega \lambda(x)^{2}) + \mathcal{F}(\sqrt{\cdot} F_{4}(x,\cdot \lambda(x)))_{1}'(\omega \lambda(x)^{2}) 2 \lambda(x) \lambda'(x) \omega.$$
Then, we recall the transference identity of \cite{kstym}, which says
$$\mathcal{F}(R \partial_{R}u)(\xi) = \begin{bmatrix} 0&
-2 \xi \partial_{\xi} \mathcal{F}(u)_{1}\end{bmatrix}^{\text{T}} + \mathcal{K}(\mathcal{F}(u))(\xi)$$
where we write
$$\mathcal{F}(u)=\begin{bmatrix} \mathcal{F}(u)_{0}&
\mathcal{F}(u)_{1}\end{bmatrix}^{\text{T}}.$$
This gives \begin{equation}\begin{split}&||2 \lambda(x) \lambda'(x) \omega \mathcal{F}(\sqrt{\cdot} F_{4}(x,\cdot \lambda(x)))_{1}'(\omega \lambda(x)^{2}) \sqrt{\omega}||^{2}_{L^{2}(\rho(\omega \lambda(x)^{2}) d\omega)} \\
&\leq \frac{C \lambda'(x)^{2}}{\lambda(x)^{6}}\left(||\mathcal{F}(R \partial_{R}(\sqrt{R} F_{4}(x,R\lambda(x))))(\xi)||^{2}_{L^{2,\frac{1}{2}}_{\rho}} + ||\mathcal{K}(\mathcal{F}(\sqrt{\cdot} F_{4}(x,\cdot \lambda(x))))(\xi)||^{2}_{L^{2,\frac{1}{2}}_{\rho}}\right).\end{split}\end{equation}
Then, we estimate
$$|\langle R \partial_{R}\left(\sqrt{R} F_{4}(x,R\lambda(x))\right),\widetilde{\phi_{0}}\rangle_{L^{2}(dR)}| \leq \frac{C}{x^{2} \log^{b}(x)}.$$
Next, we again use \eqref{ltrans} to get
\begin{equation}\begin{split}&||\mathcal{F}(R \partial_{R}(\sqrt{R} F_{4}(x,R\lambda(x))))_{1}(\xi) \langle \xi \rangle^{1/2}||^{2}_{L^{2}(\rho(\xi)) d\xi)} \\
&\leq C ||L(\sqrt{R} \partial_{R}(\sqrt{R}F_{4}(x,R\lambda(x))))||^{2}_{L^{2}(R dR)} + C ||R \partial_{R}(\sqrt{R} F_{4}(x,R\lambda(x)))||^{2}_{L^{2}(dR)}\end{split}\end{equation}
Using \eqref{f4symb}, we get
$$||L(\sqrt{R} \partial_{R}(\sqrt{R} F_{4}(x,R\lambda(x))))||^{2}_{L^{2}(R dR)} \leq \frac{C}{x^{4} \log^{2b}(x)}$$
$$||R \partial_{R}(\sqrt{R} F_{4}(x,R\lambda(x)))||_{L^{2}(dR)} \leq \frac{C}{x^{2} \log^{b}(x)}$$
which gives
$$||\mathcal{F}(R \partial_{R}(\sqrt{R} F_{4}(x,R\lambda(x))))_{1}(\xi) \langle \xi \rangle^{1/2}||^{2}_{L^{2}(\rho(\xi) d\xi)} \leq \frac{C}{x^{4} \log^{2b}(x)}.$$
Similarly
\begin{equation}\begin{split}||\mathcal{F}(\sqrt{\cdot} F_{4}(x,\cdot \lambda(x)))||^{2}_{L^{2,\frac{1}{2}}_{\rho}} &\leq C \int_{0}^{\infty} R |F_{4}(x,R\lambda(x))|^{2} dR + C \int_{0}^{\infty} |L(F_{4}(x,R\lambda(x)))|^{2} R dR\\
&\leq \frac{C}{x^{4} \log^{2b}(x)}.\end{split}\end{equation}
This, combined with Proposition 5.2 of \cite{kstym} (which gives the boundedness of $\mathcal{K}$ on $L^{2,\alpha}_{\rho}$, for $\alpha \in \mathbb{R}$) finally gives
$$||2 \lambda(x) \lambda'(x) \omega \mathcal{F}(\sqrt{\cdot} F_{4}(x,\cdot \lambda(x)))_{1}'(\omega \lambda(x)^{2}) \sqrt{\omega}||^{2}_{L^{2}(\rho(\omega \lambda(x)^{2}) d\omega)} \leq \frac{C \lambda'(x)^{2}}{x^{4} \lambda(x)^{6} \log^{2b}(x)}.$$
We again use \eqref{ltrans}, and \eqref{f4symb} to get
\begin{equation}\begin{split}||\sqrt{\omega} \mathcal{F}(\sqrt{\cdot} \partial_{x}\left(F_{4}(x,\cdot \lambda(x))\right))_{1}(\omega \lambda(x)^{2})||_{L^{2}(\rho(\omega \lambda(x)^{2})d\omega)} &\leq \frac{C}{\lambda(x)^{2}} ||L(\partial_{x}(F_{4}(x,R\lambda(x))))||_{L^{2}(R dR)} \\
&\leq \frac{C}{\lambda(x)^{2} x^{3} \log^{b}(x)}.\end{split}\end{equation}
So, using Minkowski's inequality and \eqref{rhoscaling}, we get
\begin{equation}\begin{split}&||\lambda(t)^{4} \int_{t}^{\infty} dx \cos((t-x)\sqrt{\omega}) \sqrt{\omega} \partial_{x}\left(\mathcal{F}(\sqrt{\cdot} F_{4}(x,\cdot\lambda(x)))_{1}(\omega \lambda(x)^{2})\right)||_{L^{2}(\rho(\omega \lambda(t)^{2}) d\omega)} \\
&\leq C \lambda(t)^{4} \int_{t}^{\infty} dx \left(1+\frac{\lambda(t)^{2}}{\lambda(x)^{2}}\right) \frac{1}{\lambda(x)^{2} x^{3} \log^{b}(x)} \leq \frac{C \lambda(t)^{2}}{t^{2} \log^{b}(t)}\end{split}\end{equation}
where we obtained the last inequality using the same procedure that we used to treat the integral \eqref{i2b}. In total, we then obtain
$$||\lambda(t) \left(\omega \lambda(t)^{2}\right)^{3/2} \int_{t}^{\infty} dx \frac{\sin((t-x)\sqrt{\omega})}{\sqrt{\omega}} \mathcal{F}(\sqrt{\cdot} F_{4}(x,\cdot \lambda(x))_{1}(\omega \lambda(x)^{2})||_{L^{2}(\rho(\omega \lambda(t)^{2})d\omega)} \leq \frac{C \lambda(t)^{2}}{t^{2} \log^{b}(t)}.$$
The next integral to estimate is
\begin{equation}\label{f4int2}\begin{split}&\lambda(t) \int_{t}^{\infty} \frac{\sin((t-x)\sqrt{\omega})}{\sqrt{\omega}} \mathcal{F}(\sqrt{\cdot} F_{4}(x,\cdot \lambda(x)))_{1}(\omega \lambda(x)^{2}) dx\\
&= - \lambda(t)\frac{\mathcal{F}(\sqrt{\cdot} F_{4}(t,\cdot \lambda(t)))_{1}(\omega \lambda(t)^{2})}{\omega}-\lambda(t) \int_{t}^{\infty} \frac{\cos((t-x)\sqrt{\omega})}{\omega} \partial_{x}\left(\mathcal{F}(\sqrt{\cdot} F_{4}(x,\cdot \lambda(x))_{1}(\omega \lambda(x)^{2})\right) dx.\end{split}\end{equation} 
We recall
$$\partial_{x}\left(\mathcal{F}(\sqrt{\cdot} F_{4}(x,\cdot \lambda(x)))_{1}(\omega \lambda(x)^{2})\right) = \mathcal{F}(\sqrt{\cdot} \partial_{x}\left(F_{4}(x,\cdot\lambda(x))\right))_{1}(\omega \lambda(x)^{2}) + \mathcal{F}(\sqrt{\cdot} F_{4}(x,\cdot \lambda(x)))_{1}'(\omega \lambda(x)^{2}) 2 \lambda(x) \lambda'(x) \omega.$$
This time, however, since the integrand of the second line of \eqref{f4int2} has a factor of $\frac{1}{\omega}$, we will directly estimate $\mathcal{F}(\sqrt{\cdot}F_{4}(x,\cdot\lambda(x)))_{1}'(\xi)$, rather than using an argument based on the transference identity. We have
$$\mathcal{F}(\sqrt{\cdot} F_{4}(x,\cdot \lambda(x)))_{1}'(\omega \lambda(x)^{2}) = IV + V$$ 
where
$$IV = \int_{0}^{\frac{2}{\sqrt{\omega}}} \partial_{2}\phi(\frac{u}{\lambda(x)},\omega \lambda(x)^{2}) \frac{\sqrt{u}}{\lambda(x)^{3/2}} F_{4}(x,u) du$$
$$V = \int_{\frac{2}{\sqrt{\omega}}}^{\infty} \partial_{2}\phi(\frac{u}{\lambda(x)},\omega \lambda(x)^{2}) \frac{\sqrt{u}}{\lambda(x)^{3/2}} F_{4}(x,u) du.$$
Then, we get
$$|IV| \leq \begin{cases} \frac{C g(x)^{2}}{\lambda(x)^{2} x^{2} \log^{b}(x)} + \frac{C \log^{3}(x)}{x^{2} \log^{b}(x)}, \quad \omega < \frac{16}{x^{2}}\\
\frac{C g(x)^{2}}{\lambda(x)^{2} x^{2} \log^{b}(x)} + \frac{C \log(x) \log(2+\frac{2}{\sqrt{\omega} g(x)})}{x^{2} \log^{b}(x)}\left(\log(2+\frac{2}{\sqrt{\omega}g(x)})+\frac{\log(x)}{\log^{b}(x)}\right), \quad \frac{16}{x^{2}} < \omega \leq \frac{4}{g(x)^{2}}\\
\frac{C}{\omega \lambda(x)^{2} x^{2} \log^{b}(x)}, \quad \frac{4}{g(x)^{2}} < \omega \leq \frac{4}{\lambda(x)^{2}}\\
\frac{C}{x^{2} \log^{b}(x) \lambda(x)^{8} \omega^{4}}, \quad \frac{4}{\lambda(x)^{2}} < \omega \end{cases}$$
$$|V| \leq \begin{cases}0, \quad \omega \leq \frac{16}{x^{2}}\\
\frac{C |a(\omega \lambda(x)^{2})| \log(x)\left(\log(2+\frac{2}{\sqrt{\omega} g(x)})+\frac{\log(x)}{\log^{b}(x)}\right)}{x^{2} \log^{b}(x)}, \quad \frac{16}{x^{2}} < \omega < \frac{4}{g(x)^{2}}\\
\frac{C |a(\omega \lambda(x)^{2})| \sqrt{g(x)}}{\omega^{3/4} \lambda(x)^{2} x^{2} \log^{b}(x)}, \quad \frac{4}{g(x)^{2}} < \omega \leq \frac{4}{\lambda(x)^{2}}\\
\frac{C |a(\omega \lambda(x)^{2})| \sqrt{g(x)}}{\lambda(x)^{2} \omega^{3/4} x^{2} \log^{b}(x)}, \quad \frac{4}{\lambda(x)^{2}} < \omega \end{cases}.$$
In order to estimate 
$$||\frac{\mathcal{F}(\sqrt{\cdot} \partial_{x}(F_{4}(x,\cdot \lambda(x))))_{1}(\omega \lambda(x)^{2})}{\omega}||_{L^{2}(\rho(\omega \lambda(x)^{2}) d\omega)}$$
we note that $\partial_{x}(F_{4}(x,R\lambda(x)))$ is still orthogonal to $\phi_{0}(R)$ in $L^{2}(R dR)$, and we recall our symbol type estimates on $F_{4}$, namely \eqref{f4symb}. These two observations, along with an inspection of the procedure used to obtain \eqref{f4overomega}
give
\begin{equation}\begin{split}&||\frac{\mathcal{F}(\sqrt{\cdot} \partial_{x}(F_{4}(x,\cdot \lambda(x))))_{1}(\omega \lambda(x)^{2})}{\omega}||_{L^{2}(\rho(\omega \lambda(x)^{2}) d\omega)} \\
&\leq \frac{C \lambda(x)}{x^{3}} \left(\frac{1}{\log^{2b-2\epsilon-1}(x)} + \frac{1}{\log^{3b-2-2\epsilon}(x)} + \frac{1}{\log^{2\epsilon}(x)} + \frac{1}{\log^{b}(x)}\right).\end{split}\end{equation}
Therefore,
\begin{equation}\begin{split}&||\lambda(t) \int_{t}^{\infty} \frac{\cos((t-x)\sqrt{\omega})}{\omega} \partial_{x}\left(\mathcal{F}(\sqrt{\cdot} F_{4}(x,\cdot \lambda(x)))_{1}(\omega \lambda(x)^{2})\right) dx ||_{L^{2}(\rho(\omega \lambda(t)^{2})d\omega)}\\
&\leq \frac{C \lambda(t)^{2}}{t^{2}} \left(\frac{1}{\log^{2b-2\epsilon-1}(t)} + \frac{1}{\log^{3b-2-2\epsilon}(t)} + \frac{1}{\log^{2\epsilon}(t)} + \frac{1}{\log^{b}(t)}\right)\end{split}\end{equation}
which, when combined with \eqref{f4overomega}, gives
\begin{equation}\begin{split}&||\lambda(t) \int_{t}^{\infty} \frac{\sin((t-x)\sqrt{\omega})}{\sqrt{\omega}} \mathcal{F}(\sqrt{\cdot} F_{4}(x,\cdot \lambda(x)))_{1}(\omega \lambda(x)^{2}) dx ||_{L^{2}(\rho(\omega \lambda(t)^{2})d\omega)}\\
&\leq \frac{C \lambda(t)^{2}}{t^{2}} \left(\frac{1}{\log^{2b-2\epsilon-1}(t)} + \frac{1}{\log^{3b-2-2\epsilon}(t)} + \frac{1}{\log^{2\epsilon}(t)} + \frac{1}{\log^{b}(t)}\right).\end{split}\end{equation}
The next integral to estimate is
\begin{equation}\begin{split}&\lambda(t) \int_{t}^{\infty} dx \cos((t-x)\sqrt{\omega}) \mathcal{F}(\sqrt{\cdot} F_{4}(x,\cdot \lambda(x))_{1}(\omega \lambda(x)^{2})\\
&=-\lambda(t) \frac{\partial_{t}\left(\mathcal{F}(\sqrt{\cdot} F_{4}(t,\cdot \lambda(t)))_{1}(\omega \lambda(t)^{2})\right)}{\omega} -\lambda(t) \int_{t}^{\infty} \frac{\cos((t-x)\sqrt{\omega})}{\omega} \partial_{x}^{2}\left(\mathcal{F}(\sqrt{\cdot} F_{4}(x,\cdot \lambda(x)))_{1}(\omega \lambda(x)^{2})\right) dx.\end{split}\end{equation}
First, we will estimate $\mathcal{F}(\sqrt{\cdot} F_{4}(x,\cdot \lambda(x)))_{1}''(\omega \lambda(x)^{2})$, which is one of the terms arising in the expression $\partial_{x}^{2}\left(\mathcal{F}(\sqrt{\cdot} F_{4}(x,\cdot \lambda(x)))_{1}(\omega \lambda(x)^{2})\right)$. We have
$$\mathcal{F}(\sqrt{\cdot} F_{4}(x,\cdot \lambda(x)))_{1}''(\omega \lambda(x)^{2}) = VI+VII$$
where
$$VI = \int_{0}^{\frac{2}{\sqrt{\omega} \lambda(x)}} \partial_{2}^{2} \phi(R,\omega \lambda(x)^{2}) \sqrt{R} F_{4}(x,R\lambda(x)) dR$$
\begin{equation}\label{VIIdef}VII = \int_{\frac{2}{\sqrt{\omega} \lambda(x)}}^{\infty} \partial_{2}^{2}\phi(R,\omega \lambda(x)^{2}) \sqrt{R} F_{4}(x,R\lambda(x)) dR.\end{equation}
We estimate $VI$ directly, using \eqref{f4symb}, and \eqref{phiseries} (which is from \cite{kstym}). This leads to
$$\omega |VI| \leq \begin{cases} \frac{C \log^{2}(x)}{x^{2} \log^{b}(x) \lambda(x)^{2}}, \quad \omega < \frac{16}{x^{2}}\\
\frac{C g(x)^{2}}{x^{2} \log^{b}(x) \lambda(x)^{4}} \left(1+ \frac{\lambda(x)^{2}}{g(x)^{2}} \log(x) \log(2+\frac{2}{\sqrt{\omega}g(x)})\right), \quad \frac{16}{x^{2}} < \omega < \frac{4}{g(x)^{2}}\\
\frac{C}{\lambda(x)^{2} x^{2} \log^{b}(x)} \left(1+\frac{1}{\omega \lambda(x)^{2}}\right), \quad \frac{4}{g(x)^{2}} < \omega < \frac{4}{\lambda(x)^{2}}\\
\frac{C}{\lambda(x)^{10} \omega^{4} x^{2} \log^{b}(x)}, \quad \frac{4}{\lambda(x)^{2}} < \omega\end{cases}.$$
On the other hand, to estimate $VII$, we will need to integrate by parts in $R$, when $\omega > \frac{4}{\lambda(x)^{2}}$. In particular, we have
$$\omega |VII| =0, \quad \omega \leq \frac{16}{x^{2}}$$
$$\omega |VII| \leq \frac{C |a(\omega \lambda(x)^{2})| \log(x)}{\lambda(x)^{2} x^{2} \log^{b}(x)} \left(\log(2+\frac{2}{\sqrt{\omega}g(x)}) + \frac{\log(x)}{\log^{b}(x)}\right), \quad \frac{16}{x^{2}} < \omega < \frac{4}{g(x)^{2}}$$
$$\omega |VII| \leq \frac{C |a(\omega \lambda(x)^{2})| g(x)^{3/2}}{\omega^{1/4} \lambda(x)^{4} x^{2} \log^{b}(x)}, \quad \frac{4}{g(x)^{2}} < \omega < \frac{4}{\lambda(x)^{2}}.$$
To estimate $VII$ in the region $\omega > \frac{4}{\lambda(x)^{2}}$, we first use \eqref{philgrasymp}(which follows from Lemma 4.7 of \cite{kstym}) to get, for $R^{2} \xi > 4$, $\phi(R,\xi) = 2\text{Re}\left(a(\xi) \psi^{+}(R,\xi)\right).$ Using \eqref{psiplusdef} and \eqref{asymbol}, we get
$$\partial_{2}^{2}\phi(R,\xi) = 2 \text{Re}\left(\frac{a(\xi) \cdot -R^{2}}{4 \xi^{5/4}} e^{i R \sqrt{\xi}} \sigma(R\sqrt{\xi},R)\right) + \text{Err}$$
where
$$|Err| \leq C \frac{|a(\xi)|}{\xi^{9/4}} + C \frac{|a(\xi)| R}{\xi^{7/4}}, \quad R^{2} \xi \geq 4.$$
Then, we integrate by parts in $R$ for the term 
$$\omega \int_{\frac{2}{\sqrt{\omega} \lambda(x)}}^{\infty} -\text{Re}\left(\frac{a(\omega \lambda(x)^{2})}{2} \frac{R^{2}}{\omega^{5/4} \lambda(x)^{5/2}} e^{i R \lambda(x) \sqrt{\omega}} \sigma(R \sqrt{\omega} \lambda(x),R)\right) \sqrt{R} F_{4}(x,R\lambda(x)) dR$$
which arises as part of $w\cdot VII$ (recall \eqref{VIIdef}). This leads to
$$||\omega VII||^{2}_{L^{2}(\rho(\omega \lambda(x)^{2}) d\omega)} \leq \frac{C g(x)^{3}}{\lambda(x)^{9} x^{4} \log^{2b}(x)}$$
which then leads to
$$||-\lambda(t) \int_{t}^{\infty} \frac{\cos((t-x)\sqrt{\omega})}{\omega} \cdot 4 \lambda(x)^{2} \lambda'(x)^{2} \omega^{2} \mathcal{F}(\sqrt{\cdot}F_{4}(x,\cdot \lambda(x)))_{1}''(\omega\lambda(x)^{2}) dx||_{L^{2}(\rho(\omega \lambda(t)^{2}) d\omega)} \leq \frac{C \lambda(t)^{2}}{t^{3} \log^{\frac{3b}{2}+\epsilon}(t)}.$$
The rest of the terms arising in the expression $\partial_{x}^{2}\left(\mathcal{F}(\sqrt{\cdot} F_{4}(x,\cdot \lambda(x)))_{1}(\omega \lambda(x)^{2})\right)$ can be treated by using the symbol-type nature of the estimates \eqref{f4symb}, along with the fact that $\partial_{x}^{2}(F_{4}(x,R \lambda(x)))$ is still orthogonal to $\phi_{0}(R)$ in $L^{2}(R dR)$. This observation leads to
\begin{equation}\begin{split}&||-\lambda(t) \int_{t}^{\infty} dx \frac{\cos((t-x)\sqrt{\omega})}{\omega} \mathcal{F}(\sqrt{\cdot} \partial_{x}^{2}\left(F_{4}(x,\cdot \lambda(x))\right))_{1}(\omega \lambda(x)^{2})||_{L^{2}(\rho(\omega \lambda(t)^{2})d\omega)}\\
&\leq C \frac{\lambda(t)^{2}}{t^{3}}\left(\frac{1}{\log^{b}(t)} + \frac{1}{\log^{2\epsilon}(t)} + \frac{1}{\log^{2b-2\epsilon-1}(t)} + \frac{1}{\log^{3b-2-2\epsilon}(t)}\right)\end{split}\end{equation}
and this finally gives  
\begin{equation}\begin{split} &||\lambda(t) \int_{t}^{\infty} dx \cos((t-x)\sqrt{\omega}) \mathcal{F}(\sqrt{\cdot} F_{4}(x,\cdot \lambda(x)))_{1}(\omega \lambda(x)^{2})||_{L^{2}(\rho(\omega \lambda(t)^{2})d\omega)}\\
&\leq C \frac{\lambda(t)^{2}}{t^{3}}\left(\frac{1}{\log^{b}(t)} + \frac{1}{\log^{2\epsilon}(t)} + \frac{1}{\log^{2b-2\epsilon-1}(t)} + \frac{1}{\log^{3b-2-2\epsilon}(t)}\right).\end{split}\end{equation}
The final integral to treat in this section is
\begin{equation}\begin{split}&\lambda(t) \cdot \omega \lambda(t)^{2} \cdot \int_{t}^{\infty} dx \cos((t-x)\sqrt{\omega}) \mathcal{F}(\sqrt{\cdot} F_{4}(x,\cdot \lambda(x)))_{1}(\omega \lambda(x)^{2})\\
&=-\lambda(t) \lambda(t)^{2} \partial_{t}\left(\mathcal{F}(\sqrt{\cdot} F_{4}(t,\cdot \lambda(t)))_{1}(\omega \lambda(t)^{2})\right)\\
&-\lambda(t) \cdot \int_{t}^{\infty} \lambda(t)^{2} \cos((t-x)\sqrt{\omega}) \partial_{x}^{2}\left(\mathcal{F}(\sqrt{\cdot} F_{4}(x,\cdot \lambda(x)))_{1}(\omega \lambda(x)^{2})\right)dx.\end{split}\end{equation}
Here, most terms are treated exactly as previously. We will write out in detail how to estimate the term
$\lambda(x)^{4} \omega^{2} \mathcal{F}(\sqrt{\cdot} F_{4}(x,\cdot \lambda(x)))_{1}''(\omega \lambda(x)^{2})$, which is part of $\partial_{x}^{2}\left(\mathcal{F}(\sqrt{\cdot} F_{4}(x,\cdot \lambda(x)))_{1}(\omega \lambda(x)^{2}\right)$, and for which we use a slightly different argument than previously. We start with
\begin{equation}\begin{split} \lambda(x)^{4} \omega^{2} \mathcal{F}(\sqrt{\cdot} F_{4}(x,\cdot \lambda(x)))_{1}''(\omega \lambda(x)^{2}) &= (\xi \partial_{\xi}) \circ (\xi \partial_{\xi})\left(\mathcal{F}(\sqrt{\cdot} F_{4}(x,\cdot \lambda(x)))_{1}(\xi)\right)\Bigr|_{\xi = \omega \lambda(x)^{2}}\\
& -\xi\partial_{\xi}\left(\mathcal{F}(\sqrt{\cdot} F_{4}(x,\cdot \lambda(x)))_{1}(\xi)\right)\Bigr|_{\xi = \omega \lambda(x)^{2}}.\end{split}\end{equation}
Then, we use the transference identity, which leads to
\begin{equation}\begin{split} &\lambda(x)^{4} \omega^{2} \mathcal{F}(\sqrt{\cdot} F_{4}(x,\cdot \lambda(x)))_{1}''(\omega \lambda(x)^{2})\\
&=\frac{1}{4}\left(\mathcal{F}(R \partial_{R}(R \partial_{R}(\sqrt{R} F_{4}(x,R\lambda(x)))))_{1}(\omega \lambda(x)^{2}) - \pi_{1}\circ\mathcal{K}(\mathcal{F}(R \partial_{R}(\sqrt{R}F_{4}(x,R\lambda(x))))(\omega \lambda(x)^{2})\right.\\
&+\left.2\pi_{1}\circ[\xi \partial_{\xi},\mathcal{K}] \mathcal{F}(\sqrt{\cdot}F_{4}(x,\cdot\lambda(x)))(\omega \lambda(x)^{2})\right. \\
&-\left. \pi_{1}\left( \mathcal{K}\left(\begin{bmatrix}0\\
\mathcal{F}(R \partial_{R}(\sqrt{R} F_{4}(x,R\lambda(x))))_{1} - \pi_{1}\circ \mathcal{K}(\mathcal{F}(\sqrt{\cdot} F_{4}(x,\cdot \lambda(x))))\end{bmatrix}\right)\right)\right)(\omega \lambda(x)^{2})\\
&+\frac{1}{2} \left(\mathcal{F}(R \partial_{R}(\sqrt{R} F_{4}(x,R\lambda(x)))_{1}(\omega \lambda(x)^{2}) - \pi_{1}\circ \mathcal{K}(\mathcal{F}(\sqrt{\cdot}F_{4}(x,\cdot\lambda(x))))(\omega \lambda(x)^{2})\right)\end{split}\end{equation}
where $\pi_{1}(\begin{bmatrix} v_{0}\\
v_{1}\end{bmatrix}) = v_{1}$. Using \eqref{f4symb}, and the boundedness of $\mathcal{K}$ and $[\mathcal{K},\xi\partial_{\xi}]$ on $L^{2,\alpha}_{\rho}$ (as per Proposition 5.2 of \cite{kstym}) we get
$$||\lambda(x)^{4} \omega^{2} \mathcal{F}(\sqrt{\cdot} F_{4}(x,\cdot \lambda(x))_{1}''(\omega \lambda(x)^{2})||_{L^{2}(\rho(\omega \lambda(x)^{2})d\omega)} \leq \frac{C}{\lambda(x) x^{2} \log^{b}(x)}. $$
We use the same procedure as before to estimate the other terms arising from $\partial_{x}^{2}\left(\mathcal{F}(\sqrt{\cdot} F_{4}(x,\cdot \lambda(x)))_{1}(\omega \lambda(x)^{2}\right)$. Finally, we note that
\begin{equation}\begin{split}&|\lambda(t)^{2} \partial_{t}(\mathcal{F}(\sqrt{\cdot} F_{4}(t,\cdot \lambda(t)))_{1})(\omega\lambda(t)^{2})| \\
&\leq C |\frac{\partial_{t}(\mathcal{F}(\sqrt{\cdot} F_{4}(t,\cdot \lambda(t)))_{1}(\omega \lambda(t)^{2}))}{\omega}| + C \lambda(t)^{3} \sqrt{\omega} |\partial_{t}(\mathcal{F}(\sqrt{\cdot} F_{4}(t,\cdot \lambda(t)))(\omega\lambda(t)^{2}))|\end{split}\end{equation}
and both terms on the right-hand side of the above inequality have been previously estimated. In total, we then get
\begin{equation}\begin{split}&||\lambda(t) \left(\omega \lambda(t)^{2}\right) \int_{t}^{\infty} dx \cos((t-x)\sqrt{\omega}) \mathcal{F}(\sqrt{\cdot}F_{4}(x,\cdot\lambda(x)))_{1}(\omega \lambda(x)^{2})||_{L^{2}(\rho(\omega \lambda(t)^{2})d\omega)}\\
&\leq \frac{C \lambda(t)^{2}}{t^{3}} \left(\frac{1}{\log^{2b-2\epsilon-1}(t)} + \frac{1}{\log^{3b-2\epsilon-2}(t)} + \frac{1}{\log^{2\epsilon}(t)} + \frac{1}{\log^{b}(t)}\right).\end{split}\end{equation} 
\end{proof}
\subsection{$F_{5}$ Estimates}\label{f5estimatessection}
In this section, we translate our estimates \eqref{f5l2} and \eqref{lstarlf5l2} using \eqref{l2trans} and \eqref{lstarltrans}, which gives
$$|\mathcal{F}(\sqrt{\cdot} F_{5}(x,\cdot\lambda(x)))_{0}| \leq \frac{C \lambda(x)^{3}}{x^{5} \log^{b-2}(x)}$$
$$||\mathcal{F}(\sqrt{\cdot} F_{5}(x,\cdot\lambda(x)))_{1}(\omega\lambda(x)^{2})||_{L^{2}(\rho(\omega\lambda(x)^{2})d\omega)} \leq \frac{C \lambda(x)^{2}}{x^{5} \log^{b-2}(x)}$$
$$||\omega \lambda(x)^{2} \mathcal{F}(\sqrt{\cdot}F_{5}(x,\cdot\lambda(x)))_{1}(\omega \lambda(x)^{2})||_{L^{2}(\rho(\omega \lambda(x)^{2}) d\omega)} \leq \frac{C \lambda(x)^{4} \log^{2}(x)}{g(x)^{2} \log^{b}(x) x^{5}}.$$
\subsection{Setup of the iteration}
Define $T$ on $Z$ by
\begin{equation}\label{finalTdef}T(\begin{bmatrix} y_{0}\\
y_{1}\end{bmatrix})(t,\omega) = \begin{bmatrix} -\int_{t}^{\infty} ds \int_{s}^{\infty} ds_{1} \left(F_{2,0} + \mathcal{F}\left(\sqrt{\cdot}\left(F_{3}+F_{4}+F_{5}\right)\left(s_{1},\cdot \lambda(s_{1})\right)\right)_{0}\right)\\
\int_{t}^{\infty} dx \frac{\sin((t-x)\sqrt{\omega})}{\sqrt{\omega}} \left(F_{2,1}(x,\omega) + \mathcal{F}\left(\sqrt{\cdot}\left(F_{3}+F_{4}+F_{5}\right)\left(x,\cdot \lambda(x)\right)\right)_{1}\left(\omega \lambda(x)^{2}\right)\right)\end{bmatrix}\end{equation}
where we define $F_{2,i}$ by $F_{2} = \begin{bmatrix}F_{2,0}&
F_{2,1}\end{bmatrix}^{\text{T}}$ and $\mathcal{F}\left(\sqrt{\cdot}\left(F_{3}+F_{4}+F_{5}\right)\left(x,\cdot \lambda(x)\right)\right)_{i}$ is defined by
$$\mathcal{F}\left(\sqrt{\cdot}\left(F_{3}+F_{4}+F_{5}\right)\left(x,\cdot \lambda(x)\right)\right)\left(\omega \lambda(x)^{2}\right) = \begin{bmatrix} \mathcal{F}\left(\sqrt{\cdot}\left(F_{3}+F_{4}+F_{5}\right)\left(x,\cdot \lambda(x)\right)\right)_{0}\\
\mathcal{F}\left(\sqrt{\cdot}\left(F_{3}+F_{4}+F_{5}\right)\left(x,\cdot \lambda(x)\right)\right)_{1}\left(\omega \lambda(x)^{2}\right)\end{bmatrix}.$$
Now, we can proceed with estimating $T$ on $\overline{B_{1}(0)} \subset Z$. If $\begin{bmatrix} y_{0}\\
y_{1}\end{bmatrix} \in \overline{B_{1}(0)} \subset Z$, then, we combine our estimates from Section \ref{f2estimatessection} (for $F_{2}$), Proposition \ref{f3estimatesprop} (for $F_{3}$), Lemma \ref{f4lemma} (for $F_{4}$), and Section \ref{f5estimatessection} (for $F_{5}$), along with \eqref{rhoscaling}(to estimate $\frac{\rho(\omega \lambda(t)^{2})}{\rho(\omega \lambda(x)^{2})}$) to get, for a constant $C>0$, \emph{independent} of $T_{0}$:
\begin{equation}\label{Test}||T\left(\begin{bmatrix} y_{0}\\
y_{1}\end{bmatrix}\right)||_{Z} \leq C\left(\frac{1}{\log^{b-\epsilon}(T_{0})} + \frac{1}{\log^{\epsilon}(T_{0})} + \frac{1}{\log^{2b-3\epsilon-1}(T_{0})} + \frac{1}{\log^{3b-2-3\epsilon}(T_{0})}\right), \quad \begin{bmatrix} y_{0}\\
y_{1}\end{bmatrix} \in \overline{B_{1}(0)} \subset Z. \end{equation}
Next, we will prove a Lipschitz estimate on $T$ restricted to $\overline{B_{1}(0)} \subset Z$. For this, it will be useful to use the notation
$$v_{y}(t,R\lambda(t)) = \frac{1}{\sqrt{R}} \mathcal{F}^{-1}\left(\begin{bmatrix} y_{0}(t)&
y_{1}(t,\frac{\cdot}{\lambda(t)^{2}})\end{bmatrix}^{\text{T}}\right)(R)$$
where $$y=\begin{bmatrix} y_{0}(t)&y_{1}(t,\frac{\cdot}{\lambda(t)^{2}})\end{bmatrix}^{\text{T}}.$$
Then, for $y,z \in \overline{B_{1}(0)} \subset Z$, we have
$$F_{3}(v_{y}) -F_{3}(v_{z}) = (v_{y}-v_{z})\left(\frac{6}{r^{2}} (Q_{\frac{1}{\lambda(t)}}+v_{corr}) (v_{y}+v_{z}) + \frac{2}{r^{2}} (v_{y}^{2}+v_{y}v_{z}+v_{z}^{2}) + \frac{6}{r^{2}}\left((v_{corr}+Q_{\frac{1}{\lambda(t)}})^{2}-Q_{\frac{1}{\lambda(t)}}^{2}\right)\right)$$
where, by a slight abuse of notation, we denote by $F_{3}(v_{y})$ the expression \eqref{f3def}, with $v=v_{y}$, and similarly for $F_{3}(v_{z})$. This leads to
$$||\left(F_{3}(v_{y})-F_{3}(v_{z})\right)(x,R \lambda(x))||_{L^{2}(R dR)} \leq \frac{C\lambda(x)^{2}}{x^{4} \log^{\epsilon}(x)} ||\begin{bmatrix} y_{0}-z_{0}\\
y_{1}-z_{1}\end{bmatrix}||_{Z} \left(\frac{1}{\log^{\epsilon}(x)} + \frac{1}{\log^{b}(x)}\right)$$
$$||L^{*}L\left((F_{3}(v_{y})-F_{3}(v_{z}))(t,\cdot \lambda(t))\right)||_{L^{2}(R dR)} \leq \frac{C \lambda(t)^{2}}{t^{4} \log^{\epsilon}(t)} ||\begin{bmatrix} y_{0}-z_{0}\\
y_{1}-z_{1}\end{bmatrix}||_{Z} \left(\frac{1}{\log^{\epsilon}(t)} + \frac{1}{\log^{b}(t)}\right).$$
Since $F_{2}$ depends on $y$ linearly, we get, for some $C>0$, \emph{independent} of $T_{0}$
\begin{equation}\label{Tlip}||T\left(\begin{bmatrix} y_{0}\\
y_{1}\end{bmatrix}\right)-T\left(\begin{bmatrix} z_{0}\\
z_{1}\end{bmatrix}\right)||_{Z} \leq C ||y-z||_{Z}\left(\frac{1}{\log^{\epsilon}(T_{0})} + \frac{1}{\log^{b}(T_{0})}\right), \quad y,z \in \overline{B_{1}(0)} \subset Z.\end{equation}
Combining this with \eqref{Test}, we get that there exists $M>0$ such that, for all $T_{0} > M$, $T$ is a strict contraction on $\overline{B_{1}(0)} \subset Z$. If $T_{0}>M$, then, by Banach's fixed point theorem, $T$ has a fixed point, say $y_{f}=\begin{bmatrix} y_{f,0}\\
y_{f,1}\end{bmatrix} \in \overline{B_{1}(0)}\subset Z$.
\section{Decomposition of the solution as in Theorem \ref{mainthm}}
We define $v_{f}$ by
$$v_{f}(t,r):=\begin{cases}\sqrt{\frac{\lambda(t)}{r}} \mathcal{F}^{-1}\left(\begin{bmatrix} y_{f,0}(t)\\
y_{f,1}(t,\frac{\cdot}{\lambda(t)^{2}})\end{bmatrix}\right)\left(\frac{r}{\lambda(t)}\right), \quad r >0\\
0, \quad r=0\end{cases}$$
and note that $v_{f}(t,\cdot) \in C^{1}([0,\infty))$, by Lemma \ref{ptwselemma}. By the derivation of \eqref{yeqn}, and the regularity of elements in $\overline{B_{1}(0)} \subset Z$, $u(t,r)=Q_{\frac{1}{\lambda(t)}}(r)+v_{c}(t,r)+w_{c}(t,r)+v_{f}(t,r)$ solves \eqref{ym}. It now remains to estimate the energy of $v_{c}-v_{1}+w_{c}+v_{f}$.
For example, for $v_{2}$, we have
$$v_{2}(t,r) = \int_{t}^{\infty} v_{2,s}(t,r) ds$$
where $v_{2,s}$ solves
$$\begin{cases} -\partial_{tt}v_{2,s}+\partial_{rr}v_{2,s}+\frac{1}{r}\partial_{r}v_{2,s}-\frac{4}{r^{2}}v_{2,s}=0\\
v_{2,s}(s,r)=0\\
\partial_{t}v_{2,s}(s,r) = RHS_{2}(s,r)\end{cases}.$$
By using $$\left(\partial_{x}+\frac{2}{x}\right) J_{2}(x) = J_{1}(x)$$
and the representation formula for $v_{2,s}$ using the Hankel transform of order 2, namely
$$v_{2,s}(t,r) = \int_{0}^{\infty} d\xi J_{2}(r\xi) \sin((t-s)\xi) \widehat{RHS_{2}}(s,\xi)$$
we can justify the energy estimate
$$||\partial_{t}v_{2,s}(t,r)||_{L^{2}(r dr)} +||\left(\partial_{r}+\frac{2}{r}\right) v_{2,s}(t,r)||_{L^{2}(r dr)} \leq C ||RHS_{2}(s,r)||_{L^{2}(r dr)}$$
exactly as was done in \cite{wm} for the correction denoted by $v_{4}$ in that work. Then, we have
\begin{equation}\label{energycross}\int_{0}^{\infty} \left(\left(\partial_{r}+\frac{2}{r}\right) v_{2,s}(t,r)\right)^{2} r dr = \int_{0}^{\infty} \left(\partial_{r}v_{2,s}(t,r)\right)^{2} r dr + \int_{0}^{\infty} \frac{4 v_{2,s}^{2}}{r^{2}} r dr + 2 \int_{0}^{\infty} \partial_{r}\left(v_{2,s}^{2}\right) dr.\end{equation}
Even though the pointwise estimates we recorded for $v_{2,s}$ do not imply that $v_{2,s}(t,r)\rightarrow 0, \quad r \rightarrow \infty$, we can prove that $\lim_{r \rightarrow \infty} v_{2,s}(t,r) =0$ by the Dominated convergence theorem, applied to the spherical means formula for $v_{2,s}$ (for instance, the analog of the formula \eqref{dtv2noderiv}). Then, the last integral in the expression above is zero, and we get 
$$||\partial_{t}v_{2,s}(t,r)||_{L^{2}(r dr)} +||v_{2,s}(t)||_{\dot{H}^{1}_{e}} \leq C ||RHS_{2}(s,r)||_{L^{2}(r dr)}.$$
Using Minkowski's inequality, we then get
$$||\partial_{t}v_{2}(t,r)||_{L^{2}(r dr)} + ||v_{2}(t)||_{\dot{H}^{1}_{e}} \leq C \int_{t}^{\infty} ds ||RHS_{2}(s,r)||_{L^{2}(r dr)}.$$
This same procedure can be applied for $v_{k}$ and all $w_{k}$. We recall that $v_{c}-v_{1}=\sum_{k=2}^{\infty} v_{k}$. We then get
$$||\partial_{t}\left(v_{c}-v_{1}\right)(t,r)||_{L^{2}(r dr)} + ||\left(v_{c}-v_{1}\right)(t)||_{\dot{H}^{1}_{e}} \leq \frac{C}{\log^{2b-1}(t)}$$
and
$$||\partial_{t}w_{c}(t,r)||_{L^{2}(r dr)} + ||w_{c}(t)||_{\dot{H}^{1}_{e}} \leq \frac{C \lambda(t)^{2}}{g(t) t \log^{b}(t)}.$$
Finally, the transference identity of \cite{kstym} gives
\begin{equation}\begin{split}||\partial_{1}v_{f}(t,R\lambda(t))||_{L^{2}(R dR)} &\leq \frac{C |\lambda'(t)|}{\lambda(t)} ||\mathcal{F}^{-1}\left(\begin{bmatrix}y_{f,0}(t)\\
y_{f,1}(t,\frac{\cdot}{\lambda(t)^{2}})\end{bmatrix}\right)||_{L^{2}(dR)} + C ||\mathcal{F}^{-1}\left(\begin{bmatrix} y_{f,0}'(t)\\
\partial_{1}y_{f,1}(t,\frac{\cdot}{\lambda(t)^{2}})\end{bmatrix}\right)||_{L^{2}(dR)}\\\
& + \frac{C |\lambda'(t)|}{\lambda(t)} ||\mathcal{F}^{-1}\left(\mathcal{K}\left(\begin{bmatrix} y_{f,0}(t)\\
y_{f,1}(t,\frac{\cdot}{\lambda(t)^{2}})\end{bmatrix}\right)\right)||_{L^{2}(dR)}.\end{split}\end{equation}
Therefore, we get
$$||\partial_{1}v_{f}(t,r)||_{L^{2}(r dr)} \leq \frac{C \lambda(t)^{3}}{t^{3} \log^{\epsilon}(t)}$$
$$||v_{f}(t,r)||_{\dot{H}^{1}_{e}} = ||v_{f}(t,\cdot\lambda(t))||_{\dot{H}^{1}_{e}} \leq C(||L(v_{f}(t,R\lambda(t)))||_{L^{2}(R dR)} + ||v_{f}(t,R\lambda(t))||_{L^{2}(R dR)} \leq \frac{C \lambda(t)^{2}}{t^{2} \log^{\epsilon}(t)}.$$
Next, we use the pointwise estimates recorded in Corollary \ref{vkinductioncorollary2} and \eqref{wcsymb}, to get
$$||v_{c}(t,\cdot) + w_{c}(t,\cdot)||_{L^{\infty}} \leq \frac{C}{\log^{b}(t)}.$$
For $v_{f}$, we have
$$||v_{f}(t,\cdot)||_{L^{\infty}}\leq ||v_{f}(t,\cdot \lambda(t))||_{\dot{H}^{1}_{e}} \leq \frac{C \lambda(t)^{2}}{t^{2} \log^{\epsilon}(t)}$$
We also need to verify that 
$$||\partial_{t}v_{1}(t,r)||_{L^{2}(r dr)} + ||v_{1}(t,\cdot)||_{\dot{H}^{1}_{e}} < \infty.$$
This can be done by noting that, for example, \eqref{v1fourierrep} implies
$$||\partial_{t}v_{1}||^{2}_{L^{2}(r dr)} + ||\left(\partial_{r}+\frac{2}{r}\right)v_{1}||^{2}_{L^{2}(r dr)} = ||\widehat{v_{1,1}}(\xi)||^{2}_{L^{2}(\xi d\xi)}.$$
Then, we use the same observation as in \eqref{energycross}, and the fact that $b>\frac{2}{3}$, which shows that $ \widehat{v_{1,1}}(\xi) \in L^{2}(\xi d\xi)$, to conclude
$$||\partial_{t}v_{1}(t,r)||_{L^{2}(r dr)} + ||v_{1}(t,\cdot)||_{\dot{H}^{1}_{e}} < \infty.$$
Finally, we can verify that our solution has finite energy, by noting that
$$E_{YM}(u,\partial_{t}u) \leq C\left(||\partial_{t}u(t,r)||^{2}_{L^{2}(r dr)} + ||\partial_{r}u||^{2}_{L^{2}(r dr)} + \int_{0}^{\infty} \frac{r dr}{r^{2}} \left(1-Q_{\frac{1}{\lambda(t)}}(r)^{2}\right)^{2} + \int_{0}^{\infty} \frac{r dr}{r^{2}} \left(v_{c}+w_{c}+v_{f}\right)^{2}\right)$$
where we used the fact that
$$||v_{c}(t,\cdot)+w_{c}(t,\cdot)+v_{f}(t,\cdot)||_{L^{\infty}} \rightarrow 0,\quad \text{ as }t \rightarrow \infty.$$
Also, we have
$$||\partial_{t}\left(v_{c}-v_{1}+v_{f}+w_{c}\right)||_{L^{2}(r dr)} + ||v_{c}-v_{1}+v_{f}+w_{c}||_{\dot{H}^{1}_{e}} \leq \frac{C}{\log^{2b-1}(t)}$$
which finishes the verification of the energy-related statements in theorem \ref{mainthm}.
\appendix
\section{Improved estimates on $\partial_{t}w_{j}$}\label{dtwjfinalestappendix}
This appendix contains a proof of improved estimates on $\partial_{t}w_{j}(t,r)$, which are obtained just after the preliminary estimate on $e_{0}'''(t)$, in \eqref{e0pppprelim}. We start with $\partial_{t}w_{2}$. We remind the reader of the definition of $WRHS_{2}$ (which is \eqref{wrhs2firstdef}), and get
\begin{equation}\begin{split}|\partial_{t}WRHS_{2}(t,r)| &\leq C \frac{|\chi_{\geq 1}(\frac{r}{g(t)})|}{\log^{b}(t)} \left(\frac{1}{\log^{b}(t)} + \frac{1}{\log(t)}\right) \frac{r^{2} \lambda(t)^{2}}{t^{3}(r^{2}+\lambda(t)^{2})^{2}}\\
&+ C \chi_{\geq 1}(\frac{r}{g(t)}) \begin{cases} \frac{r^{2} \lambda(t)^{2}}{t^{3} \log^{b}(t) (r^{2}+\lambda(t)^{2})^{2}} + \frac{r^{2} \lambda(t)^{2} |e_{0}'''(t)|}{(r^{2}+\lambda(t)^{2})^{2}}, \quad r \leq \frac{t}{2}\\
\frac{\lambda(t)^{2}}{r^{2} t^{2}} \left(\frac{1}{\sqrt{\langle t-r \rangle} \sqrt{r} \log^{b}(\langle t-r \rangle)} + \frac{1}{t \log^{b}(t)}+ t^{2} |e_{0}'''(t)|\right), \quad r > \frac{t}{2}\end{cases}\end{split}\end{equation} 

\begin{equation}\begin{split} &|\partial_{tr} WRHS_{2}(t,r)| \\
&\leq C \frac{\mathbbm{1}_{\{r \geq \frac{g(t)}{2}\}}}{\log^{b}(t)} \left(\frac{1}{\log^{b}(t)} + \frac{1}{\log(t)}\right) \frac{r \lambda(t)^{2}}{t^{3}(r^{2}+\lambda(t)^{2})^{2}}\\
&+ C \mathbbm{1}_{\{r \geq \frac{g(t)}{2}\}} \begin{cases} \frac{r \lambda(t)^{2}}{(r^{2}+\lambda(t)^{2})^{2}} \left(\frac{1}{t^{3} \log^{b}(t)} + |e_{0}'''(t)|\right), \quad r \leq \frac{t}{2}\\
\frac{\lambda(t)^{2}}{r^{3} t^{2}} \left(\frac{1}{\sqrt{\langle t-r \rangle} \sqrt{r} \log^{b}(\langle t-r \rangle)} + \frac{1}{t \log^{b}(t)} + t^{2}|e_{0}'''(t)|\right) \\
+ \frac{\lambda(t)^{2}}{r^{4}} \left(\frac{1}{\sqrt{r} \log^{b}(\langle t-r \rangle) \langle t-r \rangle^{3/2}} + \frac{1}{t \langle t-r \rangle \log^{b}(\langle t-r\rangle) \log^{b}(t)}\right), \quad t > r > \frac{t}{2}\end{cases}\end{split}\end{equation}

\begin{equation}\label{dtrrwrhs2forfinalest}\begin{split} |\partial_{trr}WRHS_{2}(t,r)| &\leq C \frac{\mathbbm{1}_{\{r \geq \frac{g(t)}{2}\}}}{\log^{b}(t)} \left(\frac{1}{\log^{b}(t)} + \frac{1}{\log(t)}\right) \frac{\lambda(t)^{2}}{t^{3} (r^{2}+\lambda(t)^{2})^{2}}\\
&+ C \mathbbm{1}_{\{r \geq \frac{g(t)}{2}\}} \begin{cases} \frac{\lambda(t)^{2}}{(r^{2}+\lambda(t)^{2})^{2}} \left(\frac{1}{t^{3} \log^{b}(t)} + |e_{0}'''(t)|\right), \quad r \leq \frac{t}{2}\\
\frac{\lambda(t)^{2}}{r^{4} \sqrt{t} \log^{b}(\langle t-r \rangle) \langle t-r \rangle^{5/2}} + \frac{\lambda(t)^{2} |e_{0}'''(t)|}{r^{4}}, \quad t > r > \frac{t}{2}\end{cases}.\end{split}\end{equation}

Now that we have the preliminary estimate on $e_{0}'''$, namely \eqref{e0pppprelim}, we can justify the analog of step 4 of the proof of Lemma \ref{v2lemma}, for $\partial_{t}w_{2}$, and carry out the same procedure, to get

\begin{equation} |\partial_{t}w_{2}(t,r)| \leq \begin{cases} \frac{C r^{2} \lambda(t)^{2} \log(2+\frac{r}{g(t)}) \log(t)}{(g(t)^{2}+r^{2})} \left(\frac{1}{t^{3} \log^{b}(t)} + \frac{\sup_{x \geq t} (|e_{0}'''(x)| x^{3/2})}{t^{3/2}}\right), \quad r \leq \frac{t}{2}\\
C\left(\frac{\lambda(t)^{2}}{t^{5/2} \sqrt{\langle t-r \rangle} \log^{b}(\langle t-r \rangle)} + \frac{\sup_{x \geq t}(x^{3/2} |e_{0}'''(x)|) \lambda(t)^{2}}{t^{3/2}}\right) \log^{2}(t), \quad t > r > \frac{t}{2}\end{cases}\end{equation}

\begin{equation}|\partial_{tr}w_{2}(t,r)| \leq \begin{cases} \frac{C r \lambda(t)^{2} \log(t)}{g(t)^{2}} \left(\frac{1}{t^{3} \log^{b}(t)} + \frac{\sup_{x \geq t}(x^{3/2} |e_{0}'''(x)|)}{t^{3/2}}\right), \quad r \leq g(t)\\
\frac{C \lambda(t)^{2}}{\log^{b}(\langle t-r \rangle) t^{5/2} \langle t-r \rangle^{3/2}} + \frac{C \lambda(t)^{2} \sup_{x \geq t}(x^{3/2} |e_{0}'''(x)|)}{t^{5/2}} \\
+ \frac{C \lambda(t)^{2} \log(t)}{g(t)} \left(\frac{1}{t^{3} \log^{b}(t)} + \frac{\sup_{x \geq t}(x^{3/2} |e_{0}'''(x)|)}{t^{3/2}}\right), \quad g(t) < r < t \end{cases}\end{equation}

\begin{equation} |\partial_{trr}w_{2}(t,r)| \leq \frac{C \lambda(t)^{2} \log(t)}{g(t)^{2}} \left(\frac{1}{t^{3} \log^{b}(t)} + \frac{\sup_{x \geq t}(x^{3/2} |e_{0}'''(x)|)}{t^{3/2}}\right) + \frac{C \lambda(t)^{2}}{t^{5/2} \log^{b}(\langle t-r \rangle) \langle t-r \rangle^{5/2}}.\end{equation}

Now that we have the above estimates, we proceed to estimate $\partial_{t}w_{j}$ in the region $r \leq t$. We start with
$$|\partial_{t}WRHS_{3}(t,r)| \leq \begin{cases}\frac{C r^{2} \lambda(t)^{2} \log(2+\frac{r}{g(t)}) \log(t)}{t^{2} \log^{b}(t) (g(t)^{2}+r^{2})} \left(\frac{1}{t^{3} \log^{b}(t)} + \frac{\sup_{x \geq t}(x^{3/2}|e_{0}'''(x)|)}{t^{3/2}}\right), \quad r \leq \frac{t}{2}\\
\frac{C \lambda(t)^{2} \log^{2}(t)}{r^{3/2} t^{1/2} \log^{b}(t)} \left(\frac{1}{t^{5/2} \sqrt{\langle t-r \rangle} \log^{b}(\langle t-r \rangle)} + \frac{\sup_{x \geq t}(x^{3/2} |e_{0}'''(x)|)}{t^{3/2}}\right), \quad t > r > \frac{t}{2}\end{cases}$$

$$|\partial_{tr}WRHS_{3}(t,r)| \leq \begin{cases} \frac{C r \lambda(t)^{2} \log(t)}{t^{2} \log^{b}(t) g(t)^{2}} \left(\frac{1}{t^{3} \log^{b}(t)} + \frac{\sup_{x \geq t}(x^{3/2} |e_{0}'''(x)|)}{t^{3/2}}\right), \quad r \leq g(t)\\
\frac{C \lambda(t)^{2} \log(t)}{t^{2} \log^{b}(t) g(t)} \left(\frac{\sup_{x \geq t} (x^{3/2} |e_{0}'''(x)|)}{t^{3/2}} + \frac{1}{t^{3} \log^{b}(t)}\right), \quad \frac{t}{2} > r > g(t)\\
\frac{C \lambda(t)^{2}}{r^{2} \log^{b}(t)} \left(\frac{\log^{2}(t)}{\log^{b}(\langle t-r \rangle) t^{5/2} \langle t-r \rangle^{3/2}} + \frac{\log(t)}{g(t)} \left(\frac{\sup_{x \geq t}(x^{3/2} |e_{0}'''(x)|)}{t^{3/2}}+\frac{1}{t^{3}\log^{b}(t)}\right)\right)\\
+\frac{C \lambda(t)^{2} \log(t)}{r^{5/2} t^{2} \log^{b}(t) g(t) \log^{b}(\langle t-r \rangle) \sqrt{\langle t-r \rangle}}, \quad t > r > \frac{t}{2}\end{cases}$$
\begin{equation}\nonumber\begin{split}|\partial_{trr}WRHS_{3}(t,r)| &\leq \frac{C \lambda(t)^{2} \log(t)}{t^{2} \log^{b}(t) g(t)^{2}} \left(\frac{1}{t^{3} \log^{b}(t)} + \frac{\sup_{x \geq t}(x^{3/2} |e_{0}
'''(x)|)}{t^{3/2}}\right) \\
&+ \frac{C \lambda(t)^{2}}{t^{2} \log^{b}(t) t^{5/2} \log^{b}(\langle t-r \rangle) \langle t-r \rangle^{5/2}}, \quad r \leq g(t)\end{split}\end{equation}

$$|\partial_{trr}WRHS_{3}(t,r)| \leq \frac{C \lambda(t)^{2} \log(t)}{t^{2} \log^{b}(t) g(t)^{2}} \left(\frac{\sup_{x \geq t} (x^{3/2} |e_{0}'''(x)|)}{t^{3/2}} + \frac{1}{t^{3} \log^{b}(t)}\right), \quad \frac{t}{2} > r > g(t)$$

\begin{equation}\nonumber \begin{split} |\partial_{trr}WRHS_{3}(t,r)| &\leq \frac{C}{r^{2} \log^{b}(t)} \frac{\lambda(t)^{2} \log(t)}{g(t)^{2}} \left(\frac{1}{t^{3} \log^{b}(t)} + \frac{\sup_{x \geq t}(x^{3/2} |e_{0}'''(x)|)}{t^{3/2}}\right)\\
&+ \frac{C \lambda(t)^{2} \log(t)}{r^{5/2} t^{2} \log^{b}(t) g(t)^{2} \log^{b}(\langle t-r \rangle)\sqrt{\langle t-r \rangle}} + \frac{C \lambda(t)^{2} \log^{2}(t)}{r^{5/2} t^{2} \log^{b}(t) g(t) \log^{b}(\langle t-r \rangle) \langle t -r \rangle^{3/2}} \\
&+ \frac{C \lambda(t)^{2} \log^{2}(t)}{r^{5/2} t^{2} \log^{b}(t) \log^{b}(\langle t-r \rangle) \langle t-r \rangle^{5/2}}, \quad t > r > \frac{t}{2}.\end{split}\end{equation}

Then, we use the same observation appearing after \eqref{dtrrwrhs2forfinalest}, except for $\partial_{t}w_{3}$, to get
\begin{equation}\label{dtw3finalest}|\partial_{t}w_{3}(t,r)| \leq \begin{cases} C r^{2} \left(\frac{\lambda(t)^{2} \log(t)}{t^{3} \log^{2b}(t) g(t)^{2}} + \frac{\lambda(t)^{2} \log(t) \sup_{x \geq t}(x^{3/2}|e_{0}'''(x)|)}{\log^{b}(t) g(t)^{2} t^{3/2}}\right), \quad r \leq g(t)\\
\frac{C \lambda(t)^{2} \log^{2}(t)}{\log^{b}(t) t^{5/2} \sqrt{\langle t-r \rangle}\log^{b}(\langle t-r \rangle)} + \frac{C \lambda(t)^{2} \log^{2}(t) \sup_{x \geq t} (x^{3/2} |e_{0}'''(x)|)}{t^{3/2} \log^{b}(t)}, \quad t > r > g(t)\end{cases}\end{equation}

\begin{equation}\nonumber\begin{split}&|\partial_{tr} w_{3}(t,r)| \\
&\leq \begin{cases} C r \left(\frac{\lambda(t)^{2} \log(t)}{t^{3} \log^{2b}(t) g(t)^{2}} + \frac{\lambda(t)^{2} \log(t) \sup_{x \geq t}(x^{3/2}|e_{0}'''(x)|)}{\log^{b}(t) g(t)^{2} t^{3/2}}\right), \quad r \leq g(t)\\ 
C\left(\frac{\lambda(t)^{2} \log(t)}{g(t)\log^{b}(t)} \frac{\sup_{x \geq t} (x^{3/2} |e_{0}'''(x)|)}{t^{3/2}} + \frac{\lambda(t)^{2} \log^{2}(t)}{\log^{b}(t) \log^{b}(\langle t-r \rangle) t^{5/2} \langle t-r\rangle^{3/2}} + \frac{\lambda(t)^{2} \log(t)}{t^{5/2} \sqrt{\langle t-r \rangle} \log^{b}(\langle t-r \rangle) \log^{b}(t) g(t)}\right), \quad t > r > g(t)\end{cases}\end{split}\end{equation}
and
\begin{equation}\nonumber\begin{split}&|\partial_{trr}w_{3}(t,r) \\
&\leq \begin{cases} C  \left(\frac{\lambda(t)^{2} \log(t)}{t^{3} \log^{2b}(t) g(t)^{2}} + \frac{\lambda(t)^{2} \log(t) \sup_{x \geq t}(x^{3/2}|e_{0}'''(x)|)}{\log^{b}(t) g(t)^{2} t^{3/2}}\right), \quad r \leq g(t)\\
\frac{C \lambda(t)^{2} \log(t)}{\log^{b}(t) g(t)^{2}} \frac{\sup_{x \geq t}(x^{3/2} |e_{0}'''(x)|)}{t^{3/2}} + \frac{\lambda(t)^{2} \log(t)}{t^{5/2} \log^{b}(t) \log^{b}(\langle t-r \rangle) \sqrt{\langle t-r \rangle}}\left(\frac{1}{g(t)^{2}} + \frac{\log(t)}{g(t)\langle t-r \rangle} + \frac{\log(t)}{\langle t-r \rangle^{2}}\right), \quad t > r > g(t)\end{cases}.\end{split}\end{equation}

Using an argument similar to that used to establish \eqref{wkassump}, etc., we get, after a lengthy computation, that there exists $C_{4} > \max\{1,C_{2}^{p}\}$, such that, for all $j \geq 4$,
\begin{equation}|\partial_{t}w_{j}(t,r)| \leq \begin{cases} \frac{C_{4}^{j} r^{2} \lambda(t)^{2} \log^{2}(t)}{g(t)^{2}} \left(\frac{1}{t^{3} \log^{b(j-1)}(t)} + \frac{\sup_{x \geq t}(x^{3/2} |e_{0}'''(x)|)}{t^{3/2} \log^{b(j-2)}(t)}\right), \quad r \leq g(t)\\
C_{4}^{j} \lambda(t)^{2} \log^{2}(t) \left(\frac{1}{t^{5/2} \sqrt{\langle t-r \rangle} \log^{b}(\langle t-r \rangle) \log^{b(j-2)}(t)} + \frac{\sup_{x \geq t}(x^{3/2} |e_{0}'''(x)|)}{t^{3/2} \log^{b(j-2)}(t)}\right), \quad t > r > g(t)\end{cases}\end{equation}
\begin{equation} |\partial_{tr}w_{j}(t,r)| \leq \begin{cases} \frac{C_{4}^{j} r \lambda(t)^{2} \log^{2}(t)}{g(t)^{2}} \left(\frac{1}{t^{3} \log^{b(j-1)}(t)} + \frac{\sup_{x \geq t}(x^{3/2} |e_{0}'''(x)|)}{t^{3/2} \log^{b(j-2)}(t)}\right), \quad r \leq g(t)\\
\frac{C_{4}^{j} \lambda(t)^{2} \log^{2}(t)}{g(t)} \left(\frac{1}{t^{5/2} \sqrt{\langle t-r \rangle} \log^{b}(\langle t-r \rangle) \log^{b(j-2)}(t)} + \frac{\sup_{x \geq t}(x^{3/2} |e_{0}'''(x)|)}{t^{3/2} \log^{b(j-2)}(t)}\right)\\
+\frac{C_{4}^{j} \lambda(t)^{2} \log^{2}(t)}{\log^{b(j-2)}(t) \log^{b}(\langle t-r \rangle) t^{5/2} \langle t-r \rangle^{3/2}}, \quad t > r > g(t)\end{cases}\end{equation}

\begin{equation} |\partial_{trr}w_{j}(t,r)| \leq \begin{cases} \frac{C_{4}^{j}  \lambda(t)^{2} \log^{2}(t)}{g(t)^{2}} \left(\frac{1}{t^{3} \log^{b(j-1)}(t)} + \frac{\sup_{x \geq t}(x^{3/2} |e_{0}'''(x)|)}{t^{3/2} \log^{b(j-2)}(t)}\right), \quad r \leq g(t)\\
\frac{C_{4}^{j} \lambda(t)^{2} \log^{2}(t)}{g(t)^{2}}\left(\frac{1}{t^{5/2} \sqrt{\langle t-r \rangle} \log^{b}(\langle t-r \rangle) \log^{b(j-2)}(t)} + \frac{\sup_{x \geq t}(x^{3/2} |e_{0}'''(x)|)}{t^{3/2} \log^{b(j-2)}(t)}\right)\\
+\frac{C_{4}^{j} \lambda(t)^{2} \log^{2}(t)}{t^{5/2} \log^{b(j-2)}(t) \log^{b}(\langle t-r \rangle) \sqrt{\langle t-r \rangle}} \left(\frac{1}{g(t) \langle t-r\rangle}+\frac{1}{\langle t-r\rangle^{2}}\right), \quad t > r > g(t)\end{cases}.\end{equation}
\section{Improved estimates on $\partial_{t}^{2}w_{j}$}\label{dttwjfinalestimateappendix}
This appendix contains the proof of improved estimates on $\partial_{t}^{2}w_{j}(t,r)$, which are obtained after obtaining the preliminary estimate on $e_{0}''''(t)$, namely \eqref{e0ppppprelimest}. We start with
$$|\partial_{t}^{2} WRHS_{2}(t,r)| \leq \frac{C \mathbbm{1}_{\{r \geq \frac{g(t)}{2}\}} \lambda(t)^{2}}{(g^{2}+r^{2})} \left(|e_{0}''''(t)| + \begin{cases} \frac{1}{t^{4} \log^{b}(t)}, \quad r \leq \frac{t}{2}\\
\frac{1}{r^{5/2} \log^{b}(\langle t-r \rangle) \langle t-r \rangle^{3/2}}, \quad t > r > \frac{t}{2}\end{cases}\right)$$
$$|\partial_{ttr} WRHS_{2}(t,r)| \leq C \begin{cases} \frac{\mathbbm{1}_{\{r \geq \frac{g(t)}{2}\}} \lambda(t)^{2}}{(r^{2}+g(t)^{2})r} \left(\frac{1}{t^{4} \log^{b}(t)} + |e_{0}''''(t)|\right), \quad r \leq \frac{t}{2}\\
\frac{\lambda(t)^{2} |e_{0}''''(t)|}{r^{3}} + \frac{\lambda(t)^{2}}{r^{9/2} \langle t-r \rangle^{5/2} \log^{b}(\langle t-r \rangle)},\quad t > r > \frac{t}{2}\end{cases}$$
$$|\partial_{ttrr}WRHS_{2}(t,r)| \leq C \begin{cases} \frac{\mathbbm{1}_{\{r \geq \frac{g(t)}{2}\}} \lambda(t)^{2}}{r^{2}(r^{2}+g(t)^{2})} \left(\frac{1}{t^{4} \log^{b}(t)} + |e_{0}''''(t)|\right), \quad r \leq \frac{t}{2}\\
\frac{\lambda(t)^{2} |e_{0}''''(t)|}{r^{4}} + \frac{\lambda(t)^{2}}{r^{9/2} \langle t-r \rangle^{7/2} \log^{b}(\langle t-r \rangle)}, \quad t > r > \frac{t}{2}\end{cases}$$
which gives
$$|\partial_{t}^{2} w_{2}(t,r)| \leq \begin{cases} \frac{C r^{2} \lambda(t)^{2} \log(2+\frac{r}{g(t)}) \log(t)}{(g(t)^{2}+r^{2})} \left(\frac{1}{t^{4} \log^{b}(t)} + \frac{\sup_{x \geq t}(x^{3/2} |e_{0}''''(x)|)}{t^{3/2}}\right), \quad r \leq \frac{t}{2}\\
C\lambda(t)^{2}\left(\frac{1}{t^{4} \log^{b}(t)} + \frac{\sup_{x \geq t}(x^{3/2}|e_{0}''''(x)|)}{t^{3/2}}\right) \log^{2}(t) + \frac{C \lambda(t)^{2}}{t^{5/2} \langle t-r \rangle^{3/2} \log^{b}(\langle t-r \rangle)}, \quad t > r > \frac{t}{2}\end{cases}$$
\begin{equation}\begin{split}&|\partial_{ttr}w_{2}(t,r)| \\
&\leq \begin{cases} \frac{C r \lambda(t)^{2} \log(t)}{g(t)^{2}} \left(\frac{1}{t^{4} \log^{b}(t)} + \frac{\sup_{x \geq t}(x^{3/2} |e_{0}''''(x)|)}{t^{3/2}}\right), \quad r \leq g(t)\\
\frac{C \lambda(t)^{2} \log(t)}{g(t)} \left(\frac{1}{t^{4} \log^{b}(t)} + |e_{0}''''(t)|\right) + C \lambda(t)^{2} \left(\frac{1}{t^{5/2} \langle t-r \rangle^{5/2} \log^{b}(\langle t-r \rangle)} + \frac{\sup_{x \geq t}(x^{3/2} |e_{0}''''(x)|)}{t^{5/2}}\right), \quad t > r > g(t)\end{cases}\end{split}\end{equation}
$$|\partial_{ttrr}w_{2}(t,r)| \leq \frac{C \lambda(t)^{2} \log(t)}{g(t)^{2}}\left(\frac{1}{t^{4} \log^{b}(t)} + \frac{\sup_{x \geq t}(x^{3/2} |e_{0}''''(x)|)}{t^{3/2}}\right) + \frac{C \lambda(t)^{2}}{t^{5/2} \log^{b}(\langle t-r \rangle) \langle t-r \rangle^{7/2}}, \text{  } r <t$$
$$|\partial_{t}^{2} WRHS_{3}(t,r)| \leq C \begin{cases} \frac{\lambda(t)^{2} r^{2} \log(2+\frac{r}{g(t)}) \log(t)}{t^{2} \log^{b}(t) (g(t)^{2}+r^{2})} \left(\frac{1}{t^{4} \log^{b}(t)} + \frac{\sup_{x \geq t} (x^{3/2} |e_{0}''''(x)|)}{t^{3/2}}\right), \quad r \leq \frac{t}{2}\\
\frac{\lambda(t)^{2} \log^{2}(t) \sup_{x \geq t}(x^{3/2} |e_{0}''''(x)|)}{t^{2} \log^{b}(t) t^{3/2}} + \frac{\lambda(t)^{2} \log^{2}(t)}{t^{9/2} \log^{b}(t) \log^{b}(\langle t-r \rangle) \langle t-r \rangle^{3/2}}, \quad t > r > \frac{t}{2}\end{cases}$$
$$|\partial_{ttr} WRHS_{3}(t,r)| \leq \begin{cases} \frac{C r \lambda(t)^{2} \log(t)}{t^{6} \log^{2b}(t) g(t)^{2}} + \frac{C \lambda(t)^{2} r \log(t) \sup_{x \geq t}(x^{3/2} |e_{0}''''(x)|)}{t^{2} \log^{b}(t) g(t)^{2} t^{3/2}}, \quad r \leq g(t)\\
\frac{C \lambda(t)^{2} \log(t)}{t^{2} \log^{b}(t) g(t)} \left(\frac{1}{t^{4} \log^{b}(t)} + \frac{\sup_{x \geq t}(x^{3/2} |e_{0}''''(x)|)}{t^{3/2}}\right), \quad g(t) < r < \frac{t}{2}\\
\frac{C \lambda(t)^{2} \log(t)}{t^{9/2} \langle t-r \rangle^{3/2} \log^{b}(t) \log^{b}(\langle t-r \rangle) g(t)} + \frac{C \lambda(t)^{2}}{t^{5} \langle t-r \rangle^{2} \log^{2b}(\langle t-r \rangle)}\\
+\frac{C \lambda(t)^{2} \log(t)}{g(t)r^{2}\log^{b}(t)} \left(\frac{1}{t^{4} \log^{b}(t)} + \frac{\sup_{x \geq t}(x^{3/2} |e_{0}''''(x)|)}{t^{3/2}} \right) \\
+ \frac{C \lambda(t)^{2}}{t^{9/2} \langle t-r \rangle^{5/2} \log^{b}(t) \log^{b}(\langle t-r \rangle)}, \quad \frac{t}{2} < r < t \end{cases}$$
$$|\partial_{ttrr}WRHS_{3}(t,r)| \leq C\begin{cases} \frac{\lambda(t)^{2} \log(t)}{t^{6} \log^{2b}(t) g(t)^{2}} + \frac{\lambda(t)^{2} \log(t) \sup_{x \geq t}(x^{3/2} |e_{0}''''(x)|)}{t^{2} \log^{b}(t) g(t)^{2} t^{3/2}}, \quad r \leq \frac{t}{2}\\
\frac{\lambda(t)^{2} \log(t)}{t^{9/2} \langle t-r \rangle^{5/2} \log^{b}(t) \log^{b}(\langle t-r \rangle) g(t)} + \frac{C \lambda(t)^{2} \log(t) \sup_{x \geq t}(x^{3/2} |e_{0}''''(x)|)}{t^{2} \log^{b}(t) g(t)^{2} t^{3/2}}\\
+\frac{C \lambda(t)^{2}}{t^{9/2} \langle t-r \rangle^{7/2} \log^{b}(t) \log^{b}(\langle t-r \rangle)} + \frac{C \lambda(t)^{2} \log(t)}{t^{9/2} g(t)^{2} \log^{b}(t) \langle t-r \rangle^{3/2} \log^{b}(\langle t-r \rangle)}\\
+ \frac{C \lambda(t)^{2} \log^{2}(t)}{t^{9/2} \langle t-r \rangle^{7/2} \log^{b}(\langle t-r \rangle) \log^{b}(t)}, \quad t > r > \frac{t}{2}\end{cases}$$
and these give
$$|\partial_{t}^{2}w_{3}(t,r)| \leq \begin{cases} C r^{2} \left(\frac{\lambda(t)^{2} \log(t)}{t^{4} g(t)^{2} \log^{2b}(t)} + \frac{\lambda(t)^{2} \log(t) \sup_{x \geq t}(x^{3/2} |e_{0}''''(x)|)}{t^{3/2} \log^{b}(t) g(t)^{2}}\right), \quad r \leq g(t)\\
\frac{C \lambda(t)^{2} \log^{2}(t) \sup_{x \geq t}(x^{3/2} |e_{0}''''(x)|)}{\log^{b}(t) t^{3/2}} + \frac{C \lambda(t)^{2} \log^{2}(t)}{t^{5/2} \langle t-r \rangle^{3/2} \log^{b}(t) \log^{b}(\langle t-r \rangle)}, \quad t > r > g(t)\end{cases}$$
$$|\partial_{ttr} w_{3}(t,r)| \leq \begin{cases} C r \left(\frac{\lambda(t)^{2} \log(t)}{t^{4} g(t)^{2} \log^{2b}(t)} + \frac{\lambda(t)^{2} \log(t) \sup_{x \geq t}(x^{3/2} |e_{0}''''(x)|)}{t^{3/2} \log^{b}(t) g(t)^{2}}\right), \quad r \leq g(t)\\
\frac{C \lambda(t)^{2} \log(t) \sup_{x \geq t}(x^{3/2} |e_{0}''''(x)|)}{t^{3/2} g(t) \log^{b}(t)} + \frac{C \lambda(t)^{2} \log(t)}{\log^{b}(t) g(t) t^{5/2} \langle t-r \rangle^{3/2} \log^{b}(\langle t-r \rangle)}\\
+ \frac{C \lambda(t)^{2}}{\log^{b}(t) \log^{b}(\langle t-r \rangle)t^{5/2} \langle t-r \rangle^{5/2}}, \quad t > r > g(t)\end{cases}$$
$$|\partial_{ttrr}w_{3}(t,r)| \leq \begin{cases}C \left(\frac{\lambda(t)^{2} \log(t)}{t^{4} g(t)^{2} \log^{2b}(t)} + \frac{\lambda(t)^{2} \log(t) \sup_{x \geq t}(x^{3/2} |e_{0}''''(x)|)}{t^{3/2} \log^{b}(t) g(t)^{2}}\right), \quad r \leq g(t)\\
\frac{C \lambda(t)^{2} \log(t) \sup_{x \geq t}(x^{3/2} |e_{0}''''(x)|)}{\log^{b}(g) g(t)^{2} t^{3/2}} + \frac{C\lambda(t)^{2} \log(t)}{t^{5/2} \langle t-r \rangle^{5/2} \log^{b}(t) \log^{b}(\langle t-r \rangle) g(t)}\\
+ \frac{C \lambda(t)^{2} \log(t)}{t^{5/2} g(t)^{2} \log^{b}(t) \langle t-r \rangle^{3/2} \log^{b}(\langle t-r \rangle)} + \frac{C \lambda(t)^{2} \log^{2}(t)}{t^{5/2} \langle t-r \rangle^{7/2} \log^{b}(\langle t-r \rangle) \log^{b}(t)}, \quad t > r > g(t)\end{cases}$$
Then, using an induction argument similar to that used to estimate $\partial_{t}w_{j}$, we get that there exists $C_{5} > C_{4}+C_{2}^{p}$ such that, for all $j \geq 4$,
$$|\partial_{t}^{2} w_{j}(t,r)| \leq \begin{cases} C_{5}^{j} r^{2} \left(\frac{\lambda(t)^{2} \log^{2}(t)}{t^{4} g(t)^{2} \log^{b(j-1)}(t)} + \frac{\lambda(t)^{2} \log^{2}(t) \sup_{x \geq t}(x^{3/2} |e_{0}''''(x)|)}{t^{3/2} \log^{b(j-2)}(t) g(t)^{2}}\right), \quad r \leq g(t)\\
\frac{C_{5}^{j} \lambda(t)^{2} \log^{2}(t)}{\log^{b(j-2)}(t)} \frac{\sup_{x \geq t}(x^{3/2} |e_{0}''''(x)|)}{t^{3/2}} + \frac{C_{5}^{j} \lambda(t)^{2} \log^{2}(t)}{t^{5/2} \langle t-r \rangle^{3/2} \log^{b(j-2)}(t) \log^{b}(\langle t-r \rangle)}, \quad t > r > g(t)\end{cases}$$
$$|\partial_{ttr}w_{j}(t,r)| \leq \begin{cases} C_{5}^{j} r \left(\frac{\lambda(t)^{2} \log^{2}(t)}{t^{4} g(t)^{2} \log^{b(j-1)}(t)} + \frac{\lambda(t)^{2} \log^{2}(t) \sup_{x \geq t}(x^{3/2} |e_{0}''''(x)|)}{t^{3/2} \log^{b(j-2)}(t) g(t)^{2}}\right), \quad r \leq g(t)\\
\frac{C_{5}^{j} \lambda(t)^{2} \log^{2}(t) \sup_{x \geq t}(x^{3/2} |e_{0}''''(x)|)}{t^{3/2} g(t) \log^{b(j-2)}(t)} + \frac{C_{5}^{j} \lambda(t)^{2} \log^{2}(t)}{\log^{b(j-2)}(t)  t^{5/2} \langle t-r \rangle^{3/2} \log^{b}(\langle t-r \rangle)} \left(\frac{1}{g(t)} + \frac{1}{\langle t-r \rangle}\right) , \quad t > r > g(t)\end{cases}$$
\begin{equation}\nonumber\begin{split}&|\partial_{ttrr}w_{j}(t,r)| \\
&\leq \begin{cases} C_{5}^{j} \left(\frac{\lambda(t)^{2} \log^{2}(t)}{t^{4} g(t)^{2} \log^{b(j-1)}(t)} + \frac{\lambda(t)^{2} \log^{2}(t) \sup_{x \geq t}(x^{3/2} |e_{0}''''(x)|)}{t^{3/2} g(t)^{2} \log^{b(j-2)}(t)}\right), \quad r \leq g(t)\\
\frac{C_{5}^{j} \lambda(t)^{2} \log^{2}(t) \sup_{x \geq t}(x^{3/2} |e_{0}''''(x)|)}{\log^{b)(j-2)}(t) g(t)^{2} t^{3/2}} + \frac{C_{5}^{j} \lambda(t)^{2} \log^{2}(t)}{t^{5/2} \langle t-r \rangle^{3/2} \log^{b(j-2)}(t) \log^{b}(\langle t-r \rangle)} \left(\frac{1}{\langle t-r \rangle g(t)} + \frac{1}{g(t)^{2}} + \frac{1}{\langle t-r \rangle^{2}}\right), \quad t > r > g(t)\end{cases}\end{split}\end{equation}

Department of Mathematics, University of California, Berkeley\\
\emph{E-mail address:} mpillai@berkeley.edu
\end{document}